\documentclass[11pt,a4paper]{amsart}

\usepackage{amsmath, amsthm, amsfonts, amssymb}
\usepackage{comment}
\usepackage{enumerate}
\usepackage{tikz}
\usepackage{float}
\usepackage[colorlinks=true, linkcolor=blue, citecolor=red, menucolor=black]{hyperref}

\usepackage{array,booktabs}

\numberwithin{equation}{section}
\newtheorem{letterthm}{Theorem}

\newtheorem{lettercor}[letterthm]{Corollary}

\newtheorem{letterconj}[letterthm]{Conjecture}

\newtheorem{theorem}{Theorem}[section]
\newtheorem{lemma}[theorem]{Lemma}
\newtheorem{corollary}[theorem]{Corollary}
\theoremstyle{definition}
\newtheorem{definition}[theorem]{Definition}
\newtheorem{remark}[theorem]{Remark}
\theoremstyle{plain}
\newtheorem{proposition}[theorem]{Proposition}

\newtheorem{conjecture}[theorem]{Conjecture}

\newtheorem*{flowsimportconjecture}{Conjecture}

\newcommand{\R}{\mathbb{R}}
\newcommand{\N}{\mathbb{N}}
\newcommand{\B}{\mathbb{B}}
\newcommand{\C}{\mathbb{C}}
\newcommand{\Z}{\mathbb{Z}}

\newcommand{\cA}{\mathcal{A}}
\newcommand{\cB}{\mathcal{B}}
\newcommand{\cC}{\mathcal{C}}
\newcommand{\cE}{\mathcal{E}}

\newcommand{\cI}{\mathcal{I}}
\newcommand{\cH}{\mathcal{H}}
\newcommand{\cZ}{\mathcal{Z}}
\newcommand{\cU}{\mathcal{U}}
\newcommand{\cD}{\mathcal{D}}
\newcommand{\cO}{\mathcal{O}}
\newcommand{\cL}{\mathcal{L}}

\newcommand{\ri}{\text{\rm i}}
\newcommand{\rd}{\text{\rm d}}

\newcommand{\Ad}{\operatorname{Ad}}
\newcommand{\Inn}{\operatorname{Inn}}
\newcommand{\conv}{\operatorname{conv}}

\newcommand{\Ind}{\operatorname{Ind}}
\newcommand{\Res}{\operatorname{Res}}
\newcommand{\Rep}{\operatorname{Rep}}
\newcommand{\Cor}{\operatorname{Cor}}
\newcommand{\rL}{\mathord{\text{\rm L}}}
\newcommand{\rB}{\mathord{\text{\rm B}}}
\newcommand{\rb}{\mathord{\text{\rm b}}}
\newcommand{\rC}{\mathord{\text{\rm C}}}
\newcommand{\rW}{\mathord{\text{\rm W}}}
\newcommand{\rE}{\mathord{\text{\rm E}}}
\newcommand{\rT}{\mathord{\text{\rm T}}}
\newcommand{\Sp}{\mathord{\text{\rm Sp}}}
\newcommand{\id}{\mathord{\text{\rm id}}}
\newcommand{\op}{\mathord{\text{\rm op}}}
\newcommand{\Tr}{\mathord{\text{\rm Tr}}}
\newcommand{\supp}{\mathord{\text{\rm supp}}}
\newcommand{\ovt}{\mathbin{\overline\otimes}}

\newcommand{\II}{\mathrm{II}}
\newcommand{\III}{\mathrm{III}}
\allowdisplaybreaks
\makeatletter
\@namedef{subjclassname@2020}{%
  \textup{2020} Mathematics Subject Classification}
\def\l@subsection{\@tocline{2}{0pt}{2pc}{5pc}{}}
\makeatother

\begin{document}

\title[The classification of flows and the bicentralizer problem]{The classification of flows on $\II_1$ factors\\ and Connes' bicentralizer problem}

\begin{abstract}
We solve two long-standing open problems in von Neumann algebras.
First, we classify all flows with full Connes spectrum on the hyperfinite
$\II_1$ factor up to cocycle conjugacy.
Every such flow is cocycle conjugate to an irrational rotation flow on a
noncommutative torus. In particular, there is a unique outer flow with full Connes spectrum up to cocycle conjugacy. This settles Takesaki's classification problem for flows on the hyperfinite $\II_1$ factor. To prove this result, we draw on type $\III$ theory. Notably, we develop bicentralizer machinery for trace-preserving actions of locally compact groups. In the amenable case, we relate the bicentralizer conjecture to the Rokhlin property. For abelian groups, we prove an analog of the Connes--St\o rmer transitivity theorem and we generalize the Connes--Takesaki relative commutant theorem. A new resonance phenomenon is revealed which allows us to solve the bicentralizer conjecture for actions of $\R$. We then go back to the type $\III$ world and use this new resonance phenomenon to solve Connes' bicentralizer conjecture for all type $\III_1$ factors.
\end{abstract}

\author{Cyril Houdayer}
\address{\'Ecole normale sup\'erieure \\ D\'epartement de math\'ematiques et applications \\ Universit\'e Paris-Saclay \\ 45 rue d'Ulm \\ 75230 Paris Cedex 05 \\ France}
\email{cyril.houdayer@ens.psl.eu}
\thanks{CH is supported by ERC Advanced Grant NET 101141693}

\author{Amine Marrakchi}
\address{CNRS \\ \'Ecole normale sup\'erieure de Lyon \\ Unit\'e de math\'ematiques pures et appliquées  \\ 46 all\'ee d'Italie \\ 69364 Lyon \\ France}
\email{amine.marrakchi@ens-lyon.fr}

\dedicatory{Dedicated to the memory of Uffe Haagerup}


\maketitle

\begingroup
\let\conjecture\flowsimportconjecture
\let\endconjecture\endflowsimportconjecture

\section{Introduction and statement of the main results}
\label{sec:introduction}

The classification of amenable factors with separable predual is one of
the fundamental achievements of the theory of von Neumann algebras.
Connes proved that amenability is equivalent to hyperfiniteness and
classified the amenable factors of types $\II_1$, $\II_\infty$ and
$\III_\lambda$ for $\lambda \in [0,1)$ \cite{Co76}. In particular, there
is a unique amenable $\II_1$ factor $R$, its infinite amplification is
the unique amenable $\II_\infty$ factor, and there is a unique
amenable $\III_\lambda$ factor for each $0<\lambda<1$. Connes devised a
strategy for the remaining type $\III_1$ case by introducing the
bicentralizer conjecture and showing that a positive solution in the
amenable case would imply uniqueness of the amenable $\III_1$ factor
\cite{Co85}. Haagerup proved the conjecture in the amenable case, thereby
completing the classification \cite{Ha87}.

Once the underlying factors are understood, a natural problem is to
classify their automorphisms. For the hyperfinite $\II_1$ factor $R$,
Connes classified arbitrary automorphisms up to \emph{outer conjugacy}
\cite{Co75}, which means conjugacy after perturbation by an inner automorphism.

The next natural step is to replace discrete time by continuous time
and classify \emph{flows}, that is, continuous one-parameter groups
$\alpha:\R\curvearrowright R$. The appropriate analogue of
outer conjugacy is \emph{cocycle conjugacy}: two flows $\alpha$ and $\beta$
are cocycle conjugate if there exist $\theta\in\operatorname{Aut}(R)$
and a strongly continuous family of unitaries $(u_t)_{t\in\R}$ such
that
\[
 u_{s+t}=u_s\alpha_s(u_t),
 \qquad
 \theta\circ\beta_t\circ\theta^{-1}=\Ad(u_t)\circ\alpha_t
 \quad(s,t\in\R).
\]
The classification problem for flows on the hyperfinite $\II_1$
factor $R$ up to cocycle conjugacy has been promoted by Takesaki
since the late 1970s \cite{KP25}. An important invariant for a flow
$\alpha$ is its Connes spectrum $\Gamma(\alpha)$, which is a closed
subgroup of $\R$. Connes originally introduced this invariant to
classify type $\III$ factors into subtypes $\III_\lambda$ for
$\lambda\in[0,1]$. Kawahigashi classified flows on the hyperfinite $\II_1$ factor whose Connes
spectrum is a proper subgroup of $\R$ up to stable conjugacy
\cite{Ka89}. His work is the flow analog of the classification of
amenable factors of type $\III_\lambda$ for $\lambda\in[0,1)$.
For the remaining case $\Gamma(\alpha)=\R$, he observed that the
difficulty is very similar to that of proving uniqueness of
the amenable type $\III_1$ factor \cite[p.~286]{Ka89}.
A major difference, however, is that a flow with full Connes spectrum
is not necessarily outer, while the modular flow of a type $\III_1$
factor is always outer.

A basic model of a flow with full Connes spectrum on the hyperfinite $\II_1$ factor  is an infinite
tensor product of flows on matrix algebras. Choose a diagonal
self-adjoint matrix $h\in\mathbb M_d(\C)$ such that the flow
$\sigma_t=\Ad(e^{2\pi\ri th})$ is faithful. For example, one can take
$d=3$ and $h=\operatorname{diag}(0,1,\sqrt2)$. On the infinite
tensor product of copies of $\mathbb M_d(\C)$ with their normalized
traces, which is isomorphic to $R$, define
\[
\sigma^\infty_t=\bigotimes_{n\in \N}\sigma_t
\qquad(t\in\R).
\]
This flow has full Connes spectrum, it is outer and it fixes pointwise the Cartan
subalgebra obtained by taking the infinite tensor product of the
diagonal subalgebras. Kawahigashi proved that every flow on $R$
with full Connes spectrum that fixes a Cartan subalgebra pointwise
is cocycle conjugate to $\sigma^\infty$
\cite[Theorem~1.6 and Corollary~1.9]{Ka91}.

Kawahigashi also studied irrational rotation flows on
noncommutative tori. For $\theta\in\R\setminus\mathbb Q$, let $A_\theta$
be the $\rC^*$-algebra generated by two unitaries $U,V$ satisfying
$UV=e^{2\pi\ri\theta}VU$. Its von Neumann completion with respect
to its unique tracial state is isomorphic to $R$. Given nonzero
$p,q\in\R$ with $p/q\notin\mathbb Q$, the formulas
\[
\rho_t(U)=e^{2\pi\ri pt}U,
\qquad
\rho_t(V)=e^{2\pi\ri qt}V
\qquad(t\in\R)
\]
define an ergodic flow with full Connes spectrum on this
completion. These flows are not necessarily outer.  An outer
example is obtained by taking $\theta=\sqrt2$, $p=1$ and
$q=\sqrt3$. Kawahigashi proved that every such irrational rotation
flow that is outer is cocycle conjugate to the infinite tensor
product flow $\sigma^\infty$ \cite[Theorem~16]{Ka90}.

These results led to the conjecture that all outer flows on $R$
with full Connes spectrum are cocycle conjugate.

A fundamental concept in Connes' classification of single automorphisms is a noncommutative analog of the Rokhlin property from ergodic theory. A similar Rokhlin property exists also for flows. It was first introduced by
Kishimoto in the setting of $\rC^*$-algebras
\cite{Ki96}. Szab\'o later proved that every Kirchberg
algebra carries a unique Rokhlin flow up to cocycle conjugacy,
confirming a conjecture of Kishimoto \cite{Sz21}.

Kawamuro introduced the Rokhlin property for flows on the hyperfinite
$\II_1$ factor in \cite{Kw00}. Masuda--Tomatsu extended the definition
to flows on general von Neumann algebras in
\cite[Definition~4.1]{MT16} and proved that
there is a unique Rokhlin flow on $R$ up to cocycle conjugacy,
represented by $\sigma^\infty$
\cite[Corollaries~7.9, 5.16 and 4.13]{MT16}. They conjectured that every outer flow with full
Connes spectrum on the hyperfinite $\II_1$ factor should have the Rokhlin property.

The main result of this paper solves this conjecture and classifies more generally all flows with full Connes spectrum on the hyperfinite $\II_1$ factor. This settles Takesaki's classification problem for flows on the hyperfinite $\II_1$ factor.

To state our classification theorem, we introduce, for a flow $\alpha:\R\curvearrowright M$
on a $\II_1$ factor, the \emph{characteristic invariant}
\[
\begin{split}
\Lambda(\alpha)=\{(p,t)\in\R^2\mid{}&
\text{there is }v\in\cU(M)\text{ such that }\alpha_t=\Ad(v)\\
&\text{and }\alpha_s(v)=e^{2\pi\ri ps}v\text{ for every }s\in\R\}.
\end{split}
\]
This is a subgroup of $\R^2$ and a cocycle conjugacy invariant.
It is a subgroup presentation of the characteristic invariant studied
by Connes, Jones, Jones--Takesaki and Ocneanu
\cite{Co77,Jo79,JT84,Oc85}. See Subsection~\ref{subsec:action-invariant}. It follows from its definition that the subgroup
$\Lambda(\alpha)$ is always contained in the \emph{resonance set}
\[
\mathcal R_\R=\{(p,t)\in\R^2\mid pt\in\Z\}.
\]
Any subgroup of $\R^2$ that is contained in this resonance set must be discrete or confined to one of the coordinate axes. See Figure~\ref{fig:real-resonance-set}. When $\alpha$ has full Connes spectrum, $\Lambda(\alpha)$ intersects each coordinate axis only at $(0,0)$ and is therefore discrete. The trivial subgroup
corresponds to the outer case, while the non outer cases have rank one
or two. In particular, for a flow with full Connes spectrum on an
arbitrary factor, the subgroup $$H(\alpha)=\{ t \in \R \mid \alpha_t \text{ is inner} \}$$ is generated by at most
two elements. Moreover, the resonance property imposes strong arithmetic constraints: if $s,t>0$ are two generators of $H(\alpha)$, then
$s/t$ must be a quadratic number.

We can now state the main theorem of this paper.
\begin{samepage}
\begin{letterthm}\label{main flow conjecture}\label{main non outer classification}
Let $\alpha, \beta:\R\curvearrowright R$ be two flows on the hyperfinite $\II_1$
factor with $\Gamma(\alpha)=\Gamma(\beta)=\R$. Then $\alpha$ is cocycle conjugate to $\beta$ if and only if $\Lambda(\alpha)=\Lambda(\beta)$. In particular, every such flow $\alpha$
is cocycle conjugate to an irrational rotation flow on a noncommutative torus. Moreover, the following assertions are equivalent.
\begin{enumerate}[\rm (i)]
\item $\alpha$ is outer, that is, $\Lambda(\alpha)=\{(0,0)\}$.
\item $\alpha$ is strictly outer, that is, $R'\cap(R\rtimes_\alpha\R)=\C 1$.
\item $\alpha$ has the Rokhlin property.
\item $\alpha$ is cocycle conjugate to $\sigma^\infty$.
\end{enumerate}
\end{letterthm}
\end{samepage}

Recently, we proved the outer case of Theorem~\ref{main flow conjecture} for almost periodic flows in
\cite[Theorem~A]{HM26}, but this was a very special case and the proof there could not be adapted to
arbitrary flows. Our approach in this paper is directly inspired by the theory of type $\III$ factors and their classification. The analogy guiding our present work is summarized in
Table~\ref{type III factors and flows analogy}. A crucial ingredient in the Connes--Haagerup proof of the uniqueness of the amenable $\III_1$ factor is the so-called bicentralizer problem formulated by Connes for arbitrary factors of type $\III_1$ \cite{Co85} and solved by Haagerup in the amenable case \cite{Ha87}. By analogy with their work, we develop a bicentralizer theory for actions of arbitrary locally compact groups on $\II_1$ factors, which we will now explain.

\begin{figure}[H]
\centering
\begin{tikzpicture}[x=.9cm,y=.9cm]
\begin{scope}
\clip (-4.2,-4.2) rectangle (4.2,4.2);
\foreach \n in {1,...,17} {
  \pgfmathsetmacro{\a}{\n/4.2}
  \draw[black!65,line width=.35pt]
    plot[domain=\a:4.2,samples=140,smooth] (\x,{\n/\x});
  \draw[black!65,line width=.35pt]
    plot[domain=-4.2:-\a,samples=140,smooth] (\x,{\n/\x});
  \draw[black!65,line width=.35pt]
    plot[domain=\a:4.2,samples=140,smooth] (\x,{-\n/\x});
  \draw[black!65,line width=.35pt]
    plot[domain=-4.2:-\a,samples=140,smooth] (\x,{-\n/\x});
}
\end{scope}
\draw[->,line width=.65pt] (-4.32,0)--(4.48,0) node[below] {$p$};
\draw[->,line width=.65pt] (0,-4.32)--(0,4.48) node[left] {$t$};
\foreach \j in {-4,...,4} {
  \foreach \k in {-4,...,4} {
    \pgfmathsetmacro{\pcoord}{(\j+(1+sqrt(5))*\k/2)}
    \pgfmathsetmacro{\tcoord}{(\j+(1-sqrt(5))*\k/2)}
    \pgfmathtruncatemacro{\showpoint}{
      (abs(\pcoord)<4.2) && (abs(\tcoord)<4.2)}
    \ifnum\showpoint=1
      \fill[red!80!black] (\pcoord,\tcoord) circle[radius=1.3pt];
    \fi
  }
}
\node[below left,inner sep=3pt] at (0,0) {$0$};
\end{tikzpicture}
\caption{In black, the resonance set
\[\mathcal R_\R=\{(p,t)\in\R^2\mid pt\in\Z\}.\] Every subgroup of $\R^2$ that is contained in $\mathcal R_\R$ must be discrete (such as the subgroup depicted with red dots) or confined to one of the coordinate axes.}
\label{fig:real-resonance-set}
\end{figure}

\subsection{A bicentralizer theory for trace-preserving actions}
Let $\alpha : G \curvearrowright (M,\tau)$ be a continuous trace-preserving action of a locally compact group $G$ on a tracial von Neumann algebra $(M,\tau)$. We define the bicentralizer of the action, denoted by $\rB(M,\alpha)$, to be the algebra of all
$a\in M$ such that
\[
\|x_i a-ax_i\|_2\longrightarrow0
\]
for every bounded net $(x_i)_i$ in $M$ satisfying
\[
\sup_{g\in K}\|\alpha_g(x_i)-x_i\|_2\longrightarrow0
\qquad\text{for every compact }K\subset G.
\]
Thus the bicentralizer records what commutes asymptotically with
all almost invariant elements. If we replace $\tau$ by a faithful normal state $\varphi$ and take $\alpha$ to be the modular flow $\sigma^\varphi$, we recover the bicentralizer $\rB(M,\varphi)$ defined by
Connes \cite{Co85}.

For flows with full Connes spectrum, we determine the bicentralizer
on an arbitrary $\II_1$ factor. It is generated by the unitaries
implementing the inner times of the flow.

\begin{letterthm}\label{main flow bicentralizer}
Let $\alpha:\R\curvearrowright M$ be a flow on a $\II_1$ factor
with $\Gamma(\alpha)=\R$. Then
\[
\rB(M,\alpha)
=\{v\in\cU(M)\mid \alpha_t=\Ad(v)
  \text{ for some }t\in\R\}''.
\]
More precisely, the following cases occur.
\begin{enumerate}[\rm (i)]
\item If $\Lambda(\alpha)=\{(0,0)\}$, that is, if $\alpha$ is outer, then $\rB(M,\alpha)=\C1$.
\item If $\Lambda(\alpha)=\Z(p,t)$ has rank one, then
$\rB(M,\alpha) \cong\rL^\infty(\mathbb T)$
is generated by a single unitary $v$, with $\alpha_t=\Ad(v)$ and
$\alpha_r(v)=e^{2\pi\ri pr}v$ for every $r\in\R$.
\item If $\Lambda(\alpha)$ has rank two, then $\rB(M,\alpha)$ is a noncommutative
torus and $\alpha|_{\rB(M,\alpha)}$ is an irrational rotation flow.
\end{enumerate}
\end{letterthm}

The generation formula and the outer case are proved in
Theorem~\ref{bicentralizer conjecture for flows}. The rank one and
rank two descriptions are Proposition~\ref{rank one bicentralizer structure}
and Theorem~\ref{non outer bicentralizer model}, respectively.
Thus the structure of the bicentralizer is determined without any
amenability assumption on $M$.

The argument also applies to modular flows and yields the solution
of Connes' bicentralizer problem for arbitrary type $\III_1$
factors. Its proof uses the same resonance argument as the outer
case of Theorem~\ref{main flow bicentralizer}.

\begin{letterthm}\label{main Connes bicentralizer}
Let $M$ be a type $\III_1$ factor with a faithful normal state $\varphi$.
Then
\[
\rB(M,\varphi)=\C1.
\]
\end{letterthm}

Combined with Isono's theorem \cite[Corollary~B]{Is24},
Theorem~\ref{main Connes bicentralizer} also settles the
Haagerup--St{\o}rmer conjecture: every pointwise inner automorphism
of a type $\III_1$ factor with separable predual is the composition
of an inner and a modular automorphism.
Combining Theorem~\ref{main Connes bicentralizer} with
\cite[Corollary~B]{HM25} and \cite[Theorem~4.21]{AG26} also shows that
$(M,\varphi)$ is $\rW^*$-selfless for every type $\III_1$ factor $M$
and every faithful normal state $\varphi$ on $M$.

\subsection{The relative commutant theorem and the resonance property}
The proofs of Theorems~\ref{main flow bicentralizer}
and~\ref{main Connes bicentralizer} use the resonance phenomenon for
joint actions of a group and its dual. We first explain how it
appears in crossed-product relative commutants.

For a discrete group acting on a factor, Fourier coefficients make
the passage from outerness to strict outerness straightforward.
For actions of continuous groups, these coefficients need not exist pointwise. Outerness rules out nontrivial atomic contributions to the relative commutant,
but gives no direct control over its continuous spectrum without additional assumptions (see for example \cite[Theorem~A]{MV23} or
\cite[Corollary~B]{Mo25}).

To make this more precise, let $\alpha :G \curvearrowright M$ be a continuous action of a locally compact abelian group on a factor $M$. Then it is easy to see that $\alpha$ is outer if and only if the restriction of the dual action \[\widehat{\alpha} : \widehat{G} \curvearrowright M \rtimes_\alpha G\] to $M' \cap (M \rtimes_\alpha G)$ is weakly mixing. In general, the atomic spectrum of \[\widehat{\alpha}|_{M' \cap (M \rtimes_\alpha G)}\] detects the elements of $G$ that act by inner automorphisms. 

In general, there is no reason to believe that outerness would force $\alpha$ to be strictly outer, that is, $M' \cap (M \rtimes_\alpha G)=\C 1$, even if we assume that $\Gamma(\alpha)=\widehat{G}$. But despite our efforts, with assistance from OpenAI's GPT-5.6, we could not find a counterexample for $G=\R$. We thus tried to obtain a positive result by adapting the proof of the Connes--Takesaki relative commutant theorem, \cite[Chapter ${\rm II}$, Theorem~5.1]{CT77}, a celebrated result of
modular theory. With the help of GPT-5.6, we arrived at the following theorem.
\begin{samepage}
\begin{letterthm}\label{letterthm resonance property}
Let $\alpha:G\curvearrowright M$ be a continuous action of a locally
compact abelian group on a factor $M$ that preserves a faithful normal semifinite
trace $\tau$. Consider the joint action \[\gamma : G \times \widehat{G} \curvearrowright M' \cap (M \rtimes_\alpha G)\]
given by \[\gamma_{t,p}(x)=\widehat{\alpha}_p(u_txu_t^*), \qquad (t,p) \in G \times \widehat{G},\]
where $\widehat{\alpha}$ is the dual action and $(u_t)_{t \in G}$ are the unitaries implementing the action $\alpha$ in the crossed product. Then the Arveson spectrum of $\gamma$ satisfies
\[
\Sp(\gamma)
\subset
\mathcal R_G=\{(p,t)\in\widehat G\times G\mid \langle p,t\rangle=1\}.
\]
\end{letterthm}
\end{samepage}
We call $\mathcal R_G$ the resonance set. This theorem does not have any striking consequence when $G$ is an arbitrary locally compact abelian group. But when $G=\R$, as we already mentioned, any subgroup of $\widehat\R \times \R=\R^2$ that is contained in \[\mathcal R_\R=\{ (p,t) \in \R^2 \mid pt \in \Z \}\] is either contained in one of the coordinate axes or discrete. In the latter case, this forces $\gamma$ to have atomic spectrum. When $G=\R$, we thus obtain the following dichotomy.

\begin{lettercor}\label{main relative commutant dichotomy}
Let $\alpha:\R\curvearrowright M$ be a flow on a factor $M$ that
preserves a faithful normal semifinite trace $\tau$. Then at least
one of the following assertions holds.
\begin{enumerate}[\rm (i)]
\item One has
\[
M'\cap(M\rtimes_\alpha\R)=\cZ(M\rtimes_\alpha\R).
\]
\item The algebra $M'\cap(M\rtimes_\alpha\R)$ is generated by
unitaries of the form $vu_{-t}$, where $t\in\R$, $v\in\cU(M)$
and $\alpha_t=\Ad(v)$.
\end{enumerate}
In particular, if $\Gamma(\alpha)=\R$ and $\alpha$ is
outer, then $\alpha$ is strictly outer.
\end{lettercor}

Let us now explain how we obtain Theorem~\ref{main flow bicentralizer} and Theorem~\ref{main Connes bicentralizer}. For an action of
a locally compact abelian group $G$ with full Connes spectrum, we
construct a canonical \emph{dual bicentralizer action}
\[
\beta:\widehat G\curvearrowright\rB(M,\alpha).
\]
In the case of type $\III_1$ factors, this dual bicentralizer action was discovered in \cite{AHHM20} and led to several developments, see \cite{Ma20, Ma25, Ma26}. In the case of actions on $\II_1$ factors with full Connes spectrum, we construct the dual bicentralizer action by first proving an analog of the Connes--St{\o}rmer transitivity theorem, in the form of an approximate $1$-cocycle vanishing theorem.

Similarly to the recent results in the type $\III_1$ case, we show that the dual bicentralizer action \[
\beta:\widehat G\curvearrowright\rB(M,\alpha)
\] is ergodic, and we completely determine its point spectrum.

Now, since the dual bicentralizer action $\beta$ commutes with $\alpha|_{\rB(M,\alpha)}$, we can consider the joint bicentralizer action
\[ \gamma = (\alpha|_{\rB(M,\alpha)} \times \beta) : G \times \widehat{G} \curvearrowright \rB(M,\alpha).\]
Inspired by Theorem~\ref{letterthm resonance property} and with the crucial technical help of GPT-6, we show that $\gamma$ satisfies the resonance property
\[ \Sp(\gamma) \subset \mathcal{R}_G.\]
When $G=\R$, this implies that $\gamma$ has pure point spectrum.
Identifying its eigenunitaries gives the structure in
Theorem~\ref{main flow bicentralizer}. The same resonance argument
applied to modular flows gives Theorem~\ref{main Connes bicentralizer}.

\subsection{Perspectives for actions of amenable groups}
Our bicentralizer theory also applies to actions of amenable locally
compact groups. For amenable factors, it gives the following
characterization of the Rokhlin property. The proof expresses the
Rokhlin property and triviality of the bicentralizer in terms of
weak containment of equivariant correspondences, using the
implementation of binormal states in ultrapowers
\cite[Theorem~2.1]{Ma25}.

\begin{letterthm} \label{main rokhlin and bicentralizer}
    Let $\alpha : G \curvearrowright M$ be a continuous action of an amenable locally compact group on an amenable $\II_1$ factor $M$. Then $\alpha$ has the Rokhlin property if and only if $\alpha$ is faithful and $\rB(M,\alpha)=\C 1$.
\end{letterthm}
See also Theorem \ref{Rokhlin and trivial bicentralizer}. Combining this theorem with Shimada's work \cite{Sh14}, we obtain the following corollary.
\begin{lettercor} \label{lettercor minimal actions}
    A second countable locally compact abelian group admits a unique faithful action with trivial bicentralizer on the hyperfinite $\II_1$ factor, up to cocycle conjugacy.
\end{lettercor}
This corollary applies to minimal actions, actions with full Connes spectrum fixing a maximal abelian subalgebra as well as certain tensor product actions.

These results suggest the following conjecture on the
classification of actions of amenable locally compact groups.

\begin{letterconj} \label{letterconj classification}
    Let $\alpha : G \curvearrowright M$ be a continuous action of a locally compact group on a $\II_1$ factor $M$. Suppose that $G$ and $M$ are amenable. Then the following are equivalent.
    \begin{enumerate}[\rm (i)]
        \item $\alpha$ has the Rokhlin property.
        \item $\alpha$ is faithful and has a trivial bicentralizer, i.e.\ $\rB(M,\alpha)=\C 1$.
        \item $\alpha$ is strictly outer, i.e.\ $M' \cap (M \rtimes_\alpha G)=\C 1$.
    \end{enumerate}
    Moreover, when $G$ is second countable and $M$ is separable, there is a unique action satisfying these equivalent properties up to cocycle conjugacy.
\end{letterconj}

We expect the
equivalence of assertions (ii) and (iii) to hold for arbitrary
$\II_1$ factors $M$, still assuming that $G$ is amenable. This is
the bicentralizer conjecture for actions of amenable groups. The
implication (i) $\Rightarrow$ (ii) also holds without amenability
of $M$, while our proof of the converse uses it
(Theorem~\ref{Rokhlin and trivial bicentralizer}). The uniqueness
assertion above is formulated for amenable factors.

For arbitrary $M$, the equivalence of assertions (ii) and (iii)
is known when $G$ is discrete by the work of
Popa--Shlyakhtenko--Vaes \cite[Lemmas~4.2 and~4.3]{PSV20}, using
Popa's approximate free independence theorem. We extend their
argument to the case where $G$ admits a compact open subgroup
(see Theorem~\ref{compact open amenable trivial bicentralizer}).
Conjecture~\ref{letterconj classification} is known when $G$ is
discrete by the work of Ocneanu \cite{Oc85}. Very recently,
Nishihara also proved it when $G$ has a compact open subgroup
\cite{Ni26}. The main difficulty arises when $G$ is connected and
not compact, where one cannot separate the small scale behavior
from the large scale behavior. Shimada generalized
Masuda--Tomatsu's work on Rokhlin flows to classify actions of
arbitrary locally compact abelian groups with the Rokhlin property
\cite{Sh14}. For arbitrary amenable locally compact groups, the conjecture divides into two parts that should be independent: solving the bicentralizer problem for strictly outer actions and showing the uniqueness of Rokhlin actions up to cocycle conjugacy. Simply asking AI models is not enough and new ideas are required.

\subsection*{Acknowledgments}
We are grateful to Yasuyuki Kawahigashi and Sorin Popa for their helpful
historical comments.

\subsection*{Tool and computational resource disclosure}
The bicentralizer theory for trace-preserving actions, including the connection with the Rokhlin
property, the approximate cocycle vanishing theorem,
the dual bicentralizer action and the relative fixed point theorem, was developed entirely by the authors. AI tools served only as a
writing aid and for bibliographic searches in this part of the work. Using our bicentralizer theory, we proved
the flow conjecture for minimal flows or flows fixing a maximal abelian subalgebra and certain tensor products
of flows, but the general case remained out of reach for some time. It was also clear to us from the start that a solution of the bicentralizer problem for flows on $\II_1$ factors would lead to a solution of Connes' bicentralizer problem for type $\III_1$ factors by adapting the argument.

In parallel, we used the GPT-5.6 Sol model to investigate whether every trace-preserving outer flow with full Connes
spectrum must be strictly outer. We first explored with the model
possible counterexamples among classical and less familiar
constructions of flows. When none survived scrutiny, we guided the
discussion toward adapting the distributional mechanism of
Connes--Takesaki. After further interaction, the model produced a
proof, but the argument was difficult to understand and gave little insight
into the underlying structure. Knowing that the result was true, we
set aside the original argument and reconstructed the proof
conceptually from the ground up. This led to the theory of von Neumann
algebra-valued distributions developed as a replacement for Fourier coefficients.

Most importantly, this exploration on relative commutants revealed the surprising resonance mechanism. Our dual bicentralizer flow theory and our bicentralizer conjecture relating the bicentralizer algebra to the relative commutant in the crossed product led us to think that this resonance mechanism could be applied to the bicentralizer algebra itself, equipped with its joint bicentralizer action (see Proposition \ref{spectrum identification}). This idea turned out to be decisive. Our previous attempts to solve the bicentralizer problem with the assistance of GPT-6 Astra had been unsuccessful. But once we formulated this resonance idea, the model was suddenly able to use our results to solve the bicentralizer problem. A crucial technical ingredient in the proof, found by GPT-6 Astra alone, is to introduce a singular tracial state on the crossed product algebra, which makes it possible to apply our relative fixed point theorem in the GNS representation of that state. The argument was difficult
to read and obscured the underlying structure. We rewrote it to give a conceptual proof using the Plancherel function. 

Finally, the classification strategy for non outer flows is due to the authors. We outlined the strategy step by step and asked GPT-6 Astra to work out the details. The model was very efficient. For example, in the rank one case, a single prompt, reproduced below, was enough to obtain a complete proof in just a few minutes. We then only suggested that the model use the results of \cite{SW24} to simplify the proof, and we improved the presentation. The model struggled more with the rank two case. We had to provide not only the overall strategy but also technical guidance, specifically with Haagerup's techniques from \cite{Ha16}.

\begin{quote}
   \footnotesize Now we want to classify when $\Lambda(\alpha)$ is rank one. Do it in a new section after section $8$. First, we build a model action $\sigma$ on the noncommutative torus with $\Lambda(\sigma)$ generated by $(p,t)$ for any pair $(p,t)$ with $p$ and $t$ nonzero. The key idea is that if $\alpha$ and $\beta$ are two actions with $\Lambda(\alpha)$ generated by $(p,t)$ and $\beta$ generated by $(-p,t)$, then $\alpha \otimes \beta$ is cocycle equivalent to a periodic flow (classified by Kawahigashi). Now, we consider our action $\alpha$ that we want to classify with $\Lambda(\alpha)$ generated by $(p,t)$. Take $\sigma$ a model action with $(p,t)$ and $\sigma'$ a model action with $(-p,t)$ and consider the triple tensor product $\alpha \otimes \sigma' \otimes \sigma$. Then this flow is cocycle conjugate to both $\alpha \times$ (periodic) and to (periodic) $\times \sigma$. So, we only have to check that $\alpha$ absorbs any (periodic) (and same for $\sigma$). To prove this absorption, show that there is a $p$-eigenvalue for the flow on $M_{\omega,\alpha}$. If needed, to show this, use our techniques for the Rokhlin property in the rest of the paper.
\hfill\mbox{\textit{--- Prompt sent to GPT-6 Astra}}
\end{quote}

\begin{table}[p]
\caption{Type $\III$ factors and flows on $\II_1$ factors.}
\label{type III factors and flows analogy}
\centering
\small
\renewcommand{\arraystretch}{1.1}
\begin{tabular}{@{}>{\raggedright\arraybackslash}p{0.48\textwidth}
@{\hspace{0.04\textwidth}}>{\raggedright\arraybackslash}p{0.48\textwidth}@{}}
\toprule
\textbf{Type $\III$ factors $M$.}
& \textbf{Flows on $\II_1$ factors $\alpha:\R\curvearrowright N$.}
\\
\midrule
\textbf{Type $\III_1$.}\par\smallskip
$\Gamma(\sigma^\varphi)=\R$ for some faithful normal state $\varphi$ \cite{Co73}.
&
\textbf{Outer flows with full Connes spectrum.}\par\smallskip
$\Gamma(\alpha)=\R$ and $\alpha$ is outer.
\\
\addlinespace[0.9em]
\textbf{No type $\III$ analogue.}\par\smallskip
A modular flow with full Connes spectrum is necessarily outer.
&
\textbf{Non outer flows with full Connes spectrum.}\par\smallskip
Such flows exist even on $R$, where they are classified
by $\Lambda(\alpha)$ (Theorem~\ref{main non outer classification}).
\\
\addlinespace[0.9em]
\textbf{Connes--Takesaki relative commutant theorem.}\par\smallskip
If $M$ is of type $\III_1$ then
$M'\cap(M\rtimes_{\sigma^\varphi}\R)=\C 1$
\cite{CT77}.
&
\textbf{Strict outerness.}\par\smallskip
If $\Gamma(\alpha)=\R$ and $\alpha$ is outer then
$N'\cap(N\rtimes_\alpha\R)=\C 1$ (Corollary \ref{main relative commutant dichotomy}).
\\
\addlinespace[0.9em]
\textbf{Connes--St{\o}rmer transitivity.}\par\smallskip
On a type $\III_1$ factor, any two faithful normal states are
approximately unitarily conjugate
\cite{CS78}.
&
\textbf{Our approximate $1$-cocycle vanishing theorem.}\par\smallskip
If $\Gamma(\alpha)=\R$, then
$Z^1(\alpha)=\overline{B^1(\alpha)}$: every $1$-cocycle for $\alpha$ is an approximate coboundary
(Theorem~\ref{approximate cocycle vanishing theorem}).
\\
\addlinespace[0.9em]
\textbf{$R_\infty$-stability.}\par\smallskip
$M\cong M\ovt R_\infty$ where $R_\infty$ is the Araki--Woods factor of type $\III_1$.
&
\textbf{Rokhlin property.}\par\smallskip
The flow $\alpha$ has the Rokhlin property \cite{MT16}.
\\
\addlinespace[0.9em]
\textbf{Uniqueness of amenable $R_\infty$-stable factors.}\par\smallskip
An amenable separable factor absorbing $R_\infty$ is isomorphic
to $R_\infty$ \cite{Co85}.
&
\textbf{Uniqueness of Rokhlin flows on amenable $\II_1$ factors.}\par\smallskip
There is a unique Rokhlin flow on the hyperfinite $\II_1$ factor up to cocycle conjugacy
\cite{MT16}.
\\
\addlinespace[0.9em]
\textbf{Connes' bicentralizer approach.}\par\smallskip
If $M$ is amenable of type $\III_1$ and
$\rB(M,\varphi)=\C 1$, then $M\cong M\ovt R_\infty$
\cite{Co85}.
&
\textbf{Our bicentralizer approach.}\par\smallskip
If $N$ is amenable, $\Gamma(\alpha)=\R$ and $\rB(N,\alpha)=\C 1$, then $\alpha$ has the Rokhlin property
(Theorem~\ref{main rokhlin and bicentralizer}).
\\
\addlinespace[0.9em]
\textbf{Connes' bicentralizer problem.}\par\smallskip
Every type $\III_1$ factor $M$ (not necessarily amenable) satisfies
$\rB(M,\varphi)=\C 1$
\cite{Co85}. The amenable case was solved by Haagerup \cite{Ha87}. The general case is solved in Theorem~\ref{main Connes bicentralizer}.
&
\textbf{Our bicentralizer theorem for flows.}\par\smallskip
Every outer flow $\alpha$ on a $\II_1$ factor $N$ (not necessarily amenable) with
$\Gamma(\alpha)=\R$ satisfies
$\rB(N,\alpha)=\C 1$. This is the outer case of Theorem~\ref{main flow bicentralizer}.
\\
\bottomrule
\end{tabular}
\end{table}
\clearpage

\endgroup

\begingroup
\small\linespread{0.90}\selectfont
\tableofcontents
\endgroup
\clearpage

\section{A bicentralizer theory for trace-preserving actions}
\label{sec:action-bicentralizers}

We introduce a bicentralizer theory for trace-preserving
actions, including its relative version for inclusions. The analytic
bicentralizer describes asymptotic commutation with almost invariant
elements. The ultrapower interpretation turns almost invariant
elements into fixed points and asymptotic commutation into exact
commutation. 

When the acting group is amenable, we expect the action to have a lot of almost invariant elements and the analytic bicentralizer should be as small as possible given the algebraic constraints imposed by the data. We make this intuition precise by introducing an algebraic version of the bicentralizer following the idea of \cite{Ma25}. The algebraic bicentralizer is always contained in the analytic bicentralizer and we conjecture that they agree for amenable groups. In particular, when the action is strictly outer, the conjecture says that its bicentralizer is trivial. We give examples where this conjecture is verified.

\subsection{Tracial ultrapowers and equicontinuous parts}
We recall the tracial ultrapower construction and its equivariant version.
The sequence-indexed construction is standard. See, for instance,
\cite[Section~3]{AH12}. Equicontinuous ultrapowers for continuous flows were
introduced in this context in \cite[Section~3]{MT16}. 

Let $I$ be a directed set and let $\omega$ be an ultrafilter on $I$. We write
$\lim_{i\to\omega} a_i=a$ when the net $(a_i)_{i\in I}$ converges to $a$ along
$\omega$. The ultrafilter $\omega$ is called \emph{cofinal} if
\[
\{j\in I\mid j\geq i\}\in\omega
\qquad\text{for every }i\in I.
\]
Let $(M,\tau)$ be a tracial von Neumann algebra, where $\tau$ is a faithful
normal tracial state, and set $\|x\|_2=\tau(x^*x)^{1/2}$. Define
\[
\ell^\infty(I,M)
=\left\{(x_i)_{i\in I}\mid \sup_{i\in I}\|x_i\|<\infty\right\}
\]
and
\[
\mathcal I_\omega(M)
=\left\{(x_i)_{i\in I}\in\ell^\infty(I,M)
\mathrel{\Big|}\lim_{i\to\omega}\|x_i\|_2=0\right\}.
\]
The set $\mathcal I_\omega(M)$ is a norm-closed two-sided ideal of
$\ell^\infty(I,M)$. The \emph{tracial ultrapower} of $M$ with respect to
$\omega$ is
\[
M^\omega=\ell^\infty(I,M)/\mathcal I_\omega(M).
\]
We denote the image of $(x_i)_{i\in I}$ in $M^\omega$ by $(x_i)^\omega$.
The formula
\[
\tau^\omega((x_i)^\omega)=\lim_{i\to\omega}\tau(x_i)
\]
defines a faithful normal tracial state on $M^\omega$, and we identify $M$
with the von Neumann subalgebra of $M^\omega$ given by constant nets. The trace-preserving conditional expectation $\rE^\omega : M^\omega \rightarrow M$ is given by $$\rE^\omega((x_i)^\omega)=\lim_{i\to \omega} x_i$$
where the limit is taken in the weak$^*$ topology.

We will use the following proposition throughout this paper.
\begin{proposition} \label{omega equicontinous functions}
  Let $(f_i)_{i \in I}$ be a net of continuous functions from a locally compact space $X$ into a Hilbert space $H$, with $\sup_i\|f_i(x)\|<\infty$ for every $x\in X$. If $(f_i)_{i \in I}$ is $\omega$-equicontinuous, meaning that for every $x \in X$ and every $\varepsilon > 0$, there exists a neighborhood $U$ of $x$ such that
  $$ \{ i \in I \mid \sup_{y \in U} \| f_i(x)-f_i(y)\|_2 \leq \varepsilon \} \in \omega$$ then 
  $$f=(f_i)^\omega : X\ni  x \mapsto (f_i(x))^\omega \in H^\omega$$
  is continuous. Moreover, for every compact subset $K \subset X$, we have
$$ \lim_{i \to \omega} \sup_{x \in K} \| f_i(x)\|_2=\sup_{x \in K} \|f(x)\|_2.$$
\end{proposition}

\begin{proof}
We first prove the continuity of $f$. Fix $x\in X$ and
$\varepsilon>0$. By $\omega$-equicontinuity, there is a neighborhood
$U$ of $x$ such that
\[
 \left\{i\in I\ \middle|\
 \sup_{y\in U}\|f_i(x)-f_i(y)\|_2\leq\varepsilon
 \right\}\in\omega.
\]
Hence, for every $y\in U$,
\[
 \|f(x)-f(y)\|_2
 =\lim_{i\to\omega}\|f_i(x)-f_i(y)\|_2
 \leq\varepsilon.
\]
Thus $f:X\to H^\omega$ is continuous.

Now let $K\subset X$ be compact. Pointwise, we have
\[
 \|f(x)\|_2=\lim_{i\to\omega}\|f_i(x)\|_2
 \qquad(x\in X),
\]
and therefore
\[
 \sup_{x\in K}\|f(x)\|_2
 \leq\lim_{i\to\omega}\sup_{x\in K}\|f_i(x)\|_2.
\]

For the reverse inequality, fix $\varepsilon>0$. For every $x\in K$,
choose a neighborhood $U_x$ of $x$ such that
\[
 \left\{i\in I\ \middle|\
 \sup_{y\in U_x}\|f_i(x)-f_i(y)\|_2\leq\varepsilon
 \right\}\in\omega.
\]
Choose $x_1,\ldots,x_n\in K$ such that
\[
 K\subset\bigcup_{k=1}^n U_{x_k}.
\]
After intersecting finitely many members of $\omega$, we obtain, for
$\omega$-almost every $i$,
\[
 \sup_{y\in U_{x_k}}\|f_i(x_k)-f_i(y)\|_2\leq\varepsilon,
 \qquad
 \|f_i(x_k)\|_2\leq\|f(x_k)\|_2+\varepsilon
\]
for every $1\leq k\leq n$. For such an $i$ and $y\in K$, choose $k$
with $y\in U_{x_k}$. Then
\[
 \|f_i(y)\|_2
 \leq\|f_i(x_k)\|_2+\|f_i(y)-f_i(x_k)\|_2
 \leq\sup_{x\in K}\|f(x)\|_2+2\varepsilon.
\]
It follows that
\[
 \lim_{i\to\omega}\sup_{x\in K}\|f_i(x)\|_2
 \leq\sup_{x\in K}\|f(x)\|_2+2\varepsilon.
\]
Letting $\varepsilon\to0$ proves the second assertion.
\end{proof}

Let $G$ be a locally compact group and
let $\alpha:G\curvearrowright(M,\tau)$ be a continuous trace-preserving
action. Here continuity means that $g\mapsto\alpha_g(x)$ is continuous in
$\|\cdot\|_2$ for every $x\in M$. We write $K\Subset G$ when $K$ is a compact
subset of $G$, and $U\ll G$ when $U$ is an open neighborhood of the identity.

A bounded net $(x_i)_{i\in I}$ in $M$ is called
\emph{$(\alpha,\omega)$-equicontinuous} if the net of continuous functions
$(g\mapsto\alpha_g(x_i))_{i \in I}$ is $\omega$-equicontinuous.
We define the \emph{$\alpha$-equicontinuous part} of the tracial ultrapower by
\[
M^\omega_\alpha
=\left\{(x_i)^\omega\in M^\omega
\mathrel{\Big|}(x_i)_{i\in I}\text{ is $(\alpha,\omega)$-equicontinuous}\right\}.
\]
This definition does not depend on the chosen representative. We also define
\[
M^{\omega,\alpha}
=\left\{(x_i)^\omega\in M^\omega
\mathrel{\Big|}\lim_{i\to\omega}
\sup_{g\in K}\|\alpha_g(x_i)-x_i\|_2=0
\text{ for every }K\Subset G\right\}.
\]

The following proposition follows easily from Proposition \ref{omega equicontinous functions}.
\begin{proposition}\label{equicontinuous ultrapower action}
The $\alpha$-equicontinuous part
$M^\omega_\alpha$ is a von Neumann subalgebra of
$M^\omega$, and the formula
\[
\alpha^\omega_g((x_i)^\omega)=(\alpha_g(x_i))^\omega
\qquad(g\in G)
\]
defines a continuous trace-preserving action
$$\alpha^\omega:G\curvearrowright M^\omega_\alpha$$ such that
\[
M^{\omega,\alpha}=(M^\omega_\alpha)^{\alpha^\omega}.
\]
Moreover, for every $f \in \rL^1(G)$ and $(x_i)^\omega \in M^\omega$, we have $$(\alpha_f(x_i))^\omega \in M^\omega_\alpha$$ and if $x=(x_i)^\omega \in M^\omega_\alpha$, we have the formula $$\alpha^\omega_f(x)=(\alpha_f(x_i))^\omega.$$
\end{proposition}

For a strongly continuous unitary representation $U:G\to\cU(H)$, we write
\[
U(f)=\int_G f(t)U_t\,\rd t
\qquad(f\in\rL^1(G)),
\]
and use the same notation for its extension to $\rC^*(G)$.

We identify $\rL^2(M^\omega)$ with its canonical subspace of
$\rL^2(M)^\omega$.
\begin{proposition}\label{L2 equicontinuous ultrapower}
The Hilbert space $\rL^2(M^\omega_\alpha)$ is the
$U^\alpha$-equicontinuous part of $\rL^2(M^\omega)$, that is,
\[
\rL^2(M^\omega_\alpha)
=\left\{(\xi_i)^\omega\in\rL^2(M^\omega)\ \middle|\
(g\mapsto U^\alpha_g\xi_i)_{i\in I}
\text{ is $\omega$-equicontinuous}\right\},
\]
and $$ U^{\alpha^\omega}_g\xi=(U^\alpha_g)^\omega\xi \qquad (\xi \in \rL^2(M^\omega_\alpha), \; g \in G).$$
Moreover, for every $f\in\rL^1(G)$ and
$\xi \in\rL^2(M^\omega)$, we have
\[
(U^\alpha(f))^\omega \xi \in\rL^2(M^\omega_\alpha),
\]
and if $\xi \in\rL^2(M^\omega_\alpha)$, then
\[
U^{\alpha^\omega}(f)\xi=(U^\alpha(f))^\omega \xi.
\]
\end{proposition}

\begin{proof}
Denote the right-hand side of the first display by $H_\alpha$. Proposition
\ref{omega equicontinous functions}, applied to $U^\alpha$, shows that
$H_\alpha$ is a closed subspace of $\rL^2(M^\omega)$. Since
$M^\omega_\alpha \widehat{1} \subset H_\alpha$, we have
$\rL^2(M^\omega_\alpha)\subset H_\alpha$.

For $R>0$, let
\[
 P_R:\rL^2(M)\longrightarrow
 \{a\in M\mid\|a\|\leq R\}:\xi\longmapsto
 \operatorname*{argmin}_{\substack{a\in M\\ \|a\|\leq R}}
 \|\xi-a\|_2
\]
be the metric projection in tracial $\rL^2$. The map $P_R$ is
nonexpansive and $U^\alpha$-equivariant. Hence, if
$\xi=(\xi_i)^\omega\in H_\alpha$, then
\[
 x_R=(P_R(\xi_i))^\omega\in M^\omega_\alpha.
\]
Since $\xi\in\rL^2(M^\omega)$, its representative satisfies
\[
 \lim_{R\to\infty}\lim_{i\to\omega}
 \|\xi_i-P_R(\xi_i)\|_2=0.
\]
Thus $x_R\to\xi$ in $\|\cdot\|_2$ as $R\to\infty$, and
$H_\alpha\subset\rL^2(M^\omega_\alpha)$, proving the first assertion.

Now let $f\in\rL^1(G)$ and
$\xi=(\xi_i)^\omega\in\rL^2(M^\omega)$. Put
$\xi_{i,R}=P_R(\xi_i)$. Then
\[
 \|U^\alpha(f)\xi_i-U^\alpha(f)\xi_{i,R}\|_2
 \leq\|f\|_1\|\xi_i-\xi_{i,R}\|_2.
\]
The same approximation shows that $(U^\alpha(f)\xi_i)^\omega$ belongs
to $\rL^2(M^\omega)$. Moreover,
\[
 \|U^\alpha_gU^\alpha(f)\xi_i-U^\alpha(f)\xi_i\|_2
 \leq\|f(g^{-1}\,\cdot\,)-f\|_1\|\xi_i\|_2.
\]
Translation continuity in $\rL^1(G)$ gives equicontinuity, and the
first assertion puts this vector in $\rL^2(M^\omega_\alpha)$.

Finally, the formula
\[
 U^{\alpha^\omega}(f)\widehat{x}
 =(U^\alpha(f)\widehat{x_i})^\omega
 \qquad(x=(x_i)^\omega\in M^\omega_\alpha)
\]
follows from Proposition~\ref{equicontinuous ultrapower action}. By density, it extends to every
$\xi=(\xi_i)^\omega\in\rL^2(M^\omega_\alpha)$.
\end{proof}

We end this section with a technical remark. If we make countability assumptions on von Neumann algebras and acting groups, one only needs to work with ultrafilters on $I=\N$. But since we will not make such countability assumptions, we have to assume that $I$ is large enough. We will assume this implicitly throughout this paper, relying on the following proposition.
\begin{proposition}\label{ultrapower cardinal bookkeeping}
Fix an infinite cardinal $\kappa$ and assume that there exists a cofinal net of finite subsets of $\kappa$ indexed by $I$.
\begin{enumerate}[\rm (i)]
  \item There exists a net $(\varepsilon_i)_{i \in I}$ of positive numbers that converges to $0$.
  \item If the topology of a locally compact group $G$ is $\kappa$-generated (admits a basis of cardinality less than $\kappa$) then there exists an increasing net $(K_i)_{i \in I} \Subset G$ that is cofinal (for every $K \Subset G$, there exists $i \in I$ such that $K \subset K_i$) and a decreasing net $(U_i)_{i \in I} \ll G$ that is cofinal (for every $U \ll G$, there exists $i \in I$ such that $U_i \subset U$).
  \item If $M$ is $\kappa$-generated (generated by a set of cardinality less than $\kappa$), then there exists an increasing net of finite subsets $(F_i)_{i \in I}$ of $M$ such that $\bigcup_{i \in I} F_i$ is a $*$-strongly dense self-adjoint $\mathbb{Q}+\ri \mathbb{Q}$-subalgebra of $M$.
\end{enumerate}
\end{proposition}

\subsection{The analytic bicentralizer of a trace-preserving action}
\leavevmode\par

Let $G$ be a locally compact group,
let $(M,\tau)\subset(N,\tau)$ be an inclusion of tracial von Neumann algebras,
and let $\alpha:G\curvearrowright(M,\tau)$ be a continuous trace-preserving
action. No extension of $\alpha$ to $N$ is assumed.

The \emph{analytic bicentralizer} $\rB(M\subset N,\alpha)$ is the set of all
$a\in N$ such that, for every $\varepsilon>0$, there exist $K\Subset G$ and
$\delta>0$ satisfying
\begin{equation}\label{local relative action bicentralizer}
\sup_{g\in K}\|\alpha_g(u)-u\|_2<\delta
\quad\Longrightarrow\quad
\|ua-au\|_2<\varepsilon
\qquad(u\in\cU(M)).
\end{equation}
It is a von Neumann subalgebra of $N$. When $M=N$, this is the
bicentralizer $\rB(M,\alpha)$.

\begin{proposition}\label{definition relative action bicentralizer}
Let $(x_i)_{i\in I}$ be a bounded net in $M$ such that
\[
\lim_i\sup_{g\in K}\|\alpha_g(x_i)-x_i\|_2=0
\]
for every $K\Subset G$. Then
\[
\lim_i\|x_i a-a x_i\|_2=0
\]
for every $a\in\rB(M\subset N,\alpha)$.
\end{proposition}

\begin{proof}
The assertion follows from the definition for nets of unitaries. The
bounded nets in $M$ satisfying the compact-uniform asymptotic invariance
condition form a unital $\rC^*$-algebra. Since a unital $\rC^*$-algebra is
linearly spanned by its unitaries, the result follows.
\end{proof}

\begin{proposition}\label{conditional expectation relative action bicentralizer}
Let
\[
\rE_{\rB(M\subset N,\alpha)}:N\longrightarrow
\rB(M\subset N,\alpha)
\]
be the trace-preserving conditional expectation. For every $x\in N$, the
element $\rE_{\rB(M\subset N,\alpha)}(x)$ is the unique element minimizing
$\|\cdot\|_2$ in
\[
\bigcap_{\substack{K\Subset G\\ \delta>0}}
\overline{\conv}^{\mathrm{w}^*}\left\{uxu^*\ \middle|\ u\in\cU(M),\
\sup_{g\in K}\|\alpha_g(u)-u\|_2<\delta\right\}.
\]
\end{proposition}

\begin{proof}
For $K\Subset G$ and $\delta>0$, denote the weak$^*$-closed convex hull in the
statement by $C_{K,\delta}(x)$ and put
\[
C(x)=\bigcap_{K\Subset G,\ \delta>0}C_{K,\delta}(x).
\]
This is a nonempty closed convex subset of $\rL^2(N,\tau)$, and hence has a
unique element $y$ minimizing $\|\cdot\|_2$.

Let $(v_i)_i$ be a net in $\cU(M)$ which is asymptotically
$\alpha$-invariant uniformly on compact subsets. For $u,v\in\cU(M)$,
\[
\sup_{g\in K}\|\alpha_g(vu)-vu\|_2
\leq
\sup_{g\in K}\|\alpha_g(u)-u\|_2
+\sup_{g\in K}\|\alpha_g(v)-v\|_2.
\]
It follows that every weak $\rL^2$-accumulation point of $(v_i yv_i^*)_i$
belongs to $C(x)$ and has the same norm as $y$. Thus
\[
\|v_i yv_i^*-y\|_2\longrightarrow0.
\]
The negation of \eqref{local relative action bicentralizer} would produce a
net for which this fails, so $y\in\rB(M\subset N,\alpha)$.

For $z\in\rB(M\subset N,\alpha)$ and an almost invariant unitary
$u\in\cU(M)$, traciality gives
\[
|\tau(z^*uxu^*)-\tau(z^*x)|
\leq\|uz-zu\|_2\,\|x\|_2.
\]
Using \eqref{local relative action bicentralizer}, convexity, and weak$^*$
continuity, we obtain
\[
\tau(z^*y)=\tau(z^*x)
\qquad(z\in\rB(M\subset N,\alpha)).
\]
Hence $y$ is the orthogonal projection of $x$ onto
$\rL^2(\rB(M\subset N,\alpha),\tau)$, proving the assertion.
\end{proof}

\subsection{The ultrapower interpretation of the bicentralizer}
We fix a cofinal ultrafilter
$\omega$ on a sufficiently large directed set. The inclusion $M\subset N$
induces $M^\omega\subset N^\omega$.

\begin{proposition}\label{relative bicentralizer with ultrapower}
We have
\[
\rB(M\subset N,\alpha)=(M^{\omega,\alpha})'\cap N
\]
and
\[
(M^{\omega,\alpha})'\cap N^\omega
\subset\rB(M\subset N,\alpha)^\omega.
\]
\end{proposition}

\begin{proof}
Put $P=M^{\omega,\alpha}$. If
$a\in\rB(M\subset N,\alpha)$ and
$V=(v_i)^\omega\in\cU(P)$, compact-uniform asymptotic invariance and
\eqref{local relative action bicentralizer} give $Va=aV$. Conversely, if
$a\notin\rB(M\subset N,\alpha)$, choose $\eta>0$, a compact exhaustion
$(K_i)_i$, positive numbers $\varepsilon_i\to0$, and $v_i\in\cU(M)$ such
that
\[
\sup_{g\in K_i}\|\alpha_g(v_i)-v_i\|_2<\varepsilon_i,
\qquad
\|v_i a-a v_i\|_2\geq\eta.
\]
Then $(v_i)^\omega\in P$ does not commute with $a$, proving the equality.

Let $X=(x_i)^\omega\in P'\cap N^\omega$ and set
\[
\Delta_i=\sup\left\{\|x_i-vx_i v^*\|_2\ \middle|\
v\in\cU(M),\
\sup_{g\in K_i}\|\alpha_g(v)-v\|_2<\varepsilon_i\right\}.
\]
A diagonal argument shows that $\Delta_i\to0$ along $\omega$. Otherwise
almost maximizing unitaries would define an element of $P$ not commuting
with $X$. Proposition~\ref{conditional expectation relative action bicentralizer}
gives
\[
\|x_i-\rE_{\rB(M\subset N,\alpha)}(x_i)\|_2\leq\Delta_i.
\]
Therefore $X\in\rB(M\subset N,\alpha)^\omega$.
\end{proof}

We recall the standard corner identity used below.

\begin{proposition}\label{corner relative commutant}
Let $P\subset Q$ be an inclusion of von Neumann algebras and let $e\in P$
be a projection. Then
\[
(P'\cap Q)e=(ePe)'\cap eQe.
\]
\end{proposition}

\begin{proposition}\label{corner relative action bicentralizer}
Let $e\in M^\alpha$ be a nonzero projection and let
$\alpha^e:G\curvearrowright eMe$ be the corner action. Then
\[
\rB(M\subset N,\alpha)e
=\rB(eMe\subset eNe,\alpha^e).
\]
\end{proposition}

\begin{proof}
The definition gives
$\rB(M\subset N,\alpha)\subset(M^\alpha)'\cap N$. Put
$P=M^{\omega,\alpha}$. Then
\[
(eMe)^{\omega,\alpha^e}=ePe
\]
and Proposition~\ref{relative bicentralizer with ultrapower} gives
\[
\rB(eMe\subset eNe,\alpha^e)=(ePe)'\cap eNe.
\]
This gives one inclusion. Conversely,
Proposition~\ref{corner relative commutant} gives
\[
(ePe)'\cap eN^\omega e
=\bigl(P'\cap N^\omega\bigr)e
\subset\rB(M\subset N,\alpha)^\omega e.
\]
If $x$ belongs to the relative bicentralizer of the corner, write
$x=(a_i e)^\omega$ with $a_i\in\rB(M\subset N,\alpha)$ and take the weak$^*$
ultralimit $a$ of $(a_i)_i$. Then $ae=x$.
\end{proof}

\begin{proposition}\label{equivalence projections relative bicentralizer}
Let $e,f\in M^\alpha$ be projections. Then $e$ and $f$ are equivalent in
$M^{\omega,\alpha}$ if and only if
\[
\rE_{\rB(M\subset N,\alpha)}(e)
=\rE_{\rB(M\subset N,\alpha)}(f).
\]
\end{proposition}

\begin{proof}
Put $P=M^{\omega,\alpha}$ and $Q=P'\cap N^\omega$, and let $T$ be the
center-valued trace of $P$. The restriction to $P$ of the trace-preserving
expectation $\rE_Q:N^\omega\to Q$ is $T$. By
Proposition~\ref{relative bicentralizer with ultrapower},
$Q\subset\rB(M\subset N,\alpha)^\omega$. Hence, for $x\in M^\alpha$,
\[
T(x)=\rE_Q(x)
=\rE_Q(\rE_{\rB(M\subset N,\alpha)}(x))
=\rE_{\rB(M\subset N,\alpha)}(x).
\]
The last equality holds because the constant element on the right belongs
to $Q$. The result follows from the center-valued trace criterion for
equivalence of projections in a finite von Neumann algebra.
\end{proof}

\subsection{The algebraic bicentralizer}
Throughout this subsection, $M\subset N$ is an inclusion of von Neumann
algebras, $G$ is a locally compact group, and
$\alpha:G\curvearrowright M$ is a continuous action. We introduce a
generalization of the algebraic bicentralizer from
\cite[Section~6]{Ma25}. The construction does not require $\alpha$ to
extend to $N$.

The guiding observation is that, in the tracial setting, the analytic
bicentralizer must contain every $x\in N$ satisfying
\[
xy=\alpha_g(y)x\qquad(y\in M)
\]
for some $g\in G$: almost invariant unitaries asymptotically commute
with every such $x$. This includes, in particular, any unitary in $N$
implementing $\alpha_g$ on $M$.

These intertwiners can be read as matrix coefficients of a relative
commutant. Choose $g_1,\ldots,g_n\in G$, and write
\[
D(y)=
\begin{pmatrix}
\alpha_{g_1}(y)&0&\cdots&0\\
0&\alpha_{g_2}(y)&\cdots&0\\
\vdots&\vdots&\ddots&\vdots\\
0&0&\cdots&\alpha_{g_n}(y)
\end{pmatrix},
\qquad
X=
\begin{pmatrix}
x_{11}&x_{12}&\cdots&x_{1n}\\
x_{21}&x_{22}&\cdots&x_{2n}\\
\vdots&\vdots&\ddots&\vdots\\
x_{n1}&x_{n2}&\cdots&x_{nn}
\end{pmatrix}.
\]
The condition $X\in D(M)'\cap(\mathbb M_n(\C)\ovt N)$ is equivalent to
\[
x_{ij}\alpha_{g_j}(y)=\alpha_{g_i}(y)x_{ij}
\qquad(y\in M,\ 1\leq i,j\leq n).
\]
Taking coefficients of this relative commutant therefore collects
all intertwiners between the selected copies of $M$.

For a discrete group, we can put all these copies on the diagonal of
$\B(\ell^2(G))\ovt N$ and use ordinary matrix coefficients.
For a nondiscrete group, detecting only the intertwiners attached to
individual group elements is too restrictive. There are no point-mass
vectors in $\rL^2(G)$, so operators on this space cannot be described
by matrix entries indexed by individual points of $G$.
We therefore use a continuous diagonal embedding and take the operator
coefficients of its commutant.

Denote by the same symbol $\alpha$ the normal $*$-homomorphism
\[
\alpha:M\longrightarrow\rL^\infty(G)\ovt M
\subset\B(\rL^2(G))\ovt N,
\qquad
\alpha(y)(g)=\alpha_{g^{-1}}(y).
\]
It acts by pointwise multiplication and is the continuous analogue
of the finite diagonal embeddings above. We use inverse parameters
for compatibility with the regular representations.
The analogue of a matrix coefficient is now a normal slice
$(\varphi\otimes\id)(X)\in N$, where
$\varphi\in\B(\rL^2(G))_*$ and
\[
X\in\alpha(M)'\cap\bigl(\B(\rL^2(G))\ovt N\bigr).
\]
The algebra generated by all these coefficients is the algebraic
bicentralizer. Equivalently, it is characterized as follows.

\begin{definition}
The \emph{relative algebraic bicentralizer} of the inclusion
$M\subset N$ is the smallest von Neumann subalgebra $B\subset N$ such that
\[
\alpha(M)'\cap\bigl(\B(\rL^2(G))\ovt N\bigr)
\subset \B(\rL^2(G))\ovt B.
\]
It is denoted by $\rb(M\subset N,\alpha)$. When $M=N$, we write simply
$\rb(M,\alpha)$.
\end{definition}

Writing
\[
\cC(M\subset N,\alpha)
=\alpha(M)'\cap\bigl(\B(\rL^2(G))\ovt N\bigr),
\]
the coefficient description is
\[
\rb(M\subset N,\alpha)
=\left\{(\varphi\otimes\id)(X)\ \middle|\
\varphi\in\B(\rL^2(G))_*,\ X\in\cC(M\subset N,\alpha)\right\}''.
\]
Let $\lambda,\rho:G\to\cU(\rL^2(G))$ denote the left and right
regular representations, respectively. Our convention gives
\[
\lambda_g\alpha(y)\lambda_g^*=\alpha(\alpha_g(y))
\qquad(g\in G,\ y\in M).
\]

\begin{proposition}\label{algebraic bicentralizer span characterization}
We have
\[
\B(\rL^2(G))\ovt\rb(M\subset N,\alpha)
=\overline{\operatorname{span}}^{\mathrm{w}^*}
\left\{\lambda_g X\ \middle|\
g\in G,\ X\in\cC(M\subset N,\alpha)\right\}.
\]
\end{proposition}

\begin{proof}
Put
\[
\cA=\overline{\operatorname{span}}^{\mathrm{w}^*}
\left\{\lambda_g X\ \middle|\
g\in G,\ X\in\cC(M\subset N,\alpha)\right\}.
\]
Since
\[
\lambda_g\cC(M\subset N,\alpha)\lambda_g^*
=\cC(M\subset N,\alpha)
\]
for every $g\in G$, the linear span defining $\cA$ is a $*$-algebra. Thus $\cA$ is a von Neumann algebra.

The algebra $\cC(M\subset N,\alpha)$ contains
$\rL^\infty(G)\otimes1$, and $\cA$ contains every $\lambda_g$. Since
\[
\bigl(\rL^\infty(G)\cup\{\lambda_g\mid g\in G\}\bigr)''
=\B(\rL^2(G)),
\]
we have $\B(\rL^2(G))\otimes1\subset\cA$. Hence there is a unique von
Neumann subalgebra $B\subset N$ such that
\[
\cA=\B(\rL^2(G))\ovt B.
\]
Since $\cC(M\subset N,\alpha)\subset\cA$, the definition gives
$\rb(M\subset N,\alpha)\subset B$. Conversely,
\[
\cC(M\subset N,\alpha)
\subset\B(\rL^2(G))\ovt\rb(M\subset N,\alpha),
\]
so every $\lambda_g X$ in the defining span of $\cA$ belongs to the same
tensor product. Therefore $B\subset\rb(M\subset N,\alpha)$, proving the
formula.
\end{proof}

The following is the analogue of \cite[Proposition~6.7]{Ma25}.

\begin{proposition}\label{algebraic bicentralizer fixed point commutant}
We have
\[
\rb(M\subset N,\alpha)\subset(M^\alpha)'\cap N.
\]
\end{proposition}

\begin{proof}
Let $a\in M^\alpha$. Then $\alpha(a)=1\otimes a$, so every
$X\in\cC(M\subset N,\alpha)$ commutes with $1\otimes a$. The same is true
of every $\lambda_g$, and hence of the weak$^*$-closed linear span in
Proposition~\ref{algebraic bicentralizer span characterization}. Therefore
$\B(\rL^2(G))\ovt\rb(M\subset N,\alpha)$ commutes with $1\otimes a$,
which proves the assertion.
\end{proof}

\begin{proposition}\label{algebraic bicentralizer intertwiners}
Let $x\in N$. Suppose that there exists $g\in G$ such that
\[
\alpha_g(y)x=xy
\qquad(y\in M).
\]
Then $x\in\rb(M\subset N,\alpha)$.
\end{proposition}

\begin{proof}
Let $\rho:G\to\cU(\rL^2(G))$ be the right regular representation, with
\[
(\rho_g\xi)(s)=\Delta(g)^{1/2}\xi(sg)
\qquad(\xi\in\rL^2(G),\ s\in G).
\]
We claim that
\[
\rho_g\otimes x\in\cC(M\subset N,\alpha).
\]
Indeed, for $y\in M$ and $s\in G$, apply the assumed relation to
$\alpha_{g^{-1}s^{-1}}(y)$ to obtain
\[
\alpha_{s^{-1}}(y)x=x\alpha_{g^{-1}s^{-1}}(y).
\]
It follows that
\[
\alpha(y)(\rho_g\otimes x)
=(\rho_g\otimes x)\alpha(y),
\]
which proves the claim.

By Proposition~\ref{algebraic bicentralizer span characterization},
\[
\rho_g\otimes x
\in\B(\rL^2(G))\ovt\rb(M\subset N,\alpha).
\]
Since $\rho_g$ is a nonzero operator, it follows that
$x\in\rb(M\subset N,\alpha)$.
\end{proof}

\begin{proposition}\label{intermediate algebraic action bicentralizer}
Let $M\subset N\subset P$ be an intermediate inclusion of von Neumann
algebras. Then
 \[
\rb(M\subset N,\alpha)
 \subset \rb(M\subset P,\alpha)\cap N.
\]
If there exists a normal conditional expectation
$\rE:P\to N$, then equality holds and $\rE$ restricts to a conditional expectation from $\rb(M\subset P,\alpha)$ onto $\rb(M\subset N,\alpha)$.

\end{proposition}

\begin{proof}
Put $\cH=\rL^2(G)$ and, for $Q\in\{N,P\}$, set
\[
\cC_Q
=\alpha(M)'\cap\bigl(\B(\cH)\ovt Q\bigr).
\]
Since $\cC_N \subset \cC_P$, we have
\[
\rb(M\subset N,\alpha)
\subset\rb(M\subset P,\alpha).
\]

Suppose that there exists a normal conditional expectation $\rE : P \rightarrow N$. Let $\widetilde{\rE}=\id_{\B(\cH)}\otimes \rE$. Since $\rE$ is $N$-bimodular and
$\alpha(M)\subset\rL^\infty(G)\ovt N$, we have
\[
\widetilde{\rE}(\cC_P)\subset\cC_N.
\]
Moreover,
\[
\widetilde{\rE}(\lambda_g X)=\lambda_g\widetilde{\rE}(X)
\qquad(g\in G,\ X\in\cC_P).
\]
Proposition~\ref{algebraic bicentralizer span characterization} therefore
implies
\[
\widetilde{\rE}\bigl(\B(\cH)\ovt\rb(M\subset P,\alpha)\bigr)
\subset\B(\cH)\ovt\rb(M\subset N,\alpha).
\]
If $x\in\rb(M\subset P,\alpha)\cap N$, then
$\widetilde{\rE}(1\otimes x)=1\otimes x$. Hence
$1\otimes x\in\B(\cH)\ovt\rb(M\subset N,\alpha)$, proving the reverse
inclusion.
\end{proof}

We next obtain the relative corner formula, corresponding to
\cite[Proposition~6.8]{Ma25}.

\begin{proposition}\label{corner algebraic action bicentralizer}
Let $e\in M^\alpha$ be a nonzero projection, and let
\[
\alpha^e:G\curvearrowright eMe
\]
be the corner action. Then
\[
\rb(eMe\subset eNe,\alpha^e)
=e\,\rb(M\subset N,\alpha)e.
\]
\end{proposition}

\begin{proof}
Put $p=1\otimes e\in\alpha(M)$. Since
\[
\alpha^e(eMe)=p\alpha(M)p,
\]
the corner relative-commutant identity gives
\begin{align*}
\cC(eMe\subset eNe,\alpha^e)
&=(p\alpha(M)p)'\cap
p\bigl(\B(\rL^2(G))\ovt N\bigr)p\\
&=p\cC(M\subset N,\alpha)p.
\end{align*}
The projection $p$ commutes with every $\lambda_g$. Hence
Proposition~\ref{algebraic bicentralizer span characterization} gives
\begin{align*}
\B(\rL^2(G))\ovt\rb(eMe\subset eNe,\alpha^e)
&=p\bigl(\B(\rL^2(G))
\ovt\rb(M\subset N,\alpha)\bigr)p\\
&=\B(\rL^2(G))\ovt
e\,\rb(M\subset N,\alpha)e.
\end{align*}
Cancelling
the first tensor factor proves the formula.
\end{proof}

The relative algebraic bicentralizer is particularly explicit when the
action on $M$ is implemented by unitaries in $N$.

\begin{proposition}\label{inner algebraic action bicentralizer}
Suppose that there is a continuous unitary representation
$U:G\to\cU(N)$ such that
\[
\alpha_g(x)=U_g x U_g^*
\qquad(g\in G,\ x\in M).
\]
Then
\[
\rb(M\subset N,\alpha)
=\bigl((M'\cap N)\vee\{U_g\mid g\in G\}\bigr)''.
\]
\end{proposition}

\begin{proof}
Let $\cH=\rL^2(G)$ and let
$\mathcal U\in\rL^\infty(G)\ovt N$ be the decomposable unitary defined by
\[
\mathcal U(g)=U_{g^{-1}}.
\]
Then
\[
\alpha(x)=\mathcal U(1\otimes x)\mathcal U^*
\qquad(x\in M),
\]
and hence
\[
\cC(M\subset N,\alpha)
=\mathcal U\bigl(\B(\cH)\ovt(M'\cap N)\bigr)\mathcal U^*.
\]
Put
\[
B_0=\bigl((M'\cap N)\vee\{U_g\mid g\in G\}\bigr)''.
\]
Since $\mathcal U\in\rL^\infty(G)\ovt B_0$, the preceding formula gives
\[
\cC(M\subset N,\alpha)\subset\B(\cH)\ovt B_0,
\]
and therefore $\rb(M\subset N,\alpha)\subset B_0$.

Conversely, we clearly have $M'\cap N \subset \rb(M \subset N,\alpha)$ and for every $g\in G$, we also have
\[
\alpha_g(y)U_g=U_g y
\qquad(y\in M),
\]
so Proposition~\ref{algebraic bicentralizer intertwiners} gives $U_g\in\rb(M\subset N,\alpha)$. Therefore
$B_0\subset\rb(M\subset N,\alpha)$.
\end{proof}

\begin{proposition} \label{crossed product of bicentralizer}
    We have
    $$\rb(M \subset M \rtimes_\alpha G,\alpha) = \rb(M,\alpha) \rtimes_\alpha G$$
    and this algebra is generated by $M' \cap (M \rtimes_\alpha G)$ and $\{ u_g \mid g \in G \}$.
\end{proposition}
\begin{proof}
    By the previous proposition $\rb(M \subset M \rtimes_\alpha G,\alpha)$ is generated by $M' \cap (M \rtimes_\alpha G)$ and $\{ u_g \mid g \in G \}$. In particular $\rb(M \subset M \rtimes_\alpha G,\alpha)$ is invariant under the dual coaction $$\Delta : M \rtimes_\alpha G \rightarrow (M \rtimes_\alpha G) \ovt \rL(G).$$
    This means that $\rb(M \subset M \rtimes_\alpha G,\alpha) = B \rtimes_\alpha G$ for some subalgebra $B \subset M$. Realizing $M \rtimes_\alpha G$ inside $\B(\rL^2(G)) \ovt M$, we see that $\B(\rL^2(G)) \ovt B$ is the von Neumann algebra generated by $B \rtimes_\alpha G$ and $\rL^\infty(G)$. But $B \rtimes_\alpha G$ is the von Neumann algebra generated by $M' \cap (M \rtimes_\alpha G)$ and $\rL(G)$. We conclude that $\B(\rL^2(G)) \ovt B$ is the von Neumann algebra generated by $\B(\rL^2(G))$ and $M' \cap (M \rtimes_\alpha G)$, which is $\B(\rL^2(G)) \ovt \rb(M,\alpha)$ by definition.
\end{proof}

\begin{proposition} \label{strong center bicentralizer}
    Suppose that $G$ is abelian. Then the algebra 
    $$(\rb(M,\alpha) \cap M^\alpha) \rtimes_\alpha G \cong (\rb(M,\alpha) \cap M^\alpha) \ovt \rL_\alpha(G)$$
    is generated by $\cZ(M \rtimes_\alpha G)$ and $\rL_\alpha(G)$. Moreover, if $\Gamma(\alpha)=\widehat{G}$, then $\rb(M,\alpha) \cap M^\alpha=\cZ(M)^\alpha$.
\end{proposition}
\begin{proof}
    The algebra generated by $\cZ(M \rtimes_\alpha G)$ and $\rL_\alpha(G)$ is invariant under the dual action of $\widehat{G}$, hence it is of the form $B \rtimes_\alpha G$ for some $\alpha$-invariant $B$. Clearly $B \subset \rb(M,\alpha)$ by Proposition \ref{crossed product of bicentralizer}. Since $B \rtimes_\alpha G$ is abelian, we also have $B \subset M^\alpha$.
\end{proof}

Finally, we have the tensor-product property corresponding to
\cite[Proposition~6.12]{Ma25}.

\begin{proposition}\label{tensor product algebraic action bicentralizer}
For $i\in\{1,2\}$, let $M_i\subset N_i$ be an inclusion of von Neumann
algebras and let $\alpha_i:G\curvearrowright M_i$ be a continuous action.
Then
\[
\rb(M_1\ovt M_2\subset N_1\ovt N_2,
\alpha_1\otimes\alpha_2)
\subset
\rb(M_1\subset N_1,\alpha_1)
\ovt
\rb(M_2\subset N_2,\alpha_2).
\]
\end{proposition}

\begin{proof}
Put $\cH=\rL^2(G)$ and
\[
\cC_i=\cC(M_i\subset N_i,\alpha_i),
\qquad
\cC=\cC(M_1\ovt M_2\subset N_1\ovt N_2,
\alpha_1\otimes\alpha_2).
\]
Put
\[
B_i=\rb(M_i\subset N_i,\alpha_i)
\qquad(i\in\{1,2\}).
\]
Since every element of $\cC$ commutes with
$\alpha_1(M_1)\otimes1$, the relative-commutant identity and
Proposition~\ref{algebraic bicentralizer span characterization} give
\[
\cC\subset\cC_1\ovt N_2
\subset\B(\cH)\ovt B_1\ovt N_2.
\]
By symmetry,
\[
\cC\subset\B(\cH)\ovt N_1\ovt B_2.
\]
Taking the intersection yields
\[
\cC\subset\B(\cH)\ovt B_1\ovt B_2.
\]
Since every $\lambda_g$ acts on the first tensor factor,
Proposition~\ref{algebraic bicentralizer span characterization} now gives
\[
\B(\cH)\ovt
\rb(M_1\ovt M_2\subset N_1\ovt N_2,
\alpha_1\otimes\alpha_2)
\subset\B(\cH)\ovt B_1\ovt B_2.
\]
Cancelling $\B(\cH)$ proves the result.
\end{proof}

\subsection{The bicentralizer conjecture}

\begin{proposition}
Let $(M,\tau) \subset (N,\tau)$ be a tracial inclusion. We have
\begin{equation}\label{algebraic bicentralizer contained in analytic}
\rb(M\subset N,\alpha)\subset\rB(M\subset N,\alpha).
\end{equation}
\end{proposition}
\begin{proof}
Indeed, let $(v_i)_i$ be a net in $\cU(M)$ which is asymptotically
$\alpha$-invariant, uniformly on compact subsets of $G$. On
$\rL^2(G)\otimes\rL^2(N)$, one has
\[
\alpha(v_i)-1\otimes v_i\longrightarrow0
\quad\text{strongly}^*.
\]
To check this, first test on tensors whose first coordinate has compact
support and whose second coordinate belongs to $N$. Compact uniform
invariance gives convergence on these vectors, and their density and
the uniform operator norm bound give the assertion. Every
$X\in\cC(M\subset N,\alpha)$ commutes with $\alpha(v_i)$, so
$[X,1\otimes v_i]\to0$ strongly$^*$. Taking normal slices on the first
tensor factor shows that each slice of $X$ asymptotically commutes
with $v_i$ in $\|\cdot\|_2$. Such slices therefore belong to
$\rB(M\subset N,\alpha)$. They generate the relative algebraic
bicentralizer, proving \eqref{algebraic bicentralizer contained in analytic}.
\end{proof}

We formulate a relative version of the bicentralizer conjecture
discussed after Conjecture~\ref{letterconj classification}.
\begin{conjecture}
Let $(M,\tau) \subset (N,\tau)$ be a tracial inclusion of von Neumann algebras. Let $\alpha : G \curvearrowright (M,\tau)$ be a continuous trace-preserving action of a locally compact group $G$. Suppose that $G$ is amenable. Then
$$ \rB(M \subset N,\alpha)=\rb(M \subset N,\alpha).$$
\end{conjecture}

The following result extends the discrete-group argument of
\cite[Lemmas~4.2 and~4.3]{PSV20}.  The compact open subgroup allows us
to combine Popa's free-independence argument with the discrete
homogeneous space $G/K$.

\begin{theorem}\label{compact open amenable trivial bicentralizer}
Let $G$ be an amenable locally compact group admitting a
compact open subgroup $K<G$.  Let $M$ be a $\II_1$ factor and let $\alpha:G\curvearrowright M$ be a strictly outer action.  Then
\[
 \rB(M,\alpha)=\C 1.
\]
\end{theorem}

\begin{proof}
We first construct a free Bernoulli shift inside $M^\omega$.  We claim
that there exists a Haar unitary
\[
 u\in\cU((M^K)^\omega)
\]
such that the family
\[
 M,\qquad \alpha_g^\omega(u)\quad(gK\in G/K)
 \tag{\(*\)}
\]
is freely independent.  Notice that $\alpha_g^\omega(u)$ depends only
on the coset $gK$, since $u$ is fixed by $K$.

We explain the matrix argument proving the claim.  Let
$g_1K,\ldots,g_nK$ be distinct cosets and consider
\[
 \Delta:M^K\longrightarrow \mathbb M_n(M),
 \qquad
 \Delta(x)=\operatorname{diag}
 (\alpha_{g_1}(x),\ldots,\alpha_{g_n}(x)).
\]
Strict outerness implies that $\alpha|_K$ is minimal and that $\alpha$
is properly outer relative to $K$ by \cite[Lemma~1.6]{Ni26}.  We
therefore have
\[
 \Delta(M^K)'\cap \mathbb M_n(M)=D_n(\C).
 \tag{\(**\)}
\]
Indeed, if $a=(a_{ij})$ belongs to the relative commutant, then
\[
 a_{ij}\alpha_{g_j}(x)=\alpha_{g_i}(x)a_{ij}
 \qquad(x\in M^K).
\]
After applying $\alpha_{g_j}^{-1}$, a nonzero off-diagonal entry would
intertwine the identity representation of $M^K$ with
$\alpha_{g_j^{-1}g_i}|_{M^K}$.  Relative proper outerness would then
give $g_j^{-1}g_i\in K$, contrary to the choice of the cosets.  The
diagonal entries are scalar because $(M^K)'\cap M=\C 1$.

We apply \cite[Theorem~0.1(a)]{Po14}, or its amalgamated formulation
\cite[Corollary~4.4(1)]{Po14}, to this inclusion. Its hypotheses hold
because $\Delta(M^K)$ is a diffuse factor and its relative commutant
$D_n(\C)$ is finite dimensional. Thus we can find a Haar unitary
\[
 v\in\cU((M^K)^\omega)
\]
such that $\Delta(v)$ is freely independent from $\mathbb M_n(M)$ relative to
$D_n(\C)$.  More precisely, inside $\mathbb M_n(M^\omega)$, the
algebras
\[
 \rW^*(D_n(\C),\Delta(v))
 \quad\text{and}\quad \mathbb M_n(M)
\]
are free with amalgamation over $D_n(\C)$ with respect to the diagonal
conditional expectation.

This relative freeness implies that
\[
 M,\quad \alpha_{g_1}^\omega(v),\ldots,
 \alpha_{g_n}^\omega(v)
 \tag{\(***\)}
\]
are freely independent in the usual scalar-valued sense. Put
$V=\Delta(v)$ and $v_i=\alpha_{g_i}^\omega(v)$.
Consider a cyclically reduced word
\[
w=a_1v_{i_1}^{k_1}\cdots a_rv_{i_r}^{k_r},
\qquad k_j\in\Z\setminus\{0\}.
\]
Each $a_j$ is either a centered element of $M$ or $1$. In the
latter case, reduction means that $i_{j-1}\ne i_j$, with
$i_0=i_r$. Set $b_j=e_{i_{j-1},i_j}\otimes a_j$. Then
\[
\rE_{D_n(\C)}(b_j)=0,
\qquad
\rE_{D_n(\C)}(V^{k_j})=0.
\]
The first equality follows from the zero scalar trace of a diagonal
coefficient or from its off-diagonal position. The second follows
because $v$ is Haar. Matrix multiplication gives
\[
b_1V^{k_1}\cdots b_rV^{k_r}=e_{i_r,i_r}\otimes w.
\]
Amalgamated freeness therefore gives $\tau(w)=0$. Traciality and
induction on word length reduce arbitrary centered reduced words
to these cyclic words and the individual Haar moments. This proves
\((***)\). Diagonalizing the finite moment tests over the finite subsets of
$G/K$ and the finite subsets of the unit ball of $M$ for $\|\cdot\|_2$ gives a
Haar unitary $u\in(M^K)^\omega$ satisfying \((*)\).

For $gK\in G/K$, put
\[
 P_{gK}=\alpha_g^\omega(uMu^*)=\alpha_g^\omega(u) M \alpha_g^{\omega}(u)^*.
\]
This is well defined because $u$ is fixed by $K$.  The freeness in
\((*)\) implies that the family consisting of $M$ and the algebras
$P_{gK}$ is freely independent.
Put
\[
 P=\bigvee_{s\in G/K}P_s\subset M^\omega.
\]
Then $P$ is freely independent from $M$. Since $K$ is
open and $u$ is $K$-fixed, this algebra is contained in $M_\alpha^\omega$.
The action permutes the freely independent factors in
\[
 P\cong *_{s\in G/K}(M,\tau).
\]
Since $(M^K)'\cap M=\C 1$, the algebra $M^K$ is a diffuse factor.
Choose a centered semicircular element $a=a^*\in M^K$ of variance
one and put
\[
 a_{gK}=\alpha_g^\omega(uau^*)\in P_{gK}.
\]
Both $u$ and $a$ are $K$-fixed, so this definition is independent
of the representative of $gK$ and
$\alpha_h^\omega(a_{gK})=a_{hgK}$.
Let $L\Subset G$ and let $\varepsilon>0$. Amenability,
in the Følner form for the discrete homogeneous space $G/K$, provides
a nonempty finite set $S\subset G/K$ such that
\[
 \sup_{g\in L}\frac{|gS\mathbin\triangle S|}{|S|}<\varepsilon^2.
\]
The element
\[
 b_S=|S|^{-1/2}\sum_{s\in S}a_s
\]
is again a centered semicircular element of variance one, and freeness
gives
\[
 \sup_{g\in L}\|\alpha_g^\omega(b_S)-b_S\|_2^2
 =\sup_{g\in L}\frac{|gS\mathbin\triangle S|}{|S|}
 <\varepsilon^2.
\]
Applying the usual diagonal selection to the elements $b_S$, using the
preceding compact-uniform estimate together with their joint moments
with finite subsets of $M$, we obtain an element
\[
 b\in M^{\omega,\alpha}
\]
which is a centered semicircular element of variance one and is freely
independent from the constant copy of $M$. But, by Proposition~\ref{relative bicentralizer with ultrapower}, $b$ must commute with $\rB(M,\alpha)$. This forces $\rB(M,\alpha)=\C 1$.
\end{proof}

We also verify the conjecture for free Bogoljubov actions.

\begin{theorem}\label{amenable Bogoljubov bicentralizer}
Let $G$ be an amenable locally compact group and let
$\pi:G\to\cO(H_{\R})$ be a faithful continuous orthogonal representation
with $\dim H_{\R}\geq2$. Let
$\alpha:G\curvearrowright M=\Gamma(H_{\R})''$
be the associated free Bogoljubov action. Then
\[
\rB(M,\alpha)=\C 1.
\]
\end{theorem}

\begin{proof}
Suppose first that $H_{\R}$ is finite dimensional. Put
$K=\overline{\pi(G)}\subset\cO(H_{\R})$.
The free Bogoljubov action of $K$ on $M$ is strictly outer by
\cite[Corollary~B(1)]{MV23}. Since $K$ is compact, this action is
minimal by \cite[Proposition~6.2]{Va01}. Thus
\[
\rB(M,\alpha)\subset(M^\alpha)'\cap M
=(M^K)'\cap M=\C 1.
\]

Now suppose that $H_{\R}$ is infinite dimensional, and put
$H_{\C}=H_{\R}\otimes_{\R}\C$.
Choose a Reiter net $(f_i)_i$ of nonnegative functions in
$\rL^1(G)$ with integral one, so that
\[
\sup_{s\in K}\|f_i(s^{-1}\,\cdot\,)-f_i\|_1\longrightarrow0
\qquad(K\Subset G).
\]
For each $n\geq1$, choose a real orthogonal projection $P_n$
of rank $n$, and define
\[
A_{i,n}=\int_G f_i(g)\pi_g\frac{P_n}{n}\pi_g^*\,\rd g.
\]
These positive trace-class operators satisfy
\[
\Tr(A_{i,n})=1,\qquad \|A_{i,n}\|\leq n^{-1},
\]
and
\[
\sup_{s\in K}\|\pi_sA_{i,n}\pi_s^*-A_{i,n}\|_1
\leq\sup_{s\in K}\|f_i(s^{-1}\,\cdot\,)-f_i\|_1.
\]
They commute with the real conjugation on $H_{\C}$.
On the product directed set, write $T_j=A_{i,n}^{1/2}$.
The Powers--St{\o}rmer inequality gives
\[
\|T_j\|_{\mathrm{HS}}=1,\qquad
\|T_j\|\longrightarrow0,\qquad
\sup_{s\in K}\|\pi_sT_j\pi_s^*-T_j\|_{\mathrm{HS}}
\longrightarrow0.
\]
Diagonalize $T_j$ over $H_{\R}$ as
\[
T_j=\sum_k\lambda_{j,k}
\langle\,\cdot\,,e_{j,k}\rangle e_{j,k},
\]
where $\lambda_{j,k}\geq0$ and $(e_{j,k})_k$ is orthonormal. Put
\[
\eta_j=\sum_k\lambda_{j,k}e_{j,k}\otimes e_{j,k},
\qquad
z_j=W(\eta_j)
=\sum_k\lambda_{j,k}\bigl(s(e_{j,k})^2-1\bigr).
\]
Here $s(e)=\ell(e)+\ell(e)^*$, where $\ell(e)$ is the left
creation operator on the full Fock space. The Wick formula gives
\[
\begin{aligned}
z_j&=C_j+C_j^*+D_j,\\
C_j&=\sum_k\lambda_{j,k}\ell(e_{j,k})^2,\\
D_j&=\sum_k\lambda_{j,k}\ell(e_{j,k})\ell(e_{j,k})^*.
\end{aligned}
\]
Orthogonality of the creation ranges gives
$\|C_j\|=\|T_j\|_{\mathrm{HS}}=1$ and
$\|D_j\|=\|T_j\|\leq1$. The series converge in operator norm,
so $\|z_j\|\leq3$. Vectorization intertwines conjugation on
Hilbert--Schmidt operators with $\pi\otimes\pi$. Consequently,
\[
\sup_{g\in K}\|\alpha_g(z_j)-z_j\|_2
=\sup_{g\in K}\|\pi_gT_j\pi_g^*-T_j\|_{\mathrm{HS}}
\longrightarrow0.
\]

Let $\zeta=h_1\otimes\cdots\otimes h_m$, where $m\geq1$ and
$h_1,\ldots,h_m\in H_{\R}$. The free Wick product formula gives
\[
W(\zeta)z_j\Omega=\zeta\otimes\eta_j+o(1),
\qquad
z_jW(\zeta)\Omega=\eta_j\otimes\zeta+o(1)
\]
in Fock space norm. Indeed, contracting one leg of $\eta_j$ against
a fixed vector $h$ gives $T_jh$, whose norm is at most
$\|T_j\|\|h\|$. Contracting both legs gives a scalar bounded by
$\|T_j\|\|h\|\|k\|$. Every nonempty boundary contraction therefore
tends to zero. For any two fixed elementary tensors $\zeta,\zeta'$
of the same positive degree, contraction of the concatenations gives
\[
\langle\zeta\otimes\eta_j,\eta_j\otimes\zeta'\rangle
\longrightarrow0.
\]
For degree one this follows from $\|T_j^2\|\to0$, and for higher
degrees each of the two boundary contractions contributes a factor
bounded by $\|T_j\|$. The same conclusions hold for finite complex
linear combinations of real elementary tensors. Orthogonality of
distinct homogeneous degrees now gives, for every Wick polynomial $a$,
\[
\lim_j\|az_j-z_ja\|_2^2
=2\|a-\tau(a)1\|_2^2.
\]
Since $\|z_j\|\leq3$, approximation in $\rL^2(M)$ extends this
identity to every $a\in M$. If $a\in\rB(M,\alpha)$, the bounded
almost invariant net $(z_j)_j$ satisfies
$\|az_j-z_ja\|_2\to0$. Hence $a=\tau(a)1$.
\end{proof}

\section{Connection with the Rokhlin property}
\label{sec:rokhlin-property}

We characterize triviality of the bicentralizer and the Rokhlin property
through weak containment of equivariant correspondences. The main
consequence is that an action of an amenable locally compact group on an
amenable $\II_1$ factor has the Rokhlin property precisely when it is
faithful and has trivial bicentralizer
(Theorem~\ref{Rokhlin and trivial bicentralizer}). We also identify the
Rokhlin property with faithfulness of the induced action on the
equicontinuous central sequence algebra.

The proof uses the ultrapower implementation technique to translate weak containment
into representations carried by central sequences. Induction and Fell
absorption then allow us to compare the two properties, with amenability
of the group and of the factor entering in opposite directions.

\subsection{Equivariant correspondences}
We introduce the theory of equivariant correspondences. A detailed exposition is given in \cite{DCDR24} but since it is formulated in the language of quantum groups, we give a self-contained exposition here. See also \cite{KLP10} and \cite{BCM18}.

Let $\alpha:G\curvearrowright M$ and $\beta:G\curvearrowright N$ be
two continuous actions of a locally compact group $G$ on von Neumann algebras $M$ and $N$. An equivariant $\alpha$--$\beta$--correspondence is a
quadruple
\[
 \Theta=(H,\lambda,\rho,U),
\]
where $\lambda :M\to\B(H)$ and
$\rho:N^{\op}\to\B(H)$ are commuting normal unital
representations and $U:G\to\cU(H)$ is a strongly
continuous representation such that
\[
 \begin{aligned}
 U_{g}\lambda(x)U_{g}^*
 &=\lambda(\alpha_g(x)),\\
 U_{g}\rho(y^{\op})U_{g}^*
 &=\rho(\beta_g(y)^{\op})
 \end{aligned}
\]
for $x\in M$, $y\in N$ and $g\in G$.  We denote the category of all $\alpha$--$\beta$--correspondences by
$\Cor(\alpha,\beta)$.

For an action $\alpha:G\curvearrowright M$, define the identity equivariant correspondence $\cI_\alpha \in \Cor(\alpha,\alpha)$ by
\[
 \cI_\alpha=(\rL^2(M),\lambda_M,\rho_M,U^\alpha)
\]
where $U^\alpha$ is the standard unitary implementation of $\alpha$ on $\rL^2(M)$.
Define the semicoarse equivariant correspondence $\cC_\alpha \in \Cor(\alpha,\alpha)$ by
\[
 \cC_\alpha=
 (\rL^2(M)\otimes\rL^2(M),\lambda_M\otimes1,
  1\otimes\rho_M,U^{\alpha}\otimes_G U^{\alpha}).
\]

If $\pi:G\to\cU(K)$ is a unitary representation and
$\Theta=(H,\lambda,\rho,U) \in \Cor(\alpha,\beta)$, we define the tensor product $\pi\otimes_G\Theta \in \Cor(\alpha,\beta)$ by
\[
 \pi\otimes_G\Theta
 =(K\otimes H,1\otimes\lambda,1\otimes\rho,
   \pi \otimes_G U).
\]

If $\Theta=(H,\lambda_H,\rho_H,U)\in\Cor(\alpha,\beta)$ and
$\Xi=(K,\lambda_K,\rho_K,V)\in\Cor(\beta,\gamma)$, their
\emph{Connes fusion} is the equivariant correspondence
\[
 \Theta\otimes_N\Xi
 =(H\otimes_N K,\lambda_H\otimes1,1\otimes\rho_K,
   U\otimes_N V)\in\Cor(\alpha,\gamma).
\]
Here $H\otimes_N K$ is the usual relative tensor product, and the
common action $\beta$ induces the diagonal implementation. When $N$
is tracial and $\beta$ preserves its trace, it is given on bounded
elementary tensors by
\[
 (U\otimes_N V)_g(\xi\otimes_N\eta)
 =U_g\xi\otimes_N V_g\eta.
\]
We use the usual associativity and identity identifications for fusion,
which intertwine the group actions. See
\cite[Section~5.2, especially Proposition~5.6]{DCDR24}.

\subsection{Induction and Fell absorption principle}

Let $\Cor(M,N)$ denote the category of ordinary $M$--$N$
correspondences.  Forgetting $U$ defines the restriction functor
\[
 \Res : \Cor(\alpha,\beta)\longrightarrow\Cor(M,N) :
 (H,\lambda,\rho,U)\longmapsto(H,\lambda,\rho).
\]
In particular, $\Res(\cI_\alpha)=\cI_M$ and
$\Res(\cC_\alpha)=\cC_M$.

For an ordinary $M$--$N$ correspondence
$(H,\lambda,\rho) \in \Cor(M,N)$, define
\[
 \lambda^\alpha : M\longrightarrow\B(\rL^2(G,H)) :
 a\longmapsto
 \left(\xi\longmapsto
 \bigl(g\longmapsto\lambda(\alpha_{g^{-1}}(a))\xi(g)\bigr)\right),
\]
\[
 \rho^\beta : N^{\op}\longrightarrow\B(\rL^2(G,H)) :
 b^{\op}\longmapsto
 \left(\xi\longmapsto
 \bigl(g\longmapsto\rho(\beta_{g^{-1}}(b)^{\op})\xi(g)\bigr)\right),
 \]
and
\[
 \lambda_G \otimes 1 : G\longrightarrow\cU(\rL^2(G,H)) :
 g\longmapsto
 \left(\xi\longmapsto\bigl(h\longmapsto\xi(g^{-1}h)\bigr)\right).
\]
The resulting equivariant correspondence defines the induction
functor
\[
 \Ind : \Cor(M,N)\longrightarrow\Cor(\alpha,\beta) :
 (H,\lambda,\rho)\longmapsto
 (\rL^2(G,H),\lambda^\alpha,\rho^\beta,\lambda_G \otimes 1).
\]

We have the following correspondence version of the Fell absorption principle.

\begin{proposition}
\label{equivariant Fell absorption}
For every $\Theta=(H,\lambda,\rho,U)\in\Cor(\alpha,\beta)$, there is a unitary equivalence
\[
 \lambda_G\otimes_G\Theta\cong\Ind(\Res\Theta).
\]
\end{proposition}

\begin{proof}
Identify $\rL^2(G)\otimes H$ with
$\rL^2(G,H)$ and define
\[
 V : \rL^2(G)\otimes H \longrightarrow
 \rL^2(G,H) :
 \xi\longmapsto\bigl(g\longmapsto U_{g}^*\xi(g)\bigr).
\]
The covariance relations give
\[
 V(1\otimes\lambda(a))V^*=\lambda^\alpha(a),
 \qquad
 V(1\otimes\rho(b^{\op}))V^*=\rho^\beta(b^{\op}),
\]
and a direct computation gives
\[
 V(\lambda_G(g)\otimes U_{g})V^*=\lambda_G(g) \otimes 1.
\]
Thus $V$ implements the asserted equivalence.
\end{proof}

\subsection{Weak containment}

Let $\rW^*(\alpha,\beta)$ be the universal von Neumann algebra generated by $M \odot N^{\op}$ and a continuous unitary representation $u : G \rightarrow \rW^*(\alpha,\beta)$ such that the maps $x \mapsto x \odot 1$ and $y \mapsto 1 \odot y^{\op}$ are normal on $M$ and $N$ and
$$ u_g(x \odot y^{\op})u_g^*=\alpha_g(x) \odot \beta_g(y)^{\op}$$
for all $x \in M$, $y \in N$ and $g \in G$. 
For every $\alpha$--$\beta$--correspondence $\Theta=(H,\lambda,\rho,U)$ there is a unique normal representation 
$$\Pi_{\Theta} : \rW^*(\alpha,\beta) \rightarrow \B(H)$$
such that $$\Pi_{\Theta}(x \odot y^{\op})=\lambda(x)\rho(y^{\op}), \quad x \in M, \; y \in N,$$ and $$ \Pi_\Theta(u_g)=U_g, \quad g \in G.$$

Put
\[
 \begin{aligned}
 M_\alpha&=\{x\in M\mid g\mapsto\alpha_g(x)
 \text{ is norm continuous}\},\\
 N_\beta&=\{y\in N\mid g\mapsto\beta_g(y)
 \text{ is norm continuous}\}.
 \end{aligned}
\]
Then we have a natural $*$-morphism with weak$^*$-dense range
$$\bigl(M_\alpha\otimes_{\max}N_\beta^{\op}\bigr)
   \rtimes_{\alpha\otimes_G\beta^{\op}}G \rightarrow \rW^*(\alpha,\beta).$$
The range of this $*$-morphism is denoted by $\rC^*(\alpha,\beta)$. By definition, $\rC^*(\alpha,\beta)$ is a $\rC^*$-algebra and it is the norm-closed linear span of $$ \{ (a\odot b^{\op}) \cdot u(f) \mid a \in M_\alpha, b \in N_\beta \text{ and } f \in \rL^1(G) \}.$$

We will also need the following smoothing lemma. Note that the whole point of this lemma is that we do not suppose that $a \in M_\alpha$.
\begin{lemma} \label{producing element C* algebra}
    For every $a \in M$ and every $f, h \in \rL^1(G)$, we have $$u(f) (a \odot 1) u(h) \in \rC^*(\alpha,\beta).$$
    The same is true if we replace $a \odot 1$ by $1 \odot b^{\op}$ for $b \in N$.
\end{lemma}
\begin{proof}
  Fix $a \in M$.  The map 
  $$(f,h) \mapsto (\alpha_f(a) \odot 1)u(h)$$ from $\rL^1(G) \times \rL^1(G)$ to $\rC^*(\alpha,\beta)$ is bilinear and jointly continuous since
  $$ \| (\alpha_f(a) \odot 1)u(h)\| \leq \| \alpha_f(a)\| \|u(h)\| \leq \|a\| \|f\|_1 \|h\|_1.$$
  Therefore, there is a unique continuous linear map $\Phi$ from the projective tensor product $\rL^1(G) \widehat{\otimes} \rL^1(G) =\rL^1(G \times G)$ into $\rC^*(\alpha,\beta)$ such that
$$ \Phi(f \otimes h) =  (\alpha_f(a) \odot 1)u(h) $$
for all $f,h \in \rL^1(G)$. Then we have
$$ \Phi(F) = \int_{G \times G} (\alpha_g(a) \odot 1)u_k \: F(g,k) \: \rd g \, \rd k$$
for all elementary tensors $F \in \rL^1(G) \otimes \rL^1(G)$ hence also for all $F \in \rL^1(G \times G)$ since both sides are continuous on $\rL^1(G \times G)$. In particular, the right-hand side is in $\rC^*(\alpha,\beta)$. Taking $F : (g,k) \mapsto f(g)h(g^{-1}k)$ for $f,h \in \rL^1(G)$ gives the desired conclusion.
\end{proof}

For $\Theta_1,\Theta_2 \in \Cor(\alpha,\beta)$, we thus have the following natural notion of weak containment
\[
 \Theta_1\prec\Theta_2
 \quad\Longleftrightarrow\quad 
 \Pi_{\Theta_1}|_{\rC^*(\alpha,\beta)} \prec\Pi_{\Theta_2}|_{\rC^*(\alpha,\beta) }.
\]

The following proposition provides a very convenient criterion for weak containment. It is the classical group analog of \cite[Corollary~3.8]{DCDR24}.

\begin{proposition}
\label{binormal state extension}
Let $\alpha : G \curvearrowright M$ and $\beta : G \curvearrowright N$ be two continuous actions and take $\Theta_i=(H_i,\lambda_i,\rho_i,U_i) \in \Cor(\alpha,\beta)$. Take a unit vector $\xi \in H_1$ and consider the following properties.
\begin{enumerate}[\rm (i)]
\item $\Theta_1\prec\Theta_2$.
\item There is a unital completely positive map
$\vartheta:\B(H_2)\to\B(H_1)$ such that
\[
 \vartheta(\lambda_2(a))=\lambda_1(a),\quad
 \vartheta(\rho_2(b^{\op}))=\rho_1(b^{\op}),\quad
 \vartheta(U_2(f))=U_1(f)
\]
for all $a\in M$, $b\in N$ and $f\in\rC^*(G)$.

    \item There exists a state $\Phi \in \B(H_2)^*$ such that $$\Phi( \lambda_2(a)\rho_2(b^{\op}) ) =\langle \lambda_1(a) \rho_1(b^{\op}) \xi, \xi \rangle $$ 
    and 
    $$ \Phi( U_2(f))=\langle U_1(f)\xi,\xi \rangle$$
    for all $a$ in $M$, $b \in N$ and $f \in \rC^*(G)$. 
\end{enumerate}
Then we have $(\rm i) \Leftrightarrow (\rm ii) \Rightarrow (\rm iii)$ and if $\xi$ is a cyclic vector for $\Theta_1$ that is fixed by $U_1$, then $(\rm iii) \Leftrightarrow (\rm ii) \Leftrightarrow (\rm i)$.
\end{proposition}

\begin{proof}[Proof of Proposition \ref{binormal state extension}]
$(\rm i) \Rightarrow (\rm ii)$. Suppose that $\Theta_1\prec\Theta_2$.  Kernel inclusion gives a
nondegenerate $*$-homomorphism  
\[
 q : \mathrm{Im}(\Pi_{\Theta_2}|_{\rC^*(\alpha,\beta)})\longrightarrow
   \mathrm{Im}(\Pi_{\Theta_1}|_{\rC^*(\alpha,\beta)})
\]
such that $$q \circ \Pi_{\Theta_2}|_{\rC^*(\alpha,\beta)} =\Pi_{\Theta_1}|_{\rC^*(\alpha,\beta)}.$$
After unitization, Arveson's extension theorem extends $q$ to a ucp
map $\vartheta:\B(H_2)\to\B(H_1)$.

For every $a \in M$ and $f,h \in \rL^1(G)$, we have
$$  U_i(f) \lambda_i(a) U_i(h) = \Pi_{\Theta_i}\left( u(f)(a \odot 1) u(h) \right)$$

This means that 
$$ q\left( U_2(f) \lambda_2(a) U_2(h) \right) = U_1(f) \lambda_1(a) U_1(h).$$
Since $U_2(f)$ and $U_2(h)$ are in the multiplicative domain of $\vartheta$, we thus have 

\[
 \begin{aligned}
 U_1(f) \vartheta(\lambda_2(a)) U_1(h)
 &=\vartheta(U_2(f)\lambda_2(a)U_2(h))\\
 &=U_1(f) \lambda_1(a) U_1(h).
 \end{aligned}
\]
Now, we can make $U_i(f)$ and $U_i(h)$ converge strongly to $1$, and we conclude that $\vartheta(\lambda_2(a)) =\lambda_1(a)$. The same argument applies to show that $\vartheta(\rho_2(b^{\op})) =\rho_1(b^{\op})$ for all $b \in N$.

$(\rm ii) \Rightarrow (\rm i)$. The multiplicative-domain identities in~$(\rm ii)$ give
\[
 \vartheta\circ\Pi_{\Theta_2}|_{\rC^*(\alpha,\beta)}
 =\Pi_{\Theta_1}|_{\rC^*(\alpha,\beta)}
\]
first on the standard generators and then on all of
$\rC^*(\alpha,\beta)$.

$(\rm ii) \Rightarrow (\rm iii)$. Take $\Phi = \langle \vartheta( \cdot) \xi, \xi \rangle$.

$(\rm iii) \Rightarrow (\rm ii)$. Suppose that $\Phi$ is as
in~$(\rm iii)$ and that $\xi$ is cyclic for $\Theta_1$ and fixed by
$U_1$. The restriction of $\Phi$ to $U_2(\rC^*(G))$ is the
character induced by the trivial representation $1_G$. Hence this
algebra belongs to the multiplicative domain of $\Phi$. For
$a\in M_\alpha$, $b\in N_\beta$ and $f\in\rL^1(G)$, we obtain
\[
\begin{aligned}
 \Phi\bigl(\lambda_2(a)\rho_2(b^{\op})U_2(f)\bigr)
 &=1_G(f)\Phi\bigl(\lambda_2(a)\rho_2(b^{\op})\bigr)\\
 &=1_G(f)\langle\lambda_1(a)\rho_1(b^{\op})\xi,\xi\rangle\\
 &=\langle\lambda_1(a)\rho_1(b^{\op})U_1(f)\xi,\xi\rangle.
\end{aligned}
\]
Thus $\Phi\circ\Pi_{\Theta_2}$ agrees on $\rC^*(\alpha,\beta)$
with the coefficient of $\xi$ in $\Theta_1$. Cyclicity of $\xi$
gives $\Theta_1\prec\Theta_2$, and the implication
$(\rm i)\Rightarrow(\rm ii)$ finishes the proof.
\end{proof}

\begin{proposition}\label{tensoring preserves weak containment}
For $i=1,2$ take $\Theta_i \in \Cor(\alpha,\beta)$ and $\pi_i \in \Rep(G)$. 
If $\pi_1 \prec \pi_2$ and $\Theta_1 \prec \Theta_2$, then
\[
 \pi_1\otimes_G\Theta_1
 \prec
 \pi_2\otimes_G\Theta_2.
\]
\end{proposition}

\begin{proof}
We first record how multiplicative domains behave under passage to
multipliers. Let $\Phi:\B(H)\to\B(H')$ be a ucp map whose
restriction $q$ to a nondegenerate $\rC^*$-subalgebra
$A\subset\B(H)$ is a nondegenerate $*$-homomorphism.
Represent the multiplier algebra $\mathcal M(A)$ on $H$, and let
$\overline q$ be the nondegenerate extension of $q$ to
$\mathcal M(A)$. For $m\in\mathcal M(A)$ and $a\in A$, the
multiplicative-domain identity gives
\[
\Phi(m)q(a)=\Phi(ma)=q(ma)=\overline q(m)q(a).
\]
Nondegeneracy of $q$ implies $\Phi(m)=\overline q(m)$.
Applying this also to $m^*m$ and $mm^*$ shows that
$\mathcal M(A)$ is contained in the multiplicative domain of $\Phi$.

Fix $\pi:G\to\cU(K)$ and take $\vartheta$ as in
Proposition~\ref{binormal state extension}. The multiplicative domain of
\[
\Phi=\id\otimes\vartheta:
\B(K\otimes H_2)\longrightarrow\B(K\otimes H_1)
\]
contains $1\otimes\lambda_2(M)$, $1\otimes\rho_2(N^{\op})$,
$1\otimes U_2(\rC^*(G))$ and $\pi(\rC^*(G))\otimes1$.
In particular, its restriction to the nondegenerate algebra
\[
A=\pi(\rC^*(G))\otimes_{\min}U_2(\rC^*(G))
\subset\B(K\otimes H_2)
\]
is the nondegenerate $*$-homomorphism determined by
\[
q(\pi(f)\otimes U_2(h))=\pi(f)\otimes U_1(h)
\qquad(f,h\in\rC^*(G)).
\]
The diagonal unitaries $\pi_g\otimes U_{2,g}$ form a strictly
continuous unitary representation in $\mathcal M(A)$, and
\[
\overline q(\pi_g\otimes U_{2,g})=\pi_g\otimes U_{1,g}.
\]
Their integrated representations therefore satisfy
\[
\Phi\bigl((\pi\otimes_G U_2)(f)\bigr)
=(\pi\otimes_G U_1)(f)
\qquad(f\in\rC^*(G)).
\]
Proposition~\ref{binormal state extension} now gives
$\pi\otimes_G\Theta_1\prec\pi\otimes_G\Theta_2$.

Next fix $\Theta=(H,\lambda,\rho,U)\in\Cor(\alpha,\beta)$
and suppose that $\pi_1\prec\pi_2$.
Take a ucp map $F:\B(K_2)\to\B(K_1)$ such that
$F(\pi_2(f))=\pi_1(f)$ for every $f\in\rC^*(G)$.
Apply the same multiplier argument to $F\otimes\id$ and
$\pi_2(\rC^*(G))\otimes_{\min}U(\rC^*(G))$.
This map fixes the left and right algebra representations and sends
$(\pi_2\otimes_G U)(f)$ to $(\pi_1\otimes_G U)(f)$.
Thus $\pi_1\otimes_G\Theta\prec\pi_2\otimes_G\Theta$.
Combining the two weak containments proves the proposition.
\end{proof}
 
\begin{proposition}
\label{induction preserves weak containment}
Take $\Xi_i \in \Cor(M,N)$ for $i=1,2$ and consider their inductions $\Ind(\Xi_i) \in \Cor(\alpha,\beta)$. If $\Xi_1 \prec \Xi_2$ then $\Ind(\Xi_1) \prec \Ind(\Xi_2)$.
\end{proposition}

\begin{proof}
We use again Proposition \ref{binormal state extension}. Write
$\Xi_i=(H_i,\lambda_i,\rho_i)$. Then there exists a ucp map $\vartheta : \B(H_2) \rightarrow \B(H_1)$ such that $\vartheta \circ \lambda_2=\lambda_1$ and $\vartheta \circ \rho_2=\rho_1$. One easily checks that $$\id \otimes \vartheta : \B(\rL^2(G) \otimes H_2) \rightarrow \B(\rL^2(G) \otimes H_1)$$
witnesses the weak containment $\Ind(\Xi_1) \prec \Ind(\Xi_2)$.
\end{proof}

We will also use the compatibility of weak containment with fusion. The
general statement is \cite[Proposition~6.3]{DCDR24}, applied to ordinary
locally compact groups. We give a proof only in the tracial case used
below, by adapting the coefficient argument of \cite[Lemma~1.7]{AD95}.
\begin{proposition}\label{fusion preserves weak containment}
\label{fusion preserves weak containment both variables}
Let $\alpha : G \curvearrowright M$, $\beta : G \curvearrowright N$ and
$\gamma : G \curvearrowright P$ be continuous actions. For $i=1,2$, take
$\Theta_i \in \Cor(\alpha,\beta)$ and $\Xi_i \in \Cor(\beta,\gamma)$.
If $\Theta_1 \prec \Theta_2$ and $\Xi_1 \prec \Xi_2$, then
\[
 \Theta_1\otimes_N\Xi_1
 \prec
 \Theta_2\otimes_N\Xi_2.
\]
\end{proposition}

\begin{proof}
We prove the tracial case, in which $N$ carries a faithful normal
$\beta$-invariant tracial state $\tau$.
We first fix $\Theta=(H,\lambda_H,\rho_H,U)$ and prove that
\[
\Xi_1\prec\Xi_2
\quad\Longrightarrow\quad
\Theta\otimes_N\Xi_1\prec\Theta\otimes_N\Xi_2.
\]
A vector $\xi\in H$ is right bounded if there is a constant $C\geq0$
such that $\|\xi b\|\leq C\|b\|_2$ for every $b\in N$.
Write $H_N$ for the space of right bounded vectors and $\|\xi\|_N$
for the least such constant. For $\xi,\zeta\in H_N$, the element
$\langle\xi,\zeta\rangle_N\in N$ is characterized by
\[
\tau(b^*\langle\xi,\zeta\rangle_N c)
=\langle\zeta c,\xi b\rangle
\qquad(b,c\in N),
\]
where the scalar inner product is linear in the first variable.
Then $\|\langle\xi,\xi\rangle_N\|=\|\xi\|_N^2$, and $U_g$
preserves $H_N$ and its norm.

For
$\Xi=(K,\lambda_K,\rho_K,V)\in\Cor(\beta,\gamma)$, the bounded map
\[
L_\xi^\Xi:K\longrightarrow H\otimes_N K,
\qquad L_\xi^\Xi\eta=\xi\otimes_N\eta
\]
satisfies
\[
(L_\xi^\Xi)^*L_\zeta^\Xi
=\lambda_K(\langle\xi,\zeta\rangle_N).
\]
The ranges of the maps $L_\xi^\Xi$, $\xi\in H_N$, span a dense
subspace of $H\otimes_N K$.

We may restrict to vectors $\xi\in H_N$ for which $g\mapsto U_g\xi$
is continuous in $\|\cdot\|_N$. Indeed, for $h\in\rC_c(G)$, the vector
\[
\xi_h=\int_G h(g)U_g\xi\,\rd g
\]
has this property, by continuity of left translations in $\rL^1(G)$.
If $(h_j)_j$ is a nonnegative approximate identity of integral one, then
$\xi_{h_j}\to\xi$ in $H$ and $\|\xi_{h_j}\|_N\leq\|\xi\|_N$.
Thus $\langle\xi_{h_j}-\xi,\xi_{h_j}-\xi\rangle_N\to0$
weak$^*$ in $N$, since these positive elements are uniformly bounded
and their traces are $\|\xi_{h_j}-\xi\|^2\to0$. Normality of
$\lambda_K$ and the preceding identity give
$L_{\xi_{h_j}}^\Xi\eta\to L_\xi^\Xi\eta$ for every $\eta\in K$.
The ranges corresponding to vectors with $\|\cdot\|_N$-continuous orbits therefore
still span a dense subspace of every fusion space.

Fix two such vectors $\xi,\zeta$. For $a\in M_\alpha$, set
\[
b_a(g)=\langle\xi,\lambda_H(a)U_g\zeta\rangle_N\in N.
\]
The function $g\mapsto b_a(g)$ is bounded and norm continuous. Moreover,
\[
\beta_h(b_a(g))
=\langle U_h\xi,\lambda_H(\alpha_h(a))U_{hg}\zeta\rangle_N,
\]
so $b_a(g)\in N_\beta$. For $c\in P_\gamma$ and $f\in\rC_c(G)$,
consider the integrated element
\[
F_{\xi,\zeta}\bigl((a\odot c^{\op})u(f)\bigr)
=\int_G f(g)(b_a(g)\odot c^{\op})u_g\,\rd g
\in\rC^*(\beta,\gamma).
\]
It belongs to this algebra because
$g\mapsto f(g)(b_a(g)\odot c^{\op})$ is a continuous compactly
supported function with values in $N_\beta\otimes_{\max}P_\gamma^{\op}$.
The fusion formula gives, for every $\Xi\in\Cor(\beta,\gamma)$,
\begin{equation}\label{equivariant fusion compression}
\Pi_\Xi(F_{\xi,\zeta}(z))
=(L_\xi^\Xi)^*\Pi_{\Theta\otimes_N\Xi}(z)L_\zeta^\Xi
\end{equation}
on the linear span of the standard generators of $\rC^*(\alpha,\gamma)$.
Choose $\Xi$ whose representation of $\rW^*(\beta,\gamma)$ is faithful.
Then \eqref{equivariant fusion compression} shows that $F_{\xi,\zeta}$ is
independent of the expression of $z$ as a sum of generators and that
\[
\|F_{\xi,\zeta}(z)\|\leq\|\xi\|_N\,\|\zeta\|_N\,\|z\|.
\]
It therefore extends to a bounded map from $\rC^*(\alpha,\gamma)$ to
$\rC^*(\beta,\gamma)$, and \eqref{equivariant fusion compression}
continues to hold for every $z\in\rC^*(\alpha,\gamma)$ and every $\Xi$.

Now suppose that $\Pi_{\Theta\otimes_N\Xi_2}(z)=0$. The compression
identity gives $\Pi_{\Xi_2}(F_{\xi,\zeta}(z))=0$. Since $\Xi_1\prec\Xi_2$,
we also have $\Pi_{\Xi_1}(F_{\xi,\zeta}(z))=0$. Applying the same identity
to $\Xi_1$ and using the density of the ranges of the maps $L_\xi^{\Xi_1}$,
we obtain $\Pi_{\Theta\otimes_N\Xi_1}(z)=0$. This proves the required
kernel inclusion and hence preservation of weak containment in the
second variable.

Conjugation of equivariant correspondences preserves weak containment,
as follows by conjugating their matrix coefficients, and reverses fusion:
\[
\overline{\Theta\otimes_N\Xi}
\cong\overline\Xi\otimes_N\overline\Theta.
\]
Applying the result just proved to the conjugate correspondences gives
preservation in the first variable. Finally,
\[
\Theta_1\otimes_N\Xi_1
\prec\Theta_2\otimes_N\Xi_1
\prec\Theta_2\otimes_N\Xi_2,
\]
which proves the assertion in the tracial case.
\end{proof}

Recall that a von Neumann algebra $M$ is amenable if $\cI_M \prec \cC_M$. More generally, recall that $N$ is coamenable in $M$ if $\cI_M \prec \cI_M \otimes_N \cI_M$. This does not automatically pass to the equivariant correspondences unless $G$ is amenable.
\begin{proposition}\label{amenable relative correspondence}
Let $G$ be a locally compact group and let
$\alpha:G\curvearrowright(M,\tau)$ be a continuous trace-preserving
action. If $G$ is amenable and $M$ is amenable, then
\[
 \cI_\alpha\prec\cC_\alpha.
\]
More generally, if $G$ is amenable and $N\subset M$ is a coamenable $\alpha$-invariant von Neumann subalgebra, then $$\cI_\alpha \prec \cI_\alpha \otimes_N \cI_\alpha.$$
\end{proposition}

\begin{proof}
Let $P=\langle M,e_N\rangle$ be the basic construction with its
canonical semifinite trace $\Tr$. The identification
\[
 \widehat x\otimes_N\widehat y\longmapsto\widehat{xe_Ny}
\]
identifies the correspondence in the statement with
$\rL^2(P,\Tr)$. The action $\alpha$ extends to a trace-preserving
action $\widetilde\alpha$ on $P$ fixing $e_N$, and this
identification intertwines the diagonal implementations.

By coamenability of $N$ in $M$, there is an $M$-central state on $P$
whose restriction to $M$ is $\tau$. Approximation by normal states and Hahn--Banach
separation therefore give positive $h\in\rL^1(P,\Tr)$ with
$\Tr(h)=1$ such that
\[
 \|uh-hu\|_1<\varepsilon,\qquad
 \|\Tr(h\,\cdot)|_M-\tau\|<\varepsilon
\]
for any prescribed finite subset of $\cU(M)$. Indeed, the central
state puts zero in the weak closure of the convex set consisting of
the commutators and the restriction minus $\tau$ in the relevant
finite sum of preduals, hence also in its norm closure.
These estimates hold uniformly on any prescribed compact
$\|\cdot\|_2$-subset of $\cU(M)$ by taking a finite net. The
needed estimate, with $\omega_h=\Tr(h\,\cdot)$, is
\[
 \|(u-v)h\|_1+\|h(u-v)\|_1
 \leq2\bigl(\|u-v\|_2^2+
       4\|\omega_h|_M-\tau\|\bigr)^{1/2}.
\]

Fix finite $F\subset\cU(M)$ and compact $K\subset G$. By
amenability, choose $q\in\rC_c(G)$ with $q\geq0$, $\int_Gq=1$
and
\[
 \sup_{g\in K}\|q(g^{-1}\,\cdot)-q\|_1<\varepsilon.
\]
Choose $h$ satisfying the preceding estimates uniformly on
$\{\alpha_{s^{-1}}(u)\mid s\in\supp(q),\ u\in F\}$, and set
\[
 \overline h=\int_G q(s)\widetilde\alpha_s(h)\,\rd s.
\]
Then $\Tr(\overline h)=1$, the restriction of
$\omega_{\overline h}$ to $M$ remains within $\varepsilon$ of
$\tau$, and
\[
 \max_{u\in F}\|u\overline h-\overline h u\|_1<\varepsilon,
 \qquad
 \sup_{g\in K}\|\widetilde\alpha_g(\overline h)
                    -\overline h\|_1<\varepsilon.
\]
The Powers--St{\o}rmer inequality now shows that the unit vectors
$\overline h^{1/2}\in\rL^2(P,\Tr)$ become central and invariant
uniformly on compact subsets, while their marginals converge to
$\tau$. Their vector states have accumulation points satisfying
\[
 \Phi(\lambda_M(a)\rho_M(b))=\tau(ab),\qquad
 \Phi(U^{\widetilde\alpha}(f))=\int_Gf(g)\,\rd g.
\]
Proposition~\ref{binormal state extension}, with the cyclic invariant
vector $\widehat1$ of $\cI_\alpha$, proves the claimed weak
containment. If $M$ is amenable, a hypertrace on $\B(\rL^2(M))$ restricts
to an $M$-central state on $\langle M,e_{\C 1}\rangle$, so that $\C 1$ is
coamenable in $M$, and $\rL^2(M)\otimes_{\C 1}\rL^2(M)$ is the underlying
Hilbert space of $\cC_\alpha$. This gives the first assertion.
\end{proof}

\subsection{The ultrapower implementation technique}
We recall the ultrapower implementation technique from
\cite[Theorem~2.1]{Ma25} based on a trick due to Haagerup. We will see that this technique is very well adapted to study weak containment of equivariant correspondences.

\begin{theorem}
Let $(M,\tau)$ be a tracial von Neumann algebra. Let $\Phi \in \B(\rL^2(M))^*$ be a state such that $\Phi \circ \lambda_M$ and $\Phi \circ \rho_M$ are both normal states on $M$. Then there are an abelian von Neumann algebra
$A$ with a faithful normal tracial state, a cofinal ultrafilter $\omega$
on a sufficiently large directed set,
and a unit vector
\[
 \xi\in\rL^2((A\ovt M)^\omega)
\]
such that
\begin{equation}
 \Phi(T)=\langle (1\otimes T)^\omega\xi,\xi\rangle, \quad (T \in \B(\rL^2(M)))
\end{equation}
\end{theorem}

The abelian algebra $A$ in the previous theorem is annoying. In applications, we get rid of it by using the following lemma.

\begin{lemma}\label{abelian amplification kernel lemma}
  Let $T \in \B(\rL^2(M))$ and let $\Sigma \subset \B(\rL^2(M))$ be a set of cardinal less than $\kappa$. Suppose that there exists a tracial abelian von Neumann algebra $(A,\mu)$ and some $X \in (A \ovt M)^\omega$ such that
  $(1 \otimes S)^\omega \widehat{X}=0$ for all $S \in \Sigma$.
  
  Then there exists some $x \in M^\omega$ such that 
  $S^\omega \widehat{x}=0$  for all $S \in \Sigma$
  and 
  $$ \| x\| \leq \|X\|, \quad \|T^\omega \widehat{x}\|_2 \geq \|(1\otimes T)^\omega \widehat{X}\|_2.$$
\end{lemma}
\begin{proof}
If $X=0$ or $T=0$, take $x=0$. Otherwise, by rescaling, we may
suppose that $\|X\|=1$ and $\|T\|=1$. Set
$ \eta=\|(1\otimes T)^\omega\widehat X\|_2$.
Write $A=\rL^\infty(Z,\mu)$ and choose a uniformly bounded
representative $X=(X_j)^\omega$ with $\|X_j\|\leq1$ such that every
$X_j$ is a step function on $Z$ with values in $M$. 
We have
\[
 \lim_{j\to\omega}\|(1\otimes T)\widehat{X_j}\|_2^2=\eta^2,
 \qquad
 \lim_{j\to\omega}\|(1\otimes S)\widehat{X_j}\|_2^2=0
 \quad(S\in\Sigma).
\]

Choose a net $(\varepsilon_i)_{i\in I}$ in $(0,1)$ converging
to zero and an increasing net $(\Sigma_i)_{i\in I}$ of finite subsets
of $\Sigma$ which is cofinal among its finite subsets. For every
$i\in I$, choose $j \in I$ such that
\[
 \|(1\otimes T)\widehat{X_{j}}\|_2^2\geq\eta^2-\varepsilon_i,
 \qquad
 \sum_{S\in\Sigma_i}\|(1\otimes S)\widehat{X_{j}}\|_2^2
 \leq\varepsilon_i^2.
\]
Put
\[
 E=\left\{z\in Z\ \middle|\
 \sum_{S\in\Sigma_i}\|S\widehat{X_{j}(z)}\|_2^2
 \leq\varepsilon_i\right\}.
\]
Chebyshev's inequality gives $\mu(Z\setminus E)\leq\varepsilon_i$.
Since $\|X\|=\|T\|=1$, we have
\[
 \|T\widehat{X_{j}(z)}\|_2\leq 1,
\]
hence
\[
 \int_{E}\|T\widehat{X_{j}(z)}\|_2^2\,\rd\mu(z)
 \geq  \|(1\otimes T)\widehat{X_{j}}\|_2^2 - \mu(Z\setminus E) \geq \eta^2-2 \varepsilon_i
\]
We may therefore choose an element
$z\in E$ such that
\[
 \|T\widehat{X_{j}(z)}\|_2^2\geq\eta^2-2 \varepsilon_i.
\]
Set $x_i=X_{j}(z)$ and let $x=(x_i)^\omega\in M^\omega$. Then $\|x\|\leq 1$. Moreover, the cofinality of
$(\Sigma_i)_{i\in I}$ and the fact that $\varepsilon_i\to0$ imply
\[
 S^\omega\widehat x=0
 \qquad(S\in\Sigma),
\]
whereas
\[
 \|T^\omega\widehat x\|_2\geq\eta.
\]
This proves the result.
\end{proof}

We shall also need the following application of the
ultrapower implementation technique to bicentralizers.

\subsection{The bicentralizer correspondence}

Let $\alpha:G\curvearrowright (M,\tau)$ be a continuous trace preserving action.
\begin{proposition} \label{relative bicentralizer correspondence}
Put $B=\rB(M,\alpha)$ and define
\[
 \cB_\alpha=\cI_\alpha \otimes_B \cI_\alpha. \]
Then $\cB_\alpha \prec \cI_\alpha$.
\end{proposition}
\begin{proof}
Write $\rE_B:M\to B$ for the trace-preserving conditional
expectation. Choose a net $(\varepsilon_i)_i$ of positive numbers decreasing to zero, an
increasing cofinal net $(K_i)_i$ of compact subsets of $G$ and an
increasing net $(F_i)_i$ of finite subsets of the unit ball of $M$ whose
union is $\|\cdot\|_2$-dense in it.
Thanks to Proposition~\ref{conditional expectation relative action bicentralizer},
applied to finite tuples through the diagonal inclusion of $M$ into a
finite direct sum of copies of $M$, we may choose, for every $i$,
unitaries $u_{i,1},\ldots,u_{i,n_i}\in M$ and positive numbers
$t_{i,1},\ldots,t_{i,n_i}$ with sum one such that
\[
 \sup_{g\in K_i}\|\alpha_g(u_{i,k})-u_{i,k}\|_2
 \leq\varepsilon_i
 \qquad(1\leq k\leq n_i)
\]
and
\[
 \left\|\sum_{k=1}^{n_i}t_{i,k}u_{i,k}^*au_{i,k}
 -\rE_B(a)\right\|_2\leq\varepsilon_i
 \qquad(a\in F_i).
\]
Define a state $\Phi \in \B(\rL^2(M))^*$ as a weak*-accumulation point of the normal states
$$\Phi_i = \sum_{k=1}^{n_i}t_{i,k}\langle \cdot \widehat{u_{i,k}}, \widehat{u_{i,k}} \rangle.$$
Then $\Phi$ satisfies
\begin{equation}\label{coarse product moments}
 \Phi(\lambda_M(a)\rho_M(b))=\tau(\rE_B(a)b)
 \qquad(a,b\in M)
\end{equation}
and
\begin{equation}\label{coarse identity spectrum}
 \Phi(U^\alpha(f))=\int_G f(g) \: \rd g,
 \qquad(f\in\rL^1(G)).
\end{equation}

On the other hand, the relative tensor product satisfies
\[
 \langle\widehat x\otimes_B\widehat y,
        \widehat{x'}\otimes_B\widehat{y'}\rangle
 =\tau\bigl(y'^*\rE_B(x'^*x)y\bigr)
 \qquad(x,y,x',y'\in M).
\]
The vector $\xi_B=\widehat1\otimes_B\widehat1$ is cyclic for $\cB_\alpha$ and fixed
by $U^\alpha\otimes_B U^\alpha$, and its vector state has precisely the
values \eqref{coarse product moments}--\eqref{coarse identity spectrum}.
The asserted weak containment thus follows from Proposition
\ref{binormal state extension}.
\end{proof}

\begin{lemma}\label{abelian amplification action bicentralizer}
Let $A$ be a tracial abelian von Neumann algebra. Then
\[
\rB(A\ovt M\subset A\ovt N,\id\otimes\alpha)
=A\ovt\rB(M\subset N,\alpha).
\]
\end{lemma}
\begin{proof}
Since
\[
1\otimes M^{\omega,\alpha}
\subset(A\ovt M)^{\omega,\id\otimes\alpha},
\]
Proposition
\ref{relative bicentralizer with ultrapower} gives
\[
\begin{aligned}
\rB(A\ovt M\subset A\ovt N,\id\otimes\alpha)
&=\bigl((A\ovt M)^{\omega,\id\otimes\alpha}\bigr)'
  \cap(A\ovt N)\\
&\subset(1\otimes M^{\omega,\alpha})'\cap(A\ovt N)\\
&=A\ovt\rB(M\subset N,\alpha).
\end{aligned}
\]

For the reverse inclusion, we have to show that $1\otimes\rB(M\subset N,\alpha)$ commutes with
$(A\ovt M)^{\omega,\id\otimes\alpha}$. Choose a dense subgroup $G_0\subset G$
of cardinality less than $\kappa$, and choose $f\in\rL^1(G)$ with
$f\geq0$ and $\int_Gf(g)\,\rd g=1$. Take
\[
 \Sigma=\{U^\alpha(f)-\id \}\cup\{U^\alpha_g-\id \mid g\in G_0\}.
\]

Fix $b\in\rB(M\subset N,\alpha)$ and define
$D_b:\rL^2(M)\to\rL^2(N)$ by
\[
 D_b\widehat x=\widehat{xb-bx}.
\]
Let $T=(D_b^*D_b)^{1/2}\in\B(\rL^2(M))$. If $x\in M^\omega$
satisfies $S^\omega\widehat x=0$ for every $S\in\Sigma$, then
$(U^\alpha(f)-\id)^\omega\widehat x=0$ shows that
\[
 x=(\alpha_f(x_i))^\omega\in M^\omega_\alpha.
\]
Since $(U^\alpha_g-\id)^\omega\widehat x=0$ for every $g\in G_0$ and the action
$\alpha^\omega$ is continuous on $M^\omega_\alpha$, it follows that
$x\in M^{\omega,\alpha}$. Hence $xb=bx$, and therefore
\[
 T^\omega\widehat x=0.
\]

Now let $X\in(A\ovt M)^{\omega,\id\otimes\alpha}$. Then
\[
 (1\otimes S)^\omega\widehat X=0
 \qquad(S\in\Sigma).
\]
Lemma \ref{abelian amplification kernel lemma} gives $x\in M^\omega$
such that $S^\omega\widehat x=0$ for every $S\in\Sigma$ and
\[
 \|(1\otimes T)^\omega\widehat X\|_2
 \leq\|T^\omega\widehat x\|_2=0.
\]
By the definition of $T$, this says
\[
 [X,1\otimes b]=0.
\]
Thus $1\otimes\rB(M\subset N,\alpha)$ commutes with
$(A\ovt M)^{\omega,\id\otimes\alpha}$.
\end{proof}

\begin{lemma}
\label{equivariant binormal bicentralizer lemma}
Let $\Phi\in\B(\rL^2(M))^*$ be a state whose restrictions to
$\lambda_M(M)$ and $\rho_M(M)$ are normal.  Suppose that
\begin{equation}\label{identity spectral state}
 \Phi(U^\alpha(f))=\int_G f(g)\,\rd g
 \qquad(f\in\rL^1(G)).
\end{equation}
Then, for all $x\in\rB(M,\alpha)$ and $y\in M$,
\begin{equation}\label{binormal bicentralizer identity}
 \Phi(\lambda_M(x)\rho_M(y))=\Phi(\rho_M(xy)).
\end{equation}
In particular,
$\Phi(\lambda_M(x))=\Phi(\rho_M(x))$ for every
$x\in\rB(M,\alpha)$.
\end{lemma}

\begin{proof}
By the tracial ultrapower implementation theorem, there are an abelian von Neumann algebra
$A$ with a faithful normal tracial state, a cofinal ultrafilter $\omega$
on a sufficiently large directed set,
and a unit vector
\[
 \xi\in\rL^2((A\ovt M)^\omega)
\]
such that
\begin{equation}\label{ultrapower implementation formula}
 \Phi(T)=\langle (1\otimes T)^\omega\xi,\xi\rangle, \quad (T \in \B(\rL^2(M)))
\end{equation}
The hypothesis extends by continuity to the trivial character
$1_G$ of $\rC^*(G)$. The multiplicative-domain argument therefore
gives
\[
 (1\otimes U^\alpha(f))^\omega\xi=1_G(f)\xi
 \qquad(f\in\rC^*(G)).
\]
Choose $h\in\rL^1(G)$ with $\int_Gh(g)\,\rd g=1$. Starting with a
representative $\xi=(\xi_i)^\omega$, Proposition
\ref{L2 equicontinuous ultrapower},
applied to $\id\otimes\alpha$, shows that
\[
 \xi=((1\otimes U^\alpha(h))\xi_i)^\omega
 \in\rL^2((A\ovt M)^\omega_{\id\otimes\alpha}).
\]
Here the equality follows from the preceding multiplicative-domain
relation. Proposition \ref{L2 equicontinuous ultrapower} also identifies the coordinatewise integrated
representation on this Hilbert space with the standard implementation of
$(\id\otimes\alpha)^\omega$. Its integrated representation acts
on $\xi$ by the trivial character, so $\xi$ is fixed by this action.
Therefore
\[
 \xi\in\rL^2((A\ovt M)^{\omega,\id\otimes\alpha}).
\]
Lemma~\ref{abelian amplification action bicentralizer} and
Proposition~\ref{relative bicentralizer with ultrapower} give
\[
 (1\otimes x)\xi=\xi(1\otimes x)
 \qquad(x\in\rB(M,\alpha)).
\]
Using \eqref{ultrapower implementation formula}, we obtain
\[
 \begin{aligned}
 \Phi(\lambda_M(x)\rho_M(y))
 &=\langle(1\otimes x)\xi(1\otimes y),\xi\rangle\\
 &=\langle\xi(1\otimes xy),\xi\rangle
 =\Phi(\rho_M(xy)),
 \end{aligned}
\]
as claimed.  Compare also \cite[Lemma~2.1]{Ma26}.
\end{proof}

\begin{theorem}\label{correspondance trivial bicentralizer}
We have $\rB(M,\alpha)=\C 1$ if and only if $\cC_\alpha\prec\cI_\alpha$.
\end{theorem}

\begin{proof}
If $\rB(M,\alpha)=\C1$ then $\cC_\alpha \prec \cI_\alpha$ follows from Proposition \ref{relative bicentralizer correspondence} because $\cB_\alpha=\cC_\alpha$. For the other direction, suppose that $\cC_\alpha \prec \cI_\alpha$. Then by Proposition \ref{binormal state extension}, there exists a state $\Phi\in\B(\rL^2(M))^*$ such that
\begin{equation}
 \Phi(\lambda_M(a)\rho_M(b))=\tau(a)\tau(b)
 \qquad(a,b\in M)
\end{equation}
and
\begin{equation}
 \Phi(U^\alpha(f))=\int_G f(g) \: \rd g,
 \qquad(f\in\rL^1(G)).
\end{equation}
Let $x\in\rB(M,\alpha)$. The two marginals
of $\Phi$ are $\tau$, so Lemma
\ref{equivariant binormal bicentralizer lemma} gives, for every
$y\in M$,
\[
 \tau(xy)=\Phi(\rho_M(xy))
 =\Phi(\lambda_M(x)\rho_M(y))
 =\tau(x)\tau(y).
\]
Taking $y=(x-\tau(x)1)^*$ yields
$\|x-\tau(x)1\|_2=0$.  Hence $\rB(M,\alpha)=\C 1$.
\end{proof}

We use the following result from \cite[Corollary~1.8]{BLS92}.

\begin{theorem}
\label{faithful tensor powers contain regular}
Let $\pi$ be a faithful unitary representation of a locally compact group
$G$ such that $\pi^* \prec \pi$ and $1_G \prec \pi$ and $\pi \otimes_G \pi \prec \pi$. Then $\lambda_G \prec \pi$.
\end{theorem}

\begin{corollary}
  Suppose that $\rB(M,\alpha)=\C 1$. Then $$U^\alpha \otimes_G U^\alpha \prec U^\alpha.$$
  Thus if $\alpha$ is faithful, we have $\lambda_G \prec U^\alpha$.
\end{corollary}
\begin{proof}
  Restrict the weak containment $\cC_\alpha \prec \cI_\alpha$ to the unitary representations of $G$. Also $U^\alpha$ is equal to its contragredient and it contains the trivial representation $1_G$.
\end{proof}

\begin{remark}\label{tensor product stability}
Suppose that $\rB(M,\alpha)=\C1$, and let
\[
\cL_\alpha=(\rL^2(M),\lambda_M,U^\alpha)\in\Rep(\alpha)
\]
be the standard covariant representation. Theorem
\ref{correspondance trivial bicentralizer} gives
$\cC_\alpha\prec\cI_\alpha$. Forgetting the right action
of $M$ gives
\[
 U^\alpha\otimes_G \cL_\alpha \prec \cL_\alpha.
\]
View $\cL_\alpha$ as a correspondence from $\alpha$ to
the trivial action on $\C$. Fusing on the right over
$\C$ with $\cL_\alpha^{\op}$ and applying Proposition
\ref{fusion preserves weak containment} gives
\[
 U^\alpha\otimes_G \cC_\alpha \prec \cC_\alpha.
\]
If $\alpha$ is faithful, the preceding corollary gives
$\lambda_G\prec U^\alpha$. Proposition
\ref{tensoring preserves weak containment} therefore yields
\[
 \lambda_G \otimes_G \cC_\alpha \prec \cC_\alpha.
\]
\end{remark}

\subsection{The Rokhlin property}
Let $G$ be an arbitrary locally compact group and let $\alpha:G\curvearrowright (M,\tau)$ be a continuous trace-preserving action on a tracial von Neumann algebra $(M,\tau)$.

\begin{definition}
  We say that the action $\alpha$ has the Rokhlin property if $$\lambda_G \otimes_G \cI_\alpha \prec \cI_\alpha.$$
\end{definition}

 We use the notation $M_{\omega,\alpha}=M' \cap M^{\omega}_\alpha$ for a large enough cofinal ultrafilter $\omega$ and denote by $\alpha_\omega$ the restriction of $\alpha^\omega$ to $M_{\omega,\alpha}$.

\begin{theorem}\label{Rokhlin representation criterion}
Suppose that $M$ is a $\II_1$ factor, and let $\rho=U^{\alpha_\omega}$
be the standard representation of
$\alpha_\omega:G\curvearrowright M_{\omega,\alpha}$.
For every unitary representation $\pi\in\Rep(G)$, we have
\[
\pi\otimes_G\cI_\alpha\prec\cI_\alpha
\quad\Longleftrightarrow\quad
\pi\prec\rho.
\]
\end{theorem}
\begin{proof}
We first explain why an abelian amplification does not enlarge the
central sequence representation up to weak containment. Let $(A,\mu)$
be a tracial abelian von Neumann algebra, put
$\widetilde\alpha=\id_A\otimes\alpha$, and let $\rho_A$ be the
standard representation on
\[
\rL^2(C_A),\qquad
C_A=(1\otimes M)'\cap(A\ovt M)^\omega_{\widetilde\alpha}.
\]
We claim that $\rho_A\prec\rho$.

Fix $f\in\rL^1(G)$ and put $c=\|\rho(f)\|$. If
$\|\rho_A(f)\|>c$, bounded vectors are dense, so there are
$Y\in C_A$ and $h\in\rC_c(G)$, with $h\geq0$ and $\int_Gh=1$,
such that $X=\widetilde\alpha_h^\omega(Y)$ satisfies
\[
\|\rho_A(f)\widehat X\|_2^2-c^2\|X\|_2^2>\delta
\]
for some $\delta>0$. Here $h$ is chosen from an approximate identity.
Choose a uniformly bounded representative $(Y_j)_j$ whose terms are
finite step functions in the abelian variable. This is possible by
applying conditional expectations onto finite-dimensional subalgebras
of $A$ and approximating in $\|\cdot\|_2$. Write
\[
Y_j=\sum_r e_{j,r}\otimes y_{j,r},\qquad
X_j=\sum_r e_{j,r}\otimes x_{j,r},\qquad
x_{j,r}=\alpha_h(y_{j,r}),
\]
where the $e_{j,r}$ form a partition of $1$ and
$\|y_{j,r}\|\leq C$ for a fixed $C$. The representatives $(X_j)_j$
are uniformly equicontinuous, since
\[
\|\alpha_g(x_{j,r})-x_{j,r}\|_2
\leq C\|h(g^{-1}\,\cdot)-h\|_1.
\]
Put
\[
d(x)=\|U^\alpha(f)\widehat x\|_2^2-c^2\|x\|_2^2.
\]
For $\omega$-almost every $j$, we have
\[
\sum_r\mu(e_{j,r})d(x_{j,r})>\delta,
\]
whereas, for every $a\in M$,
\[
\lim_{j\to\omega}\sum_r\mu(e_{j,r})
\|[a,x_{j,r}]\|_2^2=0.
\]
The numbers $|d(x_{j,r})|$ have a common bound. Given a finite
subset $F\subset M$ and $\varepsilon>0$, Chebyshev's inequality
therefore allows us to choose $j,r$ such that
\[
\max_{a\in F}\|[a,x_{j,r}]\|_2<\varepsilon,
\qquad d(x_{j,r})>\delta/2.
\]
Indeed, make the total weight of the components failing the first
inequality small enough that their contribution to the preceding
weighted sum has absolute value less than $\delta/2$.
Using the standing assumption on the index set, make these choices
along increasing finite subsets of a strongly dense subalgebra of
$M$ and decreasing tolerances. The selected elements are uniformly
bounded and equicontinuous. They define $x\in M_{\omega,\alpha}$
with
\[
\|\rho(f)\widehat x\|_2^2-c^2\|x\|_2^2\geq\delta/2,
\]
a contradiction. This proves the claim.

Suppose now that $\pi\otimes_G\cI_\alpha\prec\cI_\alpha$,
and fix a unit vector $\xi_0$ in the space of $\pi$.
Proposition~\ref{binormal state extension}, applied to
$\xi_0\otimes\widehat1$, gives a state
$\Phi\in\B(\rL^2(M))^*$ such that
\[
\Phi(\lambda_M(a)\rho_M(b))=\tau(ab),\qquad
\Phi(U^\alpha(f))=\langle\pi(f)\xi_0,\xi_0\rangle
\]
for $a,b\in M$ and $f\in\rL^1(G)$. By ultrapower implementation
\cite[Theorem~2.1]{Ma25}, this state is implemented by a unit vector
$\eta\in\rL^2((A\ovt M)^\omega)$ for a tracial abelian algebra $A$.
The first identity implies $(1\otimes a)\eta=\eta(1\otimes a)$
for every $a\in M$. The second identity, an approximate identity in
$\rL^1(G)$, and Proposition~\ref{L2 equicontinuous ultrapower}
show that $\eta$ belongs to the equicontinuous part. Its bounded
truncations still commute with $1\otimes M$, so $\eta\in\rL^2(C_A)$
and
\[
\langle\rho_A(f)\eta,\eta\rangle
=\langle\pi(f)\xi_0,\xi_0\rangle.
\]
Equality of these coefficients identifies the cyclic representations
generated by $\eta$ and $\xi_0$. The claim gives that this cyclic
subrepresentation of $\pi$ is weakly contained in $\rho$.
Since $\xi_0$ was arbitrary, $\pi\prec\rho$.

For the converse, factoriality implies
\[
\rE^\omega(x^*y)=\tau^\omega(x^*y)1
\qquad(x,y\in M_{\omega,\alpha}).
\]
Consequently multiplication defines an isometric equivariant
$M$--$M$ bimodule map
\[
\rL^2(M_{\omega,\alpha})\otimes\rL^2(M)
\longrightarrow\rL^2(M^\omega_\alpha),\qquad
\widehat x\otimes\widehat a\longmapsto\widehat{xa}.
\]
The correspondence on the target is weakly contained in
$\cI_\alpha$. To see this, lift a bounded equicontinuous vector to
its uniformly bounded representatives and use
Proposition~\ref{L2 equicontinuous ultrapower} to compute its
coefficients for integrated group operators. The resulting coefficients
are ultralimits of coefficients of $\cI_\alpha$.
Density of bounded vectors gives the asserted weak containment.
It follows that $\rho\otimes_G\cI_\alpha\prec\cI_\alpha$.
Thus $\pi\prec\rho$ and
Proposition~\ref{tensoring preserves weak containment} give
$\pi\otimes_G\cI_\alpha\prec\cI_\alpha$.
\end{proof}

The following corollary shows that our definition of the Rokhlin property coincides with the one given in \cite{MT16} for flows, in \cite{Sh14} for abelian groups and in \cite{Ni26} for groups admitting a compact open subgroup.

\begin{corollary}\label{Rokhlin central sequence faithfulness}
Suppose that $M$ is a $\II_1$ factor. Then $\alpha$ has the Rokhlin
property if and only if $\alpha_\omega$ is faithful.
\end{corollary}

\begin{proof}
Put $\rho=U^{\alpha_\omega}$. Since
$\rho\prec\rho$,  Theorem \ref{Rokhlin representation criterion} gives
\[
 \rho\otimes_G\cI_\alpha\prec\cI_\alpha.
\]
Tensoring this weak containment by $\rho$ and using Proposition
\ref{tensoring preserves weak containment}, we obtain
\[
 \rho\otimes_G\rho\otimes_G\cI_\alpha
 \prec \rho\otimes_G\cI_\alpha
 \prec \cI_\alpha.
\]
Another application of Theorem
\ref{Rokhlin representation criterion} yields
$$\rho\otimes_G\rho\prec\rho.$$
Since $\rho$ is a Koopman representation, it contains the constants
and is equivalent to its contragredient representation.
If $\alpha_\omega$ is faithful then Theorem \ref{faithful tensor powers contain regular} gives
\[
 \lambda_{G}\prec\rho.
\]
Applying Theorem \ref{Rokhlin representation criterion} once more gives
\[
 \lambda_{G}\otimes_G\cI_\alpha\prec\cI_\alpha.
\]
Conversely, if
$\alpha$ is Rokhlin, Theorem
\ref{Rokhlin representation criterion} gives
$\lambda_G\prec\rho$. Thus $\rho$ is faithful.
\end{proof}

\begin{theorem}\label{Rokhlin and trivial bicentralizer}
Let $M$ be a $\II_1$ factor. Consider the following properties.
\begin{enumerate}[\rm (i)]
  \item $\alpha$ has the Rokhlin property.
  \item $\alpha$ is faithful and $\rB(M,\alpha)=\C 1$.
\end{enumerate}
If $G$ is amenable then $(\rm i) \Rightarrow (\rm ii)$. If $M$ is amenable then $(\rm ii) \Rightarrow (\rm i)$.
\end{theorem}

\begin{proof}
$(\rm i) \Rightarrow (\rm ii)$ when $G$ is amenable. Since convex averages of the vector states associated with
$\widehat u$, $u\in\cU(M)$, converge on $M\odot M^{\op}$ to the
product trace, the cyclic vector state of $\cC_M$, we have
\begin{equation}\label{factor coarse below identity}
 \cC_M\prec\cI_M.
\end{equation}
Proposition
\ref{induction preserves weak containment} and Fell absorption then give
\[
 \lambda_G\otimes_G\cC_\alpha
 \cong\Ind(\cC_M)
 \prec \Ind(\cI_M)
 \cong\lambda_G\otimes_G\cI_\alpha.
\]
The Rokhlin property for $\alpha$ says
\begin{equation}\label{Rokhlin identity absorption}
 \lambda_G\otimes_G\cI_\alpha \prec \cI_\alpha
\end{equation}
and since $G$ is amenable, we have $1_G \prec \lambda_G$ hence
\begin{equation}\label{amenable coarse absorption}
 \cC_\alpha  \prec \lambda_G\otimes_G\cC_\alpha.
\end{equation}
We conclude that
\[
 \cC_\alpha
 \prec \cI_\alpha,
\]
and Theorem \ref{correspondance trivial bicentralizer} finishes the
proof.

$(\rm ii) \Rightarrow (\rm i)$ when $M$ is amenable. The amenability of $M$ means that 
$$\cI_M \prec \cC_M.$$  Proposition
\ref{induction preserves weak containment} and Fell absorption give
\[
 \lambda_G\otimes_G\cI_\alpha
 \cong\Ind(\cI_M)
 \prec \Ind(\cC_M)
 \cong\lambda_G\otimes_G\cC_\alpha.
\]
Thanks to Remark \ref{tensor product stability} we have
$$\lambda_G\otimes_G\cC_\alpha \prec \cC_\alpha$$
 and by Theorem \ref{correspondance trivial bicentralizer}, we have
$$\cC_\alpha\prec\cI_\alpha.$$  We conclude that
$$\lambda_G\otimes_G\cI_\alpha\prec\cI_\alpha.$$
\end{proof}

If $\rB(M,\alpha)$ is nontrivial, we have the following transfer
property.

\begin{proposition}\label{bicentralizer representation lifting}
Suppose that $G$ and $M$ are amenable, and put
$B=\rB(M,\alpha)$. For every unitary
representation $\pi$ of $G$,
\[
 \pi\otimes_G\cI_{\alpha|_B}\prec\cI_{\alpha|_B}
 \quad\Longrightarrow\quad
 \pi\otimes_G\cI_\alpha\prec\cI_\alpha.
\]
\end{proposition}

\begin{proof}
Associativity
and the identity property of $\cI_{\alpha|_B}$ give
\[ \cB_\alpha =\cI_\alpha \otimes_B \cI_\alpha=\cI_\alpha \otimes_B \cI_{\alpha|_B} \otimes_B \cI_\alpha \]
hence, by Proposition~\ref{fusion preserves weak containment}, we have
\[
 \pi\otimes_G\cI_{\alpha|_B}\prec\cI_{\alpha|_B}
 \quad\Longrightarrow\quad
 \pi\otimes_G\cB_\alpha\prec\cB_\alpha.
\]
Since $M$ is amenable, $B$ is coamenable in
$M$. Proposition~\ref{amenable relative correspondence} therefore gives
$\cI_\alpha\prec\cB_\alpha$, while
Proposition~\ref{relative bicentralizer correspondence} gives
$\cB_\alpha\prec\cI_\alpha$. Together with
Proposition~\ref{tensoring preserves weak containment}, this implies
\[
 \pi\otimes_G\cI_\alpha
 \prec\pi\otimes_G\cB_\alpha
 \prec\cB_\alpha\prec\cI_\alpha.
\]
\end{proof}

The following proposition answers Masuda--Tomatsu's
\cite[Problem~8.7]{MT16}, when $G=\R$ and $M=N=R$.

\begin{proposition}\label{Rokhlin trivial tensor stability}
Let $\alpha:G\curvearrowright M$ be a continuous action of a locally
compact group on a $\II_1$ factor. For every $\II_1$ factor $N$,
$\alpha$ has the Rokhlin property if and only if
$\alpha\otimes\id_N$ has the Rokhlin property on $M\ovt N$.
\end{proposition}

\begin{proof}
The standard identification
$\rL^2(M\ovt N)=\rL^2(M)\otimes\rL^2(N)$ gives
\[
\cI_{\alpha\otimes\id_N}
\cong\cI_\alpha\otimes\cI_{\id_N},
\]
where the tensor product is the spatial tensor product of
correspondences and $G$ acts trivially on $\cI_{\id_N}$.
Tensoring with $\cI_{\id_N}$ preserves weak containment. Indeed,
amplifying the completely positive map in
Proposition~\ref{binormal state extension} by
$\id_{\B(\rL^2(N))}$ preserves the left and right actions of
$M\ovt N$ and the integrated group operators.
If $\alpha$ has the Rokhlin property, we therefore obtain
\[
\begin{aligned}
\lambda_G\otimes_G\cI_{\alpha\otimes\id_N}
&\cong(\lambda_G\otimes_G\cI_\alpha)\otimes\cI_{\id_N}\\
&\prec\cI_\alpha\otimes\cI_{\id_N}
\cong\cI_{\alpha\otimes\id_N}.
\end{aligned}
\]
This is the Rokhlin property for $\alpha\otimes\id_N$.

Conversely, suppose that $\alpha\otimes\id_N$ has the Rokhlin
property. Restricting both module actions in its defining weak
containment to $M\otimes1$ gives
\[
(\lambda_G\otimes_G\cI_\alpha)\otimes\rL^2(N)
\prec\cI_\alpha\otimes\rL^2(N).
\]
Here the left and right actions of $M$ and the action of $G$ are
trivial on the $\rL^2(N)$ factor. Both sides are therefore
amplifications of the corresponding $\alpha$--$\alpha$
correspondences. We conclude that
$\lambda_G\otimes_G\cI_\alpha\prec\cI_\alpha$, so $\alpha$ has
the Rokhlin property.
\end{proof}

\section{Cocycle perturbations}
\label{sec:cocycle-perturbations}

We study the effect of cocycle perturbations on the bicentralizer.
An approximate coboundary induces a canonical isomorphism between the
bicentralizers of the original and perturbed actions. These isomorphisms
will provide the dual bicentralizer action for abelian groups. We also
give a criterion for the existence of ergodic perturbations by approximate
coboundaries in terms of the fixed part of the bicentralizer, construct
minimal perturbations when the bicentralizer is trivial, and prove
approximate $1$-cocycle vanishing for Rokhlin actions of amenable groups.

The basic observation is that approximate coboundaries become exact
coboundaries in an equicontinuous ultrapower. On the bicentralizer,
conjugation by the implementing unitaries converges pointwise to a map
independent of their choice, giving the canonical isomorphism. A Baire category argument
produces ergodic perturbations. For minimal perturbations, we successively
fix matrix subalgebras whose union has trivial relative commutant.
For approximate cocycle vanishing,
a matrix amplification reduces the problem to equivalence of
equal-trace projections in a fixed point ultrapower.

\subsection{Cocycles and approximate coboundaries}

Throughout this section, $G$ is a locally compact group,
$(M,\tau)$ is a tracial von Neumann algebra, and
$\alpha:G\curvearrowright(M,\tau)$ is a continuous trace-preserving action.
We use the standard cocycle conventions. Compare
\cite[Sections~2--3]{MT16}.

A \emph{$1$-cocycle} for $\alpha$ is a $\|\cdot\|_2$-continuous map
$u:G\to\cU(M)$ satisfying
\[
u_{gh}=u_g\alpha_g(u_h)
\qquad\text{for all }g,h\in G.
\]
We denote the set of all $1$-cocycles for $\alpha$ by $Z^1(\alpha)$. Notice
that every $u\in Z^1(\alpha)$ satisfies $u_e=1$.

Every cocycle $u\in Z^1(\alpha)$ gives a new continuous trace-preserving
action $\alpha^u:G\curvearrowright(M,\tau)$, called the \emph{cocycle
perturbation} of $\alpha$ by $u$, defined by
\[
\alpha^u_g=\Ad(u_g)\circ\alpha_g,
\qquad
\alpha^u_g(x)=u_g\alpha_g(x)u_g^*
\quad(g\in G,\ x\in M).
\]
Indeed, the cocycle relation gives
\[
\alpha^u_g\circ\alpha^u_h
=\Ad\bigl(u_g\alpha_g(u_h)\bigr)\circ\alpha_{gh}
=\Ad(u_{gh})\circ\alpha_{gh}
=\alpha^u_{gh}.
\]
Continuity follows from the continuity of $u$ and $\alpha$, and
$\alpha^u$ preserves $\tau$ because $\alpha$ is trace preserving and inner
automorphisms preserve $\tau$.

For every $v\in\cU(M)$, the map
\[
\partial_\alpha(v)_g=v\alpha_g(v^*)
\qquad(g\in G)
\]
belongs to $Z^1(\alpha)$. Such a cocycle is called a \emph{coboundary}, and
we write
\[
B^1(\alpha)
=\{\partial_\alpha(v)\mid v\in\cU(M)\}.
\]
If $u=\partial_\alpha(v)$ is a coboundary, then its cocycle perturbation is
conjugate to the original action:
\[
\alpha^u_g
=\Ad(v)\circ\alpha_g\circ\Ad(v^*)
\qquad(g\in G).
\]

The natural topology on $Z^1(\alpha)$ is the topology of uniform
$\|\cdot\|_2$-convergence on compact subsets of $G$.
A neighborhood basis at $u\in Z^1(\alpha)$ is given by
\[
\mathcal V(u,K,\varepsilon)
=\left\{w\in Z^1(\alpha)\ \middle|\
\sup_{g\in K}\|u_g-w_g\|_2<\varepsilon\right\},
\qquad K\Subset G,\quad \varepsilon>0.
\]
This description by compact sets, rather than by a sequence of compact sets,
does not require $G$ to be second countable or $\sigma$-compact.

A cocycle $u\in Z^1(\alpha)$ is called an \emph{approximate coboundary} if
it belongs to the closure of $B^1(\alpha)$ in this topology. Equivalently,
$u$ is an approximate coboundary if, for every $K\Subset G$ and every
$\varepsilon>0$, there exists $v\in\cU(M)$ such that
\[
\sup_{g\in K}
\|u_g-v\alpha_g(v^*)\|_2<\varepsilon.
\]
We denote the set of approximate coboundaries by
$\overline{B^1(\alpha)}$.

\begin{proposition}\label{approximate coboundary ultrapower}
Let $\omega$ be a cofinal ultrafilter on a large enough directed set $I$. Let $u\in Z^1(\alpha)$. We have 
\begin{align*}
\left\{ V\in\cU(M^\omega_\alpha)\ \mid\ u =\partial_{\alpha^\omega}(V) \right\} = \left\{ (v_i)^{\omega} \in \cU(M^\omega)\ \mid\
\lim_{i \to \omega} \partial_{\alpha}(v_i)=u \right\}.
\end{align*}
In particular, $u \in \overline{B^1(\alpha)}$ if and only if $u \in B^1(\alpha^\omega)$.
\end{proposition}

\begin{proof}
Let $V=(v_i)^\omega$ belong to the set on the right, and suppose that
$g\mapsto v_i\alpha_g(v_i)^*$ converges to $u$ uniformly on compact
subsets along $\omega$. Pointwise
convergence gives $u=\partial_{\alpha^\omega}(V)$. Moreover, for
$h\in G$,
\[
 \|\alpha_h(v_i)-v_i\|_2
 =\|v_i\alpha_h(v_i)^*-1\|_2.
\]
The compact-uniform convergence, together with the continuity of $u$
and $u_e=1$, therefore shows that
$(v_i)_{i\in I}$ is $(\alpha,\omega)$-equicontinuous. Hence
$V\in\cU(M^\omega_\alpha)$.

Conversely, let $V\in\cU(M_\alpha^\omega)$ satisfy
$u=\partial_{\alpha^\omega}(V)$, and choose an
$(\alpha,\omega)$-equicontinuous unitary representative
$V=(v_i)^\omega$. The net $(f_i)_{i \in I}$ of continuous maps
\[
 f_i:G\longrightarrow M,
 \qquad f_i(g)=v_i\alpha_g(v_i)^*-u_g,
\]
is $\omega$-equicontinuous. Since
$u=\partial_{\alpha^\omega}(V)$, we have
$(f_i(g))^\omega=0$ for every $g\in G$. Proposition
\ref{omega equicontinous functions} therefore gives, for every
$K\Subset G$,
\[
 \lim_{i\to\omega}\sup_{g\in K}
 \|v_i\alpha_g(v_i)^*-u_g\|_2=0.
\]
This proves the equality of the two sets.
\end{proof}

\begin{proposition}\label{topology and action on cocycles}
The space $Z^1(\alpha)$ is complete for the uniformity of compact-uniform
$\|\cdot\|_2$-convergence.  If $G$ is second countable and $M$ has
separable predual, then $Z^1(\alpha)$ is a Polish space.

The group $\cU(M)$ acts continuously on $Z^1(\alpha)$ by
\[
 (v\mathbin{\cdot} w)_g
 =v w_g\alpha_g(v^*)
 \qquad(v\in\cU(M),\ w\in Z^1(\alpha),\ g\in G).
\]
For every $u\in Z^1(\alpha)$, right multiplication by $u$ defines a
$\cU(M)$-equivariant homeomorphism
\[
Z^1(\alpha^u)\longrightarrow Z^1(\alpha) : w \mapsto wu.
\]
\end{proposition}

\subsection{Bicentralizer isomorphisms for approximate coboundaries}

Throughout this subsection, $G$ is a locally compact group,
$(M,\tau)\subset(N,\tau)$ is an inclusion of tracial von Neumann algebras,
and $\alpha:G\curvearrowright(M,\tau)$ is a continuous trace-preserving
action.

\begin{proposition}\label{relative bicentralizer isomorphism}
Let $u\in\overline{B^1(\alpha)}$. If $(v_i)_i$ is a net in $\cU(M)$ such
that
\[
\partial_\alpha(v_i)\longrightarrow u
\]
uniformly on compact subsets of $G$, then, for every
$x\in\rB(M\subset N,\alpha)$, the limit
\[
\beta_u(x)=\lim_i v_i x v_i^*
\]
exists in $\|\cdot\|_2$ and is independent of the choice of $(v_i)_{i \in I}$. This
formula defines a trace-preserving isomorphism
\[
\beta_u:\rB(M\subset N,\alpha)
\longrightarrow\rB(M\subset N,\alpha^u).
\]
Moreover, the following statements hold.
\begin{enumerate}[\rm (i)]
\item If $v\in\overline{B^1(\alpha^u)}$, then $vu\in\overline{B^1(\alpha)}$ and
\[
\beta_{vu}=\beta_v\circ\beta_u.
\]
\item Suppose that $\alpha$ extends to a continuous trace-preserving action
on $(N,\tau)$, still denoted by $\alpha$. Then the isomorphism $\beta_u$ is
equivariant: for every $g\in G$,
\[
 \beta_u\circ\alpha_g=\alpha_g^u\circ\beta_u
 \quad\text{on }\rB(M\subset N,\alpha).
\]
\item For every $x\in\rB(M\subset N,\alpha)$, the map
\[
\overline{B^1(\alpha)}\longrightarrow N,
\qquad
w\longmapsto\beta_w(x),
\]
is continuous for compact-uniform $\|\cdot\|_2$-convergence on cocycles
and the $\|\cdot\|_2$-topology on $N$.
\end{enumerate}
\end{proposition}

\begin{proof}
Choose a net $(u_i)_{i \in I}$ in $\cU(M)$ such that $\partial_\alpha (u_i)$ converges to $u$. Then $\partial_\alpha(u_j^*u_i) \to 1$ when $i,j$ go to infinity.

Consequently,
\[
\|u_i x u_i^*-u_j x u_j^*\|_2
=\|(u_j^*u_i)x-x(u_j^*u_i)\|_2 \to 0
\]
because $x \in \rB(M \subset N,\alpha)$. This means that $(u_i x u_i^*)_i$ is Cauchy in $\rL^2(N,\tau)$. Its terms are uniformly
bounded, so its limit belongs to $N$. The same comparison, applied to two
different implementing nets, proves that the
limit is independent of the choice of $(u_i)_{i \in I}$. Denote it by
$\beta_u(x)$.

 If $(v_i)_{i \in I}$ satisfies $\partial_{\alpha^u}(v_i) \to v \in \overline{B^1(\alpha^u)}$, then $\partial_{\alpha}(v_iu_i) \to vu$, hence 
$$ \beta_{vu}(x)=\lim_i v_iu_ixu_i^*v_i^*=\lim_i v_i \beta_u(x)v_i^*.$$
This shows first, when $v=1$, that $\beta_u(x) \in \rB(M \subset N, \alpha^u)$ and secondly for an arbitrary $v$ that
$$\beta_{vu}(x)=\beta_v(\beta_u(x)).$$

The map $\beta_u$ is unital, trace preserving, $*$-preserving, and
multiplicative, by the defining limit and continuity of multiplication
on bounded sets for $\|\cdot\|_2$. Moreover, $\beta_{u^*}$ is an inverse of $\beta_u$ by the previous relation. Hence $\beta_u$ is a trace-preserving
isomorphism.

We prove item~$(\rm ii)$.  The definition of the relative bicentralizer
shows that $\rB(M\subset N,\alpha)$ is $\alpha$-invariant.  Let $(v_i)_i$
implement $u$ and put $d_{i,g}=v_i\alpha_g(v_i^*)$.  For
$x\in\rB(M\subset N,\alpha)$ and $g\in G$, we have
\[
 \|\alpha_g^u(v_i x v_i^*)-v_i\alpha_g(x)v_i^*\|_2
 \leq2\|x\|_\infty\|u_g-d_{i,g}\|_2.
\]
Passing to the limit gives
\[
 \alpha_g^u(\beta_u(x))=\beta_u(\alpha_g(x)).
\]

Finally, fix $u$, $x$, and $\varepsilon>0$. The comparison used above to
prove that the defining limit is independent of the implementing net
gives $K\Subset G$ and $\delta>0$ such that
\[
 \sup_{g\in K}\|\partial_\alpha(s)_g-u_g\|_2<\delta
 \quad\Longrightarrow\quad
 \|sxs^*-\beta_u(x)\|_2<\varepsilon
\]
for every $s\in\cU(M)$. If $w\in\overline{B^1(\alpha)}$ satisfies
\[
\sup_{g\in K}\|w_g-u_g\|_2<\delta/2,
\]
choose an implementing unitary $s\in\cU(M)$ such that
\[
\sup_{g\in K}\|\partial_\alpha(s)_g-w_g\|_2<\delta/2,
\qquad
\|sxs^*-\beta_w(x)\|_2<\varepsilon.
\]
The displayed implication gives
\[
\|\beta_w(x)-\beta_u(x)\|_2<2\varepsilon,
\]
which proves item~$(\rm iii)$.
\end{proof}

\begin{proposition}\label{tensor product relative action bicentralizer}
For $i\in\{1,2\}$, let $M_i\subset N_i$ be inclusions of tracial von
Neumann algebras and let
$\alpha_i:G\curvearrowright(M_i,\tau_i)$ be trace-preserving actions. Put
\[
M=M_1\ovt M_2,
\quad N=N_1\ovt N_2,
\quad\alpha=\alpha_1\otimes\alpha_2.
\]
Then
\[
\rB(M\subset N,\alpha)
\subset
\rB(M_1\subset N_1,\alpha_1)
\ovt\rB(M_2\subset N_2,\alpha_2).
\]
Moreover, if
$u_i\in\overline{B^1(\alpha_i)}$ for $i\in\{1,2\}$, then
$u=u_1 \otimes u_2 \in \overline{B^1(\alpha)}$ and
\[
\beta_u(x)
=\bigl(\beta_{u_1}\otimes
\beta_{u_2}\bigr)(x)
\]
for every $x\in\rB(M\subset N,\alpha)$.
\end{proposition}

\begin{proof}
For a sufficiently large ultrapower,
\[
M_1^{\omega,\alpha_1}\ovt M_2^{\omega,\alpha_2}
\subset M^{\omega,\alpha}.
\]
Proposition~\ref{relative bicentralizer with ultrapower} gives
\begin{align*}
\rB(M\subset N,\alpha)
&\subset
\bigl(M_1^{\omega,\alpha_1}\ovt
M_2^{\omega,\alpha_2}\bigr)'\cap N\\
&=\rB(M_1\subset N_1,\alpha_1)
\ovt\rB(M_2\subset N_2,\alpha_2).
\end{align*}
For $i\in\{1,2\}$, choose a net $(v_{i,j_i})_{j_i\in J_i}$ in
$\cU(M_i)$ such that
\[
\partial_{\alpha_i}(v_{i,j_i})\longrightarrow u_i
\]
uniformly on compact subsets of $G$. On the product directed set
$J_1\times J_2$, put
\[
w_{(j_1,j_2)}=v_{1,j_1}\otimes v_{2,j_2}.
\]
Then
\[
\partial_\alpha(w_{(j_1,j_2)})_g
=\partial_{\alpha_1}(v_{1,j_1})_g
 \otimes\partial_{\alpha_2}(v_{2,j_2})_g,
\]
so $\partial_\alpha(w_{(j_1,j_2)})\to u$ uniformly on compact subsets.
In particular, $u\in\overline{B^1(\alpha)}$. By
Proposition~\ref{relative bicentralizer isomorphism}, the left-hand side
of the asserted formula is the limit in $\|\cdot\|_2$ of
\[
(v_{1,j_1}\otimes v_{2,j_2})x
(v_{1,j_1}^*\otimes v_{2,j_2}^*).
\]
For elementary tensors in
$\rB(M_1\subset N_1,\alpha_1)\odot
\rB(M_2\subset N_2,\alpha_2)$, this limit is
$\beta_{u_1}\otimes\beta_{u_2}$. The same conclusion
holds on the von Neumann tensor product by density and continuity in $\|\cdot\|_2$,
and hence for every $x\in\rB(M\subset N,\alpha)$.
\end{proof}

\subsection{Ergodic cocycle perturbations}

We will now prove an analog of \cite[Theorem~A]{MV24} that every type $\III_1$ factor with separable
predual admits a faithful normal state with trivial centralizer.
We will apply this result in particular to show that every continuous
action of a noncompact second countable locally compact abelian group
with full Connes spectrum on a $\II_1$ factor with separable predual
admits an ergodic cocycle perturbation
(Corollary~\ref{full Connes spectrum ergodic cocycle perturbation}).

\begin{definition}
The action $\alpha:G\curvearrowright(M,\tau)$ is called \emph{full} if the
coboundary map
\[
 \partial_\alpha:\cU(M)\longrightarrow Z^1(\alpha),
 \qquad
 \partial_\alpha(v)_g=v\alpha_g(v^*),
\]
is open onto its range.  Equivalently, $\alpha$ is full if, for every
$\varepsilon>0$, there exist $K\Subset G$ and $\delta>0$ such that, for
every $v\in\cU(M)$,
\begin{equation}\label{explicit openness of the coboundary map}
 \sup_{g\in K}\|\alpha_g(v)-v\|_2<\delta
 \quad\Longrightarrow\quad
 \operatorname{dist}_2(v,\cU(M^\alpha))<\varepsilon.
\end{equation}
\end{definition}

\begin{proposition}\label{closed coboundaries and fixed point spectral gap}
Let $\omega$ be a cofinal ultrafilter on a sufficiently large directed
set. Then $\alpha$ is full if and only if
\[
 (M^\alpha)^\omega=M^{\omega,\alpha}.
\]
If $\alpha$ is full then $B^1(\alpha)$ is closed in $Z^1(\alpha)$. If $G$ is
second countable and $M$ has separable predual, then the converse also
holds: $B^1(\alpha)$ is closed if and only if $\alpha$ is full.

Finally, if $u \in \overline{B^1(\alpha)}$ then $\alpha$ is full if and only if $\alpha^u$ is full.
\end{proposition}

\begin{proof}
Assume first that $\alpha$ is full. Let
$V=(v_i)^\omega\in\cU(M^{\omega,\alpha})$, where the $v_i$ are unitaries
and are asymptotically $\alpha$-invariant uniformly on compact subsets of
$G$. Fullness implies
\[
 \lim_{i\to\omega}
 \operatorname{dist}_2(v_i,\cU(M^\alpha))=0.
\]
Choosing $w_i\in\cU(M^\alpha)$ asymptotically realizing this distance, we
obtain $V=(w_i)^\omega\in(M^\alpha)^\omega$. Since a von Neumann algebra
is linearly spanned by its unitaries, this proves
$M^{\omega,\alpha}\subset(M^\alpha)^\omega$. The reverse inclusion is
automatic.

Conversely, suppose that $\alpha$ is not full. There are $\varepsilon_0>0$
and a net $(v_i)_i$ in $\cU(M)$ such that
\[
 \partial_\alpha(v_i)\longrightarrow1
\]
uniformly on compact subsets of $G$, whereas
\[
 \operatorname{dist}_2(v_i,\cU(M^\alpha))\geq\varepsilon_0
 \qquad(i).
\]
This net defines a unitary
$V=(v_i)^\omega\in M^{\omega,\alpha}$. If
$(M^\alpha)^\omega=M^{\omega,\alpha}$, then
$V\in(M^\alpha)^\omega$. Unitary lifting inside $(M^\alpha)^\omega$ would
give unitaries $w_i\in M^\alpha$ such that
$\|v_i-w_i\|_2\to0$ along $\omega$, a contradiction. This proves the
equivalence.

We next show that fullness implies the closedness of $B^1(\alpha)$. Put
$H=\cU(M^\alpha)$. The quotient
$\cU(M)/H$ is complete for the metric
\[
 d(vH,wH)=\inf_{a\in H}\|v-wa\|_2,
\]
because the metric induced by $\|\cdot\|_2$ on $\cU(M)$ is complete and bi-invariant and
$H$ is closed. Let $(\partial_\alpha(v_i))_i$ be a net in $B^1(\alpha)$
converging to $u\in Z^1(\alpha)$. For $i,j$ and $g\in G$, we have
\[
 \|\partial_\alpha(v_i)_g-\partial_\alpha(v_j)_g\|_2
 =\|\partial_\alpha(v_j^*v_i)_g-1\|_2.
\]
Fullness therefore implies that $(v_iH)_i$ is a Cauchy net in
$\cU(M)/H$. Let $vH$ be its limit. By continuity of the coboundary map,
\[
 u=\lim_i\partial_\alpha(v_i)=\partial_\alpha(v),
\]
so $u\in B^1(\alpha)$. Hence $B^1(\alpha)$ is closed.

Assume finally that $G$ is second countable and $M_*$ is separable.
Proposition~\ref{topology and action on cocycles} shows that $Z^1(\alpha)$
is Polish and that the Polish group $\cU(M)$ acts continuously on it.
The orbit of the trivial cocycle is precisely $B^1(\alpha)$, the range of
the coboundary map $\partial_\alpha$. If $B^1(\alpha)$ is closed, this orbit is a
closed,
hence Polish, subspace of $Z^1(\alpha)$. Effros' open mapping theorem for
Polish group actions \cite[Theorem~2.1]{Ef65} implies that the coboundary map
$\partial_\alpha:\cU(M)\to B^1(\alpha)$ is open. Thus $\alpha$ is full.

It remains to prove the final assertion. Let $u\in\overline{B^1(\alpha)}$. Suppose first that $\alpha$ is
full. By the preceding implication,
$B^1(\alpha)$ is closed, and hence $u\in B^1(\alpha)$. Thus $\alpha^u$ is
conjugate to $\alpha$, so it is full. Now, observe that $u^* \in \overline{B^1(\alpha^u)}$ and $\alpha=(\alpha^u)^{u^*}$. So by the same argument if $\alpha^u$ is full then $\alpha$ is also full.
\end{proof}

\begin{theorem}\label{ergodic approximate coboundaries and fixed bicentralizer}
Let $G$ be a second countable locally compact group and let $(M,\tau)$ be
a tracial von Neumann algebra with separable predual.  Let
$\alpha:G\curvearrowright(M,\tau)$ be a faithful continuous
trace-preserving action. Suppose that $\alpha$ is not full, that is, $B^1(\alpha)$ is not closed in $Z^1(\alpha)$. Then the following properties are equivalent.
\begin{enumerate}[\rm (i)]
\item There is $v\in\overline{B^1(\alpha)}$ such that $\alpha^v$ is
ergodic.
\item The set
\[
 \{v\in\overline{B^1(\alpha)}\mid \alpha^v\text{ is ergodic}\}
\]
is a dense $G_\delta$ subset of $\overline{B^1(\alpha)}$.
\item $\rB(M,\alpha)\cap M^\alpha=\C 1$.
\end{enumerate}
\end{theorem}

\begin{proof}
Proposition \ref{relative bicentralizer isomorphism}, item~$(\rm ii)$, yields
\begin{equation}\label{fixed bicentralizer cocycle invariance}
 \beta_v\bigl(\rB(M,\alpha)\cap M^\alpha\bigr)
 =\rB(M,\alpha^v)\cap M^{\alpha^v}.
\end{equation}
If $\alpha^v$ is ergodic, the algebra on the right of
\eqref{fixed bicentralizer cocycle invariance} is scalar.  This proves
$(\rm i)\Rightarrow(\rm iii)$. Since $(\rm ii)\Rightarrow(\rm i)$ is immediate, it remains only to show that $(\rm iii) \Rightarrow (\rm ii)$. Let
$$\cE=\{u\in\overline{B^1(\alpha)}\mid M^{\alpha^u}=\C 1\}.$$
First we show that $\cE$ is a $G_\delta$-subset of the Polish space $\overline{B^1(\alpha)}$. For $u\in\overline{B^1(\alpha)}$ and $x \in M$, put
\[
 F_x(u)=\|\rE_{M^{\alpha^u}}(x)\|_2.
\]
The closed convex hull of the $\alpha^u$-orbit of $x$ has
$\rE_{M^{\alpha^u}}(x)$ as its unique element of minimal norm. 

 Therefore
\[
 F_x(u)=\inf\left\{
 \left\|\sum_{j=1}^rt_j\alpha^u_{g_j}(x)\right\|_2
 \ \middle|\ r\geq1,\ t_j\geq0,\ \sum_jt_j=1,\ g_j\in G
 \right\}.
\]
Every expression inside the infimum is continuous in $u$.  Thus $F_x$ is
upper semicontinuous. 

Choose a $\|\cdot\|_2$-dense sequence $(x_n)_n$ in the centered part of
the unit ball of $M$. Then we have
$$\cE
 =\bigcap_{n,m\geq1}\{u\in\overline{B^1(\alpha)}\mid F_{x_n}(u)<m^{-1}\}$$
 and  the upper semicontinuity of the functions $F_{x_n}$ shows that this is a $G_\delta$ subset of $\overline{B^1(\alpha)}$.

Now, assume that property~$(\rm iii)$ is satisfied.  We prove that $\cE$ is dense. Let
\[
 \cC=\bigcap_{x \in M}\{v\in\overline{B^1(\alpha)}\mid F_x
 \text{ is continuous at }v\}.\]
 The estimate
\begin{equation}\label{fixed expectation norm Lipschitz estimate}
 \left|\|\rE_{M^{\alpha^v}}(x_n)\|_2
       -\|\rE_{M^{\alpha^v}}(y)\|_2\right|
 \leq\|x_n-y\|_2
\end{equation}
shows that
 \[ \cC = \bigcap_{n\geq1}\{v\in\overline{B^1(\alpha)}\mid F_{x_n}
 \text{ is continuous at }v\}.
\]
Since an upper semicontinuous real-valued function on a Polish space has a
dense $G_\delta$ set of continuity points, we have that $\cC$
is a dense $G_\delta$ subset of $\overline{B^1(\alpha)}$. We will show that $\cC \subset \cE$.

Take $v\in\cC$ and put $\beta=\alpha^v$. We want to show that $\beta$ is ergodic. Since $\alpha$ is not full, $\beta$ is not full, hence
\begin{equation}\label{generic no spectral gap}
 (M^\beta)^\omega\subsetneq M^{\omega,\beta}.
\end{equation}
Equation \eqref{fixed bicentralizer cocycle invariance} and property~$(\rm iii)$
give
\begin{equation}\label{generic fixed bicentralizer intersection}
 \rB(M,\beta)\cap M^\beta=\C 1.
\end{equation}
The algebra $\rB(M,\beta)$ is $\beta$-invariant.  Consequently, the
conditional expectations onto $M^\beta$ and $\rB(M,\beta)$ commute.  It follows from
\eqref{generic fixed bicentralizer intersection} that
\begin{equation}\label{ultrapower fixed bicentralizer intersection}
 (M^\beta)^\omega\cap\rB(M,\beta)^\omega=\C 1.
\end{equation}
Proposition \ref{relative bicentralizer with ultrapower} gives
\[
 (M^{\omega,\beta})'\cap M^\omega
 \subset\rB(M,\beta)^\omega.
\]
Combining this inclusion with
\eqref{ultrapower fixed bicentralizer intersection}, we obtain
\begin{equation}\label{center of fixed ultrapower misses fixed sequences}
 \cZ(M^{\omega,\beta})\cap(M^\beta)^\omega=\C 1.
\end{equation}

Suppose that we have some
$y\in M^\beta \ominus \C 1$.  We claim that there is $U\in\cU(M^{\omega,\beta})$ such that
\begin{equation}\label{fixed element escapes fixed point ultrapower}
 UyU^*\notin(M^\beta)^\omega.
\end{equation}
Otherwise, the von Neumann algebra
\[
 D=\{UyU^*\mid U\in\cU(M^{\omega,\beta})\}''
\]
would be contained in $(M^\beta)^\omega$ and invariant under every inner
automorphism of $M^{\omega,\beta}$. The standard description of
von Neumann subalgebras invariant under every inner automorphism gives a projection $e\in\cZ(M^{\omega,\beta})\cap D$ such that
$$ De=M^{\omega,\beta}e, \qquad De^{\perp} \subset \cZ(M^{\omega,\beta})e^{\perp}.$$
Equation
\eqref{center of fixed ultrapower misses fixed sequences} forces $e=1$, hence $M^{\omega,\beta}=D$. In particular,
$M^{\omega,\beta}\subset(M^\beta)^\omega$, in contradiction with
\eqref{generic no spectral gap}.  The claim follows.

Choose $U$ as in
\eqref{fixed element escapes fixed point ultrapower}, write
$U=(u_i)^\omega$, and put
\[
 d=\operatorname{dist}_2(UyU^*,(M^\beta)^\omega)>0.
\]
Define
\[
 c^{(i)}_g=u_i^*\beta_g(u_i),
 \qquad
 q^{(i)}_g=c^{(i)}_g v_g.
\]
Then $q^{(i)}\in\overline{B^1(\alpha)}$ and $q^{(i)}\to v$
compact-uniformly along
$\omega$ because $U \in M^{\omega,\beta}$.  Moreover,
\[
 M^{\alpha^{q^{(i)}}}=u_i^*M^\beta u_i
\]
and hence 
\begin{equation}\label{fixed expectation drops along almost invariant conjugacy}
 \lim_{i\to\omega}
 \|\rE_{M^{\alpha^{q^{(i)}}}}(y)\|_2^2
 =\|\rE_{(M^\beta)^\omega}(UyU^*)\|_2^2
 =\|y\|_2^2-d^2.
\end{equation}
The left side is therefore strictly smaller in the limit than its value
$\|y\|_2^2$ at $v$.  This contradicts the continuity of $F_y$ at $v$.  Thus $M^\beta=\C 1$.
We have shown that $\cC$ is contained in $\cE$.  The set $\cE$ is consequently dense,
which proves $(\rm iii)\Rightarrow(\rm ii)$.
\end{proof}

\subsection{Minimal cocycle perturbations}
\label{subsec:minimal-perturbations}

Throughout this subsection, $G$ is a second countable locally compact
group, $M$ is a $\II_1$ factor with separable predual, and $\omega$ is a
free ultrafilter on $\N$. An action $\alpha:G\curvearrowright M$ is
\emph{minimal} if it is faithful and
$(M^\alpha)'\cap M=\C1$.
We will show that every faithful action with trivial bicentralizer
has a minimal cocycle perturbation.
This is an analog of Haagerup's theorem \cite[Theorem~3.1]{Ha87}:
every type $\III_1$ factor $N$ with separable predual and trivial
bicentralizer admits a faithful normal state $\varphi$ such that
$(N^\varphi)'\cap N=\C1$.

\begin{lemma}\label{almost invariant finite dimensional conjugacy}
Suppose that $M^{\omega,\alpha}$ is a factor, and let
$(F,\tau_F)$ be a finite-dimensional von Neumann algebra with a
faithful tracial state. For every compact $K\subset G$ and
$\varepsilon>0$, there are a compact $L\supset K$ and $\delta>0$
with the following property. If $\varphi,\psi:F\to M$ are unital
trace-preserving embeddings satisfying
\[
\begin{gathered}
\sup_{g\in L,\ a\in F,\ \|a\|\leq1}
\|\alpha_g(\varphi(a))-\varphi(a)\|_2<\delta,\\
\sup_{g\in L,\ a\in F,\ \|a\|\leq1}
\|\alpha_g(\psi(a))-\psi(a)\|_2<\delta,
\end{gathered}
\]
then there is $v\in\cU(M)$ such that
\begin{equation}\label{almost invariant finite dimensional conjugator}
\Ad(v)\circ\varphi=\psi,
\qquad
\sup_{g\in K}\|\alpha_g(v)-v\|_2<\varepsilon.
\end{equation}
\end{lemma}

\begin{proof}
We use factoriality to conjugate the embeddings in the fixed
ultrapower, then lift the conjugating unitary.
If the assertion fails, choose increasing compact sets $L_n\supset K$
whose interiors cover $G$, and counterexamples $\varphi_n,\psi_n$
with invariance errors less than $1/n$ on $L_n$.
They induce trace-preserving embeddings
$\varphi^\omega,\psi^\omega:F\to M^{\omega,\alpha}$.
Since the latter algebra is a factor, there is
$V\in\cU(M^{\omega,\alpha})$ with
$\Ad(V)\circ\varphi^\omega=\psi^\omega$.
Lift $V$ to unitaries $v_n\in M$. Then
\[
\lim_{n\to\omega}\sup_{a\in F,\ \|a\|\leq1}
\|v_n\varphi_n(a)v_n^*-\psi_n(a)\|_2=0,
\qquad
\lim_{n\to\omega}\sup_{g\in K}\|\alpha_g(v_n)-v_n\|_2=0.
\]
The standard polar-decomposition correction for finite-dimensional
embeddings gives unitaries $z_n$ with $\|z_n-1\|_2\to_\omega0$
and $\Ad(z_nv_n)\circ\varphi_n=\psi_n$.
Thus $z_nv_n$ satisfies
\eqref{almost invariant finite dimensional conjugator} for some $n$,
a contradiction.
\end{proof}

\begin{lemma}\label{small cocycle fixing a finite dimensional embedding}
Suppose that $M^{\omega,\alpha}$ is a $\II_1$ factor, and let
$(F,\tau_F)$ be a finite-dimensional von Neumann algebra with a
faithful tracial state. For every unital trace-preserving embedding
$\varphi:F\to M$, there is $u\in\overline{B^1(\alpha)}$ such that
$\alpha^u$ fixes $\varphi(F)$ pointwise.
Moreover, for every compact $K\subset G$ and $\varepsilon>0$, there
are a compact $L\supset K$ and $\delta>0$ such that
\[
\sup_{g\in L,\ a\in F,\ \|a\|\leq1}
\|\alpha_g(\varphi(a))-\varphi(a)\|_2<\delta
\]
allows us to choose $u$ with
\begin{equation}\label{small finite dimensional fixing cocycle}
\alpha_g^u\circ\varphi=\varphi\quad(g\in G),
\qquad
\sup_{g\in K}\|u_g-1\|_2<\varepsilon.
\end{equation}
\end{lemma}

\begin{proof}
We conjugate $\varphi$ to increasingly invariant embeddings so that
the associated coboundaries converge.
Choose increasing compact sets $K_n\supset K$, indexed by $n\geq0$,
whose interiors cover $G$, and positive numbers $\varepsilon_n$ with
$\sum_n\varepsilon_n<\varepsilon$.
Apply Lemma~\ref{almost invariant finite dimensional conjugacy} to
$(K_n,\varepsilon_n)$, obtaining $(L_n,\delta_n)$.
Enlarging $L_n$ and decreasing $\delta_n$, assume that
$L_n\supset K_n$, $L_n\subset L_{n+1}$, and $\delta_n\downarrow0$.

The $\II_1$ factor $M^{\omega,\alpha}$ contains a trace-preserving
copy of $(F,\tau_F)$. Lifting its matrix units and adjusting their
traces gives, for every $n$, a unital trace-preserving embedding
$\varphi_n:F\to M$ with invariance error less than $\delta_n$ on
$L_n$. If $\varphi$ satisfies this condition for $(L_0,\delta_0)$,
take $\varphi_0=\varphi$ and $w_0=1$. Otherwise choose
$w_0\in\cU(M)$ with $\Ad(w_0)\circ\varphi=\varphi_0$.
The preceding lemma gives unitaries $v_n$ such that
\[
\Ad(v_n)\circ\varphi_n=\varphi_{n+1},
\qquad
\sup_{g\in K_n}\|\alpha_g(v_n)-v_n\|_2<\varepsilon_n.
\]
Put $w_{n+1}=v_nw_n$ and $u^{(n)}_g=w_n^*\alpha_g(w_n)$.
Then $\Ad(w_n)\circ\varphi=\varphi_n$ and
\begin{equation}\label{successive finite dimensional fixing coboundaries}
\|u^{(n+1)}_g-u^{(n)}_g\|_2
=\|\alpha_g(v_n)-v_n\|_2.
\end{equation}
Thus $u^{(n)}$ converges uniformly on compact subsets to
$u\in\overline{B^1(\alpha)}$. For $a\in F$ and $g\in G$,
\[
\|\alpha_g^{u^{(n)}}(\varphi(a))-\varphi(a)\|_2
=\|\alpha_g(\varphi_n(a))-\varphi_n(a)\|_2\longrightarrow0,
\]
so $\alpha^u$ fixes $\varphi(F)$ pointwise.
When $w_0=1$, summing
\eqref{successive finite dimensional fixing coboundaries} gives
\eqref{small finite dimensional fixing cocycle}.
\end{proof}

\begin{lemma}\label{small cocycle fixing a matrix algebra}
Suppose that $\rB(M,\alpha)=\C1$. For every finite set $X\subset M$,
compact set $K\subset G$, and $\varepsilon>0$, there are a unital
matrix subalgebra $F\subset M$ and $u\in\overline{B^1(\alpha)}$
such that $\alpha^u$ fixes $F$ pointwise and
\begin{equation}\label{small fixed matrix algebra estimates}
\sup_{g\in K}\|u_g-1\|_2<\varepsilon,
\qquad
\|\rE_{F'\cap M}(x)-\tau(x)1\|_2<\varepsilon
\quad(x\in X).
\end{equation}
\end{lemma}

\begin{proof}
\emph{First choose the matrix algebra in the fixed ultrapower.}
Proposition~\ref{relative bicentralizer with ultrapower} gives
\[
(M^{\omega,\alpha})'\cap M^\omega=\C1.
\]
Thus $M^{\omega,\alpha}$ is a $\II_1$ factor, since a
finite-dimensional factor in $M^\omega$ cannot have scalar relative commutant.
Popa's classical averaging argument \cite{Po81}, applied to the
irreducible inclusion $M^{\omega,\alpha}\subset M^\omega$,
provides a unital matrix algebra $D\subset M^{\omega,\alpha}$ such that
\begin{equation}\label{fixed ultrapower matrix expectation}
\|\rE_{D'\cap M^\omega}(x)-\tau(x)1\|_2<\varepsilon/2
\quad(x\in X).
\end{equation}

\emph{Next lift the matrix algebra and fix it by a small cocycle.}
Lift the inclusion $D\subset M^{\omega,\alpha}$ to unital embeddings
$\varphi_n:D\to M$, and put $F_n=\varphi_n(D)$.
These embeddings are asymptotically invariant, uniformly on compact
sets, and $(\rE_{F_n'\cap M}(x))^\omega=\rE_{D'\cap M^\omega}(x)$
for $x\in X$.
Choose $n$ for which the expectation bounds in
\eqref{small fixed matrix algebra estimates} hold and $\varphi_n$
is sufficiently invariant to apply
Lemma~\ref{small cocycle fixing a finite dimensional embedding} with
$(K,\varepsilon)$. It gives $u\in\overline{B^1(\alpha)}$ with
$\sup_{g\in K}\|u_g-1\|_2<\varepsilon$ such that $\alpha^u$ fixes
$F=F_n$ pointwise.
\end{proof}

\begin{theorem}\label{trivial bicentralizer minimal perturbation}
Let $\alpha:G\curvearrowright M$ be a faithful continuous action.
If $\rB(M,\alpha)=\C1$, then there is
$u\in\overline{B^1(\alpha)}$ such that $\alpha^u$ is minimal.
\end{theorem}

\begin{proof}
\emph{Construct fixed matrix algebras with decreasing relative commutants.}
Choose a $\|\cdot\|_2$-dense sequence $(x_j)_{j\geq1}$ in the
centered unit ball of $M$, increasing compact sets $K_n$ whose
interiors cover $G$, and positive numbers $\varepsilon_n$ with
$\sum_n\varepsilon_n<\infty$.
Start with $F_0=\C1$ and $u^{(0)}=1$.
Suppose that $u^{(n)}\in\overline{B^1(\alpha)}$ and
$\alpha_n=\alpha^{u^{(n)}}$ fixes the matrix algebra $F_n$ pointwise.
Put $N_n=F_n'\cap M$ and $\gamma_n=\alpha_n|_{N_n}$.
Propositions \ref{relative bicentralizer isomorphism} and
\ref{relative bicentralizer with ultrapower}, together with
$M=F_n\ovt N_n$, give
\[
(M^{\omega,\alpha_n})'\cap M^\omega=\C1,
\qquad
(N_n^{\omega,\gamma_n})'\cap N_n^\omega=\C1.
\]
Thus $\rB(N_n,\gamma_n)=\C1$. Apply Lemma
\ref{small cocycle fixing a matrix algebra} in $N_n$, with the finite
set $\{\rE_{N_n}(x_j)\mid 1\leq j\leq n+1\}$, compact set $K_n$,
and accuracy $\varepsilon_n$.
It gives a matrix algebra $Q_n\subset N_n$ and an approximate
coboundary $c^{(n)}$ for $\gamma_n$ such that
$\gamma_n^{c^{(n)}}$ fixes $Q_n$ pointwise.
Set $F_{n+1}=F_n\vee Q_n$ and
$u^{(n+1)}=c^{(n)}u^{(n)}$.
Then $u^{(n+1)}\in\overline{B^1(\alpha)}$, the new action fixes
$F_{n+1}$ pointwise, and
\begin{equation}\label{minimal perturbation induction estimates}
\sup_{g\in K_n}\|u^{(n+1)}_g-u^{(n)}_g\|_2<\varepsilon_n,
\qquad
\|\rE_{F_{n+1}'\cap M}(x_j)\|_2<\varepsilon_n
\quad(1\leq j\leq n+1).
\end{equation}

\emph{Pass to the limit and check minimality.}
The first estimate gives a compact-uniform limit
$u\in\overline{B^1(\alpha)}$.
Every subsequent action fixes $F_n$, so $\alpha^u$ fixes
$R=(\bigcup_nF_n)''$ pointwise.
The second estimate and density imply
\[
R'\cap M=\bigcap_n(F_n'\cap M)=\C1.
\]
Consequently $(M^{\alpha^u})'\cap M=\C1$.

Finally, if $\alpha_g=\Ad(v)$ for some $v\in\cU(M)$, then
$\|[v,w_i]\|_2\to0$ for every asymptotically $\alpha$-invariant
unitary net $(w_i)_i$. Thus $v\in\rB(M,\alpha)=\C1$.
Faithfulness of $\alpha$ therefore
implies that $\alpha_g$ is outer for every $g\ne e$. This property
is preserved by cocycle perturbations. Hence $\alpha^u$ is faithful
and therefore minimal.
\end{proof}

\subsection{Approximate 1-cocycle vanishing for Rokhlin actions}

\begin{theorem}
Let $\alpha : G \curvearrowright M$ be a continuous action on a $\II_1$ factor $M$. Suppose that $G$ is amenable and that $\alpha$ has the Rokhlin property. Then 
$$Z^1(\alpha)=\overline{B^1(\alpha)}.$$
\end{theorem}

\begin{proof}
Let $u\in Z^1(\alpha)$.  On
$N=\mathbb M_2(M)$ consider the cocycle perturbation
\[
 \theta_g
 =\Ad\begin{pmatrix}1&0\\0&u_g\end{pmatrix}
  \circ(\id_{\mathbb M_2}\otimes\alpha_g).
\]
Thus
\[
 \theta_g\begin{pmatrix}a&b\\c&d\end{pmatrix}
 =\begin{pmatrix}
   \alpha_g(a)&\alpha_g(b)u_g^*\\
   u_g\alpha_g(c)&u_g\alpha_g(d)u_g^*
  \end{pmatrix}.
\]
Under the canonical identification
\[
 N_{\omega,\theta}=1_{\mathbb M_2}\otimes M_{\omega,\alpha},
\]
the action $\theta_\omega$ corresponds to $\alpha_\omega$: the cocycle
commutes with every element of $M_{\omega,\alpha}$.  Thus the Koopman
representations of $\theta_\omega$ and $\alpha_\omega$ agree.
Corollary~\ref{Rokhlin central sequence faithfulness} shows that
$\theta$ is Rokhlin.

Since $G$ is amenable, Theorem \ref{Rokhlin and trivial bicentralizer}
gives
\[
 \rB(N,\theta)=\C 1.
\]
In particular, $N^{\omega,\theta}$ is a factor.  The projections
$e_{11},e_{22}\in N^{\omega,\theta}$ have the same trace, so they are
equivalent in this factor.  Hence there is
$W\in N^{\omega,\theta}$ such that
\[
 W^*W=e_{11},\qquad WW^*=e_{22}.
\]
Consequently, $W=e_{21}V$ for some unitary
$V\in M^\omega_\alpha$.  The equality $\theta_g^\omega(W)=W$ gives
\[
 u_g\alpha_g^\omega(V)=V,
\]
and therefore
\[
 u_g=V\alpha_g^\omega(V^*)
 \qquad(g\in G).
\]
Thus $u\in B^1(\alpha^\omega)$.  Proposition
\ref{approximate coboundary ultrapower} implies that
$u\in\overline{B^1(\alpha)}$.  The reverse inclusion is automatic.
\end{proof}

\section{The dual bicentralizer action for abelian groups}
\label{sec:abelian-groups}
In this section, the locally compact group $G$ is assumed to be abelian, $(M,\tau)$ is a tracial von Neumann algebra, and $\alpha:G\curvearrowright(M,\tau)$ is a continuous trace-preserving action.

For trace-preserving actions of locally compact abelian groups with full
Connes spectrum, we prove that every unitary $1$-cocycle is an approximate
coboundary. The proof compares the fixed parts of the analytic and algebraic
bicentralizers. This gives an analogue of Connes--St{\o}rmer transitivity \cite{CS78}
and defines a canonical action of the dual group on the bicentralizer. 

We then determine the fixed points and eigenoperators of this dual bicentralizer action. Approximate eigenstates reduce these questions to commutation and
intertwining relations in the original algebra. In particular, the dual bicentralizer
action is ergodic on the bicentralizer of a factor, and proper outerness
of the original action is equivalent to weak mixing of the dual bicentralizer action.
We then interpret the bicentralizer conjecture as a fixed point property
inside the crossed product, explaining both its relation to the relative
fixed point theorem and the obstruction to applying that theorem directly.

We conclude with an application of the dual bicentralizer action to the
intermediate subfactor property.

For a strongly continuous unitary representation $U:G\to\cU(H)$, let
$\mathbb{E}_U$ be its spectral measure, with the convention
\[
U_t=\int_{\widehat G}\langle p,t\rangle\,\rd\mathbb{E}_U(p).
\]
We denote the corresponding spectral representation by
\[
\widehat U:\rC_0(\widehat G)\longrightarrow\B(H),
\qquad
\widehat U(\varphi)=\int_{\widehat G}\varphi(p)\,\rd\mathbb{E}_U(p).
\]
Thus, for $f\in\rL^1(G)$,
\[
\widehat U(\check f)=U(f),
\qquad
\check f(p)=\int_G f(t)\langle p,t\rangle\,\rd t.
\]
The spectral calculus extends to bounded Borel functions; for a Borel set
$\Omega\subset\widehat G$, we write its spectral projection as
$\mathbb{E}_U(\Omega)=\widehat U(1_\Omega)$.

\subsection{Approximate 1-cocycle vanishing for abelian groups}

\begin{lemma}\label{conditional expectation on algebraic bicentralizer}
Let $\widetilde\tau$ be the canonical trace on $M \rtimes_\alpha G$. For $a,b\in M$,
the equality $\rE_{\rb(M,\alpha)}(a)=\rE_{\rb(M,\alpha)}(b)$ holds if and only if
\[
 \widetilde\tau(ahz)=\widetilde\tau(bhz)
\]
for all $z\in M' \cap (M \rtimes_\alpha G)$ and $h\in\rL_\alpha(G)$ with
$\widetilde\tau(|h|)<\infty$. If $a,b\in M^\alpha$, it suffices to
require these equalities for $z\in\cZ(M \rtimes_\alpha G)$.
\end{lemma}

\begin{proof}
By Proposition \ref{crossed product of bicentralizer}, we have $$\widetilde{\tau}(ahz)=\widetilde{\tau}(\rE_{\rb(M,\alpha) \rtimes_\alpha G}(a) hz).$$
Moreover, since $a \in M$, we have $$\rE_{\rb(M,\alpha) \rtimes_\alpha G}(a)=\rE_{\rb(M,\alpha)}(a).$$
Therefore
$$\widetilde{\tau}(ahz)=\widetilde{\tau}(bhz)$$
if and only if
$$\widetilde{\tau}(\rE_{\rb(M,\alpha) \rtimes_\alpha G}(a) hz) =\widetilde{\tau}(\rE_{\rb(M,\alpha) \rtimes_\alpha G}(b) hz)$$
for all $z\in M' \cap (M \rtimes_\alpha G)$ and $h\in\rL_\alpha(G)$ with
$\widetilde\tau(|h|)<\infty$.

The conclusion follows from the fact that the normal linear functionals $\widetilde{\tau}(hz \cdot)$ for $h \in \rL_\alpha(G)$ and $z \in M' \cap (M \rtimes_\alpha G)$ separate the points of $\rb(M,\alpha) \rtimes_\alpha G$.

If $a \in M^\alpha$, then $\rE_{\rb(M,\alpha) \rtimes_\alpha G}(a) \in \rb(M,\alpha) \cap M^\alpha$. Therefore, we can apply Proposition \ref{strong center bicentralizer} and conclude with the same argument.
\end{proof}

\begin{lemma} \label{lemma equivalence in the crossed product}
Take $a \in M_+ \cap M^\alpha$ with $\tau(a)=1$. For every neighbourhood
$\Omega\subset\widehat G$ of $0$, there exists
$\Phi\in\B(\rL^2(M))^+_*$ such that
\[ \Phi(\rho_M(x))=\tau(a x), \quad (x \in M^\alpha),\]
\[\Phi(\lambda_M(x))=\tau(\rE_{\rb(M,\alpha)}(a)x), \quad (x \in M^\alpha),\]
\[\Phi(\mathbb{E}_{U^\alpha}(\Omega))=1.
\] 
\end{lemma}

\begin{proof}
Choose a compact symmetric neighborhood of the identity $K \Subset\widehat G$ such
that $K+K\subset\Omega$ and define a projection $p=\mathbb{E}_{u^\alpha}(K) \in M \rtimes_\alpha G$. 

Let $b=\rE_{\rb(M,\alpha)}(a)$. The positive elements $ap=pa$ and $bp=pb$ in $M \rtimes_\alpha G$ satisfy the relation $\widetilde{\tau}(apz)=\widetilde{\tau}(bpz)$ for every $z \in \cZ(M \rtimes_\alpha G)$. Thus $ap$ and $bp$ have the same center-valued trace in the finite
corner $p(M\rtimes_\alpha G)p$. Thus by a classical maximality argument, we can find partial isometries $v_i$
in this corner and numbers $\mu_i>0$ such that
\[
 \sum_i\mu_i v_i^*v_i= ap,
 \qquad
 \sum_i\mu_i v_iv_i^*= bp.
\]
The sums converge strongly, and
$\sum_i\mu_i\widetilde\tau(v_i^*v_i)=\mu(K)$.
Identify $\rL^2(M\rtimes_\alpha G)$ with
$\rL^2(M)\otimes\rL^2(G)$ and let
\[\Phi \in \B(\rL^2(M))_* : T \mapsto \mu(K)^{-1}\sum_{i \in I} \mu_i \langle (T \otimes 1) \widehat{v_i},\widehat{v_i} \rangle. \]
Take $x \in M^\alpha$. By using the relation $\rho_{M \rtimes_\alpha G}(x)=\rho_M(x) \otimes 1$, we see that $\Phi$ satisfies
\[
 \Phi(\rho_M(x))=\mu(K)^{-1}\sum_{i \in I} \mu_i \widetilde{\tau}(v_i^*v_ix)=\mu(K)^{-1} \widetilde{\tau}(pax)=\tau(ax).
\]
Similarly, by using the relation $\lambda_{M \rtimes_\alpha G}(x)=\lambda_M(x) \otimes 1$, which holds because $x \in M^\alpha$, we see that
\[
 \Phi(\lambda_M(x))=\mu(K)^{-1}\sum_{i \in I} \mu_i \widetilde{\tau}(xv_iv_i^*)=\mu(K)^{-1} \widetilde{\tau}(xbp)=\tau(bx).
\]
It remains to check that
\[
\Phi(\mathbb{E}_{U^\alpha}(\Omega))=1.
\]
Since $v_i=p v_i p$ for all $i \in I$, the vector $\widehat{v_i}$ lies in the range of the following two projections 
\[
\begin{aligned}
\lambda_{M\rtimes_\alpha G}(p)&=\mathbb{E}_{1\otimes\lambda_G}(K),\\
\rho_{M\rtimes_\alpha G}(p)&=\mathbb{E}_{U^\alpha\otimes_G\lambda_G}(K),
\end{aligned}
\]
hence
$$ \Sp_{1 \otimes \lambda_G}(\widehat{v_i}) \subset K, \quad \Sp_{U^\alpha \otimes_G \lambda_G}(\widehat{v_i}) \subset K.$$
Then the product rule implies that 
$$ \Sp_{U^\alpha \otimes 1}(\widehat{v_i}) \subset K+K \subset \Omega.$$
In particular, we have $\Phi(\mathbb{E}_{U^\alpha}(\Omega))=1$.
\end{proof}

\begin{theorem}\label{fixed point analytic algebraic bicentralizer}
We have $\rB(M,\alpha)\cap M^\alpha=\rb(M,\alpha)\cap M^\alpha$.
In particular, if $\Gamma(\alpha)=\widehat G$, then
\[
\rB(M,\alpha)\cap M^\alpha=\cZ(M)^\alpha.
\]
\end{theorem}
\begin{proof}
By homogeneity, take $a \in M_+ \cap M^\alpha$ with $\tau(a)=1$. For every $\Omega \ll \widehat{G}$, take a state $\Phi_\Omega$ as in Lemma
\ref{lemma equivalence in the crossed product}. Let $\Phi$ be a weak$^*$ accumulation point of the net $(\Phi_\Omega)_{\Omega \ll \widehat{G}}$ where we direct the neighbourhoods of the identity $\Omega \ll \widehat G$ by reverse inclusion.  Then we have
\[ \Phi(\rho_M(x))=\tau(a x), \quad (x \in M^\alpha),\]
\[\Phi(\lambda_M(x))=\tau(\rE_{\rb(M,\alpha)}(a)x), \quad (x \in M^\alpha),\]
\[
 \Phi(\widehat U^\alpha(f))=f(0)
 \qquad(f\in\rC_0(\widehat G)).
\]
Now, take a ucp $\vartheta \in \mathrm{UCP}(\B(\rL^2(M)))$ in the weak$^*$-closed convex hull of $\{ \Ad(U^{\alpha}_t) \mid t \in  G\}$ such that $$\vartheta(\lambda_M(x))=\lambda_M(\rE_{M^\alpha}(x)),$$
$$\vartheta(\rho_M(x))=\rho_M(\rE_{M^\alpha}(x)).$$
Then, after replacing $\Phi$ by $\Phi \circ \vartheta$, and using the fact that $a$ and $\rE_{\rb(M,\alpha)}(a)$ are in $M^\alpha$, we obtain 
\[ \Phi(\rho_M(x))=\tau(a x), \quad (x \in M),\]
\[\Phi(\lambda_M(x))=\tau(\rE_{\rb(M,\alpha)}(a)x), \quad (x \in M),\]
\[
 \Phi(\widehat U^\alpha(f))=f(0)
 \qquad(f\in\rC_0(\widehat G)).
\]
 Lemma
\ref{equivariant binormal bicentralizer lemma} therefore gives for every $x \in \rB(M,\alpha)$
\[
 \tau(ax)=\Phi(\rho_M(x))=\Phi(\lambda_M(x))
 =\tau(\rE_{\rb(M,\alpha)}(a)x).
\]
This means that $\rE_{\rB(M,\alpha)}(a)=\rE_{\rb(M,\alpha)}(a)$.
If, in addition, $a\in\rB(M,\alpha)$, the left-hand side equals $a$.
Decomposing arbitrary fixed elements of $\rB(M,\alpha)$ into linear
combinations of positive ones gives
$\rB(M,\alpha)\cap M^\alpha\subset\rb(M,\alpha)\cap M^\alpha$.
The reverse inclusion follows from
$\rb(M,\alpha)\subset\rB(M,\alpha)$.
The final assertion follows from Proposition~\ref{strong center bicentralizer}.
\end{proof}

\begin{corollary}\label{fixed projection equivalence algebraic bicentralizer}
Let $\alpha:G\curvearrowright M$ be a trace-preserving action of a
locally compact abelian group, and let $\omega$ be a large enough cofinal
ultrafilter. Two projections $p,q\in M^\alpha$ are equivalent in
$M^{\omega,\alpha}$ if and only if
\[
 \rE_{\rb(M,\alpha)}(p)=\rE_{\rb(M,\alpha)}(q).
\]
If $\Gamma(\alpha)=\widehat G$, they are equivalent in
$M^{\omega,\alpha}$ if and only if they are equivalent in $M$.
\end{corollary}

\begin{proof}
The proof of Theorem~\ref{fixed point analytic algebraic bicentralizer}
shows that the expectations onto the analytic and algebraic
bicentralizers agree on $M^\alpha$. The first assertion therefore
follows from Proposition~\ref{equivalence projections relative bicentralizer}.

For the second assertion, by Proposition \ref{strong center bicentralizer} we have
$$ \rE_{\rb(M,\alpha)}(a)=\rE_{\cZ(M)}(a)
 \qquad(a\in M^\alpha).$$
The map $\rE_{\cZ(M)}$ is the normalized center-valued trace of the
finite algebra $M$. Consequently, the first assertion and the
center-valued trace criterion for equivalence in $M$ give the result.
\end{proof}

\begin{corollary} \label{haagerup stormer for flows}
Suppose that
$(M^\alpha)'\cap M\subset M^\alpha$. Then $\rB(M,\alpha)=\rb(M,\alpha)$ and if we assume that $\Gamma(\alpha)=\widehat{G}$ then $\rB(M,\alpha)=\cZ(M)^\alpha$.
\end{corollary}

\begin{proof}
We have
\[
 \rB(M,\alpha)\subset(M^\alpha)'\cap M\subset M^\alpha.
\]
Proposition \ref{algebraic bicentralizer fixed point commutant} likewise
gives $\rb(M,\alpha)\subset M^\alpha$.  Theorem
\ref{fixed point analytic algebraic bicentralizer} now yields
$\rB(M,\alpha)=\rb(M,\alpha)$. If $\Gamma(\alpha)=\widehat G$, then Proposition~\ref{strong center bicentralizer} gives $\rb(M,\alpha)=\cZ(M)^\alpha$.
\end{proof}

The following theorem is the analog of the Connes--St\o rmer transitivity theorem for type $\III_1$ factors \cite{CS78}. In fact, one can adapt our proof to obtain a new proof of the Connes--St\o rmer theorem. Corollary \ref{haagerup stormer for flows} provides an analog of the main result of Haagerup--St\o rmer \cite{HS90}. Again our proof can be adapted to provide a unified proof of their result without disintegration and type distinction.
\begin{theorem}\label{approximate cocycle vanishing theorem}
Let $\alpha:G\curvearrowright M$ be a trace-preserving action of a
locally compact abelian group. If $\Gamma(\alpha)=\widehat G$, then
$Z^1(\alpha)=\overline{B^1(\alpha)}$.
\end{theorem}

\begin{proof}
Take $v\in Z^1(\alpha)$ and equip $N=\mathbb M_2(\C)\ovt M$ with
its normalized tensor product trace. Define
\[
 \gamma_t
 =\Ad\begin{pmatrix}1&0\\0&v_t\end{pmatrix}
   \circ(\id\otimes\alpha_t).
\]
The standard isomorphism
\[
 N\rtimes_\gamma G
 \longrightarrow\mathbb M_2(\C)\ovt(M\rtimes_\alpha G)
\]
fixes $N$ and sends its canonical group unitary at $t$ to
$\operatorname{diag}(u_t,v_tu_t)$. It intertwines the dual actions.
Since the center of a matrix amplification is the amplified center,
we have
$\Gamma(\gamma)=\Gamma(\alpha)=\widehat G$.

The projections $e_{11},e_{22}$ belong to $N^\gamma$ and are equivalent
in $N$. Corollary~\ref{fixed projection equivalence algebraic bicentralizer}
gives a partial isometry $W\in N^{\omega,\gamma}$ satisfying
\[
 W^*W=e_{11},\qquad WW^*=e_{22}.
\]
Write $W=e_{21}V$ with $V\in\cU(M^\omega)$. Its invariance gives
\[
 v_t\alpha_t^\omega(V)=V,
 \qquad
 v_t=V\alpha_t^\omega(V^*)
 \qquad(t\in G).
\]
The cocycle $t\mapsto\operatorname{diag}(1,v_t)$ is continuous, so
cocycle perturbation preserves the equicontinuous ultrapower.
Consequently $V\in M^\omega_\alpha$. Proposition
\ref{approximate coboundary ultrapower} now gives
$v\in\overline{B^1(\alpha)}$.
\end{proof}

\begin{corollary}\label{full Connes spectrum ergodic cocycle perturbation}
Let $G$ be a noncompact second countable locally compact abelian group
and let $\alpha:G\curvearrowright M$ be a continuous action on a
$\II_1$ factor with separable predual. If $\Gamma(\alpha)=\widehat G$,
then $\alpha$ is cocycle equivalent to an ergodic action.
\end{corollary}

\begin{proof}
We first show that $\alpha$ is not full. Otherwise,
Proposition~\ref{closed coboundaries and fixed point spectral gap} and
Theorem~\ref{approximate cocycle vanishing theorem} would give
$B^1(\alpha)=Z^1(\alpha)$. Applying this to scalar cocycles, for each
$p\in\widehat G$ we obtain $u_p\in\cU(M)$ such that
\[
\alpha_t(u_p)=\langle p,t\rangle u_p\qquad(t\in G).
\]
Since $G$ is noncompact, $\widehat G$ is nondiscrete. Choose nonzero
$p_i\to0$ in $\widehat G$. The unitaries $u_{p_i}$ are asymptotically
$\alpha$-invariant uniformly on compact subsets of $G$, whereas
$\rE_{M^\alpha}(u_{p_i})=0$. Hence
\[
\operatorname{dist}_2(u_{p_i},\cU(M^\alpha))=\sqrt2,
\]
contradicting fullness.

The full Connes spectrum assumption also makes $\alpha$ faithful, and
Theorem~\ref{fixed point analytic algebraic bicentralizer} gives
$\rB(M,\alpha)\cap M^\alpha=\C1$. Theorem
\ref{ergodic approximate coboundaries and fixed bicentralizer} therefore
provides a cocycle $u\in Z^1(\alpha)$ such that $\alpha^u$ is ergodic.
\end{proof}

\begin{corollary}\label{Connes spectrum approximate eigenunitaries}
For every $p\in\Gamma(\alpha)$, there exists a net
$(u_i)_{i\in I}$ in $\cU(M)$ such that
\[
\lim_i\sup_{t\in K}
\|\alpha_t(u_i)-\langle p,t\rangle u_i\|_2=0
\qquad\text{for every compact }K\subset G.
\]
\end{corollary}

\begin{proof}
Fix $p\in\Gamma(\alpha)$, put $P=M\rtimes_\alpha G$, and let
$\widetilde\tau_P$ be its canonical trace. The characterization of the
Connes spectrum by the dual action on the center gives
$\widehat\alpha_{-p}|_{\cZ(P)}=\id$.
Equip $N=\mathbb M_2(\C)\ovt M$ with its normalized tensor product
trace and define
\[
\gamma_t=\Ad\begin{pmatrix}1&0\\0&\langle-p,t\rangle\end{pmatrix}
\circ(\id\otimes\alpha_t).
\]
The standard identification
$N\rtimes_\gamma G\cong\mathbb M_2(\C)\ovt P$ fixes $N$ and sends
the canonical group unitary at $t$ to
$\operatorname{diag}(u_t^\alpha,\widehat\alpha_{-p}(u_t^\alpha))$.
Thus it sends $\rL_\gamma(G)$ onto
\[
\{\operatorname{diag}(h,\widehat\alpha_{-p}(h))
\mid h\in\rL_\alpha(G)\}.
\]
Since the dual action preserves $\widetilde\tau_P$ and fixes
$\cZ(P)$ at $-p$, we have
\[
\widetilde\tau_P(hz)
=\widetilde\tau_P(\widehat\alpha_{-p}(h)z)
\]
for every $z\in\cZ(P)$ and $h\in\rL_\alpha(G)$ with
$\widetilde\tau_P(|h|)<\infty$.
Lemma~\ref{conditional expectation on algebraic bicentralizer}, applied
to the fixed projections $e_{11},e_{22}\in N^\gamma$, therefore gives
$\rE_{\rb(N,\gamma)}(e_{11})=\rE_{\rb(N,\gamma)}(e_{22})$.
Corollary~\ref{fixed projection equivalence algebraic bicentralizer}
gives a partial isometry $W=e_{21}U\in N^{\omega,\gamma}$ with
$W^*W=e_{11}$ and $WW^*=e_{22}$, for a large enough cofinal ultrafilter
$\omega$. As in the preceding proof, $U\in\cU(M^\omega_\alpha)$, and
invariance gives
$\alpha_t^\omega(U)=\langle p,t\rangle U$ for every $t\in G$.
Proposition~\ref{approximate coboundary ultrapower}, applied to the
scalar cocycle $t\mapsto\langle-p,t\rangle1$, now gives the required net.
\end{proof}

We can now define the dual bicentralizer action explicitly using
approximate eigenunitaries.
\begin{corollary} \label{full dual bicentralizer action}\label{dual unitary}
Suppose that $\Gamma(\alpha)=\widehat G$ and let $M\subset N$ be a
tracial inclusion. For $p\in\widehat G$, choose a net
$(u_i)_i\subset\cU(M)$ of approximate $p$-eigenunitaries as in
Corollary~\ref{Connes spectrum approximate eigenunitaries}. The formula
\[
\beta_p(x)=\lim_i u_i^*xu_i
\qquad(x\in\rB(M\subset N,\alpha)),
\]
where the limit is in $\|\cdot\|_2$, is independent of the chosen net
and defines a continuous trace-preserving action
\[
\beta:\widehat{G}\curvearrowright\rB(M\subset N,\alpha) : p\longmapsto\beta_p.
\]
Moreover, for
$u\in\overline{B^1(\alpha)}$ and $p\in \widehat G$, we have
\[
\beta_p\circ\beta_u
=\beta_u\circ\beta_p.
\]
\end{corollary}

\begin{proof}
For $p\in\widehat G$, denote by $\chi_p$ the scalar cocycle
$\chi_p(t)=\langle p,t\rangle1$. The approximate eigenunitary relation gives
\[
\partial_\alpha(u_i^*)_t=u_i^*\alpha_t(u_i)\longrightarrow\chi_p(t)
\]
uniformly on compact subsets of $G$. Proposition
\ref{relative bicentralizer isomorphism} therefore gives the stated limit
and identifies $\beta_p$ with $\beta_{\chi_p}$. Since
$\alpha^{\chi_p}=\alpha$, each $\beta_p$ is an automorphism.
The composition and continuity assertions follow from the same proposition
and the continuity of $p\mapsto\chi_p$ for compact-uniform convergence.
Finally, scalar cocycles commute with $u$, and hence
\[
\beta_p\circ\beta_u
=\beta_{\chi_pu}
=\beta_{u\chi_p}
=\beta_u\circ\beta_p.
\]
\end{proof}

Our sign convention makes $\beta_p$ the bicentralizer isomorphism
associated with the scalar cocycle $t\mapsto\langle p,t\rangle1$.
It also ensures that a unitary $v\in\cU(M)$ implementing $\alpha_t$
has dual frequency $t$: indeed,
\[
\beta_p(v)=\lim_i u_i^*vu_i
=\lim_i u_i^*\alpha_t(u_i)v=\langle p,t\rangle v.
\]

\begin{remark}\label{fixed ultrapower factoriality LCA counterexample}
Assume that $M$ is a $\II_1$ factor and $\Gamma(\alpha)=\widehat G$.
Theorem~\ref{fixed point analytic algebraic bicentralizer} gives
$\rB(M,\alpha)\cap M^\alpha=\C1$. By
Proposition~\ref{relative bicentralizer with ultrapower}, this means that
\[
(M^{\omega,\alpha})'\cap M^\alpha=\C1
\]
for a sufficiently large cofinal ultrafilter $\omega$.
One might ask whether the stronger assertion that $M^{\omega,\alpha}$
is a factor also holds. For modular flows on type $\III_1$ factors,
the analogous assertion is true: the asymptotic centralizer
$(N^\omega)_{\varphi^\omega}$ is a $\II_1$ factor for every faithful
normal state $\varphi$ on a type $\III_1$ factor $N$. This is an easy
consequence of Connes--St{\o}rmer transitivity \cite{CS78}. See
\cite[Lemma~2.3]{MV24}.

For actions of arbitrary locally compact abelian groups, however,
factoriality can fail even when $G=\R^2$ and $M$ is hyperfinite,
as the following counterexample shows.
Let $D=\mathbb Q^2$, viewed as a discrete group, and put
\[
M=\rL_c(D),\qquad c(p,q)=e^{2\pi\ri p_1q_2}.
\]
Write $U_p$ for the canonical unitaries. The commutator bicharacter
\[
e^{2\pi\ri(p_1q_2-p_2q_1)}
\]
is nondegenerate, so $M$ is a factor.
It is hyperfinite of type $\II_1$ since $D$ is infinite and amenable.
The action
\[
\alpha_t(U_p)=e^{2\pi\ri(p_1t_1+p_2t_2)}U_p
\qquad(t\in\R^2,\ p\in D)
\]
is ergodic and satisfies
$\Gamma(\alpha)=\Sp(\alpha)=\overline D=\R^2$.
For $u_n=U_{(1/n,0)}$, we have
\[
\alpha_{(s,t)}(u_n)=e^{2\pi\ri s/n}u_n,
\qquad \Ad(u_n)=\alpha_{(0,1/n)}.
\]
Thus $u=(u_n)^\omega\in M^{\omega,\alpha}$ for every free ultrafilter
$\omega$ on $\N$. If $x=(x_n)^\omega\in M^{\omega,\alpha}$, then
\[
\|[u_n,x_n]\|_2
=\|\alpha_{(0,1/n)}(x_n)-x_n\|_2\longrightarrow_\omega0.
\]
Hence $u$ is central in $M^{\omega,\alpha}$. Since
$\tau(u_n^k)=0$ for every nonzero integer $k$, this central unitary
is Haar, so $M^{\omega,\alpha}$ has diffuse center.
One can even show that $u$ belongs to the center of the whole
equicontinuous part $M^\omega_\alpha$.
Remarkably, when $G=\R$, the fixed ultrapower is always a $\II_1$
factor under the full Connes spectrum assumption. See
Theorem~\ref{full Connes spectrum fixed ultrapower factor}.
Its proof relies on the explicit description of the bicentralizer in
Theorem~\ref{bicentralizer conjecture for flows} and, in the rank two case,
its identification with the irrational rotation model in
Theorem~\ref{non outer bicentralizer model}.
\end{remark}

\subsection{Approximate eigenstates for locally compact abelian groups}

Throughout this subsection, let $G$ be a locally compact abelian group and
let $U:G\to\cU(H)$ be a strongly continuous unitary representation.
No countability assumption is imposed on $G$.

\begin{definition}\label{dual translation definition}
For $p\in\widehat G$ and $f\in\rC_0(\widehat G)$, define the left
translate $L_pf\in\rC_0(\widehat G)$ by
\[
(L_pf)(q)=f(p+q)
\qquad(q\in\widehat G).
\]
\end{definition}

\begin{lemma}\label{LCA spectral gap density lemma}
For every
neighbourhood $V$ of $0$ in $\widehat G$, there is a nonnegative
probability density $h\in\rL^1(G)$ such that
\begin{equation}\label{LCA spectral gap density}
\inf_{q\notin V}|1-\check h(q)|>0.
\end{equation}
\end{lemma}

\begin{proof}
Choose a compact set $K\subset G$ and $\delta>0$ such that
\[
\left\{q\in\widehat G\ \middle|\
\sup_{t\in K}|\langle q,t\rangle-1|<\delta\right\}\subset V.
\]
Let $H<G$ be the open $\sigma$-compact subgroup generated by a
relatively compact identity neighbourhood. The image of $K$ in the
discrete group $G/H$ is finite. Adjoining finitely many
representatives gives an open $\sigma$-compact subgroup $G_0$ which
contains $K$. Take a probability density which is positive almost
everywhere on $G_0$ and vanishes off $G_0$. Its inverse Fourier transform is
equal to one precisely on $G_0^\perp$. Since
$G_0^\perp\subset V$ and this transform vanishes at infinity,
compactness gives \eqref{LCA spectral gap density}.
\end{proof}

\begin{definition}\label{strongly invariant state definition}
For a Radon probability measure $\mu$ on $G$, denote by
\[
\sigma_\mu(T)=\int_G U_tTU_t^*\,\rd\mu(t)
\]
the integrated normal map on $\B(H)$. A state $\Phi$ on $\B(H)$ is
called \emph{strongly $\Ad(U)$-invariant} when
$\Phi\circ\sigma_\mu=\Phi$ for every Radon probability measure $\mu$ on
$G$.
\end{definition}

\begin{definition}\label{approximate eigenstate definition}
A state $\Phi$ on $\B(H)$ is an \emph{approximate $U$-eigenstate} if there
are unit vectors $\xi_i\in H$ and characters $p_i\in\widehat G$ such
that
\[
\Phi=\mathop{\mathrm{w}^*\!-\lim}_i
\langle\,\cdot\,\xi_i,\xi_i\rangle
\]
and
\begin{equation}\label{general approximate eigenvectors}
\lim_i\sup_{t\in K}
\|U_t\xi_i-\langle p_i,t\rangle\xi_i\|=0
\qquad(K\Subset G).
\end{equation}
\end{definition}

\begin{lemma}\label{LCA approximate eigenstate spectral characterization}
Let $\Phi$ be represented by unit vectors $\xi_i\in H$ and characters
$p_i\in\widehat G$ as in Definition~\ref{approximate eigenstate definition}.
Then $\Phi$ is an approximate $U$-eigenstate if
and only if
\[
\lim_i\|\mathbb{E}_U(p_i+V)\xi_i-\xi_i\|=0
\]
for every neighbourhood $V$ of $0\in\widehat G$. In particular,
\begin{equation}\label{translated spectral concentration}
\lim_i\|\widehat U(L_{-p_i}f)\xi_i-f(0)\xi_i\|=0
\qquad(f\in\rC_0(\widehat G)).
\end{equation}
\end{lemma}

\begin{proof}
Spectral localization implies
\eqref{general approximate eigenvectors} by choosing an identity
neighbourhood on which the characters are uniformly close to one on
a prescribed compact subset of $G$. Conversely, fix $V$ and choose
$h$ as in Lemma \ref{LCA spectral gap density lemma}. Put
\[
A_i=\int_Gh(t)\langle -p_i,t\rangle U_t\,\rd t.
\]
Compact-uniform approximation and tightness of $h(t)\,\rd t$ give
$\|(A_i-1)\xi_i\|\to0$. The spectral multiplier of $A_i-1$ at
$q\in\widehat G$ is $\check h(q-p_i)-1$. On the complement of
$p_i+V$, its modulus is bounded below by the constant in
\eqref{LCA spectral gap density}. Hence
$\|(1-\mathbb{E}_U(p_i+V))\xi_i\|\to0$. This proves the equivalence, and the
continuous functional calculus gives
\eqref{translated spectral concentration}.
\end{proof}

\begin{proposition}\label{LCA approximate eigenstate theorem}
A state on $\B(H)$ is strongly $\Ad(U)$-invariant if and only if it
belongs to the weak$^*$ closed convex hull of the approximate
$U$-eigenstates.
\end{proposition}

\begin{proof}
This is the locally compact abelian version of the approximate
eigenstate theorem from \cite[Theorem~3.2]{Ma20}. We give the argument because
the spectral characterization in Lemma
\ref{LCA approximate eigenstate spectral characterization}
replaces the translations of a self-adjoint generator.

An approximate eigenstate is strongly invariant. Indeed,
Definition \ref{approximate eigenstate definition} shows that the corresponding
vector states are asymptotically invariant in norm, uniformly on
compact subsets of $G$. Every Radon probability measure on $G$ is
tight, so the same holds after integration against any Radon probability
measure. Strong invariance is preserved by convex combinations and
weak$^*$ limits.

Conversely, let $\Phi$ be strongly invariant. Choose a Reiter net of
probability densities $h_i\in\rL^1(G)$. Normal states are weak$^*$
dense in the state space of $\B(H)$. For each $i$, approximate $\Phi$
by normal states and compose them with $\sigma_{h_i}$. A diagonal net
of the resulting normal states $\Phi_i$ converges weak$^*$ to $\Phi$
because $\Phi\circ\sigma_{h_i}=\Phi$. The Reiter property gives
\[
\sup_{t\in K}
\|\Phi_i-\Phi_i\circ\Ad(U_t)\|\longrightarrow0
\qquad(K\Subset G).
\]

Let $d_i$ be the density of $\Phi_i$ and regard
$\zeta_i=d_i^{1/2}$ as a unit vector in
\[
\operatorname{HS}(H)=H\otimes\overline H.
\]
The Powers--St\o rmer inequality shows that $(\zeta_i)_i$ is almost
invariant, uniformly on compact sets, for
\[
W_t=U_t\otimes\overline{U_t}.
\]
Suppose that a self-adjoint $T\in\B(H)$ satisfies
$\Psi(T)\leq0$ for every approximate $U$-eigenstate $\Psi$. Fix
$\varepsilon>0$. There is a symmetric neighbourhood $V$ of
$0\in\widehat G$ such that
\begin{equation}\label{localized approximate eigenstate inequality}
\mathbb{E}_U(p+V)T\mathbb{E}_U(p+V)\leq\varepsilon \mathbb{E}_U(p+V)
\qquad(p\in\widehat G).
\end{equation}
Otherwise, by directing the identity neighbourhoods by reverse
inclusion, we could find $p_V\in\widehat G$ and a unit vector
$\xi_V\in \mathbb{E}_U(p_V+V)H$ with
$\langle T\xi_V,\xi_V\rangle>\varepsilon$. Any weak$^*$ accumulation
point of the corresponding vector states would be an approximate
$U$-eigenstate with approximate eigencharacters $p_V$, which is a
contradiction.

Put $c=\|T\|+\|T\|^2/\varepsilon$. The usual two by two block
estimate applied to \eqref{localized approximate eigenstate inequality}
gives
\[
T\leq2\varepsilon1+c(1-\mathbb{E}_U(p+V))
\qquad(p\in\widehat G).
\]
Applying the spectral theorem to the second tensor factor gives
\begin{equation}\label{tensor localized approximate eigenstate inequality}
T\otimes1\leq2\varepsilon1+cQ_V,
\end{equation}
where $Q_V=\mathbb{E}_W(\widehat G\setminus V)$.
Almost invariance of $\zeta_i$ implies
$\|Q_V\zeta_i\|\to0$. To see this directly, choose a nonnegative
probability density $h$ as in Lemma
\ref{LCA spectral gap density lemma}. Since
$\|W(h)\zeta_i-\zeta_i\|\to0$, the spectral theorem gives the claimed
convergence.

It follows from \eqref{tensor localized approximate eigenstate inequality}
that
\[
\Phi(T)=\lim_i\langle(T\otimes1)\zeta_i,\zeta_i\rangle
\leq2\varepsilon.
\]
Letting $\varepsilon$ tend to zero and applying Hahn--Banach
separation proves that $\Phi$ belongs to the weak$^*$ closed convex
hull of the approximate $U$-eigenstates.
\end{proof}

\begin{lemma}\label{LCA spectral centering lemma}
Let $S\subset\B(H)$ be a subspace. Suppose that, for
every $p\in\widehat G$, there is a unital completely positive map
$\vartheta_p:\B(H)\to\B(H)$ such that
\[
\vartheta_p|_S=\id_S,
\qquad
\vartheta_p(\widehat U(f))=\widehat U(L_pf)
\quad(f\in\rC_0(\widehat G)).
\]
If $\Phi$ is a strongly $\Ad(U)$-invariant state on $\B(H)$, then
there is a state $\widetilde\Phi$ on $\B(H)$ such that
\[
\widetilde\Phi|_S=\Phi|_S,
\qquad
\widetilde\Phi(\widehat U(f))=f(0)
\quad(f\in\rC_0(\widehat G)).
\]
\end{lemma}

\begin{proof}
Suppose first that $\Phi$ is an approximate $U$-eigenstate and choose
$\xi_i$ and $p_i$ as in Definition
\ref{approximate eigenstate definition}. Any weak$^*$ accumulation
point of the states
\[
T\longmapsto
\langle\vartheta_{-p_i}(T)\xi_i,\xi_i\rangle
\]
agrees with $\Phi$ on $S$. Lemma
\ref{LCA approximate eigenstate spectral characterization}, specifically
\eqref{translated spectral concentration}, shows that its restriction
to $\widehat U(\rC_0(\widehat G))$ is evaluation at the trivial character.

The set of states for which such a centered state exists is weak$^*$
closed and convex. Proposition \ref{LCA approximate eigenstate theorem}
therefore gives the conclusion for every strongly invariant state.
\end{proof}

\subsection{Fixed points of the dual bicentralizer action}

Throughout this subsection, let $G$ be a locally compact abelian group,
let $(M,\tau)\subset(N,\tau)$ be an inclusion of tracial von Neumann
algebras, and let
$\alpha:G\curvearrowright(M,\tau)$ be a continuous trace-preserving
action. We assume that $\Gamma(\alpha)=\widehat{G}$. By Corollary
\ref{full dual bicentralizer action}, we have a continuous
trace-preserving dual bicentralizer action
\[
\beta:\widehat G\curvearrowright
\rB(M\subset N,\alpha).
\]
Our goal is to prove that its fixed point algebra is $M'\cap N$. No
extension of $\alpha$ to $N$ will be used.

Let $\rE_M:N\to M$ be the trace-preserving conditional expectation,
let $\lambda_M,\rho_M$ and $\lambda_N,\rho_N$ denote the standard
left and right representations of $M$ and $N$, respectively,
and let
\[
V:\rL^2(M)\longrightarrow\rL^2(N),
\qquad
V\widehat x=\widehat x\quad(x\in M),
\]
be the canonical isometry. For every $x\in M$, we have
\[
V\lambda_M(x)=\lambda_N(x)V,
\qquad
V\rho_M(x)=\rho_N(x)V,
\]
whereas, for every $x\in N$,
\[
V^*\lambda_N(x)V=\lambda_M(\rE_M(x)),
\qquad
V^*\rho_N(x)V=\rho_M(\rE_M(x)).
\]
For $a,b\in N$, put
\[
\Lambda(a,b)=V^*\lambda_N(a)\rho_N(b)V
\in\B(\rL^2(M)).
\]
Thus
\[
\Lambda(x,1)=\lambda_M(\rE_M(x)),
\qquad
\Lambda(1,x)=\rho_M(\rE_M(x))
\qquad(x\in N).
\]
We denote by $U^\alpha$ the standard implementation of $\alpha$ on
$\rL^2(M)$ and by $\widehat U^\alpha$ its spectral representation on
$\rC_0(\widehat G)$.

We first record the compressed version of the binormal implementation
argument.

\begin{lemma}\label{relative binormal implementation lemma compressed}
Let $\Phi$ be a state on $\B(\rL^2(M))$ such that
\[
\Phi\circ\lambda_M=\tau,
\qquad
\Phi\circ\rho_M=\tau,
\]
and
\begin{equation}\label{relative identity spectral state compressed}
\Phi(\widehat U^\alpha(f))=f(0)
\qquad(f\in\rC_0(\widehat G)).
\end{equation}
Then
\begin{equation}\label{relative binormal implementation identity compressed}
\Phi(\Lambda(x,y))=\tau(xy)
\end{equation}
for every $x\in\rB(M\subset N,\alpha)$ and $y\in N$.
\end{lemma}

\begin{proof}
The proof is similar to the proof of Lemma
\ref{equivariant binormal bicentralizer lemma}. The tracial ultrapower implementation theorem
\cite[Theorem~2.1]{Ma25} gives an abelian tracial von Neumann algebra
$A$, a cofinal ultrafilter $\omega$, and a unit vector
\[
\xi\in\rL^2((A\ovt M)^\omega)
\]
such that
\begin{equation}\label{relative implementation formula compressed}
\Phi(T)=
\langle(1\otimes T)^\omega\xi,\xi\rangle
\qquad(T\in\B(\rL^2(M))).
\end{equation}
Exactly as in the proof of Lemma
\ref{equivariant binormal bicentralizer lemma}, we obtain

\begin{equation}\label{relative abelian amplification commutation compressed}
(1\otimes x)\xi=\xi(1\otimes x)
\qquad(x\in\rB(M\subset N,\alpha)).
\end{equation}
Now, for $x\in\rB(M\subset N,\alpha)$ and $y\in N$, formulas
\eqref{relative implementation formula compressed} and
\eqref{relative abelian amplification commutation compressed} give
\[
\begin{aligned}
\Phi(\Lambda(x,y))
&=\langle(1\otimes\lambda_N(x)\rho_N(y))^\omega\xi,\xi\rangle\\
&=\langle(1\otimes x)\xi(1\otimes y),\xi\rangle\\
&=\langle\xi(1\otimes xy),\xi\rangle\\
&=\Phi\bigl(V^*\rho_N(xy)V\bigr)
=\Phi\bigl(\rho_M(\rE_M(xy))\bigr)
=\tau(xy).
\end{aligned}
\]
\end{proof}

\begin{lemma}\label{relative LCA spectral shift maps compressed}
For every $p\in\widehat G$, there is a unital completely positive map
\[
\theta_p:\B(\rL^2(M))\longrightarrow\B(\rL^2(M))
\]
such that, for $a,b\in N$ and $f\in\rC_0(\widehat G)$,
\begin{align}
\theta_p(\Lambda(a,b))
&=\Lambda\bigl(\beta_p(
\rE_{\rB(M\subset N,\alpha)}(a)),b\bigr),
\label{relative LCA theta legs compressed}\\
\theta_p(\widehat U^\alpha(f))&=\widehat U^\alpha(L_pf).
\label{relative LCA theta corner spectrum compressed}
\end{align}
There is also a unital completely positive map
\[
\Pi:\B(\rL^2(M))\longrightarrow\B(\rL^2(M))
\]
such that
\begin{align}
\Pi(\Lambda(a,z))
&=\Lambda(\rE_{\rB(M\subset N,\alpha)^\beta}(a),z),
\label{relative LCA fixed map legs compressed}\\
\Pi(\Lambda(z,b))
&=\Lambda(z,\rE_{\rB(M\subset N,\alpha)^\beta}(b)),\notag\\
\Pi(\widehat U^\alpha(f))&=\widehat U^\alpha(f).
\label{relative LCA fixed map spectrum compressed}
\end{align}
for all $a,b\in N$ and
$z\in\rB(M\subset N,\alpha)^\beta$.
\end{lemma}

\begin{proof}
We first construct a unital completely positive map
$\zeta:\B(\rL^2(M))\to\B(\rL^2(M))$ satisfying
\begin{equation}\label{relative LCA averaging map compressed}
\zeta(\Lambda(a,b))
=\Lambda\bigl(\rE_{\rB(M\subset N,\alpha)}(a),b\bigr),
\qquad
\zeta(\widehat U^\alpha(f))=\widehat U^\alpha(f).
\end{equation}
Fix a cofinal ultrafilter $\omega$ on a sufficiently large directed
set and put
\[
P=M^{\omega,\alpha},
\qquad
Q=P'\cap N^\omega.
\]
Let $\rE_Q:N^\omega\to Q$ be the trace-preserving conditional
expectation and let $\rE^\omega:N^\omega\to N$ be the canonical
expectation. Proposition \ref{relative bicentralizer with ultrapower}
gives
$Q\subset\rB(M\subset N,\alpha)^\omega$. Since every constant element
of $\rB(M\subset N,\alpha)$ belongs to $Q$, we have
\begin{equation}\label{relative ultrapower expectation descent compressed}
\rE^\omega\circ\rE_Q|_N
=\rE_{\rB(M\subset N,\alpha)}.
\end{equation}

The relative Dixmier theorem places $\rE_Q$ in the
point-$\|\cdot\|_2$ closure of the finite convex combinations of
$\Ad(v)$ with $v\in\cU(P)$. Lifting the finitely many unitaries and
selecting coordinates gives finite convex averages
\[
F_i(a)=\sum_{j=1}^{n_i}t_{i,j}v_{i,j}av_{i,j}^*,
\qquad
v_{i,j}\in\cU(M),
\]
which converge point-weak$^*$ to
$\rE_{\rB(M\subset N,\alpha)}$ and satisfy
\begin{equation}\label{relative averaging compact fixedness compressed}
\max_{1\leq j\leq n_i}\sup_{s\in K}
\|\alpha_s(v_{i,j})-v_{i,j}\|_2\longrightarrow0
\qquad(K\Subset G).
\end{equation}
Consider
\[
\zeta_i(T)=\sum_{j=1}^{n_i}t_{i,j}
\lambda_M(v_{i,j})T\lambda_M(v_{i,j}^*)
\qquad(T\in\B(\rL^2(M))).
\]
The intertwining relations for $V$ give
\[
\zeta_i(\Lambda(a,b))=\Lambda(F_i(a),b).
\]
Moreover,
\[
\lambda_M(v_{i,j})U^\alpha_s\lambda_M(v_{i,j}^*)
=\lambda_M(v_{i,j}\alpha_s(v_{i,j}^*))U^\alpha_s.
\]
Equation \eqref{relative averaging compact fixedness compressed},
followed by integration against $\rL^1(G)$, shows that every
point-weak$^*$ accumulation point of $(\zeta_i)_i$ satisfies
\eqref{relative LCA averaging map compressed}.

Fix $p\in\widehat G$. By Corollary
\ref{full dual bicentralizer action}, the scalar cocycle $p$ is an
approximate coboundary. Choose a net $(w_i)_i$ in $\cU(M)$ such that
\begin{equation}\label{relative character implementers compressed}
\sup_{s\in K}
\|w_i\alpha_s(w_i^*)-\langle p,s\rangle1\|_2\longrightarrow0
\qquad(K\Subset G).
\end{equation}
A point-weak$^*$ accumulation point of the maps
\[
T\longmapsto\lambda_M(w_i)\zeta(T)\lambda_M(w_i^*)
\]
is a map $\theta_p$ satisfying
\eqref{relative LCA theta legs compressed}. The approximate unitary
formula in Proposition \ref{relative bicentralizer isomorphism}
identifies the first entry with $\beta_p$. Equation
\eqref{relative character implementers compressed} gives
\[
\lambda_M(w_i)U^\alpha_s\lambda_M(w_i^*)\longrightarrow \langle p,s\rangle U^\alpha_s
\]
strongly, uniformly for $s$ in compact sets. Under the chosen Fourier
convention, multiplication by $\langle p,s\rangle$ replaces the spectral variable
$q$ by $p+q$. Integration therefore gives
\eqref{relative LCA theta corner spectrum compressed}.

Let $J_M$ and $J_N$ be the canonical conjugations on $\rL^2(M)$
and $\rL^2(N)$, respectively, and define
\[
\theta'_p(T)=J_M\theta_p(J_MTJ_M)J_M.
\]
Using $J_NV=VJ_M$, one obtains
\[
\theta'_p(\Lambda(a,b))
=\Lambda\bigl(a,\beta_p(
\rE_{\rB(M\subset N,\alpha)}(b))\bigr),
\qquad
\theta'_p(\widehat U^\alpha(f))=\widehat U^\alpha(L_{-p}f).
\]
Consequently, $\Theta_p=\theta_p\circ\theta'_p$ satisfies
\begin{align*}
\Theta_p(\Lambda(a,b))
&=\Lambda\Bigl(
\beta_p(\rE_{\rB(M\subset N,\alpha)}(a)),
\beta_p(\rE_{\rB(M\subset N,\alpha)}(b))\Bigr),\\
\Theta_p(\widehat U^\alpha(f))&=\widehat U^\alpha(f).
\end{align*}

Choose a left invariant mean $m$ on $\ell^\infty(\widehat G)$ and
define $\Pi$ by
\[
\langle\Pi(T)\xi,\eta\rangle
=m\bigl(p\mapsto\langle\Theta_p(T)\xi,\eta\rangle\bigr).
\]
The invariant mean applied to the matrix coefficients of
$\beta$ gives the trace-preserving expectation onto its fixed
point algebra. If one entry of $\Lambda(a,b)$ belongs to
$\rB(M\subset N,\alpha)^\beta$, the displayed formula for
$\Theta_p$ is linear in the only nonconstant orbit. Averaging
therefore gives \eqref{relative LCA fixed map legs compressed}. The
formula for $\widehat U^\alpha(f)$ gives
\eqref{relative LCA fixed map spectrum compressed}.
\end{proof}

\begin{lemma}\label{strong relative Dixmier state compressed}
There is a strongly $\Ad(U^\alpha)$-invariant state $\Phi_0$ on
$\B(\rL^2(M))$ such that
\begin{equation}\label{relative Dixmier moments compressed}
\Phi_0(\Lambda(a,b))
=\tau(\rE_{M'\cap N}(a)b)
\qquad(a,b\in N),
\end{equation}
where $\rE_{M'\cap N}:N\to M'\cap N$ is the trace-preserving
conditional expectation.
\end{lemma}

\begin{proof}
Take $K \Subset G$ and $\varepsilon >0$. Since $G$ is amenable, we can choose a probability measure $\mu$ on $G$ that is absolutely continuous with respect to Haar measure and such that 
$$ \sup_{r \in K} \| \mu(\,\cdot-r)- \mu \|_1 \leq \varepsilon.$$
Consider the embedding $$\pi : M \rightarrow \rL^\infty(G,\mu) \ovt N$$
that sends $x \in M$ to $s \mapsto \alpha_{s}(x)$. Observe that $$\pi(M)' \cap (\rL^\infty(G,\mu) \ovt N)=\rL^\infty(G,\mu) \ovt (M' \cap N).$$
Moreover, the $\mu \otimes \tau$-preserving conditional expectation on this relative commutant is $\id \otimes \rE_{M' \cap N}$. By applying the usual Hilbert space convexity argument to this inclusion, for every finite set $F \subset N$,  we obtain unitaries
$u_1,\ldots,u_n\in\cU(M)$ and numbers $t_1,\ldots,t_n\geq0$ with
$\sum_it_i=1$ such that
$$ \sum_{a \in F} \left \| 1\otimes \rE_{M' \cap N}(a) - \sum_{i=1}^n t_i\pi(u_i)(1 \otimes a) \pi(u_i)^* \right \|_{\mu \otimes \tau}^2 < \varepsilon $$
or equivalently
\begin{equation}\label{integrated relative Dixmier estimate}
\sum_{a\in F}\int_G
\left\|
\rE_{M' \cap N}(a) - \sum_{i=1}^nt_i\alpha_s(u_i)a\alpha_s(u_i)^* \right\|_2^2 \: \rd\mu(s) 
<\varepsilon.
\end{equation}

For these data, define a normal state on $\B(\rL^2(M))$ by
\begin{equation}\label{Reiter relative state}
\Phi_{K,F,\varepsilon}(T)
=\int_G\sum_{i=1}^nt_i
\left\langle
T\widehat{\alpha_s(u_i)^*},
\widehat{\alpha_s(u_i)^*}
\right\rangle \, \rd \mu(s).
\end{equation}
Then we have
\begin{equation}\label{moments estimate}
\Phi_{K,F,\varepsilon}
  (V^*\lambda_N(a)\rho_N(b)V)
=\int_G\sum_{i=1}^nt_i
\tau\bigl(\alpha_s(u_i)a\alpha_s(u_i)^*b\bigr)
\, \rd \mu(s).
\end{equation}
It follows from \eqref{integrated relative Dixmier estimate} and \eqref{moments estimate} that
every weak$^*$ accumulation point $\Phi_0$ of the net $\Phi_{K,F,\varepsilon}$ over the directed set of all triples $(K,F,\varepsilon)$ satisfies
\eqref{relative Dixmier moments compressed}.

Moreover, for $r\in K$, composing the state in
\eqref{Reiter relative state} with $\Ad(U^\alpha_r)$ only translates the measure $\mu$. Consequently,
\begin{equation}\label{reiter estimate}
\|\Phi_{K,F,\varepsilon}\circ\Ad(U^{\alpha}_r)
-\Phi_{K,F,\varepsilon}\|
\leq\|\mu(\,\cdot-r) -\mu\|_1 \leq \varepsilon.
\end{equation}

The estimate is uniform for $r$ in compact subsets. Tightness of
Radon probability measures then gives the same convergence after
integration against any Radon probability measure on $G$. The accumulation point $\Phi_0$ is therefore strongly
$\Ad(U^\alpha)$-invariant and has all the required properties.
\end{proof}

\begin{theorem}\label{fixed points relative dual bicentralizer compressed}
The fixed point algebra of the dual relative bicentralizer action is
the relative commutant:
\[
\rB(M\subset N,\alpha)^\beta=M'\cap N.
\]
\end{theorem}

\begin{proof}
Every element of $M'\cap N$ belongs to
$\rB(M\subset N,\alpha)$. If $p\in\widehat G$ and $(w_i)_i$ satisfies
\eqref{relative character implementers compressed}, the approximate
unitary formula gives
\[
\beta_p(d)=\lim_iw_idw_i^*=d
\qquad(d\in M'\cap N).
\]
Thus
\[
M'\cap N\subset\rB(M\subset N,\alpha)^\beta.
\]

Let $\Phi_0$ be the state from Lemma
\ref{strong relative Dixmier state compressed}, and let $\mathcal S$
be the linear span
\[
\left\{\Lambda(z,a)\ \middle|\
z\in\rB(M\subset N,\alpha)^\beta,\ a\in N\right\}.
\]
Equations \eqref{relative LCA theta legs compressed} and
\eqref{relative LCA theta corner spectrum compressed} show that every
$\theta_p$ fixes $\mathcal S$ pointwise and sends
$\widehat U^\alpha(f)$ to $\widehat U^\alpha(L_pf)$. Lemma
\ref{LCA spectral centering lemma} gives a state
$\Phi_1$ on $\B(\rL^2(M))$ which agrees with $\Phi_0$ on
$\mathcal S$ and satisfies
\[
\Phi_1(\widehat U^\alpha(f))=f(0)
\qquad(f\in\rC_0(\widehat G)).
\]
In particular,
\begin{equation}\label{centered relative moments compressed}
\Phi_1(\Lambda(z,a))
=\tau(\rE_{M'\cap N}(z)a)
\end{equation}
for $z\in\rB(M\subset N,\alpha)^\beta$ and $a\in N$.

Let $\Pi$ be the map from Lemma
\ref{relative LCA spectral shift maps compressed} and put
$\Phi_2=\Phi_1\circ\Pi$. For $a\in M$, equations
\eqref{relative LCA fixed map legs compressed} and
\eqref{centered relative moments compressed} give
\[
\begin{aligned}
\Phi_2(\lambda_M(a))
&=\Phi_1(\Lambda(\rE_{\rB(M\subset N,\alpha)^\beta}(a),1))\\
&=\tau(\rE_{M'\cap N}(\rE_{\rB(M\subset N,\alpha)^\beta}(a)))\\
&=\tau(a),\\
\Phi_2(\rho_M(a))
&=\Phi_1(\Lambda(1,\rE_{\rB(M\subset N,\alpha)^\beta}(a)))
=\tau(\rE_{\rB(M\subset N,\alpha)^\beta}(a))=\tau(a).
\end{aligned}
\]
Equation \eqref{relative LCA fixed map spectrum compressed} also gives
\[
\Phi_2(\widehat U^\alpha(f))=f(0)
\qquad(f\in\rC_0(\widehat G)).
\]
Lemma \ref{relative binormal implementation lemma compressed} now
yields
\begin{equation}\label{centered relative bicentralizer identity compressed}
\Phi_2(\Lambda(x,a))=\tau(xa)
\qquad(x\in\rB(M\subset N,\alpha),\ a\in N).
\end{equation}

Fix $z\in\rB(M\subset N,\alpha)^\beta$. Since
$\rE_{\rB(M\subset N,\alpha)^\beta}(z)=z$, formula
\eqref{relative LCA fixed map legs compressed} gives, for every
$a\in N$,
\[
\begin{aligned}
\tau(za)
&=\Phi_2(\Lambda(z,a))\\
&=\Phi_1(\Lambda(z,\rE_{\rB(M\subset N,\alpha)^\beta}(a)))\\
&=\tau(\rE_{M'\cap N}(z)\rE_{\rB(M\subset N,\alpha)^\beta}(a))
=\tau(\rE_{M'\cap N}(z)a).
\end{aligned}
\]
The last equality uses
$\rE_{M'\cap N}(z)\in M'\cap N
\subset\rB(M\subset N,\alpha)^\beta$. Taking
$a=(z-\rE_{M'\cap N}(z))^*$ gives
$z=\rE_{M'\cap N}(z)\in M'\cap N$. This proves the reverse inclusion.
\end{proof}

\begin{theorem}\label{LCA bicentralizer eigenvector intertwining compressed}
Take $t\in G$ and $x \in N$. The following are equivalent :
\begin{enumerate}[\rm (i)]
  \item $x\in\rB(M\subset N,\alpha)$ and $\beta_p(x)=\langle p,t\rangle x$ for all $p\in\widehat G$.
  \item $xy=\alpha_t(y)x$ for all $y\in M$.
\end{enumerate}
In consequence, the almost periodic part of 
$$\beta : \widehat{G} \curvearrowright \rB(M \subset N,\alpha)$$ is contained in $\rb(M \subset N,\alpha)$.
\end{theorem}

\begin{proof}
Suppose first that $(\rm ii)$ holds. For every asymptotically
$\alpha$-invariant net $(v_i)_i$ in $\cU(M)$, one has
\[
\|v_ix-xv_i\|_2
=\|(v_i-\alpha_t(v_i))x\|_2\longrightarrow0.
\]
Thus $x\in\rB(M\subset N,\alpha)$. If $(w_i)_i$ approximately
implements a character $p\in\widehat G$, then
\[
w_ixw_i^*=w_i\alpha_t(w_i^*)x\longrightarrow \langle p,t\rangle x,
\]
which gives $\beta_p(x)=\langle p,t\rangle x$ and proves $(\rm i)$.

Conversely, assume $(\rm i)$.
Define a trace-preserving embedding
\[
\theta:M\longrightarrow M_2(N),
\qquad
\theta(a)=
\begin{pmatrix}
a&0\\
0&\alpha_t(a)
\end{pmatrix}.
\]
Since $G$ is abelian, the formula
\[
\gamma_s(\theta(a))=\theta(\alpha_s(a))
\qquad(s\in G,\ a\in M)
\]
defines a continuous trace-preserving action on $\theta(M)$. The systems
$(M,\alpha)$ and $(\theta(M),\gamma)$ are equivariantly isomorphic, so
$\Gamma(\gamma)=\Gamma(\alpha)=\widehat G$. Put
\[
C=\rB(\theta(M)\subset M_2(N),\gamma)
\]
and denote its dual relative bicentralizer action by $\widetilde\beta$.

Consider the lower-left corner matrix
\[
X=
\begin{pmatrix}
0&0\\
x&0
\end{pmatrix}
\in M_2(N).
\]
We claim that $X\in C$. Indeed, for $u\in\cU(M)$,
\[
\left[\theta(u),X\right]
=
\begin{pmatrix}
0&0\\
\alpha_t(u)x-xu&0
\end{pmatrix}.
\]
If $\theta(u)$ is sufficiently $\gamma$-invariant on a compact set
containing $t$, then $u$ is sufficiently $\alpha$-invariant on the same
compact set. Hence $\alpha_t(u)-u$ is small in $\|\cdot\|_2$, while the
definition of $\rB(M\subset N,\alpha)$ makes $ux-xu$ small. Since
\[
\alpha_t(u)x-xu=(\alpha_t(u)-u)x+(ux-xu),
\]
the claim follows.

Fix $p\in\widehat G$ and choose a net $(w_i)_i$ in $\cU(M)$ such that
\[
w_i\alpha_s(w_i^*)\longrightarrow \langle p,s\rangle1
\]
in $\|\cdot\|_2$, uniformly for $s$ in compact subsets of $G$. The
unitaries $\theta(w_i)$ implement $p$ approximately for $\gamma$. The
approximate unitary formula and the eigenvector assumption give
\[
\begin{aligned}
\widetilde\beta_p(X)
&=\lim_i
\begin{pmatrix}
0&0\\
\alpha_t(w_i)xw_i^*&0
\end{pmatrix}\\
&=\begin{pmatrix}
0&0\\
\langle -p,t\rangle\,\langle p,t\rangle x&0
\end{pmatrix}
=X.
\end{aligned}
\]
Here we used
$\alpha_t(w_i)w_i^*\to\langle -p,t\rangle1$ and
$w_ixw_i^*\to\beta_p(x)=\langle p,t\rangle x$ in $\|\cdot\|_2$.
Thus $X\in C^{\widetilde\beta}$. Theorem
\ref{fixed points relative dual bicentralizer compressed} gives
\[
X\in\theta(M)'\cap M_2(N).
\]
Comparing the lower-left corners in $X\theta(y)=\theta(y)X$ yields
\[
xy=\alpha_t(y)x
\qquad(y\in M),
\]
as required.
\end{proof}

\begin{corollary}\label{LCA outer iff dual bicentralizer weakly mixing compressed}
Assume that $N=M$ and that $M$ is a factor. Then $\alpha$ is outer if
and only if the dual bicentralizer action
\[
\beta:\widehat G\curvearrowright\rB(M,\alpha)
\]
is weakly mixing.
\end{corollary}

\begin{proof}
Theorem~\ref{fixed points relative dual bicentralizer compressed} gives
$\rB(M,\alpha)^\beta=\C 1$. For an ergodic trace-preserving action,
the absolute values of an $\rL^2$ eigenvector are scalar, so every
nonzero eigenvector is a scalar multiple of a unitary in the algebra.
Since $\widehat G$ is abelian, weak mixing is therefore equivalent
to the absence of nonscalar eigenoperators. By
Theorem~\ref{LCA bicentralizer eigenvector intertwining compressed},
an eigenoperator at $t\ne0$ exists exactly when $\alpha_t$ is inner.
On a factor, this is precisely the obstruction to proper outerness.
\end{proof}

\begin{corollary}\label{LCA outer diagonal action trivial bicentralizer compressed}
Suppose that $M_1$ and $M_2$ are tracial factors and that
\[
(M,\tau)=(M_1 \ovt M_2,\tau_1 \otimes \tau_2), \quad \alpha=\alpha_1\otimes_G \alpha_2
\]
for continuous trace-preserving actions $\alpha_i:G\curvearrowright
(M_i,\tau_i)$ satisfying
$\Gamma(\alpha)=\Gamma(\alpha_1)=\Gamma(\alpha_2)=\widehat G$.
Then $\alpha$ is outer if and only if
\[
\rB(M,\alpha)=\C 1.
\]
\end{corollary}

\begin{proof}
Let $\beta : \widehat{G} \curvearrowright \rB(M,\alpha)$ be the dual bicentralizer action of $\alpha$. For $i\in\{1,2\}$, denote by
$\beta^{(i)}:\widehat G\curvearrowright\rB(M_i,\alpha_i)$
the dual bicentralizer action of $\alpha_i$.
Proposition \ref{tensor product relative action bicentralizer} gives
\[
\rB(M,\alpha)
\subset
\rB(M_1,\alpha_1)\ovt\rB(M_2,\alpha_2)
\]
and
$$ \beta_p(x)=(\beta^{(1)}_p \otimes \id)(x)=(\id \otimes \beta^{(2)}_p)(x), \quad (p \in\widehat G).$$
In particular,
\[
(\beta^{(1)}_p\otimes\beta^{(2)}_{-p})(x)=x
\qquad
(x\in\rB(M,\alpha)).
\]
This shows that $\rB(M,\alpha)$ is contained in the tensor product of the almost periodic parts of $\beta^{(1)}$ and $\beta^{(2)}$. In particular, $\beta$ is almost periodic. By Corollary~\ref{LCA outer iff dual bicentralizer weakly mixing compressed}, if $\alpha$ is properly outer, $\beta$ must be weakly mixing. We conclude that $\rB(M,\alpha)$ must be trivial.

Conversely, if $\rB(M,\alpha)=\C 1$, then $\beta$ is weakly mixing.
The same corollary therefore implies that $\alpha$ is properly outer.
\end{proof}

\subsection{A fixed point interpretation of the bicentralizer conjecture}
\label{subsec:fixed-point-interpretation}
Let $\alpha : G \curvearrowright (M,\tau)$ be a trace-preserving action of a locally compact abelian group with $\Gamma(\alpha)=\widehat{G}$.

We denote by $$u^\alpha : G \rightarrow \cU(M \rtimes_\alpha G)$$
the natural unitary implementation of $\alpha$ in the crossed product.

We denote by $\widehat{\alpha} : \widehat{G} \curvearrowright M \rtimes_\alpha G$ the dual action of $\alpha$. We denote by $\widetilde{\alpha}=\Ad(u^\alpha)$ the natural extension of $\alpha$ to an inner action on $M \rtimes_\alpha G$.

Fix a large enough cofinal ultrafilter $\omega$. We view
$M\rtimes_\alpha G$ in the common ambient algebra
$M^\omega_\alpha \rtimes_{\alpha^\omega}G$. Since $M^{\omega,\alpha}$ is fixed pointwise by
$\alpha^\omega$, we have
\[
 (M^{\omega,\alpha})'\cap(M^\omega_\alpha \rtimes_{\alpha^\omega}G)
 =((M^{\omega,\alpha})'\cap M_\alpha^\omega)\rtimes_{\alpha^\omega}G.
\]
Taking the intersection with $M \rtimes_\alpha G$, we obtain
\[
 (M^{\omega,\alpha})'\cap(M \rtimes_\alpha G)
 =((M^{\omega,\alpha})'\cap M)\rtimes_\alpha G = \rB(M,\alpha) \rtimes_\alpha G.
\]
Therefore, we can think of $\rB(M,\alpha) \rtimes_\alpha G$ as a relative bicentralizer in the semifinite crossed
product:
\[
 \rB(M\subset M\rtimes_\alpha G,\alpha)
 =\rB(M,\alpha)\rtimes_\alpha G.
\]
It is also possible to define a dual bicentralizer action $$\widehat{G} \curvearrowright \rB(M \subset M \rtimes_\alpha G, \alpha)$$
but it is not given by the natural extension $$\widetilde{\beta} : \widehat{G} \curvearrowright \rB(M,\alpha) \rtimes_\alpha G$$ that fixes $u^\alpha$. Indeed, if we take $p \in \widehat{G}$ and $v \in\cU(M^\omega_\alpha)$ such that $$\alpha^\omega_t(v)=\langle-p,t\rangle v, \qquad (t \in G),$$
then we have $\Ad(v)=\widetilde{\beta}_p \circ \widehat{\alpha}_p$. Thus, the dual bicentralizer action $$\delta : \widehat{G} \curvearrowright \rB(M \subset M \rtimes_\alpha G, \alpha)$$ is given by
$$ \delta_p = \widetilde{\beta}_p \circ \widehat{\alpha}_p, \qquad (p \in \widehat{G}).$$

The following proposition highlights the importance of this construction. It is the analog of \cite[Proposition 8.6]{Ma25} (in the case $M=N$).
\begin{proposition}\label{bicentralizer conjecture fixed point criterion}
    We have $\rB(M,\alpha)=\rb(M,\alpha)$ if and only if $$\rB(M \subset M \rtimes_\alpha G, \alpha)^\delta=M' \cap (M \rtimes_\alpha G).$$
\end{proposition}
\begin{proof}
The preceding construction gives
\[
 \delta_p(u_t)=\langle p,t\rangle u_t,
 \qquad \delta_p(z)=z
 \qquad(p\in\widehat G,\ t\in G,\ z\in M' \cap (M \rtimes_\alpha G)).
\]
The second equality also follows directly from the ultrapower
implementation, since every representing unitary lies in $M$ and hence
commutes with $M'\cap(M\rtimes_\alpha G)$. Takesaki's characterization
\cite[Theorem~1.5]{Ma25} gives
\[
 \rB(M \subset M \rtimes_\alpha G,\alpha)=\bigl(\rB(M \subset M \rtimes_\alpha G,\alpha)^\delta\cup\{u_t\mid t\in G\}\bigr)''.
\]
On the other hand, Proposition~\ref{crossed product of bicentralizer} gives
\[
 \rb(M \subset M \rtimes_\alpha G,\alpha)
 =\bigl((M' \cap (M \rtimes_\alpha G))\cup\{u_t\mid t\in G\}\bigr)''.
\]
Thus $(\rb(M \subset M \rtimes_\alpha G,\alpha),M' \cap (M \rtimes_\alpha G),u)$ is a split $G$-extension and $\delta|_{\rb(M \subset M \rtimes_\alpha G,\alpha)}$ fixes $M' \cap (M \rtimes_\alpha G)$
while multiplying $u_t$ by $\langle p,t\rangle$. The same characterization therefore
gives $\bigl(\rb(M,\alpha)\rtimes_\alpha G\bigr)^\delta=M' \cap (M \rtimes_\alpha G)$.
If $\rB(M,\alpha)=\rb(M,\alpha)$, then
$\rB(M \subset M \rtimes_\alpha G,\alpha)^\delta=M' \cap (M \rtimes_\alpha G)$. Conversely, $\rB(M \subset M \rtimes_\alpha G,\alpha)^\delta=M' \cap (M \rtimes_\alpha G)$ makes the displayed generating
formulas for $\rB(M \subset M \rtimes_\alpha G,\alpha)$ and $\rb(M \subset M \rtimes_\alpha G,\alpha)$ identical. Intersecting with $M$,
or taking fixed points of the original dual action $\widehat\alpha$,
gives $\rB(M,\alpha)=\rb(M,\alpha)$.
\end{proof}
Thus the bicentralizer conjecture asks for an analogue of
Theorem~\ref{fixed points relative dual bicentralizer compressed} for the
inclusion $M\subset M\rtimes_\alpha G$. That theorem applies to tracial
inclusions, whereas the crossed product need not carry a faithful normal
tracial state. In Section~\ref{sec:joint-bicentralizer-action}, a singular
state on the crossed product will allow us to apply the theorem in a
finite GNS algebra and obtain the commutation relations needed for the
resonance argument.

\subsection{Intermediate subfactor property}

Let $G$ be a locally compact abelian group and let
$\alpha:G\curvearrowright(M,\tau)$ be a continuous trace-preserving
action on a $\II_1$ factor.  We denote by $(u_s)_{s\in G}$ the canonical
implementing representation in
\[
 P=M\rtimes_\alpha G
\]
and put
\[
 \rL_\alpha(G)=\{u_s\mid s\in G\}''\subset P.
\]
The canonical faithful normal conditional expectation
$\rE_{\rL_\alpha(G)}:P\to\rL_\alpha(G)$ is characterized, on compactly
supported Fourier integrals, by
\begin{equation}\label{canonical expectation onto group algebra}
 \rE_{\rL_\alpha(G)}\left(\int_G x(s)u_s\,\rd s\right)
 =\int_G\tau(x(s))u_s\,\rd s.
\end{equation}

\begin{theorem}\label{intermediate subfactor property}
Assume that $\alpha$ is strictly outer and that
$\rB(M,\alpha)=\C 1$.  If $N$ is an intermediate factor
\[
 M\subset N\subset M\rtimes_\alpha G,
\]
then there is a closed subgroup $H<G$ such that
\[
 N=M\rtimes_{\alpha|_H}H.
\]
\end{theorem}

\begin{proof}
Write $E=\rE_{\rL_\alpha(G)}$ and let $\Tr_P$ be the canonical
semifinite trace on $P$. We first prove that $E(N)\subset N$.
The proof of Theorem
\ref{correspondance trivial bicentralizer} provides finite convex averages
\[
 F_i=\sum_{j=1}^{n_i}t_{i,j}\Ad(v_{i,j}),
 \qquad v_{i,j}\in\cU(M),
\]
such that
\[
 \max_j\sup_{s\in K}\|\alpha_s(v_{i,j})-v_{i,j}\|_2\longrightarrow0
 \qquad(K\Subset G)
\]
and $F_i(a)\to\tau(a)1$ in $\|\cdot\|_2$ for every $a\in M$.
Both $F_i$ and $E$ preserve $\Tr_P$ and induce contractions on
$\rL^2(P,\Tr_P)$. For $a\in M$, uniformly for $s$ in compact
subsets of $G$, we have
\[
 \|F_i(au_s)u_s^*-\tau(a)1\|_2
 =\left\|\sum_jt_{i,j}v_{i,j}a\alpha_s(v_{i,j}^*)-\tau(a)1\right\|_2
 \longrightarrow0.
\]
The Plancherel formula therefore gives $F_i(y)\to E(y)$ in
$\rL^2(P,\Tr_P)$ for compactly supported Fourier integrals.
Such integrals are dense in this Hilbert space, so the same convergence
holds for every $y\in\rL^2(P,\Tr_P)$.

Let $e\in\rL_\alpha(G)$ be a projection with $\Tr_P(e)<\infty$,
and let $f_e\in\rL^2(G)$ be its Fourier transform. Thus
$e=\int_G f_e(s)u_s\,\rd s$ in the $\rL^2$ sense.
For $v\in\cU(M)$, Plancherel gives
\[
\|[v,e]\|_{2,\Tr_P}^2
=\int_G|f_e(s)|^2\|v-\alpha_s(v)\|_2^2\,\rd s.
\]
This follows first for inverse Fourier transforms of compactly supported
functions and then by $\rL^2$ approximation. Compact-uniform almost
invariance and the bound $\|v-\alpha_s(v)\|_2\leq2$ imply
\[
\max_j\|[v_{i,j},e]\|_{2,\Tr_P}\longrightarrow0.
\]
For every $x\in P$, it follows that
\[
\|eF_i(x)e-F_i(exe)\|_{2,\Tr_P}
\leq2\|x\|\max_j\|[v_{i,j},e]\|_{2,\Tr_P}
\longrightarrow0.
\]
Since $exe\in\rL^2(P,\Tr_P)$ and $E(exe)=eE(x)e$, we obtain
$eF_i(x)e\to eE(x)e$ in $\rL^2(P,\Tr_P)$.
Finite-trace projections in $\rL_\alpha(G)$ increase strongly to $1$.
The uniform bound $\|F_i(x)\|\leq\|x\|$ now gives
$F_i(x)\to E(x)$ weak$^*$ for every $x\in P$.
In particular,
\begin{equation}\label{group expectation almost invariant convex hull}
 \begin{aligned}
 E(x)&\in
 \bigcap_{\substack{K\Subset G\\ \delta>0}}
 \overline{\conv}^{\,\mathrm{w}^*}
 \Bigl\{v x v^* \Bigm| v\in\cU(M),\;
 \sup_{s\in K}\|\alpha_s(v)-v\|_2<\delta\Bigr\}.
 \end{aligned}
\end{equation}

Since $M\subset N$, every conjugate $v x v^*$ occurring in
\eqref{group expectation almost invariant convex hull} belongs to $N$ when
$x\in N$.  Hence
\[
 \rE_{\rL_\alpha(G)}(N)\subset N.
\]
Consequently
\[
 \rE_A=\rE_{\rL_\alpha(G)}|_N:N\longrightarrow A,
 \qquad A=N\cap\rL_\alpha(G),
\]
is a faithful normal conditional expectation.

We next prove that $A$ is invariant under the dual action.  Strict
outerness implies that $P$ is a factor, and hence that
$\Gamma(\alpha)=\widehat G$.  The approximate cocycle vanishing theorem,
Theorem \ref{approximate cocycle vanishing theorem}, applied to the scalar
cocycle $p\in\widehat G$, yields a net $(w_i)_i$ in $\cU(M)$ such that
\begin{equation}\label{intermediate approximate eigenunitaries}
 \sup_{s\in K}\|w_i\alpha_s(w_i^*)-\langle p,s\rangle1\|_2\longrightarrow0
 \qquad(K\Subset G).
\end{equation}
Put $\delta=\widehat\alpha_p$. Our dual-action convention
gives $\delta(u_s)=\langle p,s\rangle u_s$. Plancherel and
\eqref{intermediate approximate eigenunitaries} imply
\[
\Ad(w_i)(y)\longrightarrow\delta(y)
\quad\text{in }\rL^2(P,\Tr_P)
\qquad(y\in\rL^2(\rL_\alpha(G),\Tr_P)).
\]
Indeed, the squared error for a Fourier integral with coefficient
$f\in\rL^2(G)$ is
\[
\int_G|f(s)|^2
\|w_i\alpha_s(w_i^*)-\langle p,s\rangle1\|_2^2\,\rd s,
\]
which tends to zero by compact-uniform convergence and an
$\rL^2$ tail estimate. For a finite-trace projection
$e\in\rL_\alpha(G)$ and $x\in\rL_\alpha(G)$, we have
\[
\begin{aligned}
&\|\delta(e)\Ad(w_i)(x)\delta(e)-\Ad(w_i)(exe)\|_{2,\Tr_P}\\
&\qquad\leq2\|x\|\|\Ad(w_i)(e)-\delta(e)\|_{2,\Tr_P}
\longrightarrow0.
\end{aligned}
\]
Applying the preceding $\rL^2$ convergence to $exe$ and then letting
$e$ increase to $1$ proves point-weak$^*$ convergence
$\Ad(w_i)\to\delta$ on the whole group algebra.
If $x\in A$, every $w_i xw_i^*$ belongs to $N$, so
$\widehat\alpha_p(x)\in N$. Since the dual action preserves
$\rL_\alpha(G)$ and $p$ was arbitrary, we conclude that
\[
 \widehat\alpha_p(A)=A
 \qquad(p\in\widehat G).
\]
Under the Fourier identification
$\rL_\alpha(G)=\rL^\infty(\widehat G)$, the dual action is the translation
action of $\widehat G$.  The translation-invariant von Neumann
subalgebras of $\rL^\infty(\widehat G)$ are precisely the algebras
associated with quotient groups.  Therefore there is a closed subgroup
$H<G$ such that
\begin{equation}\label{intermediate group algebra}
 A=\rL_\alpha(H).
\end{equation}

We finish with the standard Fourier-support characterization
\begin{equation}\label{crossed product Fourier support characterization}
 M\rtimes_{\alpha|_H}H
 =\left\{x\in P\ \middle|\
 \rE_{\rL_\alpha(G)}(a x b)\in\rL_\alpha(H)
 \text{ for all }a,b\in M\right\}.
\end{equation}
The inclusion from left to right follows first on the elements $a u_t$
with $t\in H$, and then by normality of $E$.
For the converse, let $H^\perp<\widehat G$ be the annihilator of $H$.
The condition
on the right implies, for every $p\in H^\perp$ and $a,b\in M$, that
\[
 \rE_{\rL_\alpha(G)}(a\widehat\alpha_p(x)b)
 =\widehat\alpha_p(\rE_{\rL_\alpha(G)}(a x b))
 =\rE_{\rL_\alpha(G)}(a x b).
\]
The maps $y\mapsto E(ayb)$, with $a,b\in M$, separate the points
of $P$. Indeed, if they all vanish at $y$, then
\[
\Tr_P\bigl(\lambda(f)^*ayb\lambda(h)\bigr)
=\Tr_P\bigl(\lambda(f)^*E(ayb)\lambda(h)\bigr)=0
\]
for $f,h\in\rC_c(G)$, where $\lambda(f)=\int_Gf(s)u_s\,\rd s$.
The vectors $a\lambda(f)$ span a dense subspace of
$\rL^2(P,\Tr_P)$, so these identities force $y=0$.
Hence $\widehat\alpha_p(x)=x$ for every
$p\in H^\perp$, and duality gives
\[
 x\in P^{\widehat\alpha|_{H^\perp}}
 =M\rtimes_{\alpha|_H}H.
\]
This proves \eqref{crossed product Fourier support characterization}.

Clearly $M\rtimes_{\alpha|_H}H\subset N$ by
\eqref{intermediate group algebra}.  Conversely, if $x\in N$ and
$a,b\in M$, then $a x b\in N$ and therefore
\[
 \rE_{\rL_\alpha(G)}(a x b)\in A=\rL_\alpha(H).
\]
Formula \eqref{crossed product Fourier support characterization} gives
$x\in M\rtimes_{\alpha|_H}H$, and the proof is complete.
\end{proof}

\section{The relative commutant theorem and the resonance property}
\label{sec:relative-commutants}

We begin with the characteristic invariant of actions of locally compact abelian groups
on factors, defined by the eigenvalues of unitaries implementing inner
automorphisms. Reflecting its second coordinate gives the joint point spectrum of the
extended action and the dual action on the relative commutant. The
invariant satisfies a resonance equation. To study the full Arveson spectrum, we construct Fourier
coefficient distributions and prove the same resonance equation for
actions preserving a faithful normal semifinite weight.

For flows preserving such a weight, the geometry of the resonance set gives the relative
commutant dichotomy and shows that outerness together with full Connes
spectrum implies strict outerness. The same framework recovers the
Connes--Takesaki relative commutant theorem.

For a faithful normal semifinite weight $\phi$ on a von Neumann algebra
$N$, we use the standard bimodule notation
$x\phi^{1/2}\in\rL^2(N)$ for $x\in\mathfrak n_\phi$, and denote its
modular conjugate by $\phi^{1/2}x^*$. Left and right multiplication refer
to the canonical $N$--$N$ bimodule structure on $\rL^2(N)$.

\subsection{The characteristic invariant \texorpdfstring{$\Lambda(\alpha)$}{Lambda(alpha)}}
\label{subsec:action-invariant}
\label{subsec:crossed-product-actions}

Let $G$ be a locally compact abelian group, written additively, and
write its Pontryagin dual $\widehat G$ additively. Let
$\alpha:G\curvearrowright M$ be a continuous action on a factor.
We define its \emph{characteristic invariant} by
\begin{equation}\label{non outer implementer lattice}
\begin{split}
\Lambda(\alpha)=\{(p,t)\in\widehat G\times G\mid{}&
\text{there is }v\in\cU(M)\text{ such that }\alpha_t=\Ad(v)\\
&\text{and }\alpha_s(v)=\langle p,s\rangle v\text{ for every }s\in G\}.
\end{split}
\end{equation}

\begin{proposition}\label{characteristic invariant properties}
The set $\Lambda(\alpha)$ is a subgroup of $\widehat G\times G$
and is a cocycle conjugacy invariant of $\alpha$. Its projection onto $G$ is
injective and has image $\{t\in G\mid\alpha_t\in\Inn(M)\}$.
In particular, $\Lambda(\alpha)=\{(0,0)\}$ if and only if $\alpha$
is outer.

Moreover $\Lambda(\alpha)$ is contained in the resonance set
\begin{equation}\label{eq:implementer-resonance}
\mathcal{R}_G=\{ (p,t) \in \widehat{G} \times G \mid \langle p,t\rangle=1\}.
\end{equation}
\end{proposition}
\begin{proof}
If $\alpha_t=\Ad(v)$, then commutation of $\alpha_s$ with
$\alpha_t$ implies that $\alpha_s(v)v^*$ is scalar for every
$s\in G$. These scalars form a continuous character $p\in\widehat G$.
Since any two implementing unitaries differ by a scalar, $p$ is
uniquely determined by $t$. This proves the assertions about the
projection onto $G$. Products and adjoints of implementing
eigenunitaries show that $\Lambda(\alpha)$ is a subgroup.

Let $z\in Z^1(\alpha)$ and write
$\alpha^z_s=\Ad(z_s)\circ\alpha_s$. If $v$ implements $(p,t)$
for $\alpha$, then $z_tv$ implements $\alpha^z_t$. The cocycle
identity and commutativity of $G$ give
\[
\begin{aligned}
\alpha^z_s(z_tv)
&=\langle p,s\rangle z_s\alpha_s(z_t)vz_s^*\\
&=\langle p,s\rangle z_{s+t}vz_s^*\\
&=\langle p,s\rangle z_t\alpha_t(z_s)vz_s^*
=\langle p,s\rangle z_tv.
\end{aligned}
\]
Thus $\Lambda(\alpha)\subset\Lambda(\alpha^z)$. Applying the same
argument to the inverse cocycle gives equality. An isomorphism
conjugating two actions carries their implementing eigenunitaries
to one another, proving cocycle conjugacy invariance.

Finally, if $(p,t) \in \Lambda(\alpha)$ and $\alpha_t=\Ad(v)$ for some $v \in \cU(M)$ then
$v=\alpha_t(v)=\langle p,t\rangle v$, proving \eqref{eq:implementer-resonance}.
\end{proof}

\begin{remark}\label{classical characteristic invariant}
Let $\alpha_n=\theta^n$ be a $\Z$-action with finite
outer period $n_0>0$. If $\theta^{n_0}=\Ad(v)$ and
$\theta(v)=\gamma v$, then, identifying $\widehat\Z$ with
$\mathbf T$, we have
$\Lambda(\alpha)=\{(\gamma^k,kn_0)\mid k\in\Z\}$.
This recovers Connes's outer invariant $(n_0,\gamma)$, and the
resonance relation is $\gamma^{n_0}=1$
\cite[Proposition~1.4]{Co77}.

When $\alpha$ is an action of a discrete abelian group $G$, our
invariant $\Lambda(\alpha)$ is directly related to the classical
characteristic invariant \cite{Jo79,JT84,Oc85}. Indeed, put
$H=\{t\in G\mid\alpha_t\in\Inn(M)\}$ and choose
$v_t\in\cU(M)$ with $\alpha_t=\Ad(v_t)$ and $v_0=1$.
The classical data are the scalar functions $\lambda$ and $\mu$ defined by
\begin{equation}\label{characteristic cocycle data}
\alpha_s(v_t)=\lambda(s,t)v_t,
\qquad v_tv_r=\mu(t,r)v_{t+r}
\quad(s\in G,\ t,r\in H).
\end{equation}
Their class $[\lambda,\mu]$ modulo changes of phase of the $v_t$ is
the characteristic invariant, denoted by $\chi_\alpha$ in
\cite[Section~1.1]{JT84} and by $\Lambda(\alpha)$ in
\cite[Section~2.1]{Oc85}. Since $G$ is abelian, $\lambda(s,t)$ is
unchanged by these phase choices. Writing
$\lambda(s,t)=\langle p_t,s\rangle$, we obtain
\begin{equation}\label{characteristic invariant graph}
\Lambda(\alpha)=\{(p_t,t)\mid t\in H\},
\qquad
\mu(t,r)\overline{\mu(r,t)}=\langle p_r,t\rangle.
\end{equation}
Thus our subgroup records $H$, $\lambda$, and the commutators of the
implementing unitaries. For a discrete abelian group $H$, a scalar
$2$-cocycle is determined up to coboundary by its commutator
bicharacter.
Consequently, for discrete abelian actions our invariant contains
exactly the classical characteristic data.
\end{remark}

In the case $G=\R$, subgroups of the resonance set are strongly constrained. We identify $\widehat\R$ with $\R$ by associating
$p\in\R$ with the character $t\mapsto e^{2\pi\ri pt}$, and the resonance equation becomes
$pt\in\Z$.

\begin{proposition} \label{resonant group for R}
    Let $\Sigma < \R \times \R$ be a subgroup such that $$\Sigma \subset \mathcal{R}_\R=\{ (p,t) \in \R^2 \mid pt \in \Z \}.$$
    Then either $\Sigma \subset \R \times \{0\}$, or $\Sigma \subset \{0\} \times \R$ or $\Sigma$ is discrete.
\end{proposition}
\begin{proof}
Suppose that $\Sigma$ is not discrete. Choose a sequence
$((p_n,t_n))_{n\in\N}$ in $\Sigma\setminus\{(0,0)\}$ converging to
$(0,0)$. Since
$p_nt_n\in\Z$, one has $p_nt_n=0$ eventually. After passing to a
subsequence, either $t_n=0$ for every $n\in\N$, or $p_n=0$ for every
$n\in\N$.

In the first case, fix $(p,t)\in\Sigma$. Since $(p_n,0)\in\Sigma$ and
$\Sigma$ is a subgroup, one has $(p+p_n,t)\in\Sigma$ for every
$n\in\N$. Applying the resonance relation to $(p,t)$ and
$(p+p_n,t)$ gives $p_nt\in\Z$. Letting $n$ tend to infinity gives
$t=0$, and hence
$\Sigma\subset\R\times\{0\}$. 

In the second case, the same argument gives $\Sigma \subset \{0\} \times \R$.
\end{proof}

The resonance property imposes a surprisingly strong arithmetic restriction
on the inner times of a flow.

\begin{proposition}\label{inner times quadratic generators}
Let $\alpha:\R\curvearrowright M$ be a flow on a factor with
$\Gamma(\alpha)=\R$. Then the subgroup
\[
\{t\in\R\mid\alpha_t\in\Inn(M)\}
\]
is free abelian of rank at most two. If it has rank two and is generated
by $s,t>0$, then $s/t$ is a quadratic number.
\end{proposition}

\begin{proof}
If $(0,t)\in\Lambda(\alpha)$ and $v$ is an implementing unitary,
then $v^*u_t$ is central in $M\rtimes_\alpha\R$, where $u_t$ is the
canonical crossed-product unitary. Since $\Gamma(\alpha)=\R$,
this crossed product is a factor, forcing $t=0$. Thus
Proposition~\ref{characteristic invariant properties} shows that
$\Lambda(\alpha)$ meets both coordinate axes only at the origin.
By Proposition~\ref{resonant group for R}, it is discrete, hence free
abelian of rank at most two. Its projection onto the second coordinate
identifies it with the subgroup of inner times.

In the rank-two case, let $(p,s),(q,t)\in\Lambda(\alpha)$ lift the
positive generators. Resonance applied to these elements and their sum
gives
\[
m=ps\in\Z\setminus\{0\},\qquad
n=qt\in\Z\setminus\{0\},\qquad
k=pt+qs\in\Z.
\]
Consequently, $r=s/t$ satisfies $nr^2-kr+m=0$. Since $s,t$ generate
a group of rank two, $r$ is irrational, proving the assertion.
\end{proof}

Write $(u_t)_{t\in G}$ for the canonical unitaries in
$M\rtimes_\alpha G$, so that $u_txu_t^*=\alpha_t(x)$ for $x\in M$.
The extended action $\widetilde\alpha_t=\Ad(u_t)$ and the dual action
$\widehat\alpha$, normalized by
\[
\widehat\alpha_p(x)=x,
\qquad
\widehat\alpha_p(u_t)=\langle p,t\rangle u_t
\qquad(x\in M,\ t\in G,\ p\in\widehat G),
\]
commute and preserve $M'\cap(M\rtimes_\alpha G)$.
The dual fixed-point formula \cite[Theorem ${\rm X}$.2.3]{Ta03}
and the crossed-product generators give
\begin{equation}\label{eq:fixed-algebras}
\begin{aligned}
\bigl(M'\cap(M\rtimes_\alpha G)\bigr)^{\widehat\alpha}&=\C1,\\
\bigl(M'\cap(M\rtimes_\alpha G)\bigr)^{\widetilde\alpha}
&=\cZ(M\rtimes_\alpha G).
\end{aligned}
\end{equation}
Indeed, the first fixed-point algebra is $M'\cap M$, while an
element of the second commutes with $M$ and every $u_t$.

We define the \emph{joint point spectrum} of these two actions on
$M'\cap(M\rtimes_\alpha G)$ to be the set of all
$(p,t)\in\widehat G\times G$ for which there is a nonzero
$x\in M'\cap(M\rtimes_\alpha G)$ satisfying
\[
\widetilde\alpha_s(x)=\langle p,s\rangle x,
\qquad
\widehat\alpha_q(x)=\langle q,t\rangle x
\qquad(s\in G,\ q\in\widehat G).
\]
We denote this set by
$\Sp_{\mathrm p}\bigl((\widetilde\alpha\times\widehat\alpha)
|_{M'\cap(M\rtimes_\alpha G)}\bigr)$.

\begin{proposition}\label{atomic coefficients}

For each $t\in G$, there is a nonzero eigenoperator for the dual action
on $M'\cap(M\rtimes_\alpha G)$ with eigenvalue $t$ if and only if
$\alpha_t$ is inner. If $\alpha_t=\Ad(v)$, the eigenspace for the
dual action with eigenvalue $t$ is $\C v^*u_t$. Moreover,
\[
\{(p,-t)\mid(p,t)\in\Lambda(\alpha)\}=
\Sp_{\mathrm p}\bigl((\widetilde\alpha\times\widehat\alpha)
|_{M'\cap(M\rtimes_\alpha G)}\bigr).
\]
\end{proposition}

\begin{proof}
Suppose that $0\ne x\in M'\cap(M\rtimes_\alpha G)$ satisfies
$\widehat\alpha_p(x)=\langle p,t\rangle x$ for all
$p\in\widehat G$. Then $xu_{-t}$ is fixed by the dual action, so
$x=au_t$ for some $a\in M$. Commutation with $M$ is equivalent to
\begin{equation}\label{eq:atomic-intertwiner}
ya=a\alpha_t(y)
\qquad(y\in M).
\end{equation}
Taking adjoints shows that $a^*a$ and $aa^*$ are central. They are
equal positive scalar multiples of $1$, so $a$ is a scalar multiple
of the adjoint of a unitary implementing $\alpha_t$. Thus $\alpha_t$ is
inner. Conversely, if $\alpha_t=\Ad(v)$, then $v^*u_t$ commutes
with $M$ and is an eigenoperator for the dual action with eigenvalue $t$. Formula
\eqref{eq:atomic-intertwiner} makes $av$ scalar, so the dual
eigenspace is exactly $\C v^*u_t$.

If $\alpha_s(v)=\langle p,s\rangle v$ for every $s\in G$, then
\[
\widetilde\alpha_s(vu_{-t})=\langle p,s\rangle vu_{-t},
\qquad
\widehat\alpha_q(vu_{-t})=\langle q,-t\rangle vu_{-t}.
\]
This proves the joint point spectrum identity.
\end{proof}
Since $\mathcal R_G$ is invariant under reflection of either coordinate,
the joint point spectrum is contained in $\mathcal R_G$.
However, joint eigenoperators need not generate the relative
commutant, and the point spectrum need not determine its full
Arveson spectrum. We will therefore study
\[
\Sp\bigl((\widetilde\alpha\times\widehat\alpha)
|_{M'\cap(M\rtimes_\alpha G)}\bigr)
\]
and prove, under the hypotheses stated below, that it too is
contained in $\mathcal R_G$. For a nondiscrete group, an arbitrary
element of the crossed product need not admit a Fourier expansion
with coefficients in $M$. We replace these coefficients by
$M$-valued distributions whose support records the joint Arveson
spectrum. This lets us use commutation with $M$ to establish
resonance beyond the joint eigenoperators. We first recall the
distributional and spectral tools needed for this construction.

\subsection{Preliminaries for locally compact abelian group actions}
\label{sec:preliminaries}

Throughout Subsections~\ref{sec:preliminaries}--\ref{sec:relative-commutant-theorems},
$G$ is a locally compact abelian group. We write $G$ and its Pontryagin
dual $\widehat G$ additively, with identity $0$, and denote evaluation
by $\langle p,t\rangle$ for $p\in\widehat G$ and $t\in G$. Thus
$\langle p+q,t\rangle=\langle p,t\rangle\langle q,t\rangle$ and
$\langle-p,t\rangle=\overline{\langle p,t\rangle}$.
We identify $\widehat\R$ with $\R$ by
$\langle p,t\rangle=e^{2\pi\ri pt}$.

Haar measures on $G$ and $\widehat G$ are in Plancherel duality.
For $G=\R$, both are Lebesgue measure. For $\varphi\in\rL^1(G)$ and
$\psi\in\rL^1(\widehat G)$, our Fourier conventions are
\[
\widehat\varphi(p)=\int_G \varphi(t)\langle-p,t\rangle\,\rd t,
\qquad
\check\psi(t)=\int_{\widehat G}\psi(p)\langle p,t\rangle\,\rd p.
\]
We use the same conventions on $\widehat G$, identifying its dual with
$G$. In particular,
$\check\varphi(p)=\widehat\varphi(-p)
=\int_G \varphi(t)\langle p,t\rangle\,\rd t$.

\subsubsection{Schwartz--Bruhat calculus}
\label{subsec:schwartz-bruhat}

A Bruhat stage is given by an open compactly generated subgroup
$G_0\subset G$ and a compact subgroup $K\subset G_0$ such that
$G_0/K\cong\R^k\times\Z^l\times\mathbb T^m\times F$, where $F$ is
finite. The corresponding stage space consists of the
functions on $G$ that are supported in $G_0$, are $K$-invariant, and
whose descent to $G_0/K$ is smooth with all Lie derivatives rapidly
decreasing in the $\R^k\times\Z^l$ variables. The Schwartz--Bruhat
space $\mathcal S(G)$ is the inductive limit of these stage spaces.
Thus $\mathcal S(\R^k)$ is the usual Schwartz space,
$\mathcal S(\Z^l)$ consists of rapidly decreasing sequences,
$\mathcal S(\mathbb T^m)=\rC^\infty(\mathbb T^m)$, and
$\mathcal S(\mathbb Q_p^k)$ consists of the locally constant compactly
supported functions. Every element of $\mathcal S(G)$ belongs to one
Bruhat stage \cite[Section~9 and Proposition~11(a)]{Br61}.

Put $\mathcal D(G)=\mathcal S(G)\cap\rC_c(G)$, with the locally convex
inductive limit topology obtained from the spaces $\mathcal D(G_0/K)$
on the Bruhat stages, extended by zero to $G$
\cite[Sections~1--2]{Br61}. This topology need not agree with the topology induced by $\mathcal S(G)$. Write
$\mathcal S'(G)$ for the continuous dual of
$\mathcal S(G)$. More generally, if $X$ is a Banach space, an
$X$-valued tempered Schwartz--Bruhat distribution on $G$ is a continuous
linear map $T:\mathcal S(G)\to X$. We denote the space of these maps
by $\mathcal S'(G,X)$. We similarly write $\mathcal D'(G,X)$ for the
continuous linear maps from $\mathcal D(G)$ to $X$. A distribution
vanishes on an open set $U$
when $T(\varphi)=0$ for every $\varphi\in\mathcal D(G)$ supported in $U$,
and its support is the complement of the largest such set
\cite[Chapter ${\rm I}$, Section~2, p.~61]{Sc57}. Thus the support of
an element of $\mathcal S'(G)$ is tested with functions in
$\mathcal D(G)$.

Fourier transformation is a topological isomorphism from
$\mathcal S(G)$ onto $\mathcal S(\widehat G)$. We denote it by
$\mathcal F_G$ and write $\mathcal F_G(\varphi)=\widehat\varphi$. Its inverse is
$\psi\mapsto\check\psi$. On distributions, the spectral transform
is the transpose of the Fourier transform on $\widehat G$. We denote it by
$\widehat T$ and set
$\langle\widehat T,\psi\rangle=\langle T,\widehat\psi\rangle$
for $T\in\mathcal S'(G)$ and $\psi\in\mathcal S(\widehat G)$.
For $p\in\widehat G$ and $a\in X$, the $X$-valued distribution
$t\mapsto\langle p,t\rangle a$ has spectral transform $a\delta_p$.
The space $\mathcal S(G)$ is complete \cite[Section~9]{Br61} and embeds
continuously and densely
into $\rL^1(G)\cap\rL^2(G)$, while $\mathcal D(G)$ is dense in
$\mathcal S(G)$. The kernel theorem and the stage construction give
\[
\mathcal S(G_1)\mathbin{\widehat\otimes}_{\iota}
\mathcal S(G_2)
\cong\mathcal S(G_1\times G_2),
\]
where $\widehat\otimes_{\iota}$ is the completed inductive tensor
product. The corresponding result for $\mathcal D$ is
\cite[Theorem~3]{Br61}. The same stagewise proof applies to
$\mathcal S$ using \cite[Section~9, Proposition~11(a)]{Br61}. Thus algebraic
tensor products are dense in both product spaces. Pullback by a
topological group isomorphism and multiplication by a fixed
Schwartz--Bruhat function are continuous. In
particular, $\mathcal D(G)\mathcal S(G)\subset\mathcal D(G)$. Finally,
if $C\subset G$ is compact and $U$ is a neighborhood of $C$, some
$\varphi\in\mathcal D(G)$ equals $1$ near $C$ and is supported in $U$.

Call a function $m:G\to\C$ \emph{regular} if,
near every compact subset of $G$, it agrees with some
$h\in\mathcal D(G)$. This is the usual Bruhat regularity condition
\cite[Sections~1--2]{Br61}. Such a function defines a continuous
multiplier of $\mathcal D(G)$. Indeed, on a test-function space with
a fixed stage and compact support, replace $m$ by such an $h$ and
pass to a common stage. The assertion then follows from continuous
smooth multiplication on the finite-dimensional Lie quotient and
the inductive limit topology.

\begin{lemma}\label{support of annihilated distribution}
Let $G$ be a locally compact abelian group, let $X$ be a Hausdorff
locally convex space, and let $\mathcal K:\mathcal D(G)\to X$ be a
continuous linear map. If $m:G\to\C$ is regular and
\[
\langle\mathcal K,m\Phi\rangle=0
\qquad(\Phi\in\mathcal D(G)),
\]
then
\[
\supp(\mathcal K)\subset m^{-1}(\{0\}).
\]
\end{lemma}

\begin{proof}
Let $\Phi\in\mathcal D(G)$ have support disjoint from
$m^{-1}(\{0\})$. Choose $h\in\mathcal D(G)$ agreeing with $m$ near
$\supp(\Phi)$ and place $h$ and $\Phi$ in one Bruhat stage. On its
Lie quotient, $h$ is smooth and nonzero on the compact support of
$\Phi$. Dividing there and extending by zero gives
\[
\Psi=\Phi/h\in\mathcal D(G),
\qquad
\supp(\Psi)\subset\supp(\Phi),
\qquad
m\Psi=\Phi.
\]
Here the quotient is set to zero where $h$ vanishes. It is smooth
there because $\Phi$ vanishes on a neighborhood of that zero set.
The hypothesis gives
\[
\langle\mathcal K,\Phi\rangle
=\langle\mathcal K,m\Psi\rangle=0,
\]
which proves the support inclusion.
\end{proof}

The lemma applies in particular to Banach-valued tempered distributions,
after restriction to $\mathcal D(G)$. We will apply it with $m$ built from the following regular multipliers. 

Define the evaluation function
\[
w_G : \widehat G\times G\longrightarrow\mathbb T :
(p,t)\longmapsto\langle p,t\rangle.
\]
Multiplication by $w_G$ preserves $\mathcal S$ because
partial Fourier transformation in the $G$-coordinate
conjugates it on $\mathcal S$ to pullback by
$s_G(p,q)=(p,q-p)$.
For any compact set $C\subset\widehat G\times G$, choose
$\chi\in\mathcal D(\widehat G\times G)$ equal to $1$ near $C$.
Then $w_G\chi$ belongs to $\mathcal S$ and has compact support, so
it belongs to $\mathcal D$ and agrees with $w_G$ near $C$.
Thus $w_G$ is regular.

If $L$ is a locally compact abelian group and $\ell:L\to\R$ is a
continuous homomorphism, then $e^\ell$ is
regular as well. Indeed, $\ell$ vanishes on every compact subgroup,
so on each Bruhat stage it descends to a smooth homomorphism on the
Lie quotient. Multiplying $e^\ell$ by a compactly supported cutoff
therefore gives an element of $\mathcal D(L)$.

\subsubsection{Arveson spectra}

Let $\theta:G\curvearrowright M$ be a continuous action on a von Neumann
algebra. For $\varphi\in\rL^1(G)$ and $x\in M$, define
$\theta_\varphi(x)=\int_G\varphi(t)\theta_t(x)\,\rd t$, with the
integral taken in the weak$^*$ sense. Following
\cite[Definition ${\rm XI}$.1.2]{Ta03}, the annihilator ideal, the
individual Arveson spectrum and the global Arveson spectrum are
\[
\begin{aligned}
\mathcal I_\theta(x)
&=\{\varphi\in\rL^1(G)\mid\theta_\varphi(x)=0\},\\
\Sp_\theta(x)
&=\{p\in\widehat G\mid
\check\varphi(p)=0\text{ for every }\varphi\in\mathcal I_\theta(x)\},\\
\Sp(\theta)
&=\overline{\bigcup_{x\in M}\Sp_\theta(x)}.
\end{aligned}
\]
We write
$M^\theta(E)=\{x\in M\mid\Sp_\theta(x)\subset E\}$ for
$E\subset\widehat G$.
If
$\theta_t(x)=\langle p,t\rangle x$ with $x\ne0$, then
$\theta_\varphi(x)=\check\varphi(p)x$, so
$\Sp_\theta(x)=\{p\}$.

Spectral localization is governed by the cutoff formula
\cite[Lemma ${\rm XI}$.1.3(ii)]{Ta03}
\begin{equation}\label{eq:Arveson-cutoff}
\Sp_\theta(\theta_\varphi(x))
\subset
\Sp_\theta(x)\cap\supp(\check\varphi)
\qquad(\varphi\in\rL^1(G),\ x\in M).
\end{equation}
Every point of the global spectrum can be detected by a nonzero
localized element. Indeed, if $p\in\Sp(\theta)$ and
$U\subset\widehat G$ is a neighborhood of $p$, then
\cite[Lemma ${\rm XI}$.1.3(v)]{Ta03} gives a nonzero element of
$M^\theta(U)$. By the Schwartz--Bruhat calculus from
\S\,\ref{subsec:schwartz-bruhat}, this element can be chosen in the
form $\theta_\varphi(y)$ with $y\in M$,
$\varphi\in\mathcal S(G)$,
$\check\varphi\in\mathcal D(\widehat G)$ and
$\supp(\check\varphi)\subset U$.

For $p\in\widehat G$, the one-point spectral subspace is the
corresponding eigenspace \cite[Lemma ${\rm XI}$.1.11]{Ta03}:
\begin{equation}\label{eq:one-point-spectrum}
M^\theta(\{p\})
=\{x\in M\mid
\theta_t(x)=\langle p,t\rangle x\text{ for every }t\in G\}.
\end{equation}
If $\Sp(\theta)$ is discrete, spectral localization and an approximate
identity show that $M$ is generated as a von Neumann algebra by the eigenspaces
$M^\theta(\{p\})$, which are nonzero for every
$p\in\Sp(\theta)$. Thus a discrete Arveson spectrum equals the point
spectrum, with the same frequency convention.

For product groups, individual spectra also control the spectra of
the restricted actions. Suppose that $G=G_1\times G_2$, let
$\theta^{(j)}$ be the
restriction to the $j$-th factor and let
$\operatorname{pr}_j:\widehat G_1\times\widehat G_2
\longrightarrow\widehat G_j$ be the coordinate projection.
Applying an approximate identity in the other factor gives
\cite[Lemma ${\rm XI}$.1.3(i)--(ii)]{Ta03}
\begin{equation}\label{eq:coordinate-restriction}
\Sp_{\theta^{(j)}}(x)
\subset
\overline{\operatorname{pr}_j(\Sp_\theta(x))}
\qquad(x\in M).
\end{equation}

The Connes spectrum is obtained by restricting to invariant corners
\cite[Section ${\rm XI}$.2]{Ta03}. Writing $M^\theta$ for the
fixed-point algebra, it is defined by
\[
\Gamma(\theta)
=\bigcap_{0\ne e\in\operatorname{Proj}(M^\theta)}
\Sp(\theta|_{eMe}),
\]
where $\theta|_{eMe}$ denotes the restricted action on $eMe$.
If $\theta$ is ergodic, then $M^\theta=\C1$, so
$\Gamma(\theta)=\Sp(\theta)$. In this case,
\cite[Lemma ${\rm XI}$.2.2(iii) and Theorem ${\rm XI}$.2.7(i)]{Ta03}
give
\begin{equation}\label{eq:ergodic-spectrum-kernel}
\Sp(\theta)=\Gamma(\theta)=(\ker\theta)^\perp,
\end{equation}
where
\[
(\ker\theta)^\perp
=\{p\in\widehat G\mid
\langle p,t\rangle=1\text{ for every }t\in\ker\theta\}.
\]

The following proposition identifies individual Arveson spectra
with the supports of equivariant operator-valued distributions.

\begin{proposition}\label{equivariant distribution spectrum}
Let $N$ be a von Neumann algebra and let
$\mathcal J:M\longrightarrow\mathcal S'(\widehat G,N)$ be an
injective linear map. Assume that
$x\mapsto\langle\mathcal J(x),\psi\rangle$ is bounded and normal for
every $\psi\in\mathcal S(\widehat G)$ and that
\[
\mathcal J(\theta_t(x))
=\bigl[p\longmapsto\langle p,t\rangle\bigr]\mathcal J(x)
\]
for every $x\in M$ and $t\in G$. Then
\[
\supp(\mathcal J(x))=\Sp_\theta(x)
\qquad(x\in M).
\]
\end{proposition}

\begin{proof}
For $\varphi\in\mathcal S(G)$, normality of the evaluations of
$\mathcal J$ and the covariance assumption give
\begin{equation}\label{eq:equivariant-distribution-multiplier}
\mathcal J(\theta_\varphi(x))
=\check\varphi\,\mathcal J(x).
\end{equation}

We first show that
$\supp(\mathcal J(x))\subset\Sp_\theta(x)$. Let
$\psi\in\mathcal D(\widehat G)$ have support disjoint from
$\Sp_\theta(x)$. Choose
$\check\varphi\in\mathcal D(\widehat G)$ which equals $1$ near
$\supp(\psi)$ and whose support remains disjoint from
$\Sp_\theta(x)$, and let $\varphi$ be its Fourier transform.
Formula \eqref{eq:Arveson-cutoff} gives
\[
\Sp_\theta(\theta_\varphi(x))
\subset
\Sp_\theta(x)\cap\supp(\check\varphi)=\varnothing.
\]
An element with empty individual spectrum is zero, so
$\theta_\varphi(x)=0$. Equation
\eqref{eq:equivariant-distribution-multiplier} therefore gives
\[
\langle\mathcal J(x),\psi\rangle
=\langle\check\varphi\mathcal J(x),\psi\rangle=0.
\]
This proves the first inclusion.

Conversely, let $p_0\notin\supp(\mathcal J(x))$. Choose
$\check\varphi\in\mathcal D(\widehat G)$ supported in the
complement of $\supp(\mathcal J(x))$ with
$\check\varphi(p_0)\neq0$, and let $\varphi$ be its
Fourier transform. Then
$\check\varphi\mathcal J(x)=0$. Indeed, for every
$\psi\in\mathcal S(\widehat G)$, the function
$\check\varphi\psi$ belongs to $\mathcal D(\widehat G)$ and is
supported in an open set on which $\mathcal J(x)$ vanishes. Hence
\eqref{eq:equivariant-distribution-multiplier} and injectivity give
$\theta_\varphi(x)=0$. Since
$\check\varphi(p_0)\neq0$, the definition of the individual
Arveson spectrum gives $p_0\notin\Sp_\theta(x)$. This proves the
reverse inclusion.
\end{proof}

\subsection{Fourier coefficient distributions}
\label{sec:fourier-coefficient-distributions}

We use the canonical unitaries and the extended and dual actions from
Subsection~\ref{subsec:action-invariant}, now for an arbitrary von
Neumann algebra $M$. We identify $\widehat{G\times\widehat G}$ with
$\widehat G\times G$ through
\[
((s,q),(p,t))\longmapsto \langle p,s\rangle\langle q,t\rangle.
\]
By \eqref{eq:one-point-spectrum}, this agrees with the convention
used to define the joint point spectrum in
Subsection~\ref{subsec:action-invariant}.

\subsubsection{Operator-valued weights}
Let $M\subset N$ be an inclusion of von Neumann algebras and let
$\rT:N_+\longrightarrow\widehat M_+$ be a normal operator-valued
weight \cite{Ha79}. We use the standard domains
\[
\mathfrak n_{\rT}
=\{a\in N\mid \rT(a^*a)\in M_+\},
\qquad
\mathfrak m_{\rT}=\operatorname{span}(\mathfrak n_{\rT}^*\mathfrak n_{\rT}),
\]
on the second of which $\rT$ is an $M$-valued linear map. The space
$\mathfrak m_{\rT}$ is an $M$-bimodule, and $\rT$ is $M$-bimodular there:
\[
\rT(ayb)=a\rT(y)b
\qquad(a,b\in M,\ y\in\mathfrak m_{\rT}).
\]

\begin{lemma}
    If $a,b\in\mathfrak n_{\rT}$, then
\begin{equation}\label{eq:operator-valued-cauchy-schwarz-matrix}
\begin{pmatrix}
\rT(a^*a)&\rT(a^*b)\\
\rT(b^*a)&\rT(b^*b)
\end{pmatrix}\geq 0.
\end{equation}
Consequently,
\begin{equation}\label{eq:operator-valued-cauchy-schwarz}
\|\rT(a^*b)\|^2
\leq
\|\rT(a^*a)\|\,\|\rT(b^*b)\|.
\end{equation}
\end{lemma}
\begin{proof}
    The matrix
\[
\begin{pmatrix}
a^*a&a^*b\\
b^*a&b^*b
\end{pmatrix}
=
\begin{pmatrix}a^*\\ b^*\end{pmatrix}
\begin{pmatrix}a&b\end{pmatrix}
\]
is positive and has entries in $\mathfrak m_{\rT}$. Positivity of the
entrywise $2\times2$ amplification of $\rT$ is proved in
\cite[Lemma~4.5(2), pp.~200--201]{Ha79}. This gives
\eqref{eq:operator-valued-cauchy-schwarz-matrix}, and the standard
estimate for a positive $2\times2$ operator matrix gives
\eqref{eq:operator-valued-cauchy-schwarz}.
\end{proof}

\begin{lemma}\label{operator-valued weight Banach norms}
The formula
\[
q_{\rT}(a)=\|a\|+\|\rT(a^*a)\|^{1/2}
\qquad(a\in\mathfrak n_{\rT})
\]
defines a complete norm on $\mathfrak n_{\rT}$.

For $y\in\mathfrak m_{\rT}$, the formula
\[
\|y\|_{\rT}
=\inf
\sum_{j=1}^n q_{\rT}(a_j)q_{\rT}(b_j),
\]
where the infimum runs over all finite decompositions
$y=\sum_{j=1}^n a_j^*b_j$ with $n\geq1$ and
$a_j,b_j\in\mathfrak n_{\rT}$, defines a complete norm satisfying
\[
\|y\|\leq\|y\|_{\rT},
\qquad
\|\rT(y)\|\leq\|y\|_{\rT}
\qquad(y\in\mathfrak m_{\rT}).
\]
\end{lemma}
\begin{proof}
Put $p_{\rT}(a)=\|\rT(a^*a)\|^{1/2}$ for $a\in\mathfrak n_{\rT}$.
Inequality \eqref{eq:operator-valued-cauchy-schwarz} gives
\[
p_{\rT}(a+b)^2
\leq p_{\rT}(a)^2+2p_{\rT}(a)p_{\rT}(b)+p_{\rT}(b)^2.
\]
Thus $p_{\rT}$ is subadditive, and $q_{\rT}$ is a norm.

We prove that it is complete. Let $(a_n)_n$ be a $q_{\rT}$-Cauchy
sequence. It converges in norm to some $a\in N$. Fix $\varepsilon>0$
and take $n_0$ such that
$q_{\rT}(a_n-a_m)\leq\varepsilon$ for $n,m\geq n_0$. For
$m\geq n_0$ and $\omega\in M_*^+$, lower semicontinuity of the normal
weight $\omega\circ \rT$ gives
\[
\begin{aligned}
(\omega\circ \rT)((a-a_m)^*(a-a_m))
&\leq\liminf_n
(\omega\circ \rT)((a_n-a_m)^*(a_n-a_m))\\
&\leq\varepsilon^2\|\omega\|.
\end{aligned}
\]
It follows that $a-a_m\in\mathfrak n_{\rT}$ and
$p_{\rT}(a-a_m)\leq\varepsilon$. We also have
$\|a-a_m\|\leq\varepsilon$. Hence $a\in\mathfrak n_{\rT}$ and
$q_{\rT}(a-a_m)\leq2\varepsilon$, proving completeness.

Equip $\mathfrak n_{\rT}^*=\{a^*\mid a\in\mathfrak n_{\rT}\}$
with the complete norm $a^*\mapsto q_{\rT}(a)$, and consider the
completed projective tensor product
\[
E=\mathfrak n_{\rT}^*\widehat\otimes_\pi\mathfrak n_{\rT}.
\]
Since $\|a^*b\|\leq q_{\rT}(a)q_{\rT}(b)$, multiplication induces a
continuous linear map
\[
\Phi:E\longrightarrow N,\qquad \Phi(a^*\otimes b)=a^*b,
\qquad \|\Phi\|\leq1.
\]
Likewise, \eqref{eq:operator-valued-cauchy-schwarz} gives a continuous
linear map
\[
\Psi:E\longrightarrow M,\qquad
\Psi(a^*\otimes b)=\rT(a^*b),\qquad \|\Psi\|\leq1.
\]

We show that $\Phi(E)=\mathfrak m_{\rT}$ and
$\Psi=\rT\circ\Phi$. The inclusion $\mathfrak m_{\rT}\subset\Phi(E)$
follows from the definition of $\mathfrak m_{\rT}$. Conversely, let
$u\in E$. By Grothendieck's series description of the projective
tensor product, we can write
\[
u=\sum_{j=1}^\infty a_j^*\otimes b_j,
\qquad
\sum_{j=1}^\infty q_{\rT}(a_j)q_{\rT}(b_j)<+\infty,
\]
with convergence in the projective norm. Omitting zero terms and
replacing $a_j,b_j$ by $r_ja_j,r_j^{-1}b_j$ for suitable $r_j>0$,
we may assume that $q_{\rT}(a_j)=q_{\rT}(b_j)$ for every $j$.
Put $d_{j,k}=a_j+\ri^k b_j$ for $k=0,1,2,3$. Polarization gives
\[
a_j^*\otimes b_j
=\frac14\sum_{k=0}^3(-\ri)^k d_{j,k}^*\otimes d_{j,k},
\qquad
q_{\rT}(d_{j,k})^2\leq4q_{\rT}(a_j)q_{\rT}(b_j).
\]
Thus, for each $k$, both series
\[
c_k=\sum_{j=1}^\infty d_{j,k}^*d_{j,k},
\qquad
\sum_{j=1}^\infty\rT(d_{j,k}^*d_{j,k})
\]
converge in norm, since
\[
\sum_{j=1}^\infty
\bigl(\|d_{j,k}^*d_{j,k}\|+\|\rT(d_{j,k}^*d_{j,k})\|\bigr)
\leq\sum_{j=1}^\infty q_{\rT}(d_{j,k})^2<+\infty.
\]
Normality of $\rT$ now implies
\[
\rT(c_k)
=\sup_{\ell}\sum_{j=1}^{\ell}\rT(d_{j,k}^*d_{j,k})
=\sum_{j=1}^\infty\rT(d_{j,k}^*d_{j,k})\in M_+.
\]
Hence $c_k\in\mathfrak m_{\rT}^+$, and continuity of $\Phi$ and
$\Psi$ gives
\[
\Phi(u)=\frac14\sum_{k=0}^3(-\ri)^k c_k\in\mathfrak m_{\rT},
\qquad
\Psi(u)=\frac14\sum_{k=0}^3(-\ri)^k\rT(c_k)=\rT(\Phi(u)).
\]
Since $c_k^{1/2}\in\mathfrak n_{\rT}$ and
\[
q_{\rT}(c_k^{1/2})^2
\leq2\bigl(\|c_k\|+\|\rT(c_k)\|\bigr),
\]
this finite decomposition also gives
\[
\|\Phi(u)\|_{\rT}
\leq\frac14\sum_{k=0}^3q_{\rT}(c_k^{1/2})^2
\leq8\sum_{j=1}^\infty q_{\rT}(a_j)q_{\rT}(b_j).
\]

Since $\ker\Phi$ is closed in $E$, its image $\mathfrak m_{\rT}$
is a Banach space for the quotient norm
\[
Q(y)
=\inf\{\|u\|_\pi\mid u\in E,\ \Phi(u)=y\}.
\]
We verify that $Q(y)=\|y\|_{\rT}$. Each finite decomposition in the
statement gives an element $u=\sum_{j=1}^n a_j^*\otimes b_j\in E$ with
$\Phi(u)=y$ and
$\|u\|_\pi\leq\sum_{j=1}^n q_{\rT}(a_j)q_{\rT}(b_j)$.
Thus $Q(y)\leq\|y\|_{\rT}$. Conversely, for every $u\in E$ and
$\varepsilon>0$, the series description of the projective norm
gives a representation $u=\sum_j a_j^*\otimes b_j$ with
\[
\sum_j q_{\rT}(a_j)q_{\rT}(b_j)<\|u\|_\pi+\varepsilon.
\]
When $\Phi(u)=y$, adjoining the first $n$ products to the finite
decomposition of the tail constructed above gives
\[
\|y\|_{\rT}
\leq\sum_{j=1}^n q_{\rT}(a_j)q_{\rT}(b_j)
+8\sum_{j>n}q_{\rT}(a_j)q_{\rT}(b_j).
\]
Letting $n$ tend to infinity gives
$\|y\|_{\rT}<\|u\|_\pi+\varepsilon$. Taking the infimum over $u$
with $\Phi(u)=y$ and letting $\varepsilon$ tend to zero proves the
reverse inequality. Hence
$\|\cdot\|_{\rT}$ is complete.
The bounds on $\Phi$ and $\Psi=\rT\circ\Phi$ imply
\[
\|y\|\leq\|y\|_{\rT},
\qquad
\|\rT(y)\|\leq\|y\|_{\rT}.
\]
\end{proof}

\subsubsection{Two-point coefficients}
Let
\[
\rT:(M\rtimes_\alpha G)_+\longrightarrow\widehat M_+
\]
be the canonical Haar operator-valued weight.
We use Haagerup's construction of $\rT$ and the formula in
\cite[Theorem~1.1(a)--(b)]{Ha78b}, together with the weak$^*$ density
statement in \cite[Lemma~2.3(e)--(f)]{Ha78a}.

For $\varphi\in\mathcal S(G)$, put
$\lambda(\varphi)=\int_G \varphi(t)u_t\,\rd t$.
This denotes the integrated form of the strongly continuous unitary
representation and satisfies
$\|\lambda(\varphi)\|\leq\|\varphi\|_1$. By \cite[Theorem~1.1(b)]{Ha78b} and
\cite[Lemma~2.3(b)]{Ha78a}, we have
\begin{equation}\label{eq:Haar-weight-Plancherel}
\rT(\lambda(\varphi)^*\lambda(\varphi))
=\rT(\lambda(\varphi)\lambda(\varphi)^*)
=\|\varphi\|_2^2 1
\end{equation}
for every $\varphi\in\mathcal S(G)$. In particular, both
$\lambda(\varphi)$ and $\lambda(\varphi)^*$ belong to
$\mathfrak n_{\rT}$.

When $G$ is discrete, $\rT$ is the canonical conditional expectation
and the ordinary Fourier coefficients $\rT(xu_t^*)$ are defined for
every $x\in M\rtimes_\alpha G$. In the nondiscrete setting, the formal expression
$\rT(xu_t^*)$ need not be defined because $\rT$ is an unbounded
operator-valued weight and $xu_t^*$ need not belong to
$\mathfrak m_{\rT}$. Smearing on only one side does not remove this domain
obstruction. Although $x\lambda(\psi)\in\mathfrak n_{\rT}$, there is no
reason for it to belong to $\mathfrak m_{\rT}$, the domain on which $\rT$ is
$M$-valued and linear. We solve this problem by using two-sided smoothing. Indeed, we have
\[
\lambda(\varphi)x\lambda(\psi)
=\bigl(\lambda(\varphi)^*\bigr)^*x\lambda(\psi)
\in\mathfrak n_{\rT}^*\mathfrak n_{\rT}\subset\mathfrak m_{\rT}.
\]
The following theorem shows that, by applying $\rT$ after this two-sided smoothing operation, we obtain a distribution defined for every
$x\in M\rtimes_\alpha G$. 

Write
$\mathcal L_\sigma(M\rtimes_\alpha G,M)$ for the Banach space of
bounded normal linear maps from $M\rtimes_\alpha G$ to $M$, equipped
with the operator norm. 

\begin{theorem}\label{two-sided coefficient distribution}
The following assertions hold.
\begin{enumerate}[\rm (i)]
\item For every $x\in M\rtimes_\alpha G$ and
$\Phi\in\mathcal S(G\times G)$, the weak$^*$ integral
\[
\int_{G\times G}\Phi(s,t)u_sxu_t\,\rd s\,\rd t
\]
belongs to $\mathfrak m_{\rT}$.
\item For every $x \in M \rtimes_\alpha G$, there is a unique distribution $\Delta(x) \in \mathcal{S}'(G \times G, M)$ whose value $\langle\Delta(x),\Phi\rangle$ is obtained by applying
$\rT$ to the integral in~$(\rm i)$.
For every $\Phi\in\mathcal S(G\times G)$, the evaluation
$x\mapsto\langle\Delta(x),\Phi\rangle$ belongs to
$\mathcal L_\sigma(M\rtimes_\alpha G,M)$, and these evaluations depend
continuously on $\Phi$ in this space. Moreover the map
\[
\Delta:M\rtimes_\alpha G \ni x \mapsto \Delta(x) \in \mathcal S'(G\times G,M)
\]
is injective.
\item If $\varphi\in\rC_c(G,M)$ and
$$x=\int_G\varphi(t)u_t\,\rd t,$$ then $\Delta(x)$ is represented by
the bounded $M$-valued function
\begin{equation}\label{eq:two-point-core-formula}
\Delta(x)(s,t)=\alpha_s\bigl(\varphi(-s-t)\bigr).
\end{equation}
\end{enumerate}
\end{theorem}

\begin{proof}
For the duration of the proof, denote the integral in
assertion~$(\rm i)$ by $Q_\Phi(x)$.
We first prove assertion~$(\rm i)$, together with the boundedness and
normality statement in assertion~$(\rm ii)$.
 Let $\varphi,\psi\in\mathcal S(G)$ and
$x\in M\rtimes_\alpha G$. Then
$x\lambda(\psi)\in\mathfrak n_{\rT}$ because
$\mathfrak n_{\rT}$ is a left ideal, while
$\lambda(\varphi)^*\in\mathfrak n_{\rT}$. Hence
$\lambda(\varphi)x\lambda(\psi)
=(\lambda(\varphi)^*)^*x\lambda(\psi)$ belongs to
$\mathfrak n_{\rT}^*\mathfrak n_{\rT}$. Moreover,
\[
\rT(\lambda(\psi)^*x^*x\lambda(\psi))
\leq \|x\|^2\|\psi\|_2^2 1.
\]
Inequality \eqref{eq:operator-valued-cauchy-schwarz}, applied with
$a=\lambda(\varphi)^*$ and $b=x\lambda(\psi)$, gives
\[
\|\rT(\lambda(\varphi)x\lambda(\psi))\|
\leq\|\varphi\|_2\|\psi\|_2\|x\|.
\]
Thus
\[
C_{\varphi,\psi}:M\rtimes_\alpha G\longrightarrow M,
\qquad
C_{\varphi,\psi}(x)=\rT(\lambda(\varphi)x\lambda(\psi)),
\]
is bounded. For $c\in\mathfrak n_{\rT}$, the map
$x\mapsto \rT(c^*xc)$ is positive and normal. Polarization with
$c=\lambda(\varphi)^*+\ri^k\lambda(\psi)$, where $0\leq k\leq3$,
expresses $C_{\varphi,\psi}$ on $(M\rtimes_\alpha G)_+$ as a linear
combination of normal maps. It therefore proves normality of
$C_{\varphi,\psi}$.

The Schwartz--Bruhat kernel theorem recalled in
\S\,\ref{subsec:schwartz-bruhat} will be applied with a Banach
topology on $\mathfrak m_{\rT}$. Fix the Banach norm $\|\cdot\|_{\rT}$
provided by Lemma~\ref{operator-valued weight Banach norms}.
Formula \eqref{eq:Haar-weight-Plancherel} and the defining estimate for
this norm give
\[
\begin{aligned}
\|\lambda(\varphi)x\lambda(\psi)\|_{\rT}
&\leq q_{\rT}(\lambda(\varphi)^*)q_{\rT}(x\lambda(\psi))\\
&\leq\|x\|
\bigl(\|\varphi\|_1+\|\varphi\|_2\bigr)
\bigl(\|\psi\|_1+\|\psi\|_2\bigr).
\end{aligned}
\]
Consequently,
\[
(\varphi,\psi)\longmapsto
\bigl(x\mapsto\lambda(\varphi)x\lambda(\psi)\bigr)
\]
is a continuous bilinear map from
$\mathcal S(G)\times\mathcal S(G)$ to the Banach space of bounded
linear maps from $M\rtimes_\alpha G$ to
$(\mathfrak m_{\rT},\|\cdot\|_{\rT})$. The Schwartz--Bruhat kernel theorem
extends this map to $\mathcal S(G\times G)$. After composition with the
continuous inclusion $\mathfrak m_{\rT}\subset M\rtimes_\alpha G$, this
extension agrees with the weak$^*$ integral defining $Q_\Phi$. Indeed,
the two maps agree on elementary tensors, those tensors are dense, and
\[
\|Q_\Phi(x)\|\leq\|\Phi\|_1\|x\|.
\]
Hence $Q_\Phi(x)\in\mathfrak m_{\rT}$, and continuity of $\rT$ for
$\|\cdot\|_{\rT}$ proves both boundedness and operator-norm continuous
dependence of $\rT\circ Q_\Phi$ on $\Phi$. To prove normality,
approximate $\Phi$ by finite sums of elementary tensors.
The corresponding maps after applying $\rT$ are finite sums of the
maps $C_{\varphi,\psi}$ and hence are normal. They converge in
operator norm, and bounded normal maps form an operator-norm closed
space. Thus $\rT\circ Q_\Phi$ is normal.

We now complete assertion~$(\rm ii)$. For each fixed $x$, the
prescription in assertion~$(\rm ii)$ defines an
$M$-valued tempered Schwartz--Bruhat distribution. This gives the
asserted linear map $\Delta$, and the same formula proves uniqueness
as well as all the continuity and normality statements.

We next record a domain observation that will also be used below. Let
$\phi$ be a faithful normal semifinite weight on $M$ and put
$\widetilde\phi=\phi\circ \rT$. If $a,b\in\mathfrak n_\phi$ and
$y\in\mathfrak m_{\rT}$, then
\begin{equation}\label{eq:composition-weight-domain}
b^*ya\in\mathfrak m_{\widetilde\phi},
\qquad
\widetilde\phi(b^*ya)=\phi\bigl(b^*\rT(y)a\bigr).
\end{equation}
Indeed, first let $y=c^*d$ with
$c,d\in\mathfrak n_{\rT}$. Both $cb$ and $da$ belong to
$\mathfrak n_{\widetilde\phi}$. For example,
\[
\widetilde\phi((da)^*da)
=\phi\bigl(a^*\rT(d^*d)a\bigr)
\leq\|\rT(d^*d)\|\phi(a^*a).
\]
The assertion follows by linearity from
$b^*c^*da=(cb)^*(da)$.

It remains to prove injectivity in assertion~$(\rm ii)$. The vectors
\[
(\lambda(\psi)y)\widetilde\phi^{1/2},
\qquad \psi\in\mathcal S(G),\quad y\in\mathfrak n_\phi,
\]
span a dense subspace of the dual-weight GNS space. The same is true
with $\lambda(\psi)$ replaced by $\lambda(\psi)^*$
\cite[Lemma~2.5(iv), Lemma~2.12, and Theorem~3.2(i)]{Ha78a}.
Suppose that $\Delta(x)=0$. For
$\varphi,\psi\in\mathcal S(G)$ and $v,w\in\mathfrak n_\phi$, put
\[
A=\lambda(\varphi)^*v,
\qquad
C=\lambda(\psi)w.
\]
Formula \eqref{eq:Haar-weight-Plancherel} and $M$-bimodularity give
$A,C\in\mathfrak n_{\widetilde\phi}$. Applying
\eqref{eq:composition-weight-domain} to
$y=\lambda(\varphi)x\lambda(\psi)$, $a=w$, and $b=v$ gives
\[
\begin{aligned}
\left\langle
x(C\widetilde\phi^{1/2}),
A\widetilde\phi^{1/2}\right\rangle
&=\widetilde\phi(A^*xC)\\
&=\phi\left(v^*\rT(\lambda(\varphi)x\lambda(\psi))w\right)=0.
\end{aligned}
\]
These matrix coefficients vanish for vectors in two dense subspaces.
Hence the left action of $x$ on $\rL^2(M\rtimes_\alpha G)$ vanishes,
so $x=0$.

We finally prove assertion~$(\rm iii)$. Suppose that
\[
x=\int_G\varphi(r)u_r\,\rd r
\qquad\text{with }\varphi\in\rC_c(G,M).
\]
First take $\Phi=f\otimes h$ with $f,h\in\mathcal D(G)$ and put
\[
b(t)=\int_G\varphi(r)h(t-r)\,\rd r.
\]
Then $b\in\rC_c(G,M)$ and Fubini's theorem gives
\[
x\lambda(h)=\int_G b(t)u_t\,\rd t.
\]
Both $\lambda(f)^*$ and $x\lambda(h)$ are integrated compactly
supported core elements. Thus the convolution-square formula in
\cite[Theorem~1.1(b)]{Ha78b}, followed by polarization, applies to
$Q_\Phi(x)=(\lambda(f)^*)^*x\lambda(h)$ and gives
\[
\begin{aligned}
\langle\Delta(x),\Phi\rangle
&=\int_G f(s)\alpha_s(b(-s))\,\rd s\\
&=\int_{G\times G}f(s)h(t)
\alpha_s\bigl(\varphi(-s-t)\bigr)\,\rd s\,\rd t.
\end{aligned}
\]
The same identity holds for finite sums of such tensors. The left
side is continuous in the Schwartz--Bruhat topology by the
$\mathfrak m_{\rT}$-norm estimate proved above. The right side is
bounded in norm by $\|\varphi\|_\infty\|\Phi\|_1$ and is therefore
continuous in the same topology. Since
$\mathcal D(G)\otimes\mathcal D(G)$ is dense in
$\mathcal S(G\times G)$, the identity extends to every Schwartz--Bruhat
test function.
This proves assertion~$(\rm iii)$.
\end{proof}

\subsubsection{Adapted coordinates and the joint spectrum}
The formula \eqref{eq:two-point-core-formula} suggests the following change of variables. Define
\[
\Theta:G\times G\longrightarrow G\times G,
\qquad
\Theta(s,t)=(s,-s-t),
\]
which is an involutive topological group automorphism since
$\Theta^2=\id$. Define
$\nabla(x)\in\mathcal S'(G\times G,M)$ by
\begin{equation}\label{eq:adapted-distribution}
\langle\nabla(x),\Phi\rangle
=\langle\Delta(x),\Phi\circ\Theta\rangle.
\end{equation}
For $\Phi\in\mathcal S(G\times G)$, the change of variables in \eqref{eq:adapted-distribution} gives
\begin{equation}\label{eq:adapted-direct-formula}
\langle\nabla(x),\Phi\rangle
=\rT\left(
\int_{G\times G}
\Phi(s,t)\widetilde\alpha_s(x)u_t^*\,\rd s\,\rd t
\right).
\end{equation}
Indeed, the operator inside $\rT$ is the integral in assertion~$(\rm i)$
of Theorem~\ref{two-sided coefficient distribution}, with $\Phi$
replaced by $\Phi\circ\Theta$, and hence belongs to $\mathfrak m_{\rT}$.

The map $x\mapsto\nabla(x)$ is injective, every evaluation is bounded
and normal, and the evaluations depend continuously on $\Phi$ in operator
norm. If $\varphi\in\rC_c(G,M)$ and
$$x=\int_G\varphi(t)u_t\,\rd t,$$ formula
\eqref{eq:two-point-core-formula} becomes
\[
\nabla(x)(s,t)=\alpha_s(\varphi(t)).
\]

For $s\in G$, $p\in\widehat G$, and
$\Phi\in\mathcal S(G\times G)$, the two commuting actions satisfy
\begin{align}
\langle\nabla(\widetilde\alpha_s(x)),\Phi\rangle
&=\left\langle\nabla(x),
\bigl[(r,t)\longmapsto\Phi(r-s,t)\bigr]\right\rangle,
\label{eq:adapted-extension-covariance}\\
\langle\nabla(\widehat\alpha_p(x)),\Phi\rangle
&=\left\langle\nabla(x),
\bigl[(r,t)\longmapsto
\langle p,t\rangle\Phi(r,t)\bigr]\right\rangle.
\label{eq:adapted-dual-covariance}
\end{align}
Both identities follow directly from
\eqref{eq:adapted-direct-formula}. Translations and multiplication by
characters preserve $\mathcal S(G\times G)$, so every smoothed
operator arising below belongs to $\mathfrak m_{\rT}$. The identity
\[
\widetilde\alpha_r(\widetilde\alpha_s(x))u_t^*
=\widetilde\alpha_{r+s}(x)u_t^*
\]
and translation in the $r$-variable give
\eqref{eq:adapted-extension-covariance}. For the second identity,
commutation of the actions and
$\widehat\alpha_p(u_t^*)=\overline{\langle p,t\rangle}u_t^*$ give
\[
\widetilde\alpha_r(\widehat\alpha_p(x))u_t^*
=\langle p,t\rangle\,
\widehat\alpha_p\bigl(\widetilde\alpha_r(x)u_t^*\bigr).
\]
Formula \eqref{eq:adapted-direct-formula} and dual invariance of $\rT$
now give \eqref{eq:adapted-dual-covariance}.

For $\Phi\in\mathcal S(\widehat G\times G)$, define the partial Fourier transform $\mathcal K_1\Phi \in \mathcal S(G \times G)$ by
\[
(\mathcal K_1\Phi)(s,t)
=\int_{\widehat G}\overline{\langle p,s\rangle}\Phi(p,t)\,\rd p
\]
and define
\[
\mathcal J:M\rtimes_\alpha G\longrightarrow
\mathcal S'(\widehat G\times G,M),
\qquad
\langle\mathcal J(x),\Phi\rangle
=\langle\nabla(x),\mathcal K_1\Phi\rangle.
\]
Thus $\mathcal J$ is obtained by Fourier transformation in the
$s$-variable.

\begin{proposition}\label{joint distribution spectrum}
For every $\Phi\in\mathcal S(\widehat G\times G)$, the map
$x\mapsto\langle\mathcal J(x),\Phi\rangle$ is bounded and normal.
The map $\mathcal J$ is injective and satisfies
\begin{equation}\label{eq:joint-distribution-covariance}
\mathcal J(\widetilde\alpha_s\widehat\alpha_q(x))
=\bigl[(p,t)\longmapsto
\langle p,s\rangle\langle q,t\rangle\bigr]\mathcal J(x)
\end{equation}
for $x\in M\rtimes_\alpha G$, $s\in G$, and
$q\in\widehat G$.
Moreover, one has
\[
\supp(\mathcal J(x))
=\Sp_{\widetilde\alpha\times\widehat\alpha}(x)
\subset\widehat G\times G
\qquad(x\in M\rtimes_\alpha G).
\]
\end{proposition}

\begin{proof}
Boundedness, normality and injectivity follow from Theorem
\ref{two-sided coefficient distribution}, because the coordinate
change and partial Fourier transformation are topological
isomorphisms of the test-function spaces. Formulas
\eqref{eq:adapted-extension-covariance} and
\eqref{eq:adapted-dual-covariance} give
\eqref{eq:joint-distribution-covariance}.
Proposition~\ref{equivariant distribution spectrum}, applied to the action
$\widetilde\alpha\times\widehat\alpha$ of
$G\times\widehat G$, now gives the asserted support equality.
\end{proof}

\subsubsection{An orbital refinement}

Let $H$ be a second countable locally compact abelian group and let
$\delta:H\curvearrowright M$ be a pointwise weak$^*$ continuous
action commuting with $\alpha$.
Write $\widetilde\delta$ for its extension to
$M\rtimes_\alpha G$ which fixes every
$u_t$. There is a unique distribution
\[
\nabla^\delta(x)\in\mathcal S'(G\times G\times H,M)
\]
such that
\begin{equation}\label{eq:orbital-refinement}
\langle\nabla^\delta(x),\Phi\otimes\psi\rangle
=\left\langle\nabla(\widetilde\delta_\psi(x)),\Phi\right\rangle
\end{equation}
for $\Phi\in\mathcal S(G\times G)$ and $\psi\in\mathcal S(H)$. Here
\[
\widetilde\delta_\psi(x)
=\int_H\psi(h)\widetilde\delta_h(x)\,\rd h.
\]
The kernel theorem and the estimate
$\|\widetilde\delta_\psi\|\leq\|\psi\|_1$ show that every evaluation
of $\nabla^\delta$ is bounded and normal, with operator norm depending
continuously on the test function.

For
$\Phi\in\mathcal S(\widehat G\times G\times\widehat H)$, put
\[
(\mathcal K_{1,3}\Phi)(s,t,h)
=\int_{\widehat G\times\widehat H}
\overline{\langle p,s\rangle\langle q,h\rangle}\Phi(p,t,q)\,\rd p\,\rd q
\]
and define
\[
\langle\mathcal J^\delta(x),\Phi\rangle
=\langle\nabla^\delta(x),\mathcal K_{1,3}\Phi\rangle.
\]

\begin{proposition}\label{orbital joint distribution spectrum}
The map $\mathcal J^\delta$ is injective, and all its evaluations are
bounded and normal. One has
\begin{equation}\label{eq:orbital-joint-covariance}
\mathcal J^\delta(
\widetilde\alpha_s\widehat\alpha_q\widetilde\delta_h(x))
=\bigl[(p,t,\xi)\longmapsto
\langle p,s\rangle\langle q,t\rangle\langle \xi,h\rangle\bigr]\mathcal J^\delta(x)
\end{equation}
for $x\in M\rtimes_\alpha G$, $s\in G$, $q\in\widehat G$, and
$h\in H$. Moreover,
\[
\supp(\mathcal J^\delta(x))
=\Sp_{\widetilde\alpha\times\widehat\alpha\times\widetilde\delta}(x)
\subset\widehat G\times G\times\widehat H
\qquad(x\in M\rtimes_\alpha G).
\]
\end{proposition}

\begin{proof}
The continuity and normality statements follow from the construction.
If $\mathcal J^\delta(x)=0$, then
$\nabla(\widetilde\delta_\psi(x))=0$ for every
$\psi\in\mathcal S(H)$. Injectivity of $\nabla$ gives
$\widetilde\delta_\psi(x)=0$. Applying an approximate identity chosen
in $\mathcal S(H)$ and using weak$^*$ convergence gives $x=0$.

The covariance formulas above and a change of variables in the
$H$-orbit give \eqref{eq:orbital-joint-covariance}.
Proposition~\ref{equivariant distribution spectrum}, applied to the action
$\widetilde\alpha\times\widehat\alpha\times\widetilde\delta$ of
$G\times\widehat G\times H$, now gives the asserted support equality.
\end{proof}

\subsection{The relative commutant theorems}
\label{sec:relative-commutant-theorems}

\subsubsection{Trace-scaling actions}
Suppose that $\tau$ is a faithful normal semifinite trace on $M$
satisfying $\tau\circ\alpha_s=e^{\ell(s)}\tau$ for every $s\in G$,
where $\ell:G\to\R$ is a continuous homomorphism. Consider first
$$x=\int_G \varphi(t)u_t\,\rd t\in
M'\cap(M\rtimes_\alpha G)$$ with
$\varphi\in\rC_c(G,M)$. Commutation with $M$ gives
$\varphi(t)\alpha_t(a)=a\varphi(t)$ for every $a\in M$ and $t\in G$.
For $a\in M\cap\rL^1(M,\tau)$, cyclicity and the scaling relation give
\[
\begin{aligned}
\tau(\varphi(t)\alpha_t(a))
&=\tau(a\varphi(t))
=\tau(\varphi(t)a)\\
&=e^{-\ell(t)}
\tau(\alpha_t(\varphi(t))\alpha_t(a)).
\end{aligned}
\]
Hence $\alpha_t(\varphi(t))=e^{\ell(t)}\varphi(t)$. Applying $\alpha_s$ gives
the translation relation
\[
e^{\ell(t)}\nabla(x)(s,t)=\nabla(x)(s+t,t).
\]
When $\ell=0$, this is the translation invariance that produces the
resonance relation. The following theorem extends the calculation to
an arbitrary element of the relative commutant.

\begin{theorem}\label{trace scaling resonance}
Let $\alpha:G\curvearrowright M$ be a continuous action of a locally
compact abelian group. Suppose that there are a continuous
homomorphism $\ell:G\to\R$ and a faithful normal semifinite trace
$\tau$ on $M$ satisfying $\tau\circ\alpha_s=e^{\ell(s)}\tau$ for
every $s\in G$.
Put $G_0=\ker\ell$. Then
\[
M'\cap(M\rtimes_\alpha G)
\subset M\rtimes_{\alpha|_{G_0}}G_0.
\]
Moreover,
\[
\Sp\left((\widetilde\alpha\times\widehat\alpha)
\big|_{M'\cap(M\rtimes_\alpha G)}\right)
\subset
\{(p,t)\in\widehat G\times G_0\mid \langle p,t\rangle=1\}.
\]
\end{theorem}

\begin{proof}
Take $x\in M'\cap(M\rtimes_\alpha G)$. For
$\varphi\in\mathcal S(G)$ and $\psi\in\mathcal D(G)$, put
\[
\begin{aligned}
X&=\int_G\widetilde\alpha_s(x)\varphi(s)\,\rd s,\\
Y&=\int_Gu_t^*\psi(t)\,\rd t,\\
Z&=\int_Gu_t^*e^{\ell(t)}\psi(t)\,\rd t.
\end{aligned}
\]
Let $\widetilde\tau=\tau\circ \rT$. The dual-weight modular formula
\cite[Theorem~3.2(ii)]{Ha78a} shows that
$Y$ is entire for $\sigma^{\widetilde\tau}$ and that
$\sigma_{\ri}^{\widetilde\tau}(Y)=Z$. A direct calculation identifies
$XY$ and $ZX$ with the two-sided integrals associated with the kernels
\[
\begin{aligned}
\Phi_1(s,t)&=\varphi(s)\psi(-s-t),\\
\Phi_2(s,t)&=e^{\ell(-s-t)}\psi(-s-t)\varphi(-t).
\end{aligned}
\]
Since $\psi$ has compact support, $e^\ell\psi\in\mathcal D(G)$.
Each kernel is obtained from a tensor product of functions in
$\mathcal S(G)$ by pullback under a topological group automorphism of
$G\times G$. Thus both kernels belong to $\mathcal S(G\times G)$, and
$XY,ZX\in\mathfrak m_{\rT}$ by
Theorem~\ref{two-sided coefficient distribution}. Since
$\mathfrak m_{\rT}$ is an
$M$-bimodule, formula \eqref{eq:composition-weight-domain} shows that
two-sided cutdowns by elements of $M\cap\rL^2(M,\tau)$ belong to the
domain of $\widetilde\tau$. Moreover,
$X\in M'\cap(M\rtimes_\alpha G)$, the algebra $M$
lies in the centralizer of $\widetilde\tau$, and
$\sigma_{\ri}^{\widetilde\tau}(Y)=Z$. Apply the KMS identity to $X$
and $a^{1/2}Ya^{1/2}$. For every
$a\in M_+\cap\rL^1(M,\tau)$,
\[
\begin{aligned}
\tau(\rT(XY)a)
&=\widetilde\tau(a^{1/2}XYa^{1/2})\\
&=\widetilde\tau\left(
\sigma_{\ri}^{\widetilde\tau}(a^{1/2}Ya^{1/2})X
\right)\\
&=\widetilde\tau\left(
a^{1/2}\sigma_{\ri}^{\widetilde\tau}(Y)a^{1/2}X
\right)\\
&=\widetilde\tau(a^{1/2}ZXa^{1/2})
=\tau(\rT(ZX)a).
\end{aligned}
\]
All the expressions involving $\widetilde\tau$ are in
$\mathfrak m_{\widetilde\tau}$ by
\eqref{eq:composition-weight-domain}, so this calculation does not use
an undefined one-sided value of the weight. The second equality is
obtained by commuting $X$ with the left copy of $a^{1/2}$ and applying
KMS to $X$ and $a^{1/2}Ya^{1/2}$. This latter element is entire and
the third equality uses that $M$ lies in the centralizer of
$\widetilde\tau$. Finally, $\sigma_{\ri}^{\widetilde\tau}(Y)=Z$, and
commuting $X$ with the right copy of $a^{1/2}$ gives the fourth
equality.
These functionals separate $M$, so $\rT(XY)=\rT(ZX)$.

The definition of the adapted distribution gives
\[
\rT(XY)=\langle\nabla(x),\varphi\otimes\psi\rangle.
\]
On the other hand,
\[
\begin{aligned}
ZX
&=\int_{G\times G}
u_t^*\widetilde\alpha_s(x)
e^{\ell(t)}\psi(t)\varphi(s)\,\rd s\,\rd t\\
&=\int_{G\times G}
\widetilde\alpha_{s-t}(x)u_t^*
e^{\ell(t)}\psi(t)\varphi(s)\,\rd s\,\rd t\\
&=\int_{G\times G}
\widetilde\alpha_s(x)u_t^*
e^{\ell(t)}\psi(t)\varphi(s+t)\,\rd s\,\rd t.
\end{aligned}
\]
Consequently,
\begin{equation}\label{eq:trace-scaling-shear}
\langle\nabla(x),\varphi\otimes\psi\rangle
=\left\langle\nabla(x),
\bigl[(s,t)\longmapsto
e^{\ell(t)}\varphi(s+t)\psi(t)\bigr]\right\rangle.
\end{equation}
We now pass to the partial Fourier transform. Let
$a\in\mathcal D(\widehat G)$ and $b\in\mathcal D(G)$, and put
$\varphi=\widehat a\in\mathcal S(G)$ and $\psi=b$. Here
$\widehat a$ is the Fourier transform on $\widehat G$. The identity
\[
\mathcal K_1\bigl(
[(p,t)\longmapsto e^{\ell(t)}\overline{\langle p,t\rangle} a(p)b(t)]
\bigr)(s,t)
=e^{\ell(t)}\varphi(s+t)\psi(t)
\]
shows that \eqref{eq:trace-scaling-shear} gives
\[
\langle\mathcal J(x),a\otimes b\rangle
=\left\langle\mathcal J(x),
\bigl[(p,t)\longmapsto
e^{\ell(t)}\overline{\langle p,t\rangle} a(p)b(t)\bigr]\right\rangle.
\]
Write
\[
m(p,t)=1-e^{\ell(t)}\overline{\langle p,t\rangle}
\qquad((p,t)\in\widehat G\times G).
\]
The elementary-tensor identity above says that
$\langle\mathcal J(x),m(a\otimes b)\rangle=0$. The function $m$ is
regular by \S\,\ref{subsec:schwartz-bruhat}, so multiplication by $m$
is continuous on $\mathcal D(\widehat G\times G)$. Density of finite
sums of elementary tensors therefore gives
\begin{equation}\label{eq:trace-scaling-multiplier}
\left\langle\mathcal J(x),m\Phi\right\rangle=0
\qquad\bigl(\Phi\in\mathcal D(\widehat G\times G)\bigr).
\end{equation}
Lemma~\ref{support of annihilated distribution}, applied to
\eqref{eq:trace-scaling-multiplier}, gives
$\supp(\mathcal J(x))\subset m^{-1}(\{0\})$. Proposition
\ref{joint distribution spectrum} gives
\[
\Sp_{\widetilde\alpha\times\widehat\alpha}(x)
\subset
\{(p,t)\in\widehat G\times G
\mid e^{\ell(t)}\overline{\langle p,t\rangle}=1\}.
\]
Since $|\langle p,t\rangle|=1$, the equality in the last set forces
$\ell(t)=0$ and $\langle p,t\rangle=1$.

Since $G_0$ is closed, the coordinate-restriction formula
\eqref{eq:coordinate-restriction} gives
$\Sp_{\widehat\alpha}(x)\subset G_0$. If
$t\in\Sp_{\widehat\alpha}(x)$ and $p\in G_0^\perp$, then
$t\in G_0$ and hence $\langle p,t\rangle=1$. Thus the restricted action has only
the trivial frequency on $x$, and \eqref{eq:one-point-spectrum} gives
$\widehat\alpha_p(x)=x$ for every $p\in G_0^\perp$. The dual
fixed-point theorem
\cite[Theorem ${\rm X}$.2.3]{Ta03} therefore
gives
$x\in(M\rtimes_\alpha G)^{G_0^\perp}
=M\rtimes_{\alpha|_{G_0}}G_0$.

Finally, $x$ was arbitrary in $M'\cap(M\rtimes_\alpha G)$, and the
set
\[
\{(p,t)\in\widehat G\times G_0\mid \langle p,t\rangle=1\}
\]
is closed. Taking the closure of the union of the individual spectra
proves the asserted spectrum inclusion for the restricted action.
\end{proof}

\begin{corollary}\label{resonance spectral inclusion}
Assume that $\alpha$ preserves a faithful normal semifinite trace on
$M$. Then
\[
\Sp\left((\widetilde\alpha\times\widehat\alpha)
\big|_{M'\cap(M\rtimes_\alpha G)}\right)
\subset
\{(p,t)\in\widehat G\times G\mid \langle p,t\rangle=1\}.
\]
\end{corollary}

\begin{proof}
Apply Theorem~\ref{trace scaling resonance} with $\ell=0$.
\end{proof}

\subsubsection{Invariant weights}

Suppose first that $\alpha$ preserves a faithful normal semifinite
weight $\phi$ and consider
$$x=\int_G\varphi(t)u_t\,\rd t\in
M'\cap(M\rtimes_\alpha G)$$ with
$\varphi\in\rC_c(G,M)$. Commutation with $M$ gives
$\varphi(t)\alpha_t(a)=a\varphi(t)$ for every $a\in M$ and $t\in G$.
Fix $t\in G$ and put $c=\varphi(t)$. Assume formally that $c$ is
entire for $\sigma^\phi$ and take $a\in M$ so that the weight values
below are defined. The KMS identity, the commutation relation applied
to $\alpha_{-t}(a)$, and invariance of $\phi$ give
\[
\phi\bigl(a\sigma_{-\ri}^\phi(c)\bigr)
=\phi(ca)=\phi(\alpha_{-t}(a)c)=\phi(a\alpha_t(c)).
\]
Thus, formally,
\[
\alpha_t(\varphi(t))=\sigma_{-\ri}^\phi(\varphi(t)).
\]
Since $\alpha$ commutes with $\sigma^\phi$ and
$\nabla(x)(s,t)=\alpha_s(\varphi(t))$, this becomes
\[
\nabla(x)(s+t,t)
=\sigma_{-\ri}^\phi(\nabla(x)(s,t)).
\]
We use the following modular smoothing and KMS facts to justify this
calculation.

\begin{lemma}\label{modular smoothing and KMS testing}
Let $N$ be a von Neumann algebra, let $\phi$ be a faithful normal
semifinite weight on $N$, and put $\sigma=\sigma^\phi$. The following
assertions hold.
\begin{enumerate}[\rm (i)]
\item Let $y\in N$ and let $\varphi\in\mathcal S(\R)$ have compactly
supported Fourier transform. Put
\[
\sigma_\varphi(y)=\int_\R \varphi(t)\sigma_t(y)\,\rd t.
\]
Then $\sigma_\varphi(y)$ is entire for $\sigma$. Fourier inversion
extends $\varphi$ to an entire function. For $z\in\C$, put
\[
\varphi_z(t)=\varphi(t-z)
\qquad(t\in\R).
\]
Then $z\mapsto\varphi_z$ is entire with values in $\mathcal S(\R)$ and
\begin{equation}\label{eq:analytic-modular-smoothing}
\sigma_z(\sigma_\varphi(y))
=\sigma_{\varphi_z}(y)
=\int_\R \varphi(t-z)\sigma_t(y)\,\rd t
\qquad(z\in\C).
\end{equation}
Moreover,
\begin{equation}\label{eq:complex-translate-Fourier-transform}
\widehat{\varphi_z}(q)=e^{-2\pi\ri qz}\widehat\varphi(q).
\end{equation}
\item If $a\in N$ is entire for $\sigma$ and
$c\in\mathfrak m_\phi$, then
\[
ac,\ c\sigma_{-\ri}(a)\in\mathfrak m_\phi,
\qquad
\phi(ac)=\phi\bigl(c\sigma_{-\ri}(a)\bigr).
\]
\end{enumerate}
\end{lemma}

\begin{proof}
We first prove assertion~$(\rm i)$. Let $K$ be the support of
$\widehat\varphi$. For $z\in\C$, the function
\[
q\longmapsto e^{-2\pi\ri qz}\widehat\varphi(q)
\]
has support in $K$, and its dependence on $z$ is entire in the
Fr\'echet space of smooth functions supported in $K$. Fourier inversion
and
\eqref{eq:complex-translate-Fourier-transform} show that
$z\mapsto\varphi_z$ is entire with values in $\mathcal S(\R)$. The estimate
\[
\left\|\int_\R h(t)\sigma_t(y)\,\rd t\right\|
\leq\|h\|_1\|y\|
\qquad(h\in\rL^1(\R))
\]
shows that $z\mapsto\sigma_{\varphi_z}(y)$ is norm entire. For
$r\in\R$, a
change of variables gives
\[
\sigma_r(\sigma_\varphi(y))=\sigma_{\varphi_r}(y),
\]
which proves \eqref{eq:analytic-modular-smoothing}.

We now prove assertion~$(\rm ii)$. Write $c$ as a finite sum of elements
$x^*y$ with $x,y\in\mathfrak n_\phi$. It is enough to consider
$c=x^*y$. For every entire element $d\in N$ and
$x\in\mathfrak n_\phi$, we have $xd\in\mathfrak n_\phi$ and
\[
(xd)\phi^{1/2}=(x\phi^{1/2})\sigma_{\ri/2}(d).
\]
Applied to $a^*$ and $\sigma_{-\ri}(a)$, this gives
\[
xa^*,y\sigma_{-\ri}(a)\in\mathfrak n_\phi.
\]
Hence
\[
ac=(xa^*)^*y,
\qquad
c\sigma_{-\ri}(a)=x^*\bigl(y\sigma_{-\ri}(a)\bigr)
\]
belong to $\mathfrak m_\phi$. Since
$\sigma_{\ri/2}(a^*)^*=\sigma_{-\ri/2}(a)$, we obtain
\[
\begin{aligned}
\phi(ac)
&=\langle y\phi^{1/2},(x\phi^{1/2})\sigma_{\ri/2}(a^*)\rangle\\
&=\langle(y\phi^{1/2})\sigma_{-\ri/2}(a),x\phi^{1/2}\rangle\\
&=\langle(y\sigma_{-\ri}(a))\phi^{1/2},x\phi^{1/2}\rangle
=\phi\bigl(c\sigma_{-\ri}(a)\bigr).
\end{aligned}
\]
\end{proof}

The following theorem extends the calculation to an arbitrary element
of the relative commutant.

\begin{theorem}\label{invariant weight resonance}
Let $\alpha:G\curvearrowright M$ be a continuous action of a locally
compact abelian group which preserves a faithful normal semifinite
weight $\phi$. Then
\[
M'\cap(M\rtimes_\alpha G)
\subset M_\phi\rtimes_\alpha G.
\]
Moreover, one has
\[
\Sp\left((\widetilde\alpha\times\widehat\alpha)
\big|_{M'\cap(M\rtimes_\alpha G)}\right)
\subset
\{(p,t)\in\widehat G\times G\mid \langle p,t\rangle=1\}.
\]
\end{theorem}

\begin{proof}
Take $x\in M'\cap(M\rtimes_\alpha G)$. Put
$\widetilde\phi=\phi\circ \rT$ and
$\sigma=\sigma^{\widetilde\phi}$. Invariance of $\phi$ implies that
$\alpha$ commutes with $\sigma^\phi$, and the dual-weight modular
formula \cite[Theorem~3.2(ii)]{Ha78a} shows that
$\sigma$ extends $\sigma^\phi$ and fixes the canonical unitaries. Thus
$M$ is globally invariant under $\sigma$, and $\widetilde\alpha$
commutes with $\sigma$.
For $\varphi\in\mathcal S(\R)$, put
\[
\sigma_\varphi(x)=\int_\R \varphi(r)\sigma_r(x)\,\rd r.
\]
The estimate $\|\sigma_\varphi(x)\|\leq\|\varphi\|_1\|x\|$ and the
operator-norm
continuity of the evaluations of $\nabla$ show that the map
$(F,\varphi)\mapsto\langle\nabla(\sigma_\varphi(x)),F\rangle$ is jointly
continuous on
$\mathcal S(G\times G)\times\mathcal S(\R)$. Hence the kernel theorem
gives a unique distribution
\[
\nabla^\sigma(x)\in\mathcal S'(G\times G\times\R,M)
\]
satisfying
\begin{equation}\label{eq:modular-orbital-distribution}
\langle\nabla^\sigma(x),F\otimes\varphi\rangle
=\langle\nabla(\sigma_\varphi(x)),F\rangle
\end{equation}
for $F\in\mathcal S(G\times G)$ and $\varphi\in\mathcal S(\R)$. This is
\eqref{eq:orbital-refinement} for $\delta=\sigma^\phi$.

Take $\xi,\eta\in\mathcal D(G)$ and $\varphi\in\mathcal S(\R)$ whose
Fourier transform is compactly supported. Put
\[
\begin{aligned}
x_\xi&=\int_G\xi(s)\widetilde\alpha_s(x)\,\rd s,\\
X&=\sigma_\varphi(x_\xi),\\
Y&=\int_Gu_t^*\eta(t)\,\rd t.
\end{aligned}
\]
The relative commutant is globally invariant under both
$\widetilde\alpha$ and $\sigma$, so $x_\xi$ belongs to it. Lemma
\ref{modular smoothing and KMS testing} shows that $X$ is entire for
$\sigma$. Put
\[
\psi=\varphi_{-\ri}=\varphi(\,\cdot+\ri),
\qquad
Z=\sigma_{-\ri}(X)=\sigma_\psi(x_\xi).
\]
Lemma~\ref{modular smoothing and KMS testing} also gives
$\psi\in\mathcal S(\R)$ and
\begin{equation}\label{eq:imaginary-translate-Fourier-transform}
\widehat\psi(q)=e^{-2\pi q}\widehat\varphi(q).
\end{equation}
We also have $X,Z\in M'\cap(M\rtimes_\alpha G)$.

Choose $\rho\in\mathcal S(G)$ with $\int_G\rho(s)\,\rd s=1$ and put
\[
\Phi_\eta(s,t)=\rho(s)\eta(-s-t).
\]
Then $\Phi_\eta\in\mathcal S(G\times G)$, and a change of variables
identifies $Y$ with the integral in assertion~$(\rm i)$ of Theorem
\ref{two-sided coefficient distribution} for $x=1$ and kernel
$\Phi_\eta$. Hence $Y\in\mathfrak m_{\rT}$.

Since $\widetilde\alpha$ commutes with $\sigma$, a direct calculation
gives
\[
\begin{aligned}
XY&=\int_{G\times G}
\widetilde\alpha_s(\sigma_\varphi(x))u_t^*
\xi(s)\eta(t)\,\rd s\,\rd t,\\
YZ&=\int_{G\times G}
\widetilde\alpha_s(\sigma_\psi(x))u_t^*
\xi(s+t)\eta(t)\,\rd s\,\rd t.
\end{aligned}
\]
Both kernels belong to $\mathcal D(G\times G)$, so
Theorem~\ref{two-sided coefficient distribution} gives
$XY,YZ\in\mathfrak m_{\rT}$.

Fix $a,b\in\mathfrak n_\phi$ and put
$C=b^*Ya$. Formula \eqref{eq:composition-weight-domain}, applied to
$Y$, $XY$, and $YZ$, gives
\[
C,\ b^*XYa,\ b^*YZa\in\mathfrak m_{\widetilde\phi}.
\]
Since $X$ and $Z$ commute with $M$, one has
$XC=b^*XYa$ and $CZ=b^*YZa$. Assertion~$(\rm ii)$ of Lemma
\ref{modular smoothing and KMS testing}, applied to $X$ and $C$, gives
the middle equality below. Formula \eqref{eq:composition-weight-domain}
identifies the endpoint values:
\[
\begin{aligned}
\phi(b^*\rT(XY)a)
&=\widetilde\phi(b^*XYa)
=\widetilde\phi(XC)\\
&=\widetilde\phi(CZ)
=\widetilde\phi(b^*YZa)
=\phi(b^*\rT(YZ)a).
\end{aligned}
\]
Since the functionals $d\mapsto\phi(b^*da)$, with
$a,b\in\mathfrak n_\phi$, separate $M$, we obtain
$\rT(XY)=\rT(YZ)$.

By \eqref{eq:modular-orbital-distribution} and the direct formula for
$\nabla$,
\[
\rT(XY)=\langle\nabla^\sigma(x),
\xi\otimes\eta\otimes\varphi\rangle.
\]
Moreover,
\[
\begin{aligned}
YZ
&=\int_{G\times G\times\R}
u_t^*\widetilde\alpha_s(\sigma_r(x))
\eta(t)\xi(s)\varphi(r+\ri)\,\rd t\,\rd s\,\rd r\\
&=\int_{G\times G\times\R}
\widetilde\alpha_s(\sigma_r(x))u_t^*
\eta(t)\xi(s+t)\varphi(r+\ri)\,\rd t\,\rd s\,\rd r.
\end{aligned}
\]
Consequently,
\begin{equation}\label{eq:modular-orbital-shear}
\langle\nabla^\sigma(x),\xi\otimes\eta\otimes\varphi\rangle
=\left\langle\nabla^\sigma(x),
\bigl[(s,t,r)\longmapsto
\xi(s+t)\eta(t)\varphi(r+\ri)\bigr]\right\rangle.
\end{equation}

Identify $\widehat\R$ with $\R$ through the characters
$r\mapsto e^{2\pi\ri qr}$. For
$\Phi\in\mathcal S(\widehat G\times G\times\R)$, put
\[
(\mathcal K_{1,3}\Phi)(s,t,r)
=\int_{\widehat G\times\R}
\overline{\langle p,s\rangle}e^{-2\pi\ri qr}\Phi(p,t,q)\,\rd p\,\rd q
\]
and define
\[
\langle\mathcal J^\sigma(x),\Phi\rangle
=\langle\nabla^\sigma(x),\mathcal K_{1,3}\Phi\rangle.
\]
For fixed $\eta$ and $\varphi$, both sides of
\eqref{eq:modular-orbital-shear} depend continuously on
$\xi\in\mathcal S(G)$. Density of $\mathcal D(G)$ in $\mathcal S(G)$
therefore extends that identity to every $\xi\in\mathcal S(G)$.
Let $a\in\mathcal D(\widehat G)$, $b\in\mathcal D(G)$, and
$c\in\mathcal D(\R)$, and take
$\xi=\widehat a$, $\eta=b$, and $\varphi=\widehat c$, where the
transforms are taken on $\widehat G$ and $\widehat\R$. Since
$\widehat\varphi(q)=c(-q)$, formula
\eqref{eq:imaginary-translate-Fourier-transform} gives
\[
\begin{aligned}
&\mathcal K_{1,3}\bigl(
[(p,t,q)\longmapsto \overline{\langle p,t\rangle}e^{2\pi q}a(p)b(t)c(q)]\bigr)(s,t,r)\\
&\hspace{35mm}=\widehat a(s+t)b(t)\widehat c(r+\ri).
\end{aligned}
\]
Equation \eqref{eq:modular-orbital-shear} therefore gives
\[
\left\langle\mathcal J^\sigma(x),
\bigl[(p,t,q)\longmapsto
(1-\overline{\langle p,t\rangle}e^{2\pi q})a(p)b(t)c(q)\bigr]\right\rangle=0.
\]
Put
\[
m(p,t,q)=1-\overline{\langle p,t\rangle}e^{2\pi q}
\qquad((p,t,q)\in\widehat G\times G\times\R).
\]
The function $m$ is regular by \S\,\ref{subsec:schwartz-bruhat}.
Multiplication by $m$ is therefore continuous on
$\mathcal D(\widehat G\times G\times\R)$.
Density of finite sums of elementary tensors gives
\[
m\mathcal J^\sigma(x)=0
\qquad\text{in }
\mathcal D'(\widehat G\times G\times\R,M).
\]
Lemma~\ref{support of annihilated distribution} gives
$\supp(\mathcal J^\sigma(x))\subset m^{-1}(\{0\})$. Proposition
\ref{orbital joint distribution spectrum} now gives
\[
\Sp_{\widetilde\alpha\times\widehat\alpha
\times\sigma}(x)
\subset
\{(p,t,q)\mid \overline{\langle p,t\rangle}e^{2\pi q}=1\}.
\]
The last equality forces $q=0$ and $\langle p,t\rangle=1$. Projection to the
$q$-coordinate, formula \eqref{eq:coordinate-restriction} and
\eqref{eq:one-point-spectrum} show that
$x\in(M\rtimes_\alpha G)_{\widetilde\phi}$. The fixed-point
description of crossed products
\cite[Corollary~${\rm X}$.1.22]{Ta03} gives
\[
(M\rtimes_\alpha G)_{\widetilde\phi}=M_\phi\rtimes_\alpha G.
\]
This proves the first assertion. Projection to the $(p,t)$-coordinates gives
\[
\Sp_{\widetilde\alpha\times\widehat\alpha}(x)
\subset
\{(p,t)\in\widehat G\times G\mid \langle p,t\rangle=1\}.
\]
Since $x\in M'\cap(M\rtimes_\alpha G)$ was arbitrary, taking the
closure of the union of these individual spectra gives the asserted
spectrum inclusion for the restricted action.
\end{proof}

\subsubsection{The relative commutant for flows}

Theorem~\ref{invariant weight resonance} and
Proposition~\ref{resonant group for R} give the following description
of the relative commutant and the spectrum of its joint action.
In particular, this proves Corollary~\ref{main relative commutant dichotomy}.

\begin{theorem}\label{relative commutant structure for flows}
Let $\alpha:\R\curvearrowright M$ be a flow on a factor $M$ that
preserves a faithful normal semifinite trace. The Arveson spectrum
of the joint action $\widetilde\alpha\times\widehat\alpha$ on
$M'\cap(M\rtimes_\alpha\R)$ is a closed subgroup of $\R^2$
contained in $\mathcal R_\R$. At least one of the following
assertions holds.
\begin{enumerate}[\rm (i)]
\item $M'\cap(M\rtimes_\alpha\R)=\cZ(M\rtimes_\alpha\R)$.
\item The joint action is almost periodic with discrete spectrum given by
\begin{equation}\label{flow relative commutant spectrum}
\Sp\left((\widetilde\alpha\times\widehat\alpha)
\big|_{M'\cap(M\rtimes_\alpha\R)}\right)
=\{(p,-t)\mid(p,t)\in\Lambda(\alpha)\}.
\end{equation}
Moreover,
\begin{equation}\label{flow relative commutant generators}
M'\cap(M\rtimes_\alpha\R)
=\{vu_{-t}\mid t\in\R,\ v\in\cU(M),\ \alpha_t=\Ad(v)\}''.
\end{equation}
For $(p,t)\in\Lambda(\alpha)$, the joint eigenspace of frequency
$(p,-t)$ is $\C vu_{-t}$, where $v$ implements $(p,t)$.
\end{enumerate}
If $\Gamma(\alpha)=\R$, assertion~\textup{(ii)} always holds and
both coordinate projections of $\Lambda(\alpha)$ are injective.
In particular, a flow with full Connes spectrum is outer if and
only if it is strictly outer.
\end{theorem}

\begin{proof}
Since $M$ is a factor, the fixed-point formulas
\eqref{eq:fixed-algebras} give
\[
\begin{aligned}
\bigl(M'\cap(M\rtimes_\alpha\R)\bigr)^{\widehat\alpha}
&=\C1,\\
\bigl(M'\cap(M\rtimes_\alpha\R)\bigr)^{\widetilde\alpha}
&=\cZ(M\rtimes_\alpha\R).
\end{aligned}
\]
Thus the joint action is ergodic. By
\eqref{eq:ergodic-spectrum-kernel}, its Arveson spectrum is a
closed subgroup of $\R^2$. Theorem~\ref{invariant weight resonance}
puts this subgroup in $\mathcal R_\R$.
Proposition~\ref{resonant group for R} therefore shows that the
joint spectrum is discrete or is contained in one of the coordinate
axes.

If it is contained in $\R\times\{0\}$, the dual action is trivial
on the relative commutant, so the first fixed-point formula gives
$M'\cap(M\rtimes_\alpha\R)=\C1$ and hence~\textup{(i)}.
If it is contained in $\{0\}\times\R$, the extended action is
trivial there, so the second fixed-point formula gives~\textup{(i)}.

If the joint spectrum is discrete, the discussion following
\eqref{eq:one-point-spectrum} shows that the action is almost
periodic and its joint eigenspaces generate
$M'\cap(M\rtimes_\alpha\R)$. Proposition~\ref{atomic coefficients}
identifies the joint point spectrum with
$\{(p,-t)\mid(p,t)\in\Lambda(\alpha)\}$ and gives the stated
one-dimensional eigenspaces. This proves
\eqref{flow relative commutant spectrum} and
\eqref{flow relative commutant generators}, hence~\textup{(ii)}.

Finally, suppose that $\Gamma(\alpha)=\R$.
Then $M\rtimes_\alpha\R$ is a factor, so both fixed-point
algebras above are $\C1$. If the joint spectrum lies in a
coordinate axis, the relative commutant is therefore $\C1$,
and its joint spectrum is $\{(0,0)\}$. Thus~\textup{(ii)} holds
in every case. The fixed-point formulas also exclude nonzero
joint eigenvalues on either coordinate axis. By
\eqref{flow relative commutant spectrum}, the same is true of
$\Lambda(\alpha)$, so its coordinate projections are injective.
If $\alpha$ is outer, the generation formula now gives
$M'\cap(M\rtimes_\alpha\R)=\C1$.
Conversely, if $t\ne0$ and $\alpha_t=\Ad(v)$, then $vu_{-t}$
is a nonscalar element of the relative commutant, since it has
nonzero dual frequency $-t$. Hence $\alpha$ is not strictly outer.
\end{proof}

We finish Subsection~\ref{sec:relative-commutant-theorems} by making precise the
relation with the Connes--Takesaki relative commutant theorem announced
in Section~\ref{sec:introduction}.

\begin{proposition}\label{modular flow specialization}
Let $P$ be a von Neumann algebra and let $\phi$ be a faithful normal
semifinite weight on $P$. Then
\[
P'\cap(P\rtimes_{\sigma^\phi}\R)
=\cZ(P\rtimes_{\sigma^\phi}\R).
\]
\end{proposition}

\begin{proof}
Let $(u_t)_{t\in\R}$ be the canonical unitaries in
$P\rtimes_{\sigma^\phi}\R$. Since the modular action preserves
$\phi$ and is trivial on $P_\phi$, Theorem
\ref{invariant weight resonance} gives
\[
P'\cap(P\rtimes_{\sigma^\phi}\R)
\subset P_\phi\rtimes_{\id}\R
\subset P\rtimes_{\sigma^\phi}\R.
\]
The middle algebra commutes with $u_s$ for every $s\in\R$. Every
element on the left also commutes with $P$. Since the crossed product
is generated by $P$ and the unitaries $(u_s)_{s\in\R}$, this gives
\[
P'\cap(P\rtimes_{\sigma^\phi}\R)
\subset\cZ(P\rtimes_{\sigma^\phi}\R).
\]
The reverse inclusion is immediate.
\end{proof}

If $P$ has separable predual and $\phi$ is integrable, we recover the
Connes--Takesaki equality
$P_\phi'\cap P=\cZ(P_\phi)$ as follows. First suppose that
$\phi$ is dominant. In the standard form of $P$, let $J$ be the modular
conjugation and $\Delta_\phi$ the modular operator. Takesaki duality
for the continuous decomposition associated with $\phi$
\cite[Chapter~${\rm II}$, Theorem~1.3]{CT77} identifies
$P\rtimes_{\sigma^\phi}\R$ with $(JP_\phi J)'$ through
$x\mapsto x$ and $u_t\mapsto\Delta_\phi^{\ri t}$. Thus
Proposition~\ref{modular flow specialization} gives
\[
J(P_\phi'\cap P)J
=P'\cap(JP_\phi J)'
=\cZ((JP_\phi J)')
=J\cZ(P_\phi)J.
\]
The integrable case follows because every integrable weight is,
up to isomorphism, a corner of a dominant weight by a projection in its
centralizer \cite[Chapter~${\rm II}$, Theorem~2.2]{CT77}, and the
relative commutant equality passes to such corners.

\section{Resonance for the joint bicentralizer action}
\label{sec:joint-bicentralizer-action}

For actions of locally compact abelian groups on $\II_1$ factors with full
Connes spectrum, we prove a resonance theorem for the joint action of the
original group and the dual bicentralizer action. If $\Sigma$ denotes its
spectrum, the relation $\langle p,t'\rangle\langle p',t\rangle=1$ holds for all
$(p,t),(p',t')\in\Sigma$. For flows with full Connes spectrum, the geometry
of this relation gives a discrete spectrum or a spectrum contained in
one coordinate axis. The eigenoperator and fixed point results then
identify the analytic and algebraic bicentralizers in all three cases.
We identify the bicentralizer as the algebra generated by the
unitaries implementing inner times. Its restricted flow is described
in Section~\ref{sec:classification-flows}, completing the proof of
Theorem~\ref{main flow bicentralizer}. In particular, outer flows
have trivial bicentralizer, which is the key input for the outer
case of Theorem~\ref{main flow conjecture}.

Subsection~\ref{subsec:fixed-point-interpretation} reformulates the
bicentralizer conjecture as a fixed point property inside the crossed
product. To address the lack of a normal tracial state, we apply
Theorem~\ref{fixed points relative dual bicentralizer compressed} in a GNS
representation built from a singular state. This technical argument,
found by GPT-6, gives the required commutation relations in a finite
tracial algebra. The Plancherel function then expresses these relations
on a Hilbert space, where comparing left and right multiplication yields
the resonance identity.

\subsection{Why should the resonance property hold?}
Let $\alpha:G\curvearrowright(M,\tau)$ be a continuous trace-preserving
action of a locally compact abelian group with
$\Gamma(\alpha)=\widehat G$. The bicentralizer conjecture predicts that
\[
\rB(M,\alpha)=\rb(M,\alpha).
\]
The following proposition identifies the joint spectrum on the algebraic
bicentralizer with the joint spectrum on the relative commutant studied
in the previous section, after reversing the dual coordinate. Thus the
bicentralizer conjecture would imply
the resonance property for the joint action on the analytic bicentralizer.

\begin{proposition}\label{spectrum identification}
Let
\[
\gamma=(\alpha\times\beta)|_{\rb(M,\alpha)}
:G\times\widehat G\curvearrowright\rb(M,\alpha)
\]
be the restricted joint bicentralizer action. Then
\[
\Sp(\gamma)
=\left\{(p,-t)\ \middle|
(p,t)\in\Sp\bigl((\widetilde\alpha\times\widehat\alpha)
|_{M'\cap(M\rtimes_\alpha G)}\bigr)\right\}.
\]
In particular, $\Sp(\gamma)\subset\mathcal R_G$.
\end{proposition}

\begin{proof}
Put $N=\rb(M,\alpha)\rtimes_\alpha G$.
Proposition~\ref{crossed product of bicentralizer}
gives
\[
N=\bigl((M'\cap(M\rtimes_\alpha G))\cup\{u_r\mid r\in G\}\bigr)''.
\]
Extend $\beta$ to $\rB(M,\alpha)\rtimes_\alpha G$ by fixing every
$u_r$, and denote this extension by $\widetilde\beta$. For a sufficiently
large cofinal ultrafilter $\omega$, Theorem~\ref{approximate cocycle vanishing theorem}
and Propositions~\ref{approximate coboundary ultrapower}
and~\ref{relative bicentralizer isomorphism} give, inside
$M_\alpha^\omega\rtimes_{\alpha^\omega}G$,
\[
\delta_p=\widetilde\beta_p\circ\widehat\alpha_p
=\Ad(v_p),
\qquad
\alpha_r^\omega(v_p)=\langle-p,r\rangle v_p,
\]
where $v_p\in\cU(M_\alpha^\omega)$. Since $v_p$ is represented by
unitaries in $M$, it commutes with $M'\cap(M\rtimes_\alpha G)$.
Hence $\delta_p$ fixes this relative commutant pointwise, and
\[
\widetilde\beta_p|_{M'\cap(M\rtimes_\alpha G)}
=\widehat\alpha_{-p}|_{M'\cap(M\rtimes_\alpha G)}.
\]
The generating formula for $N$ now shows that $\widetilde\beta$ preserves
$N$, and taking $\widehat\alpha$-fixed points shows that $\beta$ preserves
$\rb(M,\alpha)$.

On $N$, the three commuting actions $\widehat\alpha$,
$\widetilde\alpha$ and $\widetilde\beta$ define the joint action
\[
\sigma=\widehat\alpha\times\widetilde\alpha\times\widetilde\beta
:\widehat G\times G\times\widehat G\curvearrowright N.
\]
Its restriction to $\rb(M,\alpha)$ has spectrum
$\{0\}\times\Sp(\gamma)$. Each $u_r$ is a $\sigma$-eigenunitary
of frequency $(r,0,0)$. Thus, for every $x\in N$ and $r\in G$,
\[
\Sp_\sigma(xu_r)
=\{(s+r,p,t)\mid(s,p,t)\in\Sp_\sigma(x)\}.
\]
Since the products $xu_r$ with
$x\in\rb(M,\alpha)$ and $r\in G$ span a weak$^*$ dense subspace
of $N$, we obtain
\[
\Sp(\sigma)=G\times\Sp(\gamma).
\]
On $M'\cap(M\rtimes_\alpha G)$, the action $\widetilde\beta_p$
agrees with $\widehat\alpha_{-p}$. Hence the spectrum of the restriction of
$\sigma$ is the image of
\[
\Sp\bigl((\widetilde\alpha\times\widehat\alpha)
|_{M'\cap(M\rtimes_\alpha G)}\bigr)
\]
under $(p,t)\mapsto(t,p,-t)$. The products $xu_r$ with
$x\in M'\cap(M\rtimes_\alpha G)$ and $r\in G$ also span a
weak$^*$ dense subspace of $N$. Since $r$ ranges over all of $G$,
the same formula gives
\[
\Sp(\sigma)
=G\times\left\{(p,-t)\ \middle|
(p,t)\in\Sp\bigl((\widetilde\alpha\times\widehat\alpha)
|_{M'\cap(M\rtimes_\alpha G)}\bigr)\right\}.
\]
Comparing the two expressions for $\Sp(\sigma)$ proves the asserted
equality. The resonance inclusion follows from
Corollary~\ref{resonance spectral inclusion}, since $\mathcal R_G$
is invariant under $(p,t)\mapsto(p,-t)$.
\end{proof}

\subsection{A singular tracial state on the crossed product}
\label{subsec:singular-tracial-state}
Let $\alpha : G \curvearrowright (M,\tau)$ be a continuous trace-preserving action of a locally compact abelian group with $\Gamma(\alpha)=\widehat G$. We retain this assumption throughout the rest of this section.

We now address the obstruction described in
Subsection~\ref{subsec:fixed-point-interpretation} by constructing a
singular state on the crossed product. In
Subsection~\ref{subsec:bicentralizer-gns}, its restriction to a suitable
$\rC^*$-subalgebra will give a finite GNS algebra in which the relative
fixed point theorem applies.

We use the notation
\[
\lambda(\varphi)=\int_G\varphi(t)u_t\,\rd t
\qquad(\varphi\in\rL^1(G))
\]
and the natural isomorphism
\[
\widehat\lambda:\rL^\infty(\widehat G)\longrightarrow\rL(G)
\]
characterized by
\[
\widehat\lambda(\check\varphi)=\lambda(\varphi)
\qquad(\varphi\in\rL^1(G)).
\]
In particular, $\widehat\lambda(\psi)$ is the spectral multiplier of
$u$ corresponding to the convention $u_t(p)=\langle p,t\rangle$.

Put $\widetilde\tau=\tau\circ\rT$, the faithful normal semifinite
trace on $M\rtimes_\alpha G$.
Let $\rE_{\rL(G)}:M\rtimes_\alpha G\to\rL(G)$ be the canonical
faithful normal conditional expectation from
\eqref{canonical expectation onto group algebra}, determined by
\[
\rE_{\rL(G)}(au_t)=\tau(a)u_t
\qquad(a\in M,\ t\in G).
\]
We identify $\rL(G)$ with $\rL^\infty(\widehat G)$ through
$\widehat\lambda$ and write $\tau_G$ for Haar integration. Then
$\widetilde\tau=\tau_G\circ\rE_{\rL(G)}$.

Since $\widehat G$ is amenable, choose symmetric compact sets
$(F_i)_{i\in\mathcal I}$ with $0<\mu(F_i)<\infty$ such that
\begin{equation}\label{mean Folner sets}
\frac{\mu((F_i+K)\setminus F_i)}{\mu(F_i)}\longrightarrow0
\qquad(K\Subset\widehat G,\ 0\in K).
\end{equation}
Fix a cofinal ultrafilter $\mathcal V$ on $\mathcal I$ and put
\begin{equation}\label{mean invariant mean}
m_i(f)=\frac1{\mu(F_i)}\int_{F_i}f(p)\,\rd p,
\qquad
m(f)=\lim_{i\to\mathcal V}m_i(f)
\qquad(f\in\rL^\infty(\widehat G)).
\end{equation}
The F\o lner condition makes the state $m$ translation invariant.
Symmetry of $F_i$ makes it invariant under inversion.

Let $e_i=\widehat\lambda(1_{F_i})$. Define normal states $\rho_i$
and a state $\rho$ on $M\rtimes_\alpha G$ by
\[
\rho_i(x)=\frac{\widetilde\tau(e_i x e_i)}{\mu(F_i)}
=m_i\bigl(\rE_{\rL(G)}(x)\bigr),
\qquad
\rho(x)=m\bigl(\rE_{\rL(G)}(x)\bigr)
=\lim_{i\to\mathcal V}\rho_i(x).
\]
The second expression for $\rho_i$ follows from the bimodularity
of $\rE_{\rL(G)}$. We use the seminorm
$\|x\|_{2,\rho}=\rho(x^*x)^{1/2}$.

\begin{proposition}\label{mean finite tracial realization}
The state $\rho$ has the following properties.
\begin{enumerate}[\rm (i)]
\item The $\widetilde\alpha$-norm continuous part
$(M\rtimes_\alpha G)_c$ is contained in the centralizer of $\rho$.
\item The restriction of $\rho$ to $M$ is $\tau$, and $M$ is contained
in the centralizer of $\rho$.
\end{enumerate}
\end{proposition}

\begin{proof}
We first estimate the commutators of the spectral projections $e_i$.
Suppose that $a\in M\rtimes_\alpha G$ has compact
$\widetilde\alpha$-spectrum, contained in a symmetric compact set
$K\subset\widehat G$ with $0\in K$. Spectral transfer gives
\[
(1-e_i)ae_i
=\widehat\lambda(1_{(F_i+K)\setminus F_i})ae_i.
\]
Consequently,
\[
\|(1-e_i)ae_i\|_{2,\widetilde\tau}^2
\leq\|a\|^2\mu((F_i+K)\setminus F_i).
\]
The same bound applies to $a^*$. The two off-diagonal corners of
$[e_i,a]$ are orthogonal in $\rL^2(M\rtimes_\alpha G,\widetilde\tau)$,
so
\[
\|[e_i,a]\|_{2,\widetilde\tau}^2
\leq2\|a\|^2\mu((F_i+K)\setminus F_i).
\]
It follows that
\begin{equation}\label{mean cut commutators}
\frac{\|[e_i,a]\|_{2,\widetilde\tau}}{\sqrt{\mu(F_i)}}
\longrightarrow0.
\end{equation}
Elements of compact $\widetilde\alpha$-spectrum are norm dense in
$(M\rtimes_\alpha G)_c$. Since
\[
\|[e_i,a-c]\|_{2,\widetilde\tau}
\leq2\sqrt{\mu(F_i)}\,\|a-c\|,
\]
\eqref{mean cut commutators} holds for every $a$ in this continuous part.
Splitting $[e_i,a]$ into its off-diagonal corners and applying
Cauchy--Schwarz gives
\[
|\rho_i(az)-\rho_i(za)|
\leq
\frac{2\|z\|\,\|[e_i,a]\|_{2,\widetilde\tau}}
{\sqrt{\mu(F_i)}}
\qquad(z\in M\rtimes_\alpha G).
\]
Taking the ultralimit proves $(\rm i)$.

The formula for the expectation gives $\rho_i|_M=\tau$, hence
$\rho|_M=\tau$. For $a,c\in M$, traciality of $\widetilde\tau$
and this restriction formula give
\[
\|(a-c)e_i\|_{2,\widetilde\tau}
=\|e_i(a-c)\|_{2,\widetilde\tau}
=\sqrt{\mu(F_i)}\,\|a-c\|_2.
\]
Every $a\in M$ can be approximated in $\|\cdot\|_2$ by elements of
compact $\alpha$-spectrum. Thus \eqref{mean cut commutators} holds
for all $a\in M$, and the preceding estimate proves $(\rm ii)$.
\end{proof}

\subsection{The Plancherel function}
\label{subsec:plancherel-function}
We introduce the Plancherel function on the crossed product, adapted to the norm
$\|\cdot\|_{2,\rho}$. We take the inner product on $\rL^2(M)$
to be linear in the first variable. We identify
$\rL(G)$ with $\rL^\infty(\widehat G)$ through $\widehat\lambda$. We denote by $\rL^\infty(\widehat{G},\rL^2(M))$ the Banach space of bounded measurable functions from $\widehat{G}$ into $\rL^2(M)$. It is the dual of $\rL^1(\widehat{G},\rL^2(M))$ via the bilinear trace pairing.
Hilbert-valued functions are understood to be locally strongly
measurable. Two functions are identified when they agree almost
everywhere on every measurable set of finite Haar measure.

\begin{proposition}\label{mean coefficient map}
There exists a unique weak*-weak* continuous linear map
\[
\vartheta:M\rtimes_\alpha G\longrightarrow
\rL^\infty(\widehat G,\rL^2(M))
\]
such that
$$ \vartheta(a \widehat{\lambda}(\psi))=\psi \widehat{a}, \qquad (a \in M, \; \psi \in \rL^\infty(\widehat G)).$$
The map $\vartheta$ is injective and satisfies
\[
\|\vartheta(x)\|_\infty\leq\|x\|,
\qquad
\|x\|_{2,\rho}^2
=m\bigl(p\longmapsto\|\vartheta(x)(p)\|_2^2\bigr).
\]
If $x\in(M\rtimes_\alpha G)\cap
\rL^2(M\rtimes_\alpha G,\widetilde\tau)$, then
$\vartheta(x)\in\rL^2(\widehat G,\rL^2(M))$ and it is given by the usual
Plancherel identification.

The map $\vartheta$ satisfies the following identities
\[
\begin{aligned}
\vartheta(ax)&=a\vartheta(x)
&& (a\in M,\ x\in M\rtimes_\alpha G),\\
\vartheta(x\widehat\lambda(\psi))&=\psi\vartheta(x)
&& (\psi\in\rL^\infty(\widehat G),\ x\in M\rtimes_\alpha G),\\
\langle\vartheta(x),\vartheta(y)\rangle
&=\rE_{\rL(G)}(y^*x)
&& (x,y\in M\rtimes_\alpha G).
\end{aligned}
\]
Here the inner product of two functions is taken pointwise.
It also satisfies the covariance formula
\[
\vartheta(\widehat\alpha_q(x))(p)
=\vartheta(x)(p+q)
\qquad(q\in\widehat G,\ x\in M\rtimes_\alpha G).
\]
If a trace-preserving automorphism $\theta\in\operatorname{Aut}(M)$
commutes with $\alpha$ and $\widetilde\theta$ is its extension to
$M\rtimes_\alpha G$ that fixes $u^\alpha$, then
\[
\vartheta(\widetilde{\theta}(x))
=U^\theta\vartheta(x)
\qquad(x\in M\rtimes_\alpha G),
\]
where $U^\theta\widehat a=\widehat{\theta(a)}$ for $a\in M$.
If $B \subset M$ is a subalgebra that is globally $\alpha$-invariant and $x\in B \rtimes_{\alpha|_B} G$, then the Plancherel function of $x$ takes values in
$\rL^2(B)$.
\end{proposition}

\begin{proof}
For $\psi\in\rL^1(\widehat G)$ and $a\in M$, define
$\omega(\psi,a)\in (M\rtimes_\alpha G)_*$ by
\[
\omega(\psi,a)(x)=\int_{\widehat G}\psi(p)\rE_{\rL(G)}(ax)(p)\,\rd p.
\]
The conditional Cauchy--Schwarz inequality and
$\rE_{\rL(G)}(aa^*)=\tau(aa^*)1$ give
\[
\|\rE_{\rL(G)}(ax)\|_\infty
\leq\|\rE_{\rL(G)}(aa^*)\|_\infty^{1/2}\|\rE_{\rL(G)}(x^*x)\|_\infty^{1/2}
\leq\|a\|_2\|x\|.
\]
Hence
\[
|\omega(\psi,a)(x)|\leq\|\psi\|_1\|a\|_2\|x\|.
\]
Thus $\omega$ extends to a contractive bilinear map on
$\rL^1(\widehat G)\times \rL^2(M)$, and then to a contraction
\[
\omega:\rL^1(\widehat G)\widehat\otimes \rL^2(M)
=\rL^1(\widehat G,\rL^2(M))\longrightarrow (M\rtimes_\alpha G)_*.
\]
Define $\vartheta=\omega^*$ using the bilinear pairing fixed above.
It is weak*-weak* continuous and contractive, and it satisfies
\[
\int_{\widehat G}\psi(p)
\langle\vartheta(x)(p),\widehat{a^*}\rangle\,\rd p
=\int_{\widehat G}\psi(p)\rE_{\rL(G)}(ax)(p)\,\rd p.
\]
Since $\rE_{\rL(G)}(ab\widehat\lambda(\varphi))=\tau(ab)\varphi$, this gives
\[
\vartheta(b\widehat\lambda(\varphi))=\varphi\widehat b
\qquad(b\in M,\ \varphi\in\rL^\infty(\widehat G)).
\]
The linear span of these products is weak$^*$ dense in $M\rtimes_\alpha G$, since
it contains every $bu_t$. This also proves uniqueness.

For $x=a\widehat\lambda(\psi)$ and
$y=b\widehat\lambda(\varphi)$, bimodularity of $\rE_{\rL(G)}$ gives
\[
\langle\vartheta(x),\vartheta(y)\rangle
=\psi\overline\varphi\,\tau(b^*a)
=\rE_{\rL(G)}(y^*x).
\]
After integration against any function in $\rL^1(\widehat G)$,
both sides are separately weak* continuous in $x$ and $y$.
For the left-hand side, this follows because the product of an
$\rL^1$ scalar function and a bounded $\rL^2(M)$-valued function belongs to
$\rL^1(\widehat G,\rL^2(M))$. Weak$^*$ density therefore extends this
identity to all $x,y\in M\rtimes_\alpha G$. Taking $y=x$ gives
\[
\|\vartheta(x)(p)\|_2^2=\rE_{\rL(G)}(x^*x)(p)
\quad\text{almost everywhere}.
\]
Faithfulness of $\rE_{\rL(G)}$ proves injectivity of $\vartheta$, and applying
$m$ proves the asserted formula for $\|x\|_{2,\rho}$.
Integrating against Haar measure also gives
\[
\int_{\widehat G}\|\vartheta(x)(p)\|_2^2\,\rd p
=\widetilde\tau(x^*x).
\]
Thus $\vartheta$ is an isometry on $(M\rtimes_\alpha G)\cap\rL^2(M\rtimes_\alpha G,\widetilde\tau)$.
For $a\in M$ and $f\in\rL^1(G)\cap\rL^2(G)$, we have
$\vartheta(a\lambda(f))=\check f\,\widehat a$.
The linear span of these elements is dense in
$\rL^2(M\rtimes_\alpha G,\widetilde\tau)$, so this isometry agrees with the usual
Plancherel identification.

The two module identities follow by checking them on the products
$b\widehat\lambda(\varphi)$ and using weak* continuity.
The covariance and equivariance formulas follow in the same way from
\[
\widehat\alpha_q(b\widehat\lambda(\varphi))
=b\widehat\lambda\bigl(p\mapsto\varphi(p+q)\bigr),
\qquad
\widetilde\theta(b\widehat\lambda(\varphi))
=\theta(b)\widehat\lambda(\varphi).
\]
Finally, the span of $b\widehat\lambda(\varphi)$ with $b\in B$ is
weak$^*$ dense in $B\rtimes_{\alpha|_B}G$. Its image lies in the
weak* closed subspace $\rL^\infty(\widehat G,\rL^2(B))$, which
proves the last assertion.
\end{proof}

For all $a \in M$ and $x \in M \rtimes_\alpha G$, we have the left module property $\vartheta(ax)=a\vartheta(x)$. However, the right module property $\vartheta(xa)=\vartheta(x)a$ fails. But it is crucial for us to be able to compute $\vartheta(xa)$. This is achieved by the following proposition.

\begin{proposition}\label{mean coefficient right product}
Take $x\in M\rtimes_\alpha G$ and $a\in M$ such that
$\Sp_{\widehat\alpha}(x)$ and $\Sp_\alpha(a)$ are compact.
Choose $h\in\mathcal D(G)$ and $k\in\mathcal D(\widehat G)$
equal to $1$ on neighborhoods of these respective spectra, and put
\[
\Phi(t,r)=h(t)k(r)\langle r,t\rangle.
\]
Then $\Phi\in\mathcal D(G\times\widehat G)$ and
\[
\vartheta(xa)
=\left[p\longmapsto
\int_{\widehat G\times G}\widehat\Phi(q,s)
\vartheta(x)(p+q)\alpha_s(a)\,\rd q\,\rd s\right].
\]
The integral is a Bochner integral in $\rL^2(M)$ for every $p$.
\end{proposition}

\begin{proof}
Fix $a,h,k$ as in the statement.
The Schwartz--Bruhat calculus gives
$\Phi\in\mathcal D(G\times\widehat G)$ and
$\widehat\Phi\in\rL^1(\widehat G\times G)$.
Since $\|\vartheta(x)\|_\infty\leq\|x\|$, the proposed integral
converges in $\rL^2(M)$ and defines a bounded measurable function.
We first compute it for $x=cu_t$, then extend the calculation
by weak* continuity.

Let $c\in M$ and $t\in G$. By the defining formula for $\vartheta$,
\[
\vartheta(cu_t)(p)=\langle p,t\rangle\widehat c.
\]
The crossed-product relation $u_ta=\alpha_t(a)u_t$ gives
\[
cu_ta=c\alpha_t(a)u_t,
\qquad
\vartheta(cu_ta)(p)
=\langle p,t\rangle\widehat{c\alpha_t(a)}.
\]
We now compare this expression with the proposed integral.
Since $k=1$ near $\Sp_\alpha(a)$, we have
$\alpha_{\widehat k}(a)=a$. Fourier inversion therefore gives
\[
\begin{aligned}
\int_{\widehat G\times G}
\widehat\Phi(q,s)\langle q,t\rangle\alpha_s(a)\,\rd q\,\rd s
&=h(t)\int_G\widehat k(s-t)\alpha_s(a)\,\rd s\\
&=h(t)\alpha_t(\alpha_{\widehat k}(a))\\
&=h(t)\alpha_t(a).
\end{aligned}
\]
Using $\vartheta(cu_t)(p+q)=\langle p,t\rangle\langle q,t\rangle
\widehat c$, the proposed right-hand side becomes
\[
\begin{aligned}
\langle p,t\rangle\widehat c
\int_{\widehat G\times G}
\widehat\Phi(q,s)\langle q,t\rangle\alpha_s(a)\,\rd q\,\rd s
&=h(t)\langle p,t\rangle\widehat{c\alpha_t(a)}\\
&=h(t)\vartheta(cu_ta)(p).
\end{aligned}
\]

To extend this calculation with $h$ fixed, define the normal map
\[
x \mapsto x_h=\int_{\widehat G}\widehat h(q)\widehat\alpha_q(x)\,\rd q,
\]
where the integral is taken in the weak$^*$ sense. Since
$\widehat\alpha_q(cu_t)=\langle q,t\rangle cu_t$, Fourier
inversion gives $(cu_t)_h=h(t)cu_t$: this normal map multiplies
the component at dual frequency $t$ by $h(t)$.
Thus the proposed integral agrees with $\vartheta(x_ha)$ for
every $x=cu_t$.

The map $x\mapsto\vartheta(x_ha)$ is weak* continuous. So is
the map defined by the proposed integral. Indeed, pairing it with
$Z\in\rL^1(\widehat G,\rL^2(M))$ and changing variables
$r=p+q$ tests $\vartheta(x)(r)$ against the function
\[
r\longmapsto Z(r-q)\alpha_s(a^*).
\]
This function depends continuously on $(q,s)$ in
$\rL^1(\widehat G,\rL^2(M))$ and has norm at most
$\|Z\|_1\|a\|$. Multiplication by
$\overline{\widehat\Phi(q,s)}$ and integration in $(q,s)$
therefore give a single function in
$\rL^1(\widehat G,\rL^2(M))$. Fubini's theorem and weak*
continuity of $\vartheta$ show that the resulting functional of
$x$ is normal. The complex conjugate appears because the inner
product is linear in its first variable.

The linear span of the $cu_t$ is weak$^*$ dense in the crossed
product, so the proposed integral equals $\vartheta(x_ha)$ for
every $x\in M\rtimes_\alpha G$. For the $x$ in the statement,
$h=1$ near $\Sp_{\widehat\alpha}(x)$, hence $x_h=x$.
This proves the asserted formula for $\vartheta(xa)$.
\end{proof}

\subsection{The bicentralizer in the GNS representation}
\label{subsec:bicentralizer-gns}

Put $B=\rB(M,\alpha)$ and denote its dual bicentralizer action by
$\beta:\widehat G\curvearrowright B$. Put
\[
\widetilde B=B\rtimes_\alpha G\subset M\rtimes_\alpha G
\]
and write $\rT_B$ for the restriction of $\rT$ to $\widetilde B$.
Recall the action $\delta:\widehat G\curvearrowright\widetilde B$
from Subsection~\ref{subsec:fixed-point-interpretation}, and put
\[
B^\dagger=\widetilde B^\delta,
\qquad \alpha^\dagger_t=\Ad(u_t)|_{B^\dagger}.
\]
By Proposition~\ref{bicentralizer conjecture fixed point criterion}, the
bicentralizer conjecture is equivalent to
\[
B^\dagger=M'\cap(M\rtimes_\alpha G).
\]
The inclusion from right to left already holds. Thus the conjecture
amounts to showing that $B^\dagger$ commutes with $M$.

Our goal in this subsection is to establish commutation in the GNS
representation associated with $\rho$, namely
\[
\|ax-xa\|_{2,\rho}=0
\qquad(a\in M,\ x\in B^\dagger).
\]
This representation need not be faithful on the $\rC^*$-algebra generated
by $M$ and $B^\dagger$, so the identity does not directly imply commutation
in the original crossed product. It provides the relations from which
we will derive resonance in the next subsection.

We first construct a faithful normal tracial state on $B^\dagger$ using
a second crossed-product decomposition of $\widetilde B$.

\begin{proposition}\label{dagger decomposition}
There is a crossed product decomposition
\[
\widetilde B=B^\dagger\rtimes_{\alpha^\dagger}G
\]
and a faithful normal $\alpha^\dagger$-invariant tracial state
$\tau^\dagger$ on $B^\dagger$ such that
\[
\widetilde\tau|_{\widetilde B}
=\tau|_B\circ\rT_B=\tau^\dagger\circ\rT_{B^\dagger},
\]
where
\[
\rT_{B^\dagger}=\int_{\widehat G}\delta_p\,\rd p,
\qquad
\rE_{\rL(G)}(x)=\tau^\dagger(x)1
\quad(x\in B^\dagger).
\]
The dual action is $\widehat{\alpha^\dagger}_p=\delta_p$.
\end{proposition}

\begin{proof}
Our convention for the dual action gives
\[
\delta_p(u_t)=\langle p,t\rangle u_t.
\]
Thus $u_t$ normalizes $B^\dagger$. The von Neumann algebra
crossed-product characterization \cite[Theorem~8.1]{Na77}, applied
to $\widetilde B$, the action $\delta$, and the
representation $u$, gives the claimed decomposition. The resulting
isomorphism is the identity on $B^\dagger$ and sends the canonical
implementing unitaries to $u_t$. Its dual action is therefore
$\widehat{\alpha^\dagger}_p=\delta_p$.

Since $\beta$ preserves $\tau$, we have
\[
\rE_{\rL(G)}\circ\widetilde\beta_p=\rE_{\rL(G)}
\quad\text{on }\widetilde B.
\]
Hence $\rE_{\rL(G)}$ intertwines $\delta_p$ with the translation
$f(r)\mapsto f(r+p)$ on $\rL^\infty(\widehat G)$.
Its restriction to $B^\dagger$ is scalar and defines a faithful
normal state $\tau^\dagger$. The identity
$\rE_{\rL(G)}\circ\widetilde\alpha_t=\rE_{\rL(G)}$ gives
$\alpha^\dagger$-invariance.

For $x\in\widetilde B_+$, put $f=\rE_{\rL(G)}(x)$. Equivariance,
normality, and Haar invariance give
\[
\begin{aligned}
\tau^\dagger(\rT_{B^\dagger}(x))1
&=\int_{\widehat G}\rE_{\rL(G)}(\delta_p(x))\,\rd p\\
&=\left(\int_{\widehat G}f(p)\,\rd p\right)1.
\end{aligned}
\]
These are equalities of extended positive elements. Thus
\[
\tau^\dagger\circ\rT_{B^\dagger}
=\tau_G\circ\rE_{\rL(G)}=\widetilde\tau|_{\widetilde B}.
\]
The modular group of this dual weight restricts to the modular
group of $\tau^\dagger$ on $B^\dagger$
\cite[Theorem~3.2(ii)]{Ha78a}. Since the dual weight is a trace,
$\tau^\dagger$ is tracial.
\end{proof}

The Plancherel function has a particularly simple form on
$B^\dagger$.

\begin{proposition}\label{mean coefficient orbits}
For every $x\in B^\dagger$, there is a unique
$\xi_x\in\rL^2(B)$ such that
\[
\vartheta(x)=[p\longmapsto U^\beta_{-p}\xi_x].
\]
Moreover,
\[
\tau^\dagger(x^*x)=\|\xi_x\|_2^2.
\]
The map $\widehat x\mapsto\xi_x$ extends to a unitary
\[
V:\rL^2(B^\dagger,\tau^\dagger)
\longrightarrow\rL^2(B,\tau),
\qquad V\widehat1=\widehat1.
\]
It intertwines $U^{\alpha^\dagger}$ with $U^\alpha|_{\rL^2(B)}$.
\end{proposition}

\begin{proof}
For $x\in B^\dagger$, the covariance formulas give
\[
\vartheta(x)(p)
=\vartheta(\delta_q(x))(p)
=U^\beta_q\vartheta(x)(p+q)
\qquad(q\in\widehat G).
\]
Thus the function
$F(p)=U^\beta_p\vartheta(x)(p)$ is invariant under every
translation, hence it is constant. Write $\xi_x$ for its constant value. This also proves uniqueness.
The inner-product formula and
$\rE_{\rL(G)}(x^*x)=\tau^\dagger(x^*x)1$ give
$\|\xi_x\|_2^2=\tau^\dagger(x^*x)$, so $V$ extends to an isometry.

To prove that it is onto, use the two crossed-product decompositions
of $\widetilde B$. The vectors
$x\widehat\lambda(\psi)$, with $x\in B^\dagger$ and
$\psi\in\rL^2(\widehat G)\cap\rL^\infty(\widehat G)$,
span a dense subspace of
$\rL^2(\widetilde B,\widetilde\tau)$. Under the Plancherel
identification associated with $B\rtimes_\alpha G$, they become
\[
p\longmapsto\psi(p)U^\beta_{-p}\xi_x.
\]
Suppose that $\eta\in\rL^2(B)$ is orthogonal to the range of $V$,
and choose a nonzero
$f\in\rL^2(\widehat G)\cap\rL^\infty(\widehat G)$.
The function $p\mapsto f(p)U^\beta_{-p}\eta$ is orthogonal to every
function in the preceding dense subspace. It is therefore zero, so
$\eta=0$. The range of $V$ is closed and dense, proving surjectivity.
Finally, $\vartheta(1)=\widehat1$ gives $V\widehat1=\widehat1$.
The covariance under $\widetilde\alpha$ and the commutation of
$\alpha|_B$ with $\beta$ give the intertwining assertion.
\end{proof}

We now determine how the state $\rho$ behaves on $B^\dagger$.

\begin{proposition}\label{mean dagger centralizer}
We have $\rho|_{B^\dagger}=\tau^\dagger$, and $B^\dagger$ is contained
in the centralizer of $\rho$.
\end{proposition}

\begin{proof}
The scalar expectation formula gives
$\rho_i|_{B^\dagger}=\rho|_{B^\dagger}=\tau^\dagger$.
For $x,y\in B^\dagger$, it follows that
\[
\|(x-y)e_i\|_{2,\widetilde\tau}
=\|e_i(x-y)\|_{2,\widetilde\tau}
=\sqrt{\mu(F_i)}\,\|x-y\|_{2,\tau^\dagger}.
\]
Approximate $x$ in $\|\cdot\|_{2,\tau^\dagger}$ by elements of
compact $\alpha^\dagger$-spectrum and apply the proof of
Proposition~\ref{mean finite tracial realization}. This gives
\eqref{mean cut commutators} for $x$ and puts $x$ in the centralizer
of $\rho$.
\end{proof}

We will pass to a finite tracial von Neumann algebra using the GNS
representation of $\rho$ on the $\rC^*$-algebra generated by $M$ and
$B^\dagger$. To identify $B^\dagger$ with fixed points of the
relative dual bicentralizer action there, we first need uniform
control on the approximate implementing unitaries.

Corollary~\ref{dual unitary} provides, for each $q\in\widehat G$,
a net $(v_j)_j$ in $\cU(M)$ such that
\begin{equation}\label{mean scalar implementers}
\sup_{t\in K}\|\alpha_t(v_j)-\langle-q,t\rangle v_j\|_2
\longrightarrow0
\qquad(K\Subset G).
\end{equation}
Proposition~\ref{relative bicentralizer isomorphism} then gives
$v_jbv_j^*\to\beta_q(b)$ in $\|\cdot\|_2$ for every $b\in B$.

\begin{lemma}\label{mean uniform dual orbits}
Let $b\in B$ and $\varepsilon>0$. There are $K\Subset G$ and
$\varepsilon_0>0$, depending only on $b$ and $\varepsilon$, such that
for every $q\in\widehat G$ and $v\in\cU(M)$,
\[
\sup_{t\in K}\|\alpha_t(v)-\langle-q,t\rangle v\|_2<\varepsilon_0
\quad\Longrightarrow\quad
\sup_{p\in\widehat G}
\|v\beta_p(b)v^*-\beta_{p+q}(b)\|_2<\varepsilon.
\]
In particular, every net satisfying \eqref{mean scalar implementers}
satisfies
\begin{equation}\label{mean uniform scalar implementation}
\sup_{p\in\widehat G}
\|v_j\beta_p(b)v_j^*-\beta_{p+q}(b)\|_2\longrightarrow0.
\end{equation}
\end{lemma}

\begin{proof}
By \eqref{local relative action bicentralizer}, choose $K\Subset G$
and $\eta>0$ such that
\[
\sup_{t\in K}\|\alpha_t(u)-u\|_2<\eta
\quad\Longrightarrow\quad
\|[u,b]\|_2<\varepsilon/2
\qquad(u\in\cU(M)).
\]
Fix such a unitary $u$. For $p\in\widehat G$, choose unitaries
$(w_j)_j$ satisfying \eqref{mean scalar implementers} with $q=p$.
The scalar phases cancel under conjugation, so
\[
\limsup_j\sup_{t\in K}
\|\alpha_t(w_j^*uw_j)-w_j^*uw_j\|_2
\leq\sup_{t\in K}\|\alpha_t(u)-u\|_2<\eta.
\]
The choice of $\eta$ and the convergence
$w_jbw_j^*\to\beta_p(b)$ give
\[
\sup_{p\in\widehat G}\|[u,\beta_p(b)]\|_2\leq\varepsilon/2.
\]

Take $\varepsilon_0=\eta/2$ and let $v$ satisfy the hypothesis at $q$.
Choose $(w_j)_j$ satisfying \eqref{mean scalar implementers} for this
$q$. Then
\[
\begin{aligned}
\sup_{t\in K}\|\alpha_t(w_j^*v)-w_j^*v\|_2
&\leq\sup_{t\in K}\|\alpha_t(v)-\langle-q,t\rangle v\|_2\\
&\quad+\sup_{t\in K}\|\alpha_t(w_j)-\langle-q,t\rangle w_j\|_2
<\eta
\end{aligned}
\]
eventually. Applying the preceding uniform commutator estimate to
$w_j^*v$ gives
\[
\|v\beta_p(b)v^*-w_j\beta_p(b)w_j^*\|_2
\leq\varepsilon/2
\qquad(p\in\widehat G).
\]
For each fixed $p$, Proposition~\ref{relative bicentralizer isomorphism}
gives $w_j\beta_p(b)w_j^*\to\beta_{p+q}(b)$.
Passing to the limit gives the bound $\varepsilon/2$ for every $p$,
with the same constant. This proves the first assertion, and
\eqref{mean uniform scalar implementation} follows by applying it
to all sufficiently large indices of the given net.
\end{proof}

\begin{proposition}\label{mean scalar action on dagger}
For every $q\in\widehat G$, every $x\in B^\dagger$, and every net
satisfying \eqref{mean scalar implementers}, we have
\[
\|v_jxv_j^*-x\|_{2,\rho}\longrightarrow0.
\]
\end{proposition}

\begin{proof}
We use the Plancherel function and product formula proved in Propositions~\ref{mean coefficient map},
\ref{mean coefficient orbits}, and \ref{mean coefficient right product}.
Write
\[
F(p)=\vartheta(x)(p)=U^\beta_{-p}\xi_x.
\]
First suppose that $x$ has compact $\widehat\alpha$-spectrum.
Choose $h\in\mathcal D(G)$ equal to $1$ on a neighborhood of this
spectrum. Choose $f\in\mathcal S(G)$ with compactly supported
inverse Fourier transform and $\check f(q)=1$, and put
\[
a_j=\alpha_f(v_j^*).
\]
Then $\sup_j\|a_j\|\leq\|f\|_1$, all the $a_j$ have their
$\alpha$-spectrum in the compact set $\supp\check f$, and
\[
\|a_j-v_j^*\|_2\longrightarrow0.
\]
Indeed, $\alpha_t(v_j^*)-\langle q,t\rangle v_j^*\to0$ in $\|\cdot\|_2$
uniformly on compact sets, and $\int_G f(t)\langle q,t\rangle\,\rd t=1$.
Compact approximation of the $\rL^1$ integral justifies this limit
for nets. The same estimates show that
\begin{equation}\label{mean smoothed implementer error}
\sup_{s\in C}
\|\alpha_s(a_j)-\langle q,s\rangle a_j\|_2\longrightarrow0
\qquad(C\Subset G).
\end{equation}
Choose $k\in\mathcal D(\widehat G)$ equal to $1$ near
$\supp\check f$, so in particular $k(q)=1$.
Put $\Phi(t,r)=h(t)k(r)\langle r,t\rangle$.
The product formula for the Plancherel function gives
\[
\vartheta(xa_j)(p)
=\int_{\widehat G\times G}\widehat\Phi(r,s)
F(p+r)\alpha_s(a_j)\,\rd r\,\rd s.
\]
We claim that replacing $\alpha_s(a_j)$ in this integral by
$\langle q,s\rangle a_j$ changes the resulting function by a quantity
tending to zero in the essential supremum of its $\rL^2(M)$ norm.
To see this, fix $b\in B$ and put
$d_{j,s}=\alpha_s(a_j)-\langle q,s\rangle a_j$. Uniformly in $p,r,s$,
\[
\|F(p+r)d_{j,s}\|_2
\leq2\|f\|_1\|\xi_x-\widehat b\|_2
+\|b\|\,\|d_{j,s}\|_2.
\]
Since $\widehat\Phi\in\rL^1(\widehat G\times G)$, equation
\eqref{mean smoothed implementer error} and compact approximation
of this integral make the second term tend to zero after integration.
Then approximate $\xi_x$ by $\widehat b$ to prove the claim.

Fourier inversion in the second variable gives
\[
\int_G\widehat\Phi(r,s)\langle q,s\rangle\,\rd s
=\widehat h(r-q).
\]
The spectral cutoff property of $h$ gives
\[
\int_{\widehat G}\widehat h(r)F(p+r)\,\rd r=F(p).
\]
Thus the preceding claim gives
\[
\operatorname*{ess\,sup}_{p\in\widehat G}
\|\vartheta(xa_j)(p)-F(p+q)a_j\|_2\longrightarrow0.
\]
The same approximation of $\xi_x$ by vectors $\widehat b$, using
$\|a_j-v_j^*\|_2\to0$ and a uniform operator norm bound, gives
\[
\operatorname*{ess\,sup}_{p\in\widehat G}
\|F(p+q)(a_j-v_j^*)\|_2\longrightarrow0.
\]
On the other hand,
\[
\|x(a_j-v_j^*)\|_{2,\rho}
\leq\|x\|\,\|a_j-v_j^*\|_2\longrightarrow0.
\]
Combining these estimates with the norm identity for the Plancherel function yields
\[
\bigl\|\vartheta(xv_j^*)-[p\mapsto F(p+q)v_j^*]\bigr\|_m
\longrightarrow0.
\]
Here $\|\cdot\|_m$ denotes the square root of the mean of the squared
$\rL^2$ norm, as in the next subsection.
Lemma~\ref{mean uniform dual orbits} and approximation of $\xi_x$
by $\widehat b$ also give
\[
\sup_{p\in\widehat G}
\|v_jF(p+q)v_j^*-F(p)\|_2\longrightarrow0.
\]
Left $M$-linearity of $\vartheta$ now proves the assertion for $x$
of compact dual spectrum.

The dual action preserves $B^\dagger$ and its normal trace
$\tau^\dagger$. Hence any $x\in B^\dagger$ is a
$\|\cdot\|_{2,\tau^\dagger}$ limit of elements $y\in B^\dagger$
with compact dual spectrum. Since $v_j\in M$ belongs to the
centralizer of $\rho$,
\[
\|v_jxv_j^*-x\|_{2,\rho}
\leq2\|x-y\|_{2,\tau^\dagger}
+\|v_jyv_j^*-y\|_{2,\rho}.
\]
This proves the result for arbitrary $x$.
\end{proof}

\begin{theorem}\label{mean value commutation}
For every $a\in M$ and $x\in B^\dagger$, we have
\[
\|ax-xa\|_{2,\rho}=0.
\]
\end{theorem}

\begin{proof}
Let $\mathcal A$ be the $\rC^*$-algebra generated by $M$ and $B^\dagger$.
Both generating algebras lie in the centralizer of $\rho$, which
is a norm-closed $*$-subalgebra. Thus $\rho|_{\mathcal A}$ is tracial.
Its GNS representation $\pi$ generates a finite von Neumann algebra
$N=\pi(\mathcal A)''$ with a faithful normal tracial state $\tau_N$
satisfying $\tau_N(\pi(z))=\rho(z)$ for $z\in\mathcal A$.
The restrictions of $\pi$ to $M$ and $B^\dagger$ are faithful and
trace preserving. They are also normal. Indeed, if $0\leq b_i\uparrow b$
in $M$ and $z\in\mathcal A$, traciality gives
\[
\|\pi((b-b_i)z)\widehat1\|
\leq\|z\|\,\|b-b_i\|_2\longrightarrow0.
\]
The vectors $\pi(z)\widehat1$ are dense, which proves normality on
$M$. The same argument applies to $B^\dagger$ with $\tau^\dagger$.
Identify $M$ with its normal copy in $N$.

Proposition~\ref{mean scalar action on dagger}, with $q=0$, applies
to every net of asymptotically $\alpha$-invariant unitaries in $M$.
The local characterization \eqref{local relative action bicentralizer}
therefore gives
\[
\pi(B^\dagger)\subset\rB(M\subset N,\alpha).
\]
For arbitrary $q\in\widehat G$, the same proposition and the
approximate unitary formula in
Proposition~\ref{relative bicentralizer isomorphism} show that the
dual relative bicentralizer action fixes $\pi(B^\dagger)$ pointwise.
Theorem~\ref{fixed points relative dual bicentralizer compressed} yields
\[
\pi(B^\dagger)\subset M'\cap N.
\]
Consequently,
\[
\|ax-xa\|_{2,\rho}
=\|\pi(a)\pi(x)-\pi(x)\pi(a)\|_{2,\tau_N}=0.\qedhere
\]
\end{proof}

\subsection{Resonance for the joint bicentralizer action}
\label{subsec:joint-bicentralizer-resonance}
In this subsection, we suppose that $M$ is a factor. Then the action
$\beta:\widehat G\curvearrowright B$ is ergodic.
Let $\cH_m$ be the Hilbert space obtained by separation and completion
of $\rL^\infty(\widehat G,\rL^2(B))$ for the seminorm
\[
\|F\|_m^2=m\bigl(p\mapsto\|F(p)\|_2^2\bigr).
\]

\begin{lemma}\label{mean coefficient isometries}
The formulas
\[
L(\widehat b\otimes\widehat x)=b\vartheta(x),
\qquad
R(\widehat b\otimes\widehat x)=\vartheta(x)b
\qquad(b\in B,\ x\in B^\dagger)
\]
extend to linear isometries from
$\rL^2(B)\otimes\rL^2(B^\dagger)$ into $\cH_m$.
The formula
\[
I(F)(p)=U^\beta_{-p}\bigl(F(-p)\bigr)
\]
defines a linear unitary involution on $\cH_m$. If $S$ denotes the
flip on $\rL^2(B)\otimes\rL^2(B)$, then
\[
L(1\otimes V^*)S=IR(1\otimes V^*).
\]
\end{lemma}

\begin{proof}
First define maps on $B\odot B$ by
\[
L_0(\widehat b\otimes\widehat c)(p)=b\beta_{-p}(c),
\qquad
R_0(\widehat b\otimes\widehat c)(p)=\beta_{-p}(c)b.
\]
For $z\in B$, the mean ergodic theorem and ergodicity give
\[
\frac1{\mu(F_i)}\int_{F_i}\beta_{-p}(z)\,\rd p
\longrightarrow\tau(z)1
\quad\text{in }\rL^2(B).
\]
Consequently, for $b,c,d,e\in B$, traciality gives
\[
\begin{aligned}
\langle L_0(\widehat b\otimes\widehat c),
L_0(\widehat d\otimes\widehat e)\rangle_m
&=m\bigl(p\mapsto\tau(d^*b\beta_{-p}(ce^*))\bigr)\\
&=\tau(d^*b)\tau(e^*c),\\
\langle R_0(\widehat b\otimes\widehat c),
R_0(\widehat d\otimes\widehat e)\rangle_m
&=m\bigl(p\mapsto\tau(bd^*\beta_{-p}(e^*c))\bigr)\\
&=\tau(d^*b)\tau(e^*c).
\end{aligned}
\]
Thus $L_0$ and $R_0$ extend to isometries on
$\rL^2(B)\otimes\rL^2(B)$.
For $\xi\in\rL^2(B)$, approximation by vectors from $B$ and the bound
\[
\sup_p\|bU^\beta_{-p}(\xi)\|_2\leq\|b\|\|\xi\|_2
\]
show that
\[
L_0(\widehat b\otimes\xi)
=[p\mapsto bU^\beta_{-p}(\xi)],
\qquad
R_0(\widehat b\otimes\xi)
=[p\mapsto U^\beta_{-p}(\xi)b].
\]
Since $V\widehat x=\xi_x$ and $V$ is unitary, these are precisely
$L_0=L(1\otimes V^*)$ and $R_0=R(1\otimes V^*)$.
Inversion invariance of $m$ and unitarity of $U^\beta_{-p}$ show that
$I$ preserves the seminorm. Also $I^2=1$, so $I$ extends to a
unitary involution. Finally,
\[
IR_0(\widehat b\otimes\widehat c)(p)
=c\beta_{-p}(b)
=L_0S(\widehat b\otimes\widehat c)(p).
\]
Density proves the stated identity.
\end{proof}

The next lemma uses Theorem~\ref{mean value commutation}, which gives
\[
0=\|bx-xb\|_{2,\rho}
=\|\vartheta(bx)-\vartheta(xb)\|_m
\qquad(b\in B,\ x\in B^\dagger).
\]

Let $\mathbb{E}_\alpha$ and $\mathbb{E}_\beta$ be the spectral measures on $\rL^2(B)$
with the conventions
\[
U^\alpha_s=\int_{\widehat G}\langle p,s\rangle\,\rd \mathbb{E}_\alpha(p),
\qquad
U^\beta_q=\int_G\langle q,t\rangle\,\rd \mathbb{E}_\beta(t).
\]

\begin{lemma}\label{mean tensor commutation}
Let $\Omega$ be the unitary on
$\rL^2(B)\otimes\rL^2(B)$ defined by
\[
\Omega=\int_{\widehat G\times G}
\langle p,t\rangle\,\rd(\mathbb{E}_\alpha\otimes \mathbb{E}_\beta)(p,t),
\]
where $\mathbb{E}_\alpha$ acts on the first tensor factor and $\mathbb{E}_\beta$ on
the second. Then
\[
L(1\otimes V^*)=R(1\otimes V^*)\Omega^*.
\]
\end{lemma}

\begin{proof}
Use $L_0=L(1\otimes V^*)$ and $R_0=R(1\otimes V^*)$.
Take $b\in B$ with compact $\alpha$-spectrum and
$x\in B^\dagger$ with compact $\widehat\alpha$-spectrum.
The covariance of the Plancherel function gives
\[
\xi_{\widehat\alpha_q(x)}=U^\beta_{-q}\xi_x,
\]
so the $\beta$-spectrum of $\xi_x$ is contained in
$-\Sp_{\widehat\alpha}(x)$.
Choose $h\in\cD(G)$ and $k\in\cD(\widehat G)$ equal to $1$ near
$\Sp_{\widehat\alpha}(x)$ and $\Sp_\alpha(b)$, respectively, and put
\[
\Phi(t,p)=h(t)k(p)\langle p,t\rangle.
\]
Fourier inversion and spectral calculus give the Bochner integral
\[
\Omega^*(\widehat b\otimes\xi_x)
=\int_{\widehat G\times G}\widehat\Phi(q,s)
U^\alpha_s\widehat b\otimes U^\beta_{-q}\xi_x
\,\rd q\,\rd s.
\]
Applying $R_0$ and using the right multiplication formula for
$\vartheta$ yields
\[
\begin{aligned}
R_0\Omega^*(\widehat b\otimes\xi_x)
&=\left[p\mapsto
\int_{\widehat G\times G}\widehat\Phi(q,s)
U^\beta_{-p-q}\xi_x\,\alpha_s(b)
\,\rd q\,\rd s\right]\\
&=\vartheta(xb).
\end{aligned}
\]
These integrals converge uniformly in the $\rL^2(B)$ norm, since
the integrand norm is bounded by
$|\widehat\Phi(q,s)|\|b\|\|\xi_x\|_2$.
On the other hand,
$L_0(\widehat b\otimes\xi_x)=\vartheta(bx)$.
Theorem~\ref{mean value commutation} therefore gives equality of
these vectors in $\cH_m$.
The action $\widehat\alpha$ preserves $B^\dagger$, and smoothing
by functions whose Fourier transforms have compact support makes
such $x$ dense in $\rL^2(B^\dagger)$.
The analogous smoothing makes the allowed $b$ dense in $\rL^2(B)$.
Since $V$ is onto, the tensors under consideration span a dense
subspace. The identity follows by continuity.
\end{proof}

\begin{lemma} \label{resonance for bicentralizers}
Let $\alpha : G \curvearrowright M$ be a continuous action of a locally compact abelian group $G$ on a $\II_1$ factor $M$. Suppose that $\Gamma(\alpha)=\widehat{G}$. Let $$\Sigma =\Sp(\alpha|_B\times\beta)=\Gamma(\alpha|_B \times \beta).$$
For all $(p,t),(p',t')\in\Sigma$, we have
\begin{equation}\label{mean symmetric multiplier}
\langle p,t'\rangle\langle p',t\rangle=1.
\end{equation}
In particular, $\langle p,t\rangle^2= 1$ for all $(p,t) \in \Sigma$ and $\Sigma_0=\{ (p,t) \in \Sigma \mid \langle p,t\rangle=1\}$ is a subgroup of $\Sigma$ of index at most 2.
\end{lemma}

\begin{proof}
Put $L_0=L(1\otimes V^*)$ and $R_0=R(1\otimes V^*)$.
Lemmas~\ref{mean coefficient isometries} and~\ref{mean tensor commutation} give
$L_0S=IR_0$ and $L_0=R_0\Omega^*$.
Since $I^2=1$, the first identity also gives $IL_0=R_0S$.
Consequently,
\[
L_0S\Omega^*=IR_0\Omega^*=IL_0
=R_0S=L_0\Omega S.
\]
Injectivity of $L_0$ yields $S\Omega^*=\Omega S$, hence
\[
\Omega(S\Omega S)=1.
\]
Let $\mathbb{E}$ be the joint spectral measure of $U^\alpha\times U^\beta$
on $\rL^2(B)$ and put $\Sigma=\Sp(\alpha|_B\times\beta)$.
Faithfulness of $\tau$ identifies $\Sigma$ with the support of $\mathbb{E}$.
The operator $\Omega(S\Omega S)$ is the image of the continuous
function
\[
((p,t),(p',t'))\longmapsto \langle p,t'\rangle\langle p',t\rangle
\]
under the spectral calculus for $\mathbb{E}\otimes \mathbb{E}$.
The support of $\mathbb{E}\otimes \mathbb{E}$ is $\Sigma\times\Sigma$.
Its multiplier can therefore equal $1$ only if
$\langle p,t'\rangle\langle p',t\rangle=1$ at every point of $\Sigma\times\Sigma$.
Taking $(p',t')=(p,t)$ gives the last assertion.

By ergodicity of $\alpha\times\beta$ and
\eqref{eq:ergodic-spectrum-kernel}, $\Sigma$ is a closed subgroup of $\widehat{G} \times G$.
Equation~\eqref{mean symmetric multiplier} gives
\[
\langle p+p',t+t'\rangle=\langle p,t\rangle\langle p',t'\rangle,
\]
so $d(p,t)=\langle p,t\rangle$ is a continuous character
$d:\Sigma\to\{1,-1\}$. Its kernel $\Sigma_0$ lies in
$\mathcal R_G$ and has index at most two.
\end{proof}

\begin{theorem}\label{bicentralizer conjecture for flows}
Let $\alpha:\R\curvearrowright M$ be a flow on a $\II_1$ factor
with $\Gamma(\alpha)=\R$, and let $\beta$ be its dual bicentralizer
flow. Then the joint action $\alpha\times\beta$ on $\rB(M,\alpha)$ is almost
periodic, with point spectrum $\Lambda(\alpha)$
and we have
\[
\rB(M,\alpha)=\rb(M,\alpha)
=\{v\in\cU(M)\mid \alpha_t=\Ad(v)
  \text{ for some }t\in\R\}''.
\]
In particular, $\alpha$ is outer if and only if
$\rB(M,\alpha)=\C1$.
\end{theorem}

\begin{proof}
Lemma~\ref{resonance for bicentralizers} shows that
$\Sp(\alpha|_{\rB(M,\alpha)}\times\beta)$ is a closed subgroup
of $\R^2$ whose double is contained in $\mathcal R_\R$.
Multiplication by two is a homeomorphism of $\R^2$, so
Proposition~\ref{resonant group for R} implies that this spectrum
is discrete or is contained in one of the coordinate axes.

If it is contained in $\R\times\{0\}$, then $\beta$ is trivial
on $\rB(M,\alpha)$. Theorem
\ref{fixed points relative dual bicentralizer compressed} and
factoriality give $\rB(M,\alpha)=\C1$.
If it is contained in $\{0\}\times\R$, then $\alpha$ is trivial
on $\rB(M,\alpha)$. Theorem
\ref{fixed point analytic algebraic bicentralizer} and
Proposition~\ref{strong center bicentralizer} give
\[
\rB(M,\alpha)
=\rB(M,\alpha)\cap M^\alpha
=\rb(M,\alpha)\cap M^\alpha=\C1.
\]
In either case the joint action is almost periodic with spectrum
$\{(0,0)\}$. Otherwise its spectrum is discrete, so it is again
almost periodic. Theorem
\ref{LCA bicentralizer eigenvector intertwining compressed} therefore
implies $\rB(M,\alpha)\subset\rb(M,\alpha)$.
The reverse inclusion is \eqref{algebraic bicentralizer contained in analytic}.

Let $x\in\rB(M,\alpha)$ be a nonzero joint eigenoperator of
frequency $(p,t)$. By Theorem
\ref{LCA bicentralizer eigenvector intertwining compressed},
\[
xy=\alpha_t(y)x\qquad(y\in M).
\]
Factoriality makes $x$ a nonzero scalar multiple of a unitary $v$
implementing $\alpha_t$. Since
$\alpha_r(v)=e^{2\pi\ri pr}v$, we have $(p,t)\in\Lambda(\alpha)$.
Conversely, the same theorem shows that every implementing
unitary for $(p,t)\in\Lambda(\alpha)$ belongs to $\rB(M,\alpha)$
and satisfies
\[
\beta_q(v)=e^{2\pi\ri qt}v\qquad(q\in\R).
\]
This proves the point spectrum identity. Every unitary implementing
an inner time is an $\alpha$-eigenunitary, since
$\alpha_r(v)v^*$ is scalar and these scalars form a continuous
character. The joint eigenoperators generate $\rB(M,\alpha)$,
which proves the generation formula.
\end{proof}

\section{The classification of flows on the hyperfinite \texorpdfstring{$\II_1$}{II1} factor}
\label{sec:classification-flows}

In this section, we will prove Theorem~\ref{main flow conjecture}. We first describe
irrational rotation flows studied by Kawahigashi
\cite[Section~0]{Ka90} and we realize every possible invariant
$\Lambda$ in this family. These will be our model flows. We then classify flows with full Connes
spectrum on $R$, treating the outer, rank one and rank two cases
in turn. We write $\sim_{\mathrm{cc}}$ for cocycle conjugacy.

\subsection{The model actions: irrational rotation flows}
\label{sec:irrational-rotation-models}

Let $\theta\in\R\setminus\mathbb Q$. The irrational rotation
algebra $A_\theta$ is generated by unitaries $U,V$ with
$UV=e^{2\pi\ri\theta}VU$. Let $R_\theta$ be its von Neumann
completion for the canonical trace, characterized by
\[
\tau(U^mV^n)=0\qquad((m,n)\in\Z^2\setminus\{(0,0)\}).
\]
For $p,q\in\R$, the formulas
\begin{equation}\label{irrational rotation flow}
\sigma_r(U)=e^{2\pi\ri pr}U,
\qquad
\sigma_r(V)=e^{2\pi\ri qr}V
\qquad(r\in\R)
\end{equation}
define a trace-preserving flow on $R_\theta$. Thus the model is
specified by the two rotation frequencies $p,q$ and the irrational
angle $\theta$, taken modulo $\Z$.

\begin{proposition}\label{irrational rotation properties}
The following assertions hold for the flow in
\eqref{irrational rotation flow}.
\begin{enumerate}[\rm (i)]
\item $R_\theta$ is the hyperfinite $\II_1$ factor. The flow is
almost periodic, with point spectrum $p\Z+q\Z$ and Arveson spectrum
$\overline{p\Z+q\Z}$. It is ergodic if and only if $p,q$ are
linearly independent over $\mathbb Q$. In that case
$\Gamma(\sigma)=\R$. Otherwise $\Gamma(\sigma)=\{0\}$.
\item Its invariant is
\begin{equation}\label{irrational rotation invariant}
\Lambda(\sigma)=
\left\{(mp+nq,t)\ \middle|
\begin{array}{l}
 m,n\in\Z,\ t\in\R,\\
 pt+n\theta\in\Z,\\
 qt-m\theta\in\Z
\end{array}\right\}.
\end{equation}
The pair indexed by $m,n,t$ is implemented by $U^mV^n$.
\item Suppose that $p,q$ are linearly independent over
$\mathbb Q$. Then $\sigma$ is outer if and only if
\begin{equation}\label{irrational rotation outer criterion}
\frac pq\notin
\left\{\frac{a\theta+b}{c\theta+d}\ \middle|
 a,b,c,d\in\mathbb Q,\ ad-bc\ne0\right\}.
\end{equation}
In this case $\sigma$ has the Rokhlin property.
\item If $\Gamma(\sigma)=\R$, then
\[
\rB(R_\theta,\sigma)
=\{U^mV^n\mid (mp+nq,t)\in\Lambda(\sigma)
       \text{ for some }t\in\R\}''.
\]
\end{enumerate}
\end{proposition}

\begin{proof}
The identification of $R_\theta$ with the crossed product of
$\rL^\infty(\mathbb T)$ by an irrational rotation gives factoriality
and hyperfiniteness. The monomials $U^mV^n$ form an orthonormal
basis of $\rL^2(R_\theta)$ and have frequencies $mp+nq$. This proves
the spectral and ergodicity assertions. In the ergodic case the
point spectrum is dense, and the only nonzero invariant projection
is $1$, so the Connes spectrum is $\R$.

Suppose that $p,q$ are rationally dependent and not both zero.
Choose a primitive vector $(a,b)\in\Z^2$ with $ap+bq=0$, and
complete it to a basis $(a,b),(c,d)$ of $\Z^2$. Put
$W=U^aV^b$, $Z=U^cV^d$, and $\nu=cp+dq\ne0$. Then
$R_\theta^\sigma=\{W\}''$, and the spectral subspace of frequency
$k\nu$ is $\{W\}''Z^k$. For each $k\ne0$, the automorphism
$\Ad(Z^k)$ acts on $\{W\}''$ by a nontrivial irrational rotation.
Choose a nonzero projection $e\in\{W\}''$ with
$e\Ad(Z^k)(e)=0$. The $k\nu$ spectral subspace of the restriction
to $eR_\theta e$ is zero. Since the spectrum is contained in the
discrete group $\nu\Z$, this excludes $k\nu$ from the Connes
spectrum. Thus $\Gamma(\sigma)=\{0\}$. The case $p=q=0$ is immediate.

To compute the invariant, suppose that $\sigma_t=\Ad(X)$.
Expand $X$ in the monomial basis. The relations
\[
XU=e^{2\pi\ri pt}UX,
\qquad
XV=e^{2\pi\ri qt}VX
\]
show that every nonzero coefficient at $(m,n)$ satisfies
\[
e^{2\pi\ri pt}=e^{-2\pi\ri n\theta},
\qquad
e^{2\pi\ri qt}=e^{2\pi\ri m\theta}.
\]
Irrationality of $\theta$ allows at most one such pair $(m,n)$.
Hence $X$ is a scalar multiple of $U^mV^n$, with frequency $mp+nq$.
The converse follows from the same relations, proving
\eqref{irrational rotation invariant}.

Under rational independence, a nonzero inner time gives integers
$a,b,m,n$ such that
\[
\frac pq=\frac{a-n\theta}{b+m\theta}.
\]
The matrix of this fractional transformation is invertible, since
otherwise $p/q$ would be rational. Conversely, after clearing
denominators, any equality
$p/q=(a\theta+b)/(c\theta+d)$ with $ad-bc\ne0$ gives the nonzero
inner time $t=(c\theta+d)/q$, implemented by $U^cV^{-a}$.
This proves \eqref{irrational rotation outer criterion}, which is
Kawahigashi's characterization in \cite[Section~0]{Ka90}. The Rokhlin property follows
from Kishimoto \cite[Proposition~2.5]{Ki96}, as explained for the
von Neumann completion by Kawamuro \cite[Example~2.3]{Kw00}.
Finally, the bicentralizer formula is
Theorem~\ref{bicentralizer conjecture for flows} applied to the
implementers in \eqref{irrational rotation invariant}.
\end{proof}

We now specify the models used in the classification. The trivial
invariant is realized by
\[
\theta=\sqrt2,\qquad p=1,\qquad q=\sqrt3.
\]
Indeed, $p/q\notin\mathbb Q(\sqrt2)$, whereas the
$\operatorname{GL}(2,\mathbb Q)$ orbit of $\theta$ is contained in
$\mathbb Q(\sqrt2)$.

For a cyclic invariant, we choose the generator $(p,t)$ with $t>0$.
Thus let $p\in\R\setminus\{0\}$ and $t>0$ satisfy
\begin{equation}\label{rank one resonance}
pt\in\Z\setminus\{0\}.
\end{equation}
Fix once and for all an irrational number $\theta$ which is not a
root of a quadratic polynomial over $\mathbb Q$, and put
$q=\theta/t$. All the rank one models act on the same irrational
rotation factor $R_\theta$, generated by $U,V$ with
\begin{equation}\label{rank one torus model}
UV=e^{2\pi\ri\theta}VU,
\qquad
\tau(U^mV^n)=0\quad((m,n)\ne(0,0)),
\end{equation}
and are defined by
\[
\sigma^{p,t}_r(U)=e^{2\pi\ri pr}U,
\qquad
\sigma^{p,t}_r(V)=e^{2\pi\ri qr}V
\qquad\left(q=\frac\theta t\right).
\]
Since $q/p=\theta/(pt)\notin\mathbb Q$,
Proposition~\ref{irrational rotation properties}
shows that $\sigma^{p,t}$ is ergodic and has full Connes spectrum.

\begin{corollary}\label{rank one model properties}
The flow $\sigma^{p,t}$ satisfies
\[
\Lambda(\sigma^{p,t})=\Z(p,t).
\]
Its inner times are exactly $t\Z$, and its bicentralizer is
$\{U\}''\cong\rL^\infty(\mathbb T)$.
\end{corollary}

\begin{proof}
Write $pt=k$ with $k\in\Z\setminus\{0\}$.
The defining relations give $\sigma^{p,t}_t=\Ad(U)$.
Conversely, by \eqref{irrational rotation invariant}, an inner
time $h$ has an implementer $U^mV^n$ and satisfies
\[
ph+n\theta\in\Z,
\qquad
qh-m\theta\in\Z.
\]
Since $q=\theta/t$, the second relation gives
$h=mt+\ell t/\theta$ for some
$\ell\in\Z$. Substitution into the first gives
\[
\frac{k\ell}{\theta}+n\theta=j
\qquad\text{for some }j\in\Z.
\]
Thus $n\theta^2-j\theta+k\ell=0$. The hypothesis on $\theta$
and $k\ne0$ force $n=j=\ell=0$, so $h=mt$ and the implementer is a scalar
multiple of $U^m$. The invariant and bicentralizer formulas now
follow from Proposition~\ref{irrational rotation properties}.
\end{proof}

A lattice $\Lambda<\R^2$ is called \emph{resonant} if
$\Lambda\subset\mathcal R_\R$. It is called \emph{nondegenerate} if it
intersects each coordinate axis only at $(0,0)$. Let $\Lambda<\R^2$
be a nondegenerate resonant lattice and choose a basis $(p,t),(q,s)$.
Resonance gives
\begin{equation}\label{non outer lattice relations}
pt,\ qs,\ ps+qt\in\Z.
\end{equation}
Both coordinate projections of $\Lambda$ are injective, so $p,q$
are rationally independent. Since $pt\in\Z\setminus\{0\}$,
the angle $tq=(pt)q/p$ is irrational.

Define $R_\Lambda$ to be the tracial von Neumann algebra generated
by two unitaries $U,V$ satisfying
\[
UV=e^{2\pi\ri tq}VU,
\qquad
\tau(U^mV^n)=0\quad((m,n)\in\Z^2\setminus\{(0,0)\}).
\]
Thus $R_\Lambda$ is the tracial von Neumann completion of the
irrational rotation algebra with angle $tq$. The formulas
\[
\sigma^\Lambda_r(U)=e^{2\pi\ri pr}U,
\qquad
\sigma^\Lambda_r(V)=e^{2\pi\ri qr}V
\qquad(r\in\R)
\]
define a trace-preserving flow on $R_\Lambda$.
For $(m,n)\in\Z^2$, put
\[
U_{mp+nq,mt+ns}=U^mV^n.
\]
These unitaries form an orthonormal basis of $\rL^2(R_\Lambda)$,
and their products are scalar multiples of the unitaries labelled
by the corresponding sums. By \eqref{non outer lattice relations},
for all $(p',t'),(q',s')\in\Lambda$ they satisfy
\begin{equation}\label{non outer rotation model}
U_{p',t'}U_{q',s'}=e^{2\pi\ri t'q'}U_{q',s'}U_{p',t'},
\qquad
\sigma^\Lambda_r(U_{p',t'})=e^{2\pi\ri p'r}U_{p',t'}.
\end{equation}

The conjugacy class of this flow does not depend on the chosen basis.
Indeed, for another basis $(p',t'),(q',s')$, the unitaries
$U_{p',t'},U_{q',s'}$ generate $R_\Lambda$ and have the same
commutation relation, trace moments and flow eigenvalues as the
two generators in the construction for that basis.
Matching these generators gives a trace-preserving normal
isomorphism intertwining the flows. The same argument applies
if either generator is multiplied by a scalar of modulus one.
Proposition~\ref{irrational rotation properties} identifies
$R_\Lambda$ with $R$ and shows that $\sigma^\Lambda$ is ergodic
and has full Connes spectrum.

\begin{corollary}\label{non outer model properties}
The flow $\sigma^\Lambda$ satisfies
$\Lambda(\sigma^\Lambda)=\Lambda$. Its inner times form the dense
subgroup $\{t\in\R\mid(p,t)\in\Lambda\text{ for some }p\in\R\}$,
and $\rB(R_\Lambda,\sigma^\Lambda)=R_\Lambda$.
\end{corollary}

\begin{proof}
The commutation relations \eqref{non outer rotation model} give
\[
\Ad(U_{p,t})=\sigma^\Lambda_t
\qquad((p,t)\in\Lambda).
\]
Hence every $(p,t)\in\Lambda$ belongs to
$\Lambda(\sigma^\Lambda)$. Conversely, by the irrational rotation
description and Proposition~\ref{irrational rotation properties},
every inner time $h$ is implemented by a scalar multiple of
$U_{p,t}$ for some $(p,t)\in\Lambda$. Thus
$\sigma^\Lambda_{h-t}=\id$. The first coordinate projection of
$\Lambda$ is dense, so the flow is faithful and $h=t$.
This proves the invariant formula. The inner times are the second
coordinate projection of $\Lambda$, which is also dense.
The bicentralizer contains every $U_{p,t}$, so it is $R_\Lambda$.
\end{proof}

Consequently, every discrete subgroup of $\R^2$ contained in
$\mathcal R_\R$ and meeting both coordinate axes only at the origin
is realized by an irrational rotation flow with full Connes
spectrum. Indeed, such a subgroup is trivial, cyclic, or a lattice,
and the preceding constructions cover these three cases. By
Theorem~\ref{bicentralizer conjecture for flows}, these are exactly the
possible invariants of flows on $\II_1$ factors with full Connes
spectrum.

\subsection{The outer case}
\label{sec:outer-flows}

The bicentralizer theorem reduces the outer case to the classification
of Rokhlin flows by Masuda--Tomatsu.

\begin{theorem}\label{outer flow classification}
Let $\alpha:\R\curvearrowright R$ have full Connes spectrum.
The following assertions are equivalent.
\begin{enumerate}[\rm (i)]
\item $\alpha$ is outer.
\item $\alpha$ is strictly outer.
\item $\alpha$ has the Rokhlin property.
\item $\alpha$ is cocycle conjugate to $\sigma^\infty$.
\end{enumerate}
In particular, every outer flow with full Connes spectrum on $R$
is cocycle conjugate to the outer irrational rotation models of
Subsection~\ref{sec:irrational-rotation-models}.
\end{theorem}

\begin{proof}
Theorem~\ref{bicentralizer conjecture for flows} identifies outerness with
$\rB(R,\alpha)=\C1$. A flow with full Connes spectrum is faithful.
Since both $\R$ and $R$ are amenable, Theorem
\ref{Rokhlin and trivial bicentralizer} therefore gives
$(\rm i)\Longleftrightarrow(\rm iii)$.
Theorem~\ref{relative commutant structure for flows} gives
$(\rm i)\Longrightarrow(\rm ii)$. Conversely, an inner time
$t\neq0$ with implementer $v$ produces the nonscalar relative
commutant element $v^*u_t$, so $(\rm ii)\Longrightarrow(\rm i)$.

Masuda--Tomatsu's uniqueness theorem for Rokhlin flows on $R$
gives $(\rm iii)\Longleftrightarrow(\rm iv)$
\cite[Corollaries~7.9, 5.16 and 4.13]{MT16}.
The last assertion also follows from Kawahigashi's classification
of the outer irrational rotation flows \cite[Theorem~16]{Ka90}.
\end{proof}

\subsection{The rank one case}
\label{sec:rank-one-flows}

We now treat the rank one case, using the model $\sigma^{p,t}$
on the fixed irrational rotation factor $R_\theta$ from
Subsection~\ref{sec:irrational-rotation-models}. Its
bicentralizer is abelian.
Tensoring two flows with opposite implementing eigenvalues produces
a periodic flow. We prove that the original flow absorbs this
periodic flow by constructing an equivariant copy of its product
model in the central sequence algebra. Applying this to a triple
tensor product then gives the classification.

\begin{proposition}\label{rank one bicentralizer structure}
Let $\alpha:\R\curvearrowright M$ be a flow on a $\II_1$ factor
with $\Gamma(\alpha)=\R$ and $\Lambda(\alpha)=\Z(p,t)$, where
$t>0$. If $v\in\cU(M)$ satisfies $\alpha_t=\Ad(v)$, then $v$
is Haar and
\[
\rB(M,\alpha)=\{v\}''\cong\rL^\infty(\mathbb T),
\qquad \alpha_r(v)=e^{2\pi\ri pr}v\quad(r\in\R).
\]
\end{proposition}

\begin{proof}
The definition of $\Lambda(\alpha)$ and uniqueness of implementers
up to a scalar give the eigenvalue formula. Every inner time is
an integer multiple of $t$, and its implementers are scalar
multiples of powers of $v$. Theorem
\ref{bicentralizer conjecture for flows} therefore gives
$\rB(M,\alpha)=\{v\}''$ and $p\ne0$.
For $n\ne0$, trace invariance and
$\alpha_r(v^n)=e^{2\pi\ri npr}v^n$ imply $\tau(v^n)=0$.
Thus $v$ is Haar, proving the asserted isomorphism.
\end{proof}

\subsubsection{Tensor cancellation}

\begin{proposition}\label{rank one tensor cancellation}\label{rank one cancellation Connes spectrum}
Let $\alpha:\R\curvearrowright M$ and
$\beta:\R\curvearrowright N$ be flows on $\II_1$ factors with
full Connes spectrum such that
\[
\Lambda(\alpha)=\Z(p,t),
\qquad
\Lambda(\beta)=\Z(-p,t).
\]
Assume $t>0$. Then $\alpha\otimes\beta$ is cocycle equivalent to a periodic
flow whose least positive period is $t$, every time outside
$t\Z$ is outer for this periodic flow, and
\[
\Gamma(\alpha\otimes\beta)=\frac1t\Z.
\]
\end{proposition}

\begin{proof}
Choose implementing eigenunitaries $a\in\cU(M)$ and
$b\in\cU(N)$ at time $t$, with eigenvalues $p$ and $-p$.
For $\gamma=\alpha\otimes\beta$, the unitary $v=a\otimes b$
satisfies
\[
\gamma_t=\Ad(v),
\qquad
\gamma_r(v)=v\quad(r\in\R).
\]
Bounded Borel functional calculus gives a self-adjoint
$h\in(M\ovt N)^\gamma$ with $e^{2\pi\ri th}=v$.
Thus $z_r=e^{-2\pi\ri rh}$ is a $\gamma$-cocycle and the flow
$\delta=\gamma^z$ satisfies $\delta_t=\id$.

We verify that the inner times of $\gamma$ are exactly $t\Z$.
If $\gamma_s$ is implemented by a unitary $W$, a nonzero slice
of $W$ over $N$ gives $x\in M\setminus\{0\}$ such that
\[
xy=\alpha_s(y)x\qquad(y\in M).
\]
Factoriality implies that $x$ is a scalar multiple of a unitary,
so $\alpha_s$ is inner. Slicing over $M$ similarly shows that
$\beta_s$ is inner. The converse is immediate. Hence the inner
times of a tensor product are the intersection of the inner
times of its factors, which here is $t\Z$.
Cocycle perturbation preserves innerness at every time.
Every period of $\delta$ therefore belongs to $t\Z$, proving
the first two assertions.

We now compute the Connes spectrum of $\alpha\otimes\beta$.
Put $P=M\rtimes_\alpha\R$, $Q=N\rtimes_\beta\R$,
$C=M'\cap P$, and $D=N'\cap Q$.
Write $u_r\in P$ and $u'_r\in Q$ for the canonical unitaries.
The equivariant crossed-product identification for the closed
diagonal subgroup of $\R^2$ gives
\[
(M\ovt N)\rtimes_{\alpha\otimes\beta}\R
\cong
(P\ovt Q)^\eta,
\qquad
\eta_q=\widehat\alpha_q\otimes\widehat\beta_{-q}.
\]
Here $x\otimes y$ maps to itself and the canonical unitary at
$r$ maps to $u_r\otimes u'_r$.
Taking relative commutants gives
\[
(M\ovt N)'\cap
\bigl((M\ovt N)\rtimes_{\alpha\otimes\beta}\R\bigr)
\cong(C\ovt D)^\eta.
\]
Since $\alpha$ and $\beta$ have full Connes spectrum, Theorem
\ref{relative commutant structure for flows} applies in its discrete
case. Every inner implementer is a scalar multiple of a power of
$a$ or $b$, respectively, so
\eqref{flow relative commutant generators} gives
\[
C=\{c\}'',\quad c=au_{-t},
\qquad
D=\{d\}'',\quad d=bu'_{-t}.
\]
The dual actions rotate $c,d$ with the same frequency $-t$.
Their invariant tensor products are therefore generated by
$c\otimes d$, and
\[
(C\ovt D)^\eta=\{c\otimes d\}''.
\]
The extended diagonal flow fixes $c\otimes d$, because the
eigenvalues of $a,b$ cancel. Hence this entire relative commutant
is the center of the diagonal crossed product.
The dual action on this center rotates $c\otimes d$ with
frequency $-t$. Its kernel is $(1/t)\Z$.
The characterization of the Connes spectrum as the kernel of the
dual action on the center proves the formula.
\end{proof}

For $T>0$, let $\sigma^T$ be the product flow on
$R=\overline\bigotimes_{n\in\N}\mathbb M_2(\C)$ given by
\[
\sigma^T_r=\bigotimes_{n\in\N}
\Ad\begin{pmatrix}1&0\\0&e^{2\pi\ri r/T}\end{pmatrix}.
\]

For this model, $\ker(\sigma^T)=T\Z$ and
$\Gamma(\sigma^T)=(1/T)\Z$.
Indeed, tail matrix tensors give nonzero eigenoperators at every
frequency in $(1/T)\Z$ after compression by any nonzero
invariant projection.
We use the classification of prime circle actions on $R$
\cite[Theorem on p.~182 and Section~4]{Jo83}.
It says that a flow $\delta$ on $R$ satisfying
\[
\delta_T=\id,\qquad \Gamma(\delta)=\frac1T\Z
\]
is conjugate to $\sigma^T$. Indeed, the induced circle action
has full Connes spectrum, so its fixed-point algebra is a factor.
Jones's classification of such actions has a unique class for
the circle. This is the periodic classification needed here,
within the classification of flows with nonfull Connes spectrum
considered in \cite{Ka89}.

\begin{corollary}\label{rank one periodic cancellation model}
If $M$ and $N$ are hyperfinite in Proposition~\ref{rank one tensor cancellation}, then
\[
\alpha\otimes\beta\sim_{\mathrm{cc}}\sigma^t.
\]
\end{corollary}

\begin{proof}
The periodic perturbation in that proposition has period $t$.
Its Connes spectrum is $(1/t)\Z$ by
Proposition~\ref{rank one cancellation Connes spectrum} and cocycle
invariance. Apply the preceding circle classification.
\end{proof}

\subsubsection{The central sequence spectrum}

\begin{proposition}\label{rank one central clock}
Let $\alpha$ be a flow on a separable $\II_1$ factor $M$. Suppose
that $M_{\omega,\alpha}$ contains a unitary $U$ with
$\alpha_s^\omega(U)=e^{2\pi\ri ps}U$ for some $p\ne0$.
For every $\theta\in\operatorname{Aut}(M)$ commuting with $\alpha$,
$\theta^\omega$ is trivial on $M_{\omega,\alpha}$ if and only if
it is trivial on $M_\omega$.
\end{proposition}

\begin{proof}
We adapt the Rokhlin averaging construction of Kawamuro
\cite[Lemma~3.2]{Kw00} and Masuda--Tomatsu
\cite[Section~5.1, Lemma~5.2]{MT16} to detect central triviality
on the equicontinuous central sequence algebra.

Since $\theta$ commutes with $\alpha$, its ultrapower extension
preserves $M_{\omega,\alpha}$. Suppose that $\theta^\omega$ is
trivial on this algebra. In particular, $\theta^\omega(U)=U$.
Replacing $U$ by $U^*$ if necessary, assume $p>0$ and put
$S=1/p$.
The unitary $U$ is Haar and commutes with $M$. Factoriality gives
\[
\tau^\omega(f(U)x)=\tau^\omega(f(U))\tau(x)
\qquad(f\in\rL^\infty(\mathbb T),\ x\in M).
\]
Thus multiplication gives a trace-preserving normal embedding
\[
\Theta:\rL^\infty(\R/S\Z)\ovt M\longrightarrow M^\omega_\alpha,
\qquad \Theta(f\otimes x)=f(U)x,
\]
where $r\mapsto e^{2\pi\ri pr}$ corresponds to $U$.
Its range lies in the equicontinuous ultrapower, since it is
generated by $M$ and the equicontinuous element $U$.
This embedding intertwines $\alpha^\omega$ with
\[
(\gamma_sF)(r)=\alpha_s(F(r+s)),
\]
with the circle variable taken modulo $S$.
Since $\theta^\omega(U)=U$, it also satisfies
\[
\theta^\omega\circ\Theta
=\Theta\circ(\id\otimes\theta).
\]

Take a contraction $x=(x_n)^\omega\in M_\omega$ with
contractive representatives. On $[0,S)$ put
\[
F_n(r)=\alpha_{-r}(x_n),\qquad Y_n=\Theta(F_n).
\]
For $b\in M$, trace preservation gives
\[
\|[Y_n,b]\|_2^2
=\frac1S\int_0^S\|[x_n,\alpha_r(b)]\|_2^2\,\rd r
\longrightarrow_\omega0.
\]
Here centrality is uniform on the compact $\|\cdot\|_2$-orbit
of $b$. For $|s|\leq S$, the functions $\gamma_s(F_n)$ and
$F_n$ agree away from an interval of length $|s|$ where
translation crosses an endpoint. Therefore
\begin{equation}\label{rank one clock continuity}
\|\alpha_s^\omega(Y_n)-Y_n\|_2\leq2\sqrt{|s|/S}.
\end{equation}
Commutation of $\theta$ with $\alpha$ gives
\begin{equation}\label{rank one clock norm preservation}
\begin{split}
\|\theta^\omega(Y_n)-Y_n\|_2^2
&=\frac1S\int_0^S
\|\alpha_{-r}(\theta(x_n)-x_n)\|_2^2\,\rd r\\
&=\|\theta(x_n)-x_n\|_2^2.
\end{split}
\end{equation}

Choose equicontinuous contractive representatives
$Y_n=(y_{n,m})^\omega$. Fix a countable dense subset of the
unit ball of $M$ and numbers $\delta_j\downarrow0$.
For each $n$, Proposition~\ref{omega equicontinous functions}
allows us to choose $m(n)$ so that the first $n$ commutator
norms, the norm in \eqref{rank one clock norm preservation},
and the suprema over $|s|\leq\delta_j$, $j\leq n$, differ
from their ultrapower values by less than $1/n$.
Put $y_n=y_{n,m(n)}$.
The commutator estimates make $(y_n)$ central, and
\eqref{rank one clock continuity} makes it equicontinuous.
Thus $y=(y_n)^\omega\in M_{\omega,\alpha}$ and
\[
\|\theta^\omega(y)-y\|_2
=\|\theta^\omega(x)-x\|_2.
\]
Triviality on $M_{\omega,\alpha}$ therefore implies triviality
on $M_\omega$. The converse follows from the inclusion.
\end{proof}

\begin{proposition}\label{rank one central spectrum}
Let $\alpha:\R\curvearrowright M$ be a flow on a hyperfinite
$\II_1$ factor $M$ with separable predual. Assume that
$\Gamma(\alpha)=\R$ and that
$\Lambda(\alpha)=\Z(p,t)$, where $p\ne0$ and $t>0$. Then
\[
\ker(\alpha_\omega)=t\Z,
\qquad
\Sp(\alpha_\omega)=\frac1t\Z.
\]
Every element of this spectrum is the eigenvalue of a unitary in
$M_{\omega,\alpha}$. In particular, both $p$ and $1/t$ are
such eigenvalues.
\end{proposition}

\begin{proof}
We first obtain a unitary of frequency $p$ in $M_{\omega,\alpha}$.
Choose $v\in\cU(M)$ with $\alpha_t=\Ad(v)$ and
$\alpha_s(v)=e^{2\pi\ri ps}v$. Proposition
\ref{rank one bicentralizer structure} gives
$\rB(M,\alpha)=\{v\}''$, which is abelian. Multiplication by $v$
therefore defines a unitary $V$ on $\rL^2(\rB(M,\alpha))$
commuting with both $\rB(M,\alpha)$-actions and satisfying
\[
U_s^\alpha VU_s^{\alpha *}=e^{2\pi\ri ps}V
\qquad(s\in\R).
\]
Here $U_s^\alpha$ is restricted to $\rL^2(\rB(M,\alpha))$.
Thus, for the character $\chi_p(s)=e^{2\pi\ri ps}$, we have
\[
\chi_p\otimes_\R\cI_{\alpha|_{\rB(M,\alpha)}}
\cong\cI_{\alpha|_{\rB(M,\alpha)}}.
\]
Proposition~\ref{bicentralizer representation lifting} then gives
$\chi_p\otimes_\R\cI_\alpha\prec\cI_\alpha$, and
Theorem~\ref{Rokhlin representation criterion} yields
$p\in\Sp(\alpha_\omega)$.
By \cite[Theorem~7.8]{MT16}, every point of this spectrum has a
unitary eigenvector in $M_{\omega,\alpha}$. In particular, there
is a unitary $U\in M_{\omega,\alpha}$ with
\[
\alpha_s^\omega(U)=e^{2\pi\ri ps}U\qquad(s\in\R).
\]

We next identify the kernel. If $r\in\ker(\alpha_\omega)$,
then $\alpha_r^\omega$ is trivial on $M_{\omega,\alpha}$.
Since $p\ne0$, Proposition~\ref{rank one central clock}, applied
to $\theta=\alpha_r$, shows that it is trivial on $M_\omega$.
Centrally trivial automorphisms of $M$ are inner by
\cite[Theorem~3.2(1)]{Co77}, so $\alpha_r$ is inner.
The assumption $\Lambda(\alpha)=\Z(p,t)$ says that the inner
times are exactly $t\Z$, hence $r\in t\Z$.
Conversely, every inner time acts trivially on central sequences.
Therefore
\[
\ker(\alpha_\omega)=t\Z.
\]

Finally, we recover the spectrum from the kernel. Since every
spectral value admits a unitary eigenvector, products and adjoints
show that $\Sp(\alpha_\omega)$ is a closed subgroup of $\R$.
Its annihilator is $\ker(\alpha_\omega)$, so
\[
\Sp(\alpha_\omega)=(t\Z)^\perp=\frac1t\Z.
\]
\end{proof}

\subsubsection{Absorption of the periodic flow}

We use strong self-absorption in the sense of Szab\'o and Wouters
\cite[Definitions~3.1 and 3.2]{SW24}. For $\sigma^T$, we verify
this property by constructing an equivariant isomorphism onto
its tensor square that is approximated by unitary conjugates
of the first-factor embedding. The implementing unitaries will
be invariant under the tensor square flow.

\begin{proposition}\label{periodic model strong self absorption}
For every $T>0$, there are an equivariant isomorphism
\[
\Phi:(R,\sigma^T)\longrightarrow
(R\ovt R,\sigma^T\otimes\sigma^T)
\]
and unitaries $w_n\in(R\ovt R)^{\sigma^T\otimes\sigma^T}$ such that
\[
\|w_n(x\otimes1)w_n^*-\Phi(x)\|_2\longrightarrow0
\qquad(x\in R).
\]
In particular, $\sigma^T$ is strongly self-absorbing.
\end{proposition}

\begin{proof}
Label the matrix tensor factors of $R\ovt R$ by
$I=\N\times\{1,2\}$. A bijection $b:\N\to I$ induces
a trace-preserving normal isomorphism $\Phi:R\to R\ovt R$.
It is equivariant because every tensor factor carries the same
matrix flow. For each $n$, extend the map
$(k,1)\mapsto b(k)$ for $1\leq k\leq n$ to a permutation
of $I$ with finite support. This permutation is implemented by
an invariant unitary $w_n$. Indeed, each transposition is
implemented by the flip unitary
$\sum_{i,j=1}^2e_{ij}\otimes e_{ji}$ in the two relevant
matrix factors, and this unitary is fixed by the diagonal
matrix flow. By construction,
\[
w_n(x\otimes1)w_n^*=\Phi(x)
\qquad\left(x\in\bigotimes_{k=1}^n\mathbb M_2(\C)\right).
\]
The claimed convergence follows by $\|\cdot\|_2$-density of
the finite tensor products in $R$.
\end{proof}

\begin{proposition}\label{rank one periodic absorption}
Let $\alpha:\R\curvearrowright M$ be a flow on a $\II_1$ factor
with separable predual, let $T>0$, and fix a free ultrafilter
$\omega$ on $\N$. The following conditions are equivalent.
\begin{enumerate}[\rm (i)]
\item $M$ is McDuff and there is a unitary
$u\in M_{\omega,\alpha}$ such that
\[
\alpha_r^\omega(u)=e^{2\pi\ri r/T}u\qquad(r\in\R).
\]
\item $\alpha\sim_{\mathrm{cc}}\alpha\otimes\sigma^T$.
\item $\alpha\sim_{\mathrm{cc}}\alpha\otimes\rho$ for every $T$-periodic
flow $\rho:\R\curvearrowright R$.
\end{enumerate}
\end{proposition}

\begin{proof}
The implication \textup{(iii)$\Rightarrow$(ii)} follows by taking
$\rho=\sigma^T$.

Assume~\textup{(ii)}. The underlying isomorphism
$M\cong M\ovt R$ shows that $M$ is McDuff. Since
$1/T\in\Gamma(\sigma^T)$, Corollary
\ref{Connes spectrum approximate eigenunitaries} gives unitaries
$v_n\in R$ satisfying
\[
\sup_{|r|\leq n}
\|\sigma^T_r(v_n)-e^{2\pi\ri r/T}v_n\|_2<\frac1n.
\]
Shift $v_n$ into the tensor factors of $R$ beyond the first $n$.
These shifts are equivariant, so the estimates remain valid and
the resulting unitaries form a central sequence. Hence they define
a unitary $1/T$-eigenvector in $R_{\omega,\sigma^T}$, and tensoring
with $1$ gives such an eigenvector in
$(M\ovt R)_{\omega,\alpha\otimes\sigma^T}$.
Transport it through the cocycle conjugacy to obtain the required
unitary in $M_{\omega,\alpha}$. This is legitimate because cocycle
perturbations leave the equicontinuous central sequence algebra
and its induced action unchanged. Thus~\textup{(i)} holds.

Assume~\textup{(i)}. We first construct invariant central matrix
units to prove~\textup{(ii)}. The unitary $u$
is Haar, because its nonzero powers have nonzero eigenvalues.
It commutes with the constant copy of $M$. Factoriality gives
\[
\tau^\omega(f(u)x)=\tau^\omega(f(u))\tau(x)
\qquad(f\in\rL^\infty(\mathbb T),\ x\in M).
\]
Consequently, functional calculus and multiplication give a
trace-preserving normal embedding
\[
\Theta:\rL^\infty(\R/T\Z)\ovt M
\longrightarrow M^\omega_\alpha,
\qquad
\Theta(f\otimes x)=f(u)x.
\]
Here the coordinate function $s\mapsto e^{2\pi\ri s/T}$ maps
to $u$. The map is equivariant for the action
\[
\eta_r(F)(s)=\alpha_r(F(s+r)),
\]
where the variable $s$ is taken modulo $T$.

Apply \cite[Theorem~5.5]{SW24} to the action of $\Z$ generated
by $\alpha_T$. Since $M$ is McDuff, this action absorbs the
trivial action on $R$. Pulling back matrix units from successive
tensor factors of $R$ through the cocycle conjugacy gives unital
matrix units $(e_{ij}^{(n)})_{i,j=1}^2$ in $M$ forming central
sequences and satisfying
\[
\varepsilon_n=
\max_{i,j}\|\alpha_T(e_{ij}^{(n)})-e_{ij}^{(n)}\|_2
\longrightarrow0.
\]
Passing to a subsequence, assume that $\varepsilon_n\leq n^{-3}$.
Define $F_{ij}^{(n)}\in\rL^\infty(\R/T\Z)\ovt M$
using the fundamental interval $[0,T)$ by
\[
F_{ij}^{(n)}(s)=\alpha_{-s}(e_{ij}^{(n)}),
\qquad
Y_{ij}^{(n)}=\Theta(F_{ij}^{(n)}).
\]
For each $n$, the elements $Y_{ij}^{(n)}$ are unital matrix units.
If $|r|\leq L$, the difference between $s+r$ and its representative
in $[0,T)$ is $mT$ with $|m|\leq L/T+1$. Telescoping the
$\alpha_T$-error gives
\[
\sup_{|r|\leq L}
\|\alpha_r^\omega(Y_{ij}^{(n)})-Y_{ij}^{(n)}\|_2
\leq(L/T+1)\varepsilon_n.
\]
For every $x\in M$, trace preservation gives
\[
\|[Y_{ij}^{(n)},x]\|_2^2
=\frac1T\int_0^T
\|[e_{ij}^{(n)},\alpha_s(x)]\|_2^2\,\rd s
\longrightarrow0.
\]
The convergence is uniform in $s$, because the orbit of $x$ on
$[0,T]$ is compact in $\|\cdot\|_2$.
A diagonal selection of representatives, preserving the matrix
relations and the compact-uniform estimates, therefore gives
unital matrix units $(f_{ij})_{i,j=1}^2$ in
$M_{\omega,\alpha}^{\alpha^\omega}$.

Reindex the eigenunitary $u$ so that it commutes with these matrix
units. Then
\[
g_{11}=f_{11},\quad g_{22}=f_{22},\quad
g_{21}=uf_{21},\quad g_{12}=f_{12}u^*
\]
are matrix units with
\[
\alpha_r^\omega(g_{21})=e^{2\pi\ri r/T}g_{21}.
\]
Thus they give an equivariant unital embedding of the single
matrix flow defining $\sigma^T$ into $M_{\omega,\alpha}$.
Repeated fast reindexing produces commuting copies of this
embedding. Their tensor product extends normally, by trace
preservation, to an equivariant unital embedding
\[
(R,\sigma^T)\longrightarrow(M_{\omega,\alpha},\alpha^\omega).
\]

By Proposition~\ref{periodic model strong self absorption},
$\sigma^T$ is strongly self-absorbing. The equivariant absorption theorem
\cite[Corollary~3.9]{SW24} now gives
$\alpha\sim_{\mathrm{cc}}\alpha\otimes\sigma^T$.

Finally, assume~\textup{(ii)} and let $\rho:\R\curvearrowright R$
satisfy $\rho_T=\id$. Then the flow $\sigma^T\otimes\rho$ is
$T$-periodic. Tail matrix tensors in its first factor are
asymptotically central eigenoperators at every frequency in
$(1/T)\Z$. Their compressions by any nonzero invariant
projection are nonzero for all sufficiently late tensor blocks.
Thus $\Gamma(\sigma^T\otimes\rho)=(1/T)\Z$, and the
circle classification gives
$\sigma^T\otimes\rho\cong\sigma^T$.
Consequently,
\[
\alpha\otimes\rho
\sim_{\mathrm{cc}}\alpha\otimes\sigma^T\otimes\rho
\sim_{\mathrm{cc}}\alpha\otimes\sigma^T
\sim_{\mathrm{cc}}\alpha.\qedhere
\]
\end{proof}

\subsubsection{Classification}

\begin{theorem}\label{rank one classification}
Let $\alpha:\R\curvearrowright M$ be a flow on a hyperfinite
$\II_1$ factor $M$ with separable predual. Assume that
$\alpha$ has full Connes spectrum and that
$\Lambda(\alpha)=\Z(p,t)$, where $p\ne0$ and $t>0$.
Then
\[
\alpha\sim_{\mathrm{cc}}\sigma^{p,t}.
\]
\end{theorem}

\begin{proof}
The models $\sigma^{p,t}$ and $\sigma^{-p,t}$ both act on
$R_\theta$, with the same second frequency $q=\theta/t$.
Corollary~\ref{rank one model properties} gives their invariants
$\Z(p,t)$ and $\Z(-p,t)$, respectively.
Proposition~\ref{rank one central spectrum} supplies a
$1/t$-eigenunitary for each of $\alpha$ and $\sigma^{p,t}$ on
its equicontinuous central sequence algebra. Both underlying
factors are McDuff, so Proposition~\ref{rank one periodic absorption} gives
\[
\alpha\otimes\sigma^t\sim_{\mathrm{cc}}\alpha,
\qquad
\sigma^{p,t}\otimes\sigma^t\sim_{\mathrm{cc}}\sigma^{p,t}.
\]
On the other hand, Corollary~\ref{rank one periodic cancellation model}
gives
\[
\alpha\otimes\sigma^{-p,t}\sim_{\mathrm{cc}}\sigma^t,
\qquad
\sigma^{-p,t}\otimes\sigma^{p,t}\sim_{\mathrm{cc}}\sigma^t.
\]
Associating the triple tensor product in the two ways yields
\[
\alpha
\sim_{\mathrm{cc}}\alpha\otimes\sigma^t
\sim_{\mathrm{cc}}\alpha\otimes\sigma^{-p,t}\otimes\sigma^{p,t} \sim_{\mathrm{cc}}\sigma^t\otimes\sigma^{p,t}
\sim_{\mathrm{cc}}\sigma^{p,t}.
\]
\end{proof}

For fixed $(p,t)$, the argument also shows that the cocycle
conjugacy class of the torus model does not depend on the chosen
nonquadratic irrational angle $\theta$.

\subsection{The rank two case}
\label{sec:non-outer-flows}

We now treat the rank two case. The approach is completely different from the outer case or the rank one case. Let $\alpha:\R\curvearrowright M$
be a flow on a hyperfinite $\II_1$ factor $M$ with separable predual.
Assume that $\alpha$ has full Connes spectrum and that
$\Lambda(\alpha)$ has rank two.
Theorem~\ref{non outer bicentralizer model} below
identifies $\alpha|_{\rB(M,\alpha)}$ with our model action, the
irrational rotation flow $\sigma^{\Lambda(\alpha)}$ determined by
the lattice $\Lambda(\alpha)$. This observation is the starting
point of our classification strategy: we will show that $\alpha$
is cocycle conjugate to $\alpha|_{\rB(M,\alpha)}$. To do so, we
construct a cocycle $u\in Z^1(\alpha)$ such that
\[
\rB(M,\alpha^u)=M.
\]
The equivariant bicentralizer isomorphism
$\beta_u:\rB(M,\alpha)\to\rB(M,\alpha^u)=M$ then conjugates
$\alpha|_{\rB(M,\alpha)}$ to $\alpha^u$. Hence $\alpha$ is
cocycle conjugate to $\alpha|_{\rB(M,\alpha)}=\sigma^{\Lambda(\alpha)}$.

To obtain such a cocycle $u$, Proposition~\ref{full bicentralizer perturbation criterion} shows
that it suffices to construct a nonzero partial isometry
$v\in M^{\omega,\alpha}$ with $v^*v\in M'\cap M^\omega$ and
$vMv^*\subset\rB(M,\alpha)^\omega$. We first construct
an equivariant embedding of $(M,\alpha)$ into the equicontinuous
ultrapower $\bigl(\rB(M,\alpha)^\omega_\alpha,\alpha^\omega\bigr)$
that fixes $\rB(M,\alpha)$ pointwise. We then pass to approximate
equivariant embeddings of finite matrix algebras into the
bicentralizer, each implemented by an ordinary unitary of $M$.
Following Haagerup's Fourier decomposition
\cite[proof of Lemma~5.2]{Ha16}, we localize these unitaries in
frequency. Lattice translations put the frequencies near zero, and
Haagerup's randomization trick produces bounded almost invariant approximate
intertwiners for increasing matrix algebras.  We then obtain a nonzero partial isometry $v \in M^{\omega,\alpha}$ verifying condition~\textup{(iv)}
of Proposition~\ref{full bicentralizer perturbation criterion},
which supplies the desired cocycle by a Baire category argument.

\subsubsection{The bicentralizer model}

\begin{theorem}\label{non outer bicentralizer model}
Let $\alpha:\R\curvearrowright M$ be a flow on a $\II_1$ factor
with $\Gamma(\alpha)=\R$. Suppose that $\Lambda(\alpha)$ has rank two.
Then $\rB(M,\alpha)$ is the hyperfinite $\II_1$ factor and
\[
\alpha|_{\rB(M,\alpha)}\cong\sigma^{\Lambda(\alpha)}.
\]
\end{theorem}

\begin{proof}
Choose a basis $(p,t),(q,s)$ of $\Lambda(\alpha)$ and unitaries
$v,w\in\cU(M)$ implementing these pairs. Then
\[
\alpha_t=\Ad(v),\qquad \alpha_s=\Ad(w),
\]
\[
\alpha_r(v)=e^{2\pi\ri pr}v,\quad
\alpha_r(w)=e^{2\pi\ri qr}w\quad(r\in\R).
\]
Every implementer for $(mp+nq,mt+ns)$ is a scalar multiple of
$v^mw^n$. Theorem~\ref{bicentralizer conjecture for flows} gives
\[
\rB(M,\alpha)=\{v,w\}'',
\qquad vw=\alpha_t(w)v=e^{2\pi\ri tq}wv.
\]
By the same theorem, both coordinate projections of
$\Lambda(\alpha)$ are injective. Hence
$mp+nq\ne0$ for $(m,n)\ne(0,0)$, and trace invariance gives
\[
\tau(v^mw^n)=0\qquad((m,n)\ne(0,0)).
\]
Also $pt\in\Z\setminus\{0\}$ and $q/p\notin\mathbb Q$, so
\[
tq=(pt)\frac qp\notin\mathbb Q.
\]
Thus the commutation relation and trace moments identify
$\rB(M,\alpha)$ with the tracial von Neumann completion of the
irrational rotation algebra with parameter $tq$ modulo $\Z$.
This is the crossed product of $\rL^\infty(\mathbb T)$ by an
irrational rotation, hence the hyperfinite $\II_1$ factor.
The assignment $U_{p,t}\mapsto v$, $U_{q,s}\mapsto w$ therefore extends to a
trace-preserving normal isomorphism
\[
R_{\Lambda(\alpha)}\longrightarrow\rB(M,\alpha).
\]
The eigenvalue formulas show that this isomorphism intertwines
$\sigma^{\Lambda(\alpha)}$ and $\alpha|_{\rB(M,\alpha)}$.
\end{proof}

We next give a general criterion for a cocycle perturbation to have
the whole factor as its bicentralizer. The factor need not be
hyperfinite. The last condition asks for a nonzero partial isometry
in the fixed-point ultrapower of the original flow. This is the
condition we will verify by constructing almost invariant
intertwiners in the classification proof.

\begin{proposition}\label{full bicentralizer perturbation criterion}
Let $\alpha:\R\curvearrowright M$ be a flow on a $\II_1$ factor
with separable predual. Assume $\Gamma(\alpha)=\R$, and fix a
free ultrafilter $\omega$ on $\N$.
The following conditions are equivalent.
\begin{enumerate}[\rm (i)]
\item There is $u\in Z^1(\alpha)$ such that
$\rB(M,\alpha^u)=M$.
\item The set
\[
\{u\in Z^1(\alpha)\mid\rB(M,\alpha^u)=M\}
\]
is a dense $G_\delta$ subset of $Z^1(\alpha)$.
\item There is a unitary $v\in M^{\omega,\alpha}$ such that
\[
vMv^*\subset\rB(M,\alpha)^\omega.
\]
\item There is a nonzero partial isometry $v\in M^{\omega,\alpha}$
such that
\[
v^*v\in M'\cap M^\omega,
\qquad
vMv^*\subset\rB(M,\alpha)^\omega.
\]
\end{enumerate}
\end{proposition}

\begin{proof}
For $y\in M$ and $u\in Z^1(\alpha)$, write
\[
d_2(y,\rB(M,\alpha^u))
=\|y-\rE_{\rB(M,\alpha^u)}(y)\|_2.
\]
Theorem~\ref{approximate cocycle vanishing theorem} and
Proposition~\ref{relative bicentralizer isomorphism} give
isomorphisms
\[
\beta_u:\rB(M,\alpha)\longrightarrow\rB(M,\alpha^u)
\qquad(u\in Z^1(\alpha))
\]
such that $u\mapsto\beta_u(b)$ is $\|\cdot\|_2$-continuous for
every $b\in\rB(M,\alpha)$. Hence the sets
\[
\mathcal O(x,\varepsilon)
=\{u\in Z^1(\alpha)\mid d_2(x,\rB(M,\alpha^u))<\varepsilon\}
\]
are open, since the distance is
$\inf_{b\in\rB(M,\alpha)}\|x-\beta_u(b)\|_2$.
Choose a $\|\cdot\|_2$-dense sequence $(x_j)_j$ in the unit ball
of $M$. Then
\[
\{u\in Z^1(\alpha)\mid\rB(M,\alpha^u)=M\}
=\bigcap_{j,k\geq1}\mathcal O(x_j,1/k),
\]
so this set is always a $G_\delta$.
The space $Z^1(\alpha)$ is Polish by
Proposition~\ref{topology and action on cocycles}.
We will also use that the functions
\[
f_x(u)=\|\rE_{\rB(M,\alpha^u)}(x)\|_2\qquad(x\in M)
\]
are lower semicontinuous and uniformly $1$-Lipschitz in $x$.

\textup{(i)}$\Rightarrow$\textup{(ii)}.
Choose $u$ as in~\textup{(i)}. Its orbit
\[
(w\mathbin{\cdot}u)_t=w u_t\alpha_t(w^*)
\qquad(w\in\cU(M))
\]
is dense in $Z^1(\alpha)$ by approximate cocycle vanishing for
$\alpha^u$ and Proposition~\ref{topology and action on cocycles}.
Every cocycle in this orbit has bicentralizer $M$, proving density.
The converse implication is immediate.

\textup{(ii)}$\Rightarrow$\textup{(iii)}.
For every $n\geq1$, choose $u\in Z^1(\alpha)$ with
$\rB(M,\alpha^u)=M$ and
\[
\sup_{|t|\leq n}\|u_t-1\|_2<\frac1{2n}.
\]
Choose unitaries $w_k\in\cU(M)$ with
$\partial_\alpha(w_k)\to u$ uniformly on compact intervals.
Then $w_k^*x_jw_k\to\beta_u^{-1}(x_j)$ in $\|\cdot\|_2$
for $j\leq n$. For a sufficiently large $k$, put $v_n=w_k^*$,
so that
\[
\sup_{|t|\leq n}\|\alpha_t(v_n)-v_n\|_2<\frac1n,
\qquad
\max_{j\leq n}d_2(v_nx_jv_n^*,\rB(M,\alpha))<\frac1n.
\]
Then $v=(v_n)^\omega$ is a unitary in $M^{\omega,\alpha}$
with $vMv^*\subset\rB(M,\alpha)^\omega$, proving~\textup{(iii)}.
The implication \textup{(iii)}$\Rightarrow$\textup{(iv)} is immediate.

\textup{(iv)}$\Rightarrow$\textup{(i)}.
We use continuity of the bicentralizers to show that one of
them contains all of $M$.
By Baire's theorem, the functions $f_{x_j}$ have a common
continuity point $u\in Z^1(\alpha)$. The Lipschitz estimate makes
this a continuity point of every $f_x$.
Suppose that $\rB(M,\alpha^u)\ne M$, and choose
$x\in M\ominus\rB(M,\alpha^u)$ with $\|x\|_2=1$.
Choose $w_n\in\cU(M)$ with $\partial_\alpha(w_n)\to u$
uniformly on compact intervals, and put $y_n=w_n^*xw_n$.

We first retain a nonzero part of each $y_n$ inside the
bicentralizer of $\alpha$.
Take $v$ as in~\textup{(iv)}, put $p=v^*v$, and set
$\lambda=\tau^\omega(p)>0$.
Since $p\in M'$ and $M$ is a factor,
\[
\|vyv^*\|_2^2=\tau^\omega(py^*y)
=\lambda\|y\|_2^2\qquad(y\in M).
\]
Choose partial isometries $v_n\in M$ from a representing
sequence of $v$ so that
\[
\begin{gathered}
\sup_{|t|\leq n}\|\alpha_t(v_n)-v_n\|_2<\frac1n,\qquad
 d_2(v_nv_n^*,\rB(M,\alpha))<\frac1n,\\
 d_2(v_ny_nv_n^*,\rB(M,\alpha))<\frac1n,\qquad
\|v_ny_nv_n^*\|_2^2>\frac\lambda2.
\end{gathered}
\]
For each $n$, these choices follow from the properties of $v$
and the preceding identity applied to $y_n$.
Thus $a=(v_n)^\omega\in M^{\omega,\alpha}$ and
$y=(y_n)^\omega\in M^\omega$ satisfy
\[
q=aa^*\in\rB(M,\alpha)^\omega,\qquad
aya^*\in\rB(M,\alpha)^\omega,\qquad
\|aya^*\|_2^2\geq\frac\lambda2.
\]

We now use continuity at $u$ to contradict this positive mass.
Since $M^{\omega,\alpha}$ is finite, $a$ extends to a unitary
$z=(z_n)^\omega$ in this algebra, with $z_n\in\cU(M)$.
Then $\partial_\alpha(w_nz_n^*)\to_\omega u$ uniformly on
compact intervals, so continuity gives
\begin{equation}\label{bicentralizer continuity point identity}
\|\rE_{\rB(M,\alpha)^\omega}(zyz^*)\|_2
=\lim_{n\to\omega}f_x\bigl(\partial_\alpha(w_nz_n^*)\bigr)
=f_x(u)=0.
\end{equation}
But $qz=a$, and $q,aya^*\in\rB(M,\alpha)^\omega$.
Compressing this zero conditional expectation by $q$ gives
$aya^*=0$, a contradiction. Hence $\rB(M,\alpha^u)=M$,
proving~\textup{(i)}.
\end{proof}

\subsubsection{Equivariant embeddings into the ultrapower}

The rank two irrational rotation case yields the following
factoriality theorem for arbitrary $\II_1$ factors.

\begin{theorem}\label{full Connes spectrum fixed ultrapower factor}
Let $\alpha:\R\curvearrowright M$ be a flow on a $\II_1$ factor
with $\Gamma(\alpha)=\R$, and let $\omega$ be a free ultrafilter
on $\N$. Then $M^{\omega,\alpha}$ is a $\II_1$ factor.
\end{theorem}

\begin{proof}
Suppose first that $M=R_\Lambda$ and $\alpha=\sigma^\Lambda$,
where $\Lambda<\R^2$ is a nondegenerate resonant lattice.
The defining unitaries $U_{p,t}$, indexed by $(p,t)\in\Lambda$,
form an orthonormal basis of $\rL^2(R_\Lambda)$ and
satisfy
\[
\alpha_t=\Ad(U_{p,t}),\qquad
\alpha_r(U_{p,t})=e^{2\pi\ri pr}U_{p,t}.
\]
The formulas
\[
\beta_q(U_{p,t})=e^{2\pi\ri qt}U_{p,t}
\]
define a flow on $R_\Lambda$. Resonance gives
\begin{equation}\label{rotation implementers dual flow}
\Ad(U_{p,t})=\beta_{-p}.
\end{equation}
Indeed, for every $(q,s)\in\Lambda$, the two sides act on
$U_{q,s}$ by multiplication by $e^{2\pi\ri qt}$ and
$e^{-2\pi\ri ps}$, respectively. These scalars agree because
resonance gives $qt+ps\in\Z$.

Take
$x=(x_n)^\omega\in\cZ(R_\Lambda^{\omega,\alpha})$
with $x_n\in R_\Lambda$ and $\|x_n\|\leq1$.
For $\varepsilon>0$, let $P_\varepsilon$ be the orthogonal
projection onto the closed span of the $U_{p,t}$ with
$|p|\leq\varepsilon^2$. Almost invariance
gives $\|(1-P_\varepsilon)x_n\|_2\to_\omega0$ for every fixed
$\varepsilon>0$. A diagonal choice therefore gives
$\varepsilon_n>0$ such that
\[
\varepsilon_n\to_\omega0,
\qquad
\delta_n=\|(1-P_{\varepsilon_n})x_n\|_2\to_\omega0.
\]
Averaging $\beta$ over $[-\varepsilon,\varepsilon]$ multiplies
$U_{p,t}$ by $\sin(2\pi\varepsilon t)/(2\pi\varepsilon t)$,
with value $1$ at $t=0$. For nonzero $(p,t)\in\Lambda$ with
$|p|\leq\varepsilon^2$, resonance and nondegeneracy give
$pt\in\Z\setminus\{0\}$, hence $|t|\geq\varepsilon^{-2}$.
Thus $|\sin(2\pi\varepsilon t)/(2\pi\varepsilon t)|$ is at most
$\varepsilon/(2\pi)$. Hence
\[
\left\|\frac1{2\varepsilon_n}
\int_{-\varepsilon_n}^{\varepsilon_n}\beta_q(x_n)\,\rd q
-\tau(x_n)1\right\|_2
\leq\delta_n+\frac{\varepsilon_n}{2\pi}\to_\omega0.
\]
On the other hand, $x=(x_n)^\omega\in\cZ(R_\Lambda^{\omega,\alpha})$ implies
\[
\sup_{|q|\leq\varepsilon_n}\|\beta_q(x_n)-x_n\|_2
\longrightarrow_\omega0.
\]
To see this, choose $q_n$ in these intervals. Density of the first
coordinate projection of $\Lambda$ and continuity of
$\beta$ give $(p_n,t_n)\in\Lambda$ with
\[
|p_n|\leq2\varepsilon_n,\qquad
\|\beta_{-p_n}(x_n)-\beta_{q_n}(x_n)\|_2<\varepsilon_n.
\]
The unitaries $(U_{p_n,t_n})_n$ are almost $\alpha$-invariant,
so they define an element of $R_\Lambda^{\omega,\alpha}$. Centrality
and \eqref{rotation implementers dual flow} now give
$\|\beta_{q_n}(x_n)-x_n\|_2\to_\omega0$. Since the $q_n$ were
arbitrary, this proves the asserted uniform convergence.
Combining the two estimates gives
$\|x_n-\tau(x_n)1\|_2\to_\omega0$. Thus
$R_\Lambda^{\omega,\alpha}$ is a factor.

We now return to the given flow $\alpha$ on $M$, without any
restriction on the rank of $\Lambda(\alpha)$. Proposition
\ref{relative bicentralizer with ultrapower} gives
\[
\cZ(M^{\omega,\alpha})
\subset\rB(M,\alpha)^\omega\cap M^{\omega,\alpha}
=\rB(M,\alpha)^{\omega,\alpha}.
\]
Since the left side commutes with the right side, we obtain
\[
\cZ(M^{\omega,\alpha})
\subset\cZ\bigl(\rB(M,\alpha)^{\omega,\alpha}\bigr).
\]
By Theorem~\ref{bicentralizer conjecture for flows}, there are
three cases. If $\Lambda(\alpha)=\{0\}$, then
$\rB(M,\alpha)=\C1$. If $\Lambda(\alpha)=\Z(p,t)$, with
$p\ne0$, then $\alpha|_{\rB(M,\alpha)}$ is periodic and ergodic.
Averaging over one period shows that
$\rB(M,\alpha)^{\omega,\alpha}=\C1$.
Finally, if $\Lambda(\alpha)$ has rank two,
Theorem~\ref{non outer bicentralizer model} identifies
$\alpha|_{\rB(M,\alpha)}$ with $\sigma^{\Lambda(\alpha)}$.
Applying the first part with $\Lambda=\Lambda(\alpha)$ shows
that $\rB(M,\alpha)^{\omega,\alpha}$ is a $\II_1$ factor.
In all three cases, $\cZ(M^{\omega,\alpha})=\C1$.

To see that $M^{\omega,\alpha}$ is of type $\II_1$,
apply Corollary~\ref{Connes spectrum approximate eigenunitaries}
to the frequencies $1/n$. Choose $u_n\in\cU(M)$ such that
\[
\sup_{|t|\leq n}
\|\alpha_t(u_n)-e^{2\pi\ri t/n}u_n\|_2<\frac1n.
\]
Then $(u_n)_n$ is almost $\alpha$-invariant on every compact
interval, so $u=(u_n)^\omega\in M^{\omega,\alpha}$. For
$1\leq k\leq n$, evaluating at $t=n/(2k)$ gives
\[
2|\tau(u_n^k)|
=|\tau(\alpha_{n/(2k)}(u_n^k)+u_n^k)|
\leq\frac{k}{n}.
\]
Thus $\tau^\omega(u^k)=0$ for every nonzero integer $k$, so
$u$ is Haar. This completes the proof.
\end{proof}

Fix a free ultrafilter $\omega$ on $\N$.

\begin{lemma}\label{equicontinuous ultrapower coboundaries closed}
For any trace-preserving flow $\alpha:\R\curvearrowright(M,\tau)$,
consider the induced flow $\alpha^\omega$ on $M^\omega_\alpha$.
Then $B^1(\alpha^\omega)$ is closed in $Z^1(\alpha^\omega)$ for uniform
$\|\cdot\|_2$-convergence on compact intervals.
\end{lemma}

\begin{proof}
Let $z\in\overline{B^1(\alpha^\omega)}$, enumerate $\mathbb Q$ as
$(r_\ell)_\ell$, and choose $T_j\geq j$ containing
$r_1,\ldots,r_j$ in $[-T_j,T_j]$. Take $v_j\in\cU(M^\omega_\alpha)$ with
\[
\sup_{|t|\leq T_j}\|z_t-v_j\alpha_t^\omega(v_j^*)\|_2<j^{-1}.
\]
Fix unitary representatives $v_j=(v_{j,n})^\omega$ and
$z_r=(z_{r,n})^\omega$ for $r\in\mathbb Q$.
Continuity of $z$ gives $0<\delta_k\leq1$ with
$\sup_{|t|\leq\delta_k}\|z_t-1\|_2<2^{-k}$, and hence
\[
\sup_{|t|\leq\delta_k}\|\alpha_t^\omega(v_j)-v_j\|_2
\leq2^{-k}+j^{-1}.
\]
For each fixed $j$, Proposition~\ref{omega equicontinous functions}
and the chosen representatives give, on a set in $\omega$,
\[
\begin{aligned}
\|z_{r_\ell,n}-v_{j,n}\alpha_{r_\ell}(v_{j,n}^*)\|_2
&<2j^{-1} &&(\ell\leq j),\\
\sup_{|t|\leq\delta_k}\|\alpha_t(v_{j,n})-v_{j,n}\|_2
&<2^{-k}+2j^{-1} &&(k\leq j).
\end{aligned}
\]
Intersect these sets successively with the tails $\{n\geq j\}$.
Taking the largest admissible row gives $j(n)\to_\omega\infty$.
Put $V_n=v_{j(n),n}$, using $V_n=1$ where no row is admissible.
The second estimate makes $(V_n)_n$ equicontinuous, and the first
gives $z_r=V\alpha_r^\omega(V^*)$ for $V=(V_n)^\omega$ and
$r\in\mathbb Q$. Continuity extends this equality to all of
$\R$, proving the lemma.
\end{proof}

\begin{theorem}\label{non outer Rokhlin ultrapower embedding}
Let $\alpha:\R\curvearrowright M$ be a flow on a $\II_1$ factor
with $\Gamma(\alpha)=\R$.
Then every flow $\rho:\R\curvearrowright R$ on the hyperfinite $\II_1$ factor admits a trace-preserving
unital normal equivariant embedding
\[
(R,\rho)\longrightarrow(M^\omega_\alpha,\alpha^\omega).
\]
\end{theorem}

\begin{proof}
We first prove the assertion when $\rho$ is an infinite tensor
product of matrix flows.

Since $\alpha$ has full Connes spectrum,
Theorem~\ref{full Connes spectrum fixed ultrapower factor} shows that
$M^{\omega,\alpha}$ is a $\II_1$ factor, and
Corollary~\ref{Connes spectrum approximate eigenunitaries} gives a
$p$-eigenunitary $u_p\in M^\omega_\alpha$ for every $p\in\R$.
Therefore, for every $p\in\R$ and every pair of equivalent projections
$e,f\in M^{\omega,\alpha}$, there is a partial isometry
$w\in M^\omega_\alpha$ such that
\[
w^*w=e,\qquad ww^*=f,\qquad
\alpha_t^\omega(w)=e^{2\pi\ri pt}w\quad(t\in\R).
\]
Indeed, $e$ and $u_p^*fu_p$ have the same trace in the factor
$M^{\omega,\alpha}$. Choose a partial isometry
$v\in M^{\omega,\alpha}$ with $v^*v=e$ and $vv^*=u_p^*fu_p$,
and put $w=u_pv$.

Now, given frequencies $p_1,p_2,\ldots,p_d \in \R$, choose
orthogonal projections $e_i\in M^{\omega,\alpha}$ of trace $1/d$ with sum $1$.
By the preceding observation, choose partial isometries
$v_i\in M^\omega_\alpha$ such that
\[
v_i^*v_i=e_1,\qquad v_iv_i^*=e_i,\qquad
\alpha_t^\omega(v_i)=e^{2\pi\ri(p_i-p_1)t}v_i\quad(t\in\R),
\]
with $v_1=e_1$. Then $f_{ij}=v_iv_j^*$ are matrix units satisfying
\[
\alpha_t^\omega(f_{ij})=e^{2\pi\ri(p_i-p_j)t}f_{ij}.
\]
For their matrix algebra $Q$, the map
\[
e_1M^\omega_\alpha e_1\longrightarrow Q'\cap M^\omega_\alpha,
\qquad x\longmapsto\sum_i v_i x v_i^*
\]
is an equivariant isomorphism preserving normalized traces, with
inverse given by compression. Iteration in successive matrix
commutants therefore embeds every infinite tensor product of
finite matrix flows. Trace preservation gives the normal extension
from the finite tensor products, proving the assertion in this case.

Now let $\rho$ be an arbitrary flow on $R$. Let $\sigma^\infty$ be
a faithful infinite tensor product flow as recalled in the
introduction. This flow is Rokhlin by \cite[Corollary~7.9]{MT16}.
The first part gives an
equivariant embedding
$\iota:(R,\sigma^\infty)\to(M^\omega_\alpha,\alpha^\omega)$.
The equivariant inclusion of central sequence algebras induced by
$y\mapsto1\otimes y$ and
Corollary~\ref{Rokhlin central sequence faithfulness} show that
$\rho\otimes\sigma^\infty$ is Rokhlin. By \cite[Corollary~5.16]{MT16}, we may identify
$(R\ovt R,\rho\otimes\sigma^\infty)$ with $(R,(\sigma^\infty)^c)$ for some
$c\in Z^1(\sigma^\infty)$. The full Connes spectrum of $\sigma^\infty$ and
Theorem~\ref{approximate cocycle vanishing theorem} give
$c\in\overline{B^1(\sigma^\infty)}$. Hence
$\iota(c)\in\overline{B^1(\alpha^\omega)}$, and
Lemma~\ref{equicontinuous ultrapower coboundaries closed} gives
$V\in\cU(M^\omega_\alpha)$ with
\[
\iota(c_t)=V\alpha_t^\omega(V^*)\qquad(t\in\R).
\]
The embedding $x\mapsto V^*\iota(x)V$ intertwines $(\sigma^\infty)^c$
with $\alpha^\omega$. Restricting the resulting embedding of
$R\ovt R$ to its first tensor factor proves the assertion.
\end{proof}

\begin{lemma}\label{non outer relative commutant type II}
Let $\alpha:\R\curvearrowright M$ be a flow on a $\II_1$ factor
with $\Gamma(\alpha)=\R$ and $\Lambda(\alpha)$ of rank two.
Let $\omega$ be a free ultrafilter on $\N$.
For every von Neumann subalgebra $P\subset\rB(M,\alpha)^\omega$
with separable predual, the algebra
\[
P'\cap\rB(M,\alpha)^{\omega,\alpha}
\]
is of type $\II_1$, and
\[
\cZ(P'\cap M^{\omega,\alpha})
\subset\rB(M,\alpha)^\omega.
\]
Moreover, for every projection $q\in P'\cap M^{\omega,\alpha}$,
there is a partial isometry $u\in P'\cap M^{\omega,\alpha}$ such that
\[
u^*u=q,\qquad uu^*\in\rB(M,\alpha)^\omega.
\]
\end{lemma}

\begin{proof}
\emph{The relative commutant is of type $\II_1$.}
We construct unital matrix algebras of every order in the
relative commutant. By Theorem~\ref{non outer bicentralizer model},
$\rB(M,\alpha)$ is a $\II_1$ factor, its restricted flow has full Connes
spectrum, and it has unitary generators $v_1,v_2$ with
$\alpha_{t_i}=\Ad(v_i)$ for suitable $t_i\in\R$.
Theorem~\ref{full Connes spectrum fixed ultrapower factor}
therefore makes $\rB(M,\alpha)^{\omega,\alpha}$ a $\II_1$ factor.
Its elements commute with $v_1,v_2$, hence with $\rB(M,\alpha)$.

Choose a $\|\cdot\|_2$-dense sequence $(x_k)_k$ in the unit
ball of $P$, with representatives
$x_k=(x_{k,n})^\omega$, $x_{k,n}\in\rB(M,\alpha)$,
$\|x_{k,n}\|\leq1$.
Fix $d\geq2$ and lift a unital copy of $\mathbb M_d(\C)$
in $\rB(M,\alpha)^{\omega,\alpha}$ to embeddings
$\theta_m:\mathbb M_d(\C)\to\rB(M,\alpha)$.
For standard matrix units $e_{ij}$, the $\theta_m(e_{ij})$
become invariant uniformly on compact time intervals along
$\omega$, by Proposition~\ref{omega equicontinous functions},
and commute asymptotically with each fixed element of $\rB(M,\alpha)$.
For each $n$, choose $m(n)$ so that, for all $i,j\leq d$,
\[
\begin{gathered}
\sup_{|t|\leq n}
\|\alpha_t(\theta_{m(n)}(e_{ij}))-\theta_{m(n)}(e_{ij})\|_2<1/n,\\
\max_{k\leq n}\|[\theta_{m(n)}(e_{ij}),x_{k,n}]\|_2<1/n.
\end{gathered}
\]
These finitely many conditions hold on a set in $\omega$.
Thus $a\mapsto(\theta_{m(n)}(a))^\omega$ embeds
$\mathbb M_d(\C)$ unitally into
$P'\cap\rB(M,\alpha)^{\omega,\alpha}$.
Since $d$ is arbitrary, this finite algebra has no type
$\mathrm I_n$ summand for any $n$, and is of type $\II_1$.

\smallskip
\emph{Locate the center.}
Suppose that
$h=(h_n)^\omega\in\cZ(P'\cap M^{\omega,\alpha})
\setminus\rB(M,\alpha)^\omega$.
Choose $\varepsilon>0$ and $A\in\omega$ such that
\[
\|h_n-\rE_{\rB(M,\alpha)}(h_n)\|_2>2\varepsilon
\qquad(n\in A).
\]
For each $n$, apply \eqref{local relative action bicentralizer}
to the representatives $x_{1,n},\ldots,x_{n,n}$ chosen above.
By enlarging the compact set and decreasing the tolerance, we
obtain $K_n\Subset\R$ containing $[-n,n]$ and
$0<\delta_n\leq n^{-1}$ such that
\[
\sup_{t\in K_n}\|\alpha_t(w)-w\|_2<\delta_n
\quad\Longrightarrow\quad
\max_{k\leq n}\|[w,x_{k,n}]\|_2<n^{-1}
\qquad(w\in\cU(M)).
\]
For $n\in A$, Proposition~\ref{conditional expectation relative action bicentralizer}
places $\rE_{\rB(M,\alpha)}(h_n)$ in the weak$^*$-closed convex
hull of the conjugates of $h_n$ by these unitaries. Hence one
of them, denoted by $w_n$, satisfies $\|[w_n,h_n]\|_2>\varepsilon$.
Indeed, otherwise that convex hull would lie in the closed
$\|\cdot\|_2$-ball of radius $\varepsilon$ about $h_n$.
Put $w_n=1$ for $n\notin A$.
Then $w=(w_n)^\omega$ belongs to $M^{\omega,\alpha}$ and
commutes with every $x_k$, so $w\in P'\cap M^{\omega,\alpha}$.
This contradicts the centrality of $h$ and proves the inclusion.

\smallskip
\emph{Compare projections.}
Let $\rT$ be the normalized center-valued trace on
$P'\cap M^{\omega,\alpha}$. The first two parts show that
$P'\cap\rB(M,\alpha)^{\omega,\alpha}$ is of type $\II_1$ and
contains $\cZ(P'\cap M^{\omega,\alpha})$.
Therefore, we can choose a
projection $f\in P'\cap\rB(M,\alpha)^{\omega,\alpha}$ whose
center-valued trace in this algebra is $\rT(q)$.
The restriction of $\rT$ to this algebra factors through its
normalized center-valued trace, so $\rT(f)=\rT(q)$.
Equality of center-valued traces gives a partial isometry
$u\in P'\cap M^{\omega,\alpha}$ with $u^*u=q$ and $uu^*=f$,
as required.
\end{proof}

\begin{theorem}\label{non outer embedding fixing bicentralizer}
Let $\alpha:\R\curvearrowright M$ be a flow with $M\cong R$,
$\Gamma(\alpha)=\R$, and $\Lambda(\alpha)$ of rank two.
Let $\omega$ be a free ultrafilter on $\N$.
Then there is a trace-preserving unital normal equivariant
embedding $\Phi:M\to \rB(M,\alpha)^\omega_\alpha$ such that $\Phi|_{\rB(M,\alpha)}=\id_{\rB(M,\alpha)}$.
\end{theorem}

\begin{proof}
\noindent\emph{Step 1: An equivariant embedding and its relative commutant.}
The restriction
$\alpha|_{\rB(M,\alpha)}$ is a flow on a $\II_1$ factor with full
Connes spectrum. Applying
Theorem~\ref{non outer Rokhlin ultrapower embedding} to this restriction
gives an equivariant embedding
$\psi:(M,\alpha)\to(\rB(M,\alpha)^\omega_\alpha,\alpha^\omega)$.
By Lemma~\ref{non outer relative commutant type II}, the algebra
\[
Q=\psi(M)'\cap\rB(M,\alpha)^{\omega,\alpha}
\]
is of type $\II_1$. We will use this to cut spectral projections
to prescribed traces.

\smallskip
\noindent\emph{Step 2: Approximate correction on the lattice generators.} Choose unitary generators $v_1,v_2$ of $\rB(M,\alpha)$, with
$\alpha_{t_i}=\Ad(v_i)$ for some $t_i \in \R$. Our goal in this step is to construct, for every $\varepsilon>0$, a
trace-preserving unital normal equivariant embedding
\[
\varphi_\varepsilon:M\longrightarrow\rB(M,\alpha)^\omega_\alpha
\quad\text{such that}\quad
\|\varphi_\varepsilon(v_i)-v_i\|\leq\varepsilon\qquad(i=1,2).
\]
In Step~3, we will turn these approximate identities into an
embedding fixing $\rB(M,\alpha)$ exactly. To obtain $\varphi_\varepsilon$,
we first compress $\psi$ to a corner on which its values at
$v_1,v_2$ are close to scalar multiples of $v_1,v_2$. We then
amplify this corner back to a unital embedding and remove the
scalar phases by a gauge automorphism.

The unitaries
\[
w_i=v_i^*\psi(v_i)\qquad(i=1,2)
\]
belong to $Q$. Indeed, they are fixed by
equivariance, and $v_i$ and $\psi(v_i)$ implement the same
automorphism on $\psi(M)$. Every element of $Q$ commutes
with both $\rB(M,\alpha)$ and $\psi(M)$, so in fact
$w_1,w_2\in\cZ(Q)$.

Choose $(\lambda_1,\lambda_2)\in\mathbb T^2$ in the joint spectrum of $w_1$ and $w_2$.
Fix $\varepsilon>0$ and take a nonzero joint spectral projection $p$ with
$\|(w_i-\lambda_i)p\|\leq\varepsilon$ for $i=1,2$.
Choose an integer $k$ with $1/k\leq\tau^\omega(p)$.
Since $Q$ is of type $\II_1$, there is a projection $e\in Q$
with $e\leq p$ and $\tau^\omega(e)=1/k$. This trace allows
the compressed embedding to be amplified to a unital one.

By Theorem~\ref{full Connes spectrum fixed ultrapower factor},
$\rB(M,\alpha)^{\omega,\alpha}$ is a factor. Choose partial
isometries $s_1,\ldots,s_k$ in it with common initial
projection $e$ and orthogonal final projections summing to $1$.
The map
\[
\psi_\varepsilon(x)=\sum_{\ell=1}^{k}
s_\ell\psi(x)s_\ell^*
\]
is a unital normal equivariant embedding. It preserves the trace
because $\tau^\omega(e\psi(x))=\tau^\omega(e)\tau(x)$.
All the $s_\ell$ commute with $\rB(M,\alpha)$, so
\[
\|\psi_\varepsilon(v_i)-\lambda_i v_i\|
\leq\|(w_i-\lambda_i)p\|\leq\varepsilon\qquad(i=1,2).
\]
Let $\gamma\in\operatorname{Aut}(\rB(M,\alpha))$ be the gauge automorphism
$v_i\mapsto\overline{\lambda_i}v_i$. It preserves the trace and
commutes with $\alpha$. Thus $\varphi_\varepsilon=\gamma^\omega\circ\psi_\varepsilon$
is equivariant and $\|\varphi_\varepsilon(v_i)-v_i\|\leq\varepsilon$.

\smallskip
\noindent\emph{Step 3: Exact fixing by diagonalization.}
Apply Step~2 with $\varepsilon=1/j$ for $j\geq1$.
We now choose representatives of the embeddings $\varphi_{1/j}$ and
diagonalize so that the resulting embedding remains equivariant
and fixes both lattice generators exactly.
Fix an isomorphism $\iota:M\to \rB(M,\alpha)$. Matching increasing matrix
subfactors and diagonalizing gives $W_j\in\cU(\rB(M,\alpha)^\omega)$ with
$\varphi_{1/j}=\Ad(W_j)\circ\iota$. For unitary representatives
$W_j=(W_{j,n})^\omega$, put
$\varphi_{j,n}=\Ad(W_{j,n})\circ\iota$. These are trace-preserving
normal embeddings into $\rB(M,\alpha)$. For fixed $j,x$, their values at $x$
are equicontinuous representatives of $\varphi_{1/j}(x)$, while
$t\mapsto\varphi_{j,n}(\alpha_t(x))$ has the common continuity
modulus of $\alpha_t(x)$. Proposition
\ref{omega equicontinous functions} therefore gives
\[
\lim_{n\to\omega}\sup_{|t|\leq T}
\|\alpha_t(\varphi_{j,n}(x))-\varphi_{j,n}(\alpha_t(x))\|_2=0
\qquad(T>0).
\]
Choose a $\|\cdot\|_2$-dense sequence $(x_\ell)_\ell$ in the
unit ball of $M$. For each $j$, choose $n(j)$ so that the last
supremum is less than $1/j$ for $T=j$ and $x=x_1,\ldots,x_j$,
and $\|\varphi_{j,n(j)}(v_i)-v_i\|_2<2/j$ for $i=1,2$.
These are finitely many conditions holding on a set in $\omega$.
The map
\[
\Phi(x)=(\varphi_{j,n(j)}(x))^\omega
\]
is a trace-preserving unital normal embedding. The covariance
estimates and continuity of $\alpha$ give equicontinuity and
exact equivariance first on the $x_\ell$, then on $M$ by density.
Finally, $\Phi(v_i)=v_i$ for $i=1,2$, so $\Phi$ fixes $\rB(M,\alpha)$.
\end{proof}

\subsubsection{Spectral localization and intertwiners}
Throughout this subsection, let $\alpha:\R\curvearrowright M$ be
a flow where $M$ is a separable hyperfinite $\II_1$ factor, $\Gamma(\alpha)=\R$, and
$\Lambda(\alpha)$ has rank two. Put $\Lambda=\Lambda(\alpha)$.
By Theorem~\ref{non outer bicentralizer model}, we identify
$(\rB(M,\alpha),\alpha|_{\rB(M,\alpha)})$ with $(R_\Lambda,\sigma^\Lambda)$ and write
$U_{p,t}$ for its defining unitaries. Fix a basis
$(p_1,t_1),(p_2,t_2)$ of $\Lambda$ and put
$v_\ell=U_{p_\ell,t_\ell}$ for $\ell=1,2$.
Thus $\alpha_{t_\ell}=\Ad(v_\ell)$ for $\ell=1,2$.

Our goal is to verify condition~\textup{(iv)} of
Proposition~\ref{full bicentralizer perturbation criterion}: we seek
a nonzero partial isometry $v\in M^{\omega,\alpha}$ such that
$v^*v\in M'\cap M^\omega$ and $vMv^*\subset \rB(M,\alpha)^\omega$.
We start with the equivariant embedding obtained in
Theorem~\ref{non outer embedding fixing bicentralizer}. On each
unital finite dimensional matrix subfactor $F\subset M$, which
need not be $\alpha$-invariant, this gives an approximately
equivariant embedding into $\rB(M,\alpha)$ implemented by a unitary of $M$.
Following Haagerup's proof of the uniqueness of the injective
$\III_1$ factor, we decompose this unitary into Fourier
coefficients with localized spectrum
\cite[proof of Lemma~5.2, p.~1215]{Ha16}.
The successive lemmas then transform these coefficients,
preserving a positive amount of mass while improving their
spectral and intertwining properties.
Finally, a Gaussian version of \emph{Haagerup's randomization trick}
combines them into a single almost invariant contraction with
positive mass and small intertwining error. This is the
randomization argument underlying
\cite[Theorem~2.8 and Remark~2.9, used in Lemma~5.6]{Ha16}. Passing to the ultrapower and moving the range into $\rB(M,\alpha)^\omega$
then gives the required partial isometry.

We use the following common conditions for an automorphism
$\theta\in\operatorname{Aut}(M)$ and a finite family $(a_i)_i$:
\begin{equation}\label{non outer family covariance}
\theta(F)\subset \rB(M,\alpha),\qquad
\sup_{\substack{|t|\leq L\\x\in F,\ \|x\|\leq1}}
\|\alpha_t(\theta(x))-\theta(\alpha_t(x))\|_2<\varepsilon,
\end{equation}
\begin{equation}\label{non outer family bounds}
\sum_i a_i^*a_i\leq1,\qquad
\sum_i a_i a_i^*\leq1,
\end{equation}
and the approximate intertwining condition
\begin{equation}\label{non outer family intertwining}
\sup_{\substack{x\in F\\\|x\|\leq1}}
\sum_i\|a_i x-\theta(x)a_i\|_2^2<\varepsilon^2.
\end{equation}
We first obtain an approximately equivariant unitary conjugation
of $F$ into $\rB(M,\alpha)$ that approximately fixes the lattice generators.

\begin{lemma}\label{non outer finite matrix approximations}
For every unital matrix subfactor $F\subset M$ and every
$L,\varepsilon>0$, there is an inner automorphism $\theta$ satisfying
\eqref{non outer family covariance} and
\[
\|\theta(v_\ell)-v_\ell\|_2<\varepsilon
\qquad(\ell=1,2).
\]
\end{lemma}

\begin{proof}
Theorem~\ref{non outer embedding fixing bicentralizer} supplies
an equivariant embedding
\[
\Phi:(M,\alpha)\longrightarrow(\rB(M,\alpha)^\omega_\alpha,\alpha^\omega)
\]
fixing $\rB(M,\alpha)$.
Since $F'\cap M$ is hyperfinite, choose a matrix algebra
$A\supset F$ approximating $v_1,v_2$ and the compact set
\[
\{\alpha_s(x)\mid |s|\leq L,\ x\in F,\ \|x\|\leq1\}
\]
in $\|\cdot\|_2$ to an arbitrarily small tolerance.
Lift the matrix units of $\Phi|_A$ to unital embeddings
$\eta_n:A\to \rB(M,\alpha)$, and choose $w_n\in\cU(M)$ implementing
$\eta_n$ on $A$. Then $w_nFw_n^*\subset \rB(M,\alpha)$.

For $x\in F$, compare both $\alpha_s(w_nxw_n^*)$ and
$w_n\alpha_s(x)w_n^*$ with
$\eta_n(\rE_A(\alpha_s(x)))$.
Equivariance of $\Phi$ and
Proposition~\ref{omega equicontinous functions} make the
ultralimit of the equivariance error arbitrarily small, uniformly
on $[-L,L]$ and the unit ball of $F$. Since $\Phi$ fixes $\rB(M,\alpha)$,
the same is true of $\|w_n v_\ell w_n^*-v_\ell\|_2$ for $\ell=1,2$.
Choosing the approximation tolerance sufficiently small and
then one coordinate $n$ gives the conclusion with
$\theta=\Ad(w_n)$.
\end{proof}

Following the Fourier decomposition in Haagerup's proof
\cite[proof of Lemma~5.2, p.~1215]{Ha16}, we decompose the implementing
unitary into Fourier coefficients with narrow spectral intervals
and small aggregate commutators with the lattice generators.

\begin{lemma}\label{non outer Fourier coefficients}
For every unital matrix subfactor $F\subset M$ and every
$L,\varepsilon>0$, there are an inner
automorphism $\theta$, a finite family $(a_i)_i$ in $M$, and
frequencies $r_i\in\R$ satisfying
\eqref{non outer family covariance}--\eqref{non outer family intertwining}
and
\[
\begin{gathered}
\sum_i\|a_i\|_2^2>1-\varepsilon,\qquad
\Sp_\alpha(a_i)\subset[r_i-\varepsilon,r_i+\varepsilon],
\\
\sum_i\|[a_i,v_\ell]\|_2^2<\varepsilon^2
\qquad(\ell=1,2).
\end{gathered}
\]
\end{lemma}

\begin{proof}
Take a Schwartz function $g$ of $\rL^2$-norm $1$ with Fourier
transform supported in $[-1,1]$, and put
$g_m(s)=m^{-1/2}g(s/m)$. Choose $m$ so that
$m^{-1}<\varepsilon$.
Choose an inner automorphism $\theta$ as in
Lemma~\ref{non outer finite matrix approximations} and write
$\theta=\Ad(w)$ with $w\in\cU(M)$.
For $x\in F$ with $\|x\|\leq1$, the identity
\[
\|\alpha_s(w)x-\theta(x)\alpha_s(w)\|_2
=\|\alpha_{-s}(\theta(x))-\theta(\alpha_{-s}(x))\|_2
\]
allows us to make
\[
\sup_{\substack{x\in F\\\|x\|\leq1}}
\int_\R |g_m(s)|^2
\|\alpha_s(w)x-\theta(x)\alpha_s(w)\|_2^2\,\rd s
\]
arbitrarily small, along with
$\|[w,v_\ell]\|_2
=\|\theta(v_\ell)-v_\ell\|_2$ for $\ell=1,2$.
Indeed, the tails of the integral are bounded uniformly by
$4|g_m(s)|^2$, and the preceding lemma controls its compact
part as well as the equivariance error on $[-L,L]$.

Set
\[
a(r)=\int_\R e^{-2\pi\ri rs}g_m(s)\alpha_s(w)\,\rd s.
\]
Plancherel gives
\[
\int_\R a(r)^*a(r)\,\rd r
=\int_\R a(r)a(r)^*\,\rd r=1,
\qquad
\Sp_\alpha(a(r))\subset[r-m^{-1},r+m^{-1}].
\]
For every $x\in F$, Plancherel also gives
\[
\begin{aligned}
&\int_\R\|a(r)x-\theta(x)a(r)\|_2^2\,\rd r\\
&\qquad=\int_\R |g_m(s)|^2
\|\alpha_s(w)x-\theta(x)\alpha_s(w)\|_2^2\,\rd s.
\end{aligned}
\]
Since each $v_\ell$ is an eigenunitary, we likewise obtain
\[
\int_\R\|[a(r),v_\ell]\|_2^2\,\rd r
=\|[w,v_\ell]\|_2^2\qquad(\ell=1,2).
\]
Truncate to a large compact interval and partition it into
small intervals $I_i$ with midpoints $r_i$. Put
\[
a_i=|I_i|^{-1/2}\int_{I_i}a(r)\,\rd r.
\]
The operator Cauchy--Schwarz inequality gives
\[
\sum_i a_i^*a_i\leq\int_{\bigcup_i I_i}a(r)^*a(r)\,\rd r\leq1,
\]
and likewise for the row bound. It also bounds the aggregate
intertwining and commutator errors by the corresponding
integrals above. Since $a$ is continuous, its interval averages
converge to $a$ in $\rL^2$ as the partition is refined. Thus
the mass can be kept arbitrarily close to $1$.
Finally, choose the intervals so short that
$m^{-1}+|I_i|/2<\varepsilon$. Then
$\Sp_\alpha(a_i)\subset[r_i-\varepsilon,r_i+\varepsilon]$.
\end{proof}

We next localize the frequencies near the dual lattice
\[
\Lambda^\#=\{(q,s)\in\R^2\mid
qt+sp\in\Z\text{ for every }(p,t)\in\Lambda\}.
\]
By the resonance property, this lattice contains $\Lambda$. We discard the coefficients whose frequencies are not close
to $\Lambda^\#$. The small commutators show that they carry
arbitrarily little mass.

\begin{lemma}\label{non outer near lattice frequencies}
For every unital matrix subfactor $F\subset M$ and every
$L,\varepsilon>0$, there are an
automorphism $\theta$, a finite family $(a_i)_i$ in $M$, and
frequencies $r_i\in\R$ satisfying
\eqref{non outer family covariance}--\eqref{non outer family intertwining}
and
\[
\begin{gathered}
\sum_i\|a_i\|_2^2>1-\varepsilon,\\
\Sp_\alpha(a_i)\subset[r_i-\varepsilon,r_i+\varepsilon],
\qquad \operatorname{dist}((r_i,0),\Lambda^\#)<\varepsilon.
\end{gathered}
\]
\end{lemma}

\begin{proof}
Apply Lemma~\ref{non outer Fourier coefficients} with a
sufficiently small accuracy $\delta$, obtaining $\theta$,
coefficients $a_i$, and frequencies $r_i$.
Since $\alpha_{t_\ell}=\Ad(v_\ell)$, their narrow spectral
intervals and small commutators imply
\[
\sum_i
|e^{2\pi\ri r_i t_\ell}-1|^2\|a_i\|_2^2
\longrightarrow0\qquad(\delta\to0,\ \ell=1,2).
\]
The invertible linear map
\[
(q,s)\longmapsto(qt_1+sp_1,qt_2+sp_2)
\]
sends $\Lambda^\#$ onto $\Z^2$. If
$\operatorname{dist}((r_i,0),\Lambda^\#)\geq\varepsilon$, then
$(r_i t_1,r_i t_2)$ stays uniformly away from $\Z^2$, and hence
$\bigl(e^{2\pi\ri r_i t_1},e^{2\pi\ri r_i t_2}\bigr)$ stays uniformly
away from $(1,1)$. Thus, applying the preceding estimate for $\ell=1,2$ gives
\[
\sum_{\operatorname{dist}((r_i,0),\Lambda^\#)\geq\varepsilon}
\|a_i\|_2^2\longrightarrow0\qquad(\delta\to0).
\]
Taking $\delta<\varepsilon$ sufficiently small, we obtain the
conclusion by discarding the indices for which
$\operatorname{dist}((r_i,0),\Lambda^\#)\geq\varepsilon$.
\end{proof}

We now choose frequencies in $\Lambda^\#$ and arrange the shifted
row estimate needed for lattice translation. Averaging after
cancelling each coefficient's oscillating phase keeps the mass
loss small.

\begin{lemma}\label{non outer lattice frequencies}
For every unital matrix subfactor $F\subset M$ and every
$L,\varepsilon>0$, there are an
automorphism $\theta$, a finite family $(a_i)_i$ in $M$, and $(q_i,s_i)\in\Lambda^\#$ satisfying
\eqref{non outer family covariance}, \eqref{non outer family bounds}, and
\[
\begin{gathered}
\sum_i\|a_i\|_2^2>1-\varepsilon,\qquad |s_i|<\varepsilon,\\
\Sp_\alpha(a_i)\subset[q_i-\varepsilon,q_i+\varepsilon],
\qquad \sum_i\alpha_{-s_i}(a_i a_i^*)\leq1,\\
\sup_{\substack{x\in F\\\|x\|\leq1}}
\sum_i\|a_i x-\alpha_{s_i}(\theta(x))a_i\|_2^2<\varepsilon^2.
\end{gathered}
\]
\end{lemma}

\begin{proof}
Put
\[
\omega_F(T)=\sup_{\substack{|t|\leq T\\x\in F,\ \|x\|\leq1}}
\|\alpha_t(x)-x\|_2.
\]
The unit ball of $F$ is compact, so $\omega_F(T)\to0$ as
$T\to0$. Choose a small averaging radius $0<\eta\leq\min\{L/2,1\}$.
Apply Lemma~\ref{non outer near lattice frequencies} with
accuracy $\delta<\eta$ (to be chosen sufficiently small after $\eta$).
We obtain coefficients $b_i$ with frequencies $r_i$.
Choose $(q_i,s_i)\in\Lambda^\#$ with
\[
|q_i-r_i|<\delta,\qquad |s_i|<\delta.
\]
Define
\[
a_i=\frac1{2\eta}\int_{-\eta}^{\eta}
e^{-2\pi\ri r_i t}\alpha_t(b_i)\,\rd t.
\]
The operator Cauchy--Schwarz inequality gives
\[
\sum_i a_i^*a_i
\leq\frac1{2\eta}\int_{-\eta}^{\eta}
\alpha_t\left(\sum_i b_i^*b_i\right)\,\rd t
\leq1.
\]
The same argument gives $\sum_i a_i a_i^*\leq1$.
For the shifted row bound, changing variables and enlarging
the interval of integration gives
\[
\sum_i\alpha_{-s_i}(a_i a_i^*)
\leq\frac1{2\eta}\int_{-\eta-\delta}^{\eta+\delta}
\alpha_t\left(\sum_i b_i b_i^*\right)\,\rd t
\leq(1+\delta/\eta)1.
\]

Spectral localization gives
\[
\|e^{-2\pi\ri r_i t}\alpha_t(b_i)-b_i\|_2
\leq2\pi\delta|t|\,\|b_i\|_2.
\]
Since $\sum_i\|b_i\|_2^2\leq1$, it follows that
\[
\sum_i\|a_i-b_i\|_2^2
\longrightarrow0\qquad(\delta\to0).
\]
Thus the mass loss tends to zero. Averaging does not enlarge
spectral supports, so
\[
\Sp_\alpha(a_i)\subset[r_i-\delta,r_i+\delta]
\subset[q_i-2\delta,q_i+2\delta].
\]

For $x\in F$ with $\|x\|\leq1$, we have
\[
\begin{aligned}
&\|\alpha_t(b_i)x
-\alpha_{s_i}(\theta(x))\alpha_t(b_i)\|_2\\
&\hspace{2cm}=\|b_i\alpha_{-t}(x)
-\alpha_{s_i-t}(\theta(x))b_i\|_2.
\end{aligned}
\]
Split the right side into the original intertwining error,
the source translation, and the target translation.
In the averaged direct-sum $\rL^2$-norm, the first is bounded
by $\delta$, and the column bound controls the second by
$\omega_F(\eta)$. For the third, changing variables as above,
the row bound and approximate equivariance of $\theta$ give
$\sqrt{1+\delta/\eta}\bigl(\delta+\omega_F(\eta+\delta)\bigr)$,
since $\eta+\delta\leq L$. The operator Cauchy--Schwarz inequality
and the triangle inequality therefore give
\[
\begin{aligned}
&\sup_{\substack{x\in F\\\|x\|\leq1}}
\left(\sum_i\|a_i x-\alpha_{s_i}(\theta(x))a_i\|_2^2\right)^{1/2}
\\
&\qquad\leq\delta+\omega_F(\eta)
+\sqrt{1+\delta/\eta}\bigl(\delta+\omega_F(\eta+\delta)\bigr).
\end{aligned}
\]
Choosing $\eta$ small and then $\delta$ small, and renormalizing
the family, gives all the required estimates with mass greater
than $1-\varepsilon$.
The automorphism $\theta$ is unchanged, so it still satisfies
\eqref{non outer family covariance}.
\end{proof}

Put $d=|\Lambda^\#/\Lambda|$. We retain the indices whose frequencies
lie in a single coset carrying maximal total mass, and translate
the retained frequencies near zero. Recall that we identify
$(\rB(M,\alpha),\alpha|_{\rB(M,\alpha)})$ with $(R_\Lambda,\sigma^\Lambda)$ and write
$U_{p,t}$ for its defining unitaries.

\begin{lemma}\label{non outer coset coefficient translation}
For every unital matrix subfactor $F\subset M$ and every
$L,\varepsilon>0$, there are an
automorphism $\theta$ and a finite family $(a_i)_i$ in $M$
satisfying
\eqref{non outer family covariance}--\eqref{non outer family intertwining}
and
\[
\sum_i\|a_i\|_2^2>\frac1d-\varepsilon,\qquad
\Sp_\alpha(a_i)\subset[-\varepsilon,\varepsilon]
\qquad\text{for every }i.
\]
\end{lemma}

\begin{proof}
Fix a small $\delta>0$ and apply
Lemma~\ref{non outer lattice frequencies} with accuracy
$\delta$, writing $\theta_0,b_i,(q_i,s_i)$ for its output.
Among the $d$ cosets of $\Lambda$ in $\Lambda^\#$, choose one
maximizing the sum of $\|b_i\|_2^2$ over the indices $i$ for which
$(q_i,s_i)$ belongs to that coset, and discard all other indices.
Since $p_1\Z+p_2\Z$ is dense in $\R$, we can now choose a
representative $(q,s)$ of this coset with $|q|<\delta$.
All subsequent indices and sums
refer to the retained family, whose mass is greater than
$(1-\delta)/d$.
Write
\[
(q_i,s_i)=(q,s)+(p_i,t_i),\qquad (p_i,t_i)\in\Lambda,
\]
and put
\[
a_i=U_{p_i,t_i}^*b_i,\qquad \theta=\alpha_s\circ\theta_0.
\]
The column bound and mass are unchanged. Since
$\Ad(U_{p_i,t_i})=\alpha_{t_i}=\alpha_{s_i-s}$, the additional row estimate
in the preceding lemma gives
\[
\sum_i a_i a_i^*
=\alpha_s\left(\sum_i\alpha_{-s_i}(b_i b_i^*)\right)\leq1.
\]
The same identity gives
\[
\|a_i x-\theta(x)a_i\|_2
=\|b_i x-\alpha_{s_i}(\theta_0(x))b_i\|_2.
\]
Multiplication by $U_{p_i,t_i}^*$ shifts the spectral interval
from $q_i$ to $q$, so
\[
\Sp_\alpha(a_i)\subset[q-\delta,q+\delta]\subset[-2\delta,2\delta].
\]
Composing with $\alpha_s$ preserves the equivariance error and
the condition $\theta_0(F)\subset \rB(M,\alpha)$. Taking $\delta$ sufficiently
small proves the assertion.
\end{proof}

Finally, we use a Gaussian version of \emph{Haagerup's randomization
trick} \cite[proof of Lemma~2.5, pp.~272--274]{Ha89}.
Randomization and truncation turn the finite family into one
almost invariant contraction with positive mass, while preserving
approximate intertwining.

\begin{lemma}\label{non outer finite spectral branches}
There is $c>0$ such that, for every unital matrix subfactor
$F\subset M$ and every $L,\varepsilon>0$, there are an
automorphism $\theta$ satisfying \eqref{non outer family covariance}
and a contraction $a\in M$ such that
\[
\begin{gathered}
\|a\|_2^2\geq c,\qquad
\sup_{|t|\leq L}\|\alpha_t(a)-a\|_2<\varepsilon,\\
\sup_{x\in F,\,\|x\|\leq1}\|ax-\theta(x)a\|_2<\varepsilon.
\end{gathered}
\]
\end{lemma}

\begin{proof}
Fix $0<c_0<1/d$. Apply
Lemma~\ref{non outer coset coefficient translation} with
accuracy $\delta>0$, to be chosen sufficiently small, obtaining
$\theta$ and a family $(a_i)_i$ with mass greater than $c_0$ and
$\Sp_\alpha(a_i)\subset[-\delta,\delta]$ for every $i$.
We also require $\delta<\varepsilon$, so that $\theta$ satisfies
\eqref{non outer family covariance}.

For independent standard complex Gaussian variables $g_i$,
put $X=\sum_i g_i a_i$. The row and column bounds and Gaussian
moments give
\[
\mathbb E\|X\|_2^2=\sum_i\|a_i\|_2^2\leq1,
\]
\[
\begin{aligned}
\mathbb E\tau(|X|^4)
&=\tau\left(\left(\sum_i a_i^*a_i\right)^2\right)
+\tau\left(\left(\sum_i a_i a_i^*\right)^2\right)
\\
&\leq2\sum_i\|a_i\|_2^2.
\end{aligned}
\]
Let $Z$ be the $\rL^2$-metric projection of $X$ onto the
operator-norm unit ball. Thus, if $X=v|X|$, then
$Z=v\min\{|X|,1\}$.
The scalar inequality
$\min\{s^2,1\}\geq s^2-s^4/4$ gives
\[
\mathbb E\|Z\|_2^2\geq\frac12\sum_i\|a_i\|_2^2\geq c_0/2.
\]
Metric projection is $1$-Lipschitz and commutes with $\alpha_t$
and with $z\mapsto\theta(u)zu^*$ for $u\in\cU(F)$.
For every realization, $\Sp_\alpha(X)\subset[-\delta,\delta]$,
so the spectral theorem gives
\[
\sup_{|t|\leq L}\|\alpha_t(Z)-Z\|_2
\leq\sup_{|t|\leq L}\|\alpha_t(X)-X\|_2
\leq2\pi L\delta\|X\|_2.
\]
Consequently,
\[
\mathbb E\sup_{|t|\leq L}\|\alpha_t(Z)-Z\|_2^2
\leq(2\pi L\delta)^2.
\]
For the intertwining error, put
\[
J_F(z)=\int_{\cU(F)}\|zu-\theta(u)z\|_2^2\,\rd u,
\]
where $\rd u$ is normalized Haar measure.
Gaussian orthogonality and approximate intertwining give
\[
\mathbb E J_F(Z)\leq\mathbb E J_F(X)
=\sum_i J_F(a_i)<\delta^2.
\]
The triangle inequality, averaged over $\cU(F)$ for the
unitary action $z\mapsto\theta(u)zu^*$, and convexity give
an absolute constant $C$ such that
\[
\sup_{x\in F,\,\|x\|\leq1}\|zx-\theta(x)z\|_2^2\leq C J_F(z).
\]
Fix $0<c<c_0/2$. Taking $\delta$ sufficiently small gives
\[
\mathbb E\left[\|Z\|_2^2-\frac1{\varepsilon^2}
\left(\sup_{|t|\leq L}\|\alpha_t(Z)-Z\|_2^2
+C J_F(Z)\right)\right]>c.
\]
Choose an outcome $a$ for the random variable $Z$ for which the
expression in brackets is greater than $c$.
Since $\|a\|_2^2\leq1$, this gives
$\|a\|_2^2>c$ and both required error bounds.
\end{proof}

\subsubsection{Cocycle classification}

\begin{theorem}\label{non outer classification}
Let $\alpha:\R\curvearrowright M$ be a flow on a hyperfinite
$\II_1$ factor $M$ with separable predual. Assume that
$\Gamma(\alpha)=\R$ and that $\Lambda(\alpha)$ has rank two. Then
$\alpha$ is cocycle conjugate to $\sigma^{\Lambda(\alpha)}$.
\end{theorem}

\begin{proof}
Fix a free ultrafilter $\omega$ on $\N$.
Put $B=\rB(M,\alpha)$ and $\sigma=\alpha|_B$.
By condition~\textup{(iv)} of
Proposition~\ref{full bicentralizer perturbation criterion},
it suffices to construct a nonzero partial isometry
$v\in M^{\omega,\alpha}$ such that
\[
v^*v\in M'\cap M^\omega,\qquad vMv^*\subset B^\omega.
\]

\emph{Choose almost invariant approximate intertwiners.}
We first obtain contractions that approximately intertwine on
matrix subfactors, with $\rL^2$-norms bounded away from zero.
Let $c>0$ be the constant in
Lemma~\ref{non outer finite spectral branches}.
Choose increasing matrix subfactors $F_n\subset M$ with
weakly dense union. Apply
Lemma~\ref{non outer finite spectral branches} with
$F_n$, $L=n$, and $\varepsilon=1/n$.
Denote the resulting automorphisms and contractions by
$\psi_n$ and $a_n$. Then $\psi_n(F_n)\subset B$, and
\[
\|a_n\|\leq1,\qquad \|a_n\|_2^2\geq c,\qquad
\sup_{x\in F_n,\,\|x\|\leq1}
\|a_nx-\psi_n(x)a_n\|_2<\frac1n.
\]
On $[-n,n]$, the invariance error of $a_n$ is less than $1/n$.
The covariance error of $\psi_n$ on the unit ball of $F_n$
is also less than $1/n$.

\smallskip
\emph{Pass to the ultrapower.}
We use density to obtain an equivariant embedding into
$B^\omega_\sigma$ and a nonzero fixed intertwiner.
Define $\psi(x)=(\psi_n(x))^\omega$ for $x\in M$.
Each $\psi_n$ preserves products and the trace, and
\[
\operatorname{dist}_2(\psi_n(x),B)
\leq\|x-\rE_{F_n}(x)\|_2\longrightarrow0.
\]
The covariance estimates on $F_n$ extend by
$\|\cdot\|_2$-density to every $x\in M$, uniformly on compact
time intervals. Thus $\psi:M\to B^\omega_\sigma$ is a unital
normal equivariant embedding.

The element $a=(a_n)^\omega\in M^{\omega,\alpha}$ is nonzero.
For every $x\in M$,
\[
\|a_nx-\psi_n(x)a_n\|_2
\leq2\|x-\rE_{F_n}(x)\|_2+\frac{\|x\|}{n}
\longrightarrow0.
\]
Hence $ax=\psi(x)a$ for all $x\in M$.

\smallskip
\emph{Put the range in the bicentralizer.}
We move the whole range into $B^\omega$ by a fixed partial isometry.
The polar part $v_0$ of $a$ is fixed and satisfies
$v_0x=\psi(x)v_0$. Apply
Lemma~\ref{non outer relative commutant type II} to
$P=\psi(M)$ and $q=v_0v_0^*$. It gives a partial isometry
$w\in\psi(M)'\cap M^{\omega,\alpha}$ with $w^*w=q$ and
$ww^*\in B^\omega$.
Then $v=wv_0$ is a nonzero fixed partial isometry,
$v^*v=v_0^*v_0\in M'\cap M^\omega$, and
\[
vxv^*=\psi(x)ww^*\in B^\omega\qquad(x\in M).
\]

\smallskip
\emph{Conclude the classification.}
Proposition~\ref{full bicentralizer perturbation criterion}
gives $u\in Z^1(\alpha)$ with $\rB(M,\alpha^u)=M$.
Theorem~\ref{non outer bicentralizer model} then gives
\[
\alpha^u\cong\sigma^{\Lambda(\alpha^u)}
=\sigma^{\Lambda(\alpha)},
\]
so $\alpha\sim_{\mathrm{cc}}\sigma^{\Lambda(\alpha)}$.
\end{proof}

\section{Connes' bicentralizer problem}
\label{sec:connes-diagonal-averaging}

We prove Connes' bicentralizer conjecture for $\sigma$-finite type
$\III_1$ factors by adapting the preceding resonance argument to modular
theory. The proof reduces to ruling out a type $\III_1$ factor equal to its
own bicentralizer. In this case, the modular flow and the bicentralizer
flow commute and are weakly mixing. A resonance restriction on their
joint spectrum will contradict these properties.

We replace the singular tracial state by a singular KMS state on the
continuous core. The Plancherel function again identifies a second
crossed-product base, and the relative fixed point theorem gives
commutation in the GNS representation. The KMS identity controls the
modular terms that replace tracial identities. Comparing left and right
multiplication then gives the same spectral restriction as in the tracial
case, and the geometry of the resonance set completes the proof.

Let $M_0$ be a $\sigma$-finite type $\III_1$ factor and let
$\varphi_0$ be a faithful normal state. Put
$B=\rB(M_0,\varphi_0)$ and $\psi=\varphi_0|_B$.
By \cite[discussion preceding Theorem~B]{AHHM20}, if $B\ne\C 1$,
then $B$ is a type $\III_1$ factor and $\rB(B,\psi)=B$.
It therefore suffices to rule out this case. We rename $(B,\psi)$
as $(M,\varphi)$ and assume throughout the argument that
\begin{equation}\label{connes self bicentralizing reduction}
\rB(M,\varphi)=M.
\end{equation}

We retain the standard modular normalization
$\sigma_t^\varphi=\Ad(\Delta_\varphi^{\ri t})$ and write
$\beta^\varphi:\R\curvearrowright M$ for the bicentralizer flow
in additive notation. The parameter $q$ corresponds to the positive
parameter $e^{2\pi q}$ used in \cite{AHHM20, Ma20, Ma25, Ma26}.
We use this convention throughout the section, also for relative
bicentralizer flows and their unitary implementations.

The reduction makes the bicentralizer flow act on the whole working
algebra. In particular, the proof of Lemma~\ref{connes scaled smoothing}
applies $\beta^\varphi_{-p}$ to the approximating elements $a_n\in M$
to obtain an estimate uniform in $p$.

We use
\[
\|a\|_\varphi=\|a\varphi^{1/2}\|_2,
\qquad
\|a\|_\varphi^\#=
\bigl(\varphi(a^*a)+\varphi(aa^*)\bigr)^{1/2}.
\]

For an inclusion with expectation $M\subset N$, where $N$ is
$\sigma$-finite, the relative bicentralizer flow is characterized by
\begin{equation}\label{connes scaled characterization}
\|a_n\varphi-e^{2\pi q}\varphi a_n\|\longrightarrow0
\quad\Longrightarrow\quad
a_nx-\beta^{\varphi,N}_q(x)a_n\longrightarrow0
\quad\text{strongly}^*
\end{equation}
for bounded sequences $(a_n)_n$ in $M$ and
$x\in\rB(M\subset N,\varphi)$
\cite[Theorem~A(ii)]{AHHM20}.
Here $(a\varphi)(b)=\varphi(ba)$ and
$(\varphi a)(b)=\varphi(ab)$.
The relative fixed point theorem \cite{Ma26} gives
\begin{equation}\label{connes relative fixed points}
\rB(M\subset N,\varphi)^{\beta^{\varphi,N}}=M'\cap N
\end{equation}
\cite[Theorem~C]{Ma26}. In particular, we have $M^{\beta^\varphi}=\C 1$. We also have $M^\varphi=\C 1$. Moreover both actions are weakly mixing. See \cite[Theorem~A(iv)]{AHHM20} and \cite{Ma20}.

\subsection{A singular KMS state on the continuous core}
\label{subsec:connes-singular-state}

Let $(u_t)_{t\in\R}$ be the canonical unitaries in
$M\rtimes_{\sigma^\varphi}\R$, and denote its dual action by
$\theta=\widehat{\sigma^\varphi}$. As in
Subsection~\ref{subsec:plancherel-function}, write
\[
\lambda(f)=\int_\R f(t)u_t\,\rd t,
\qquad
\widehat\lambda:\rL^\infty(\R)\longrightarrow\rL(\R),
\qquad
\widehat\lambda(\check f)=\lambda(f).
\]
Our Fourier convention is
$\widehat f(p)=\int_\R f(t)e^{-2\pi\ri pt}\,\rd t$, with inverse
$f(t)=\int_\R\widehat f(p)e^{2\pi\ri pt}\,\rd p$.
In particular, $u_t(p)=e^{2\pi\ri pt}$ and
\[
\theta_q
   \bigl(\widehat\lambda(k)\bigr)
=\widehat\lambda\bigl(p\mapsto k(p+q)\bigr).
\]
We also write $\sigma_t^\varphi$ for its canonical extension
$\Ad(u_t)$ to the crossed product. The automorphism
$\beta^\varphi_q$ extends by fixing every $u_t$, and we denote
this extension by $\widetilde\beta^\varphi_q$. These extensions commute
with the dual action.

Crossing the $\sigma^\varphi$-equivariant expectation
$\varphi:M\to\C$ with $\R$ gives the faithful normal
conditional expectation
\[
\rE_{\rL(\R)}:M\rtimes_{\sigma^\varphi}\R\longrightarrow\rL(\R),
\qquad
\rE_{\rL(\R)}(au_t)=\varphi(a)u_t.
\]
We identify its range with $\rL^\infty(\R)$ through
$\widehat\lambda$. The dual weight and the canonical trace are
normalized by
\begin{equation}\label{connes core trace formula}
\begin{aligned}
\widetilde\varphi(x)
 &=\int_\R\rE_{\rL(\R)}(x)(p)\,\rd p,\\
\Tr(x)
 &=\int_\R e^{-2\pi p}\rE_{\rL(\R)}(x)(p)\,\rd p
\end{aligned}
\qquad(x\in(M\rtimes_{\sigma^\varphi}\R)_+).
\end{equation}
Thus $\widetilde\varphi=\varphi\circ\rT$ and its modular group
is the extension $\sigma^\varphi=\Ad(u_t)$. The dual action satisfies
$\Tr\circ\theta_q=e^{2\pi q}\Tr$.

Fix a free ultrafilter $\mathcal V$ on $\N$. For integers $T\geq1$, put
\begin{equation}\label{connes invariant mean}
m_T(f)=\frac1{2T}\int_{-T}^T f(p)\,\rd p,
\qquad
m(f)=\lim_{T\to\mathcal V}m_T(f).
\end{equation}
The mean $m$ is invariant under translations and reflection.
Define normal states $\rho_T$ and a state $\rho$ by
\begin{equation}\label{connes weighted states}
\rho_T=m_T\circ\rE_{\rL(\R)},
\qquad
\rho=m\circ\rE_{\rL(\R)}.
\end{equation}
Their restrictions to $M$ are $\varphi$. We write
$\|x\|_{2,\rho}=\rho(x^*x)^{1/2}$.

Let $D$ be the self-adjoint generator determined by
$u_t=e^{\ri tD}$. Put $P=D/(2\pi)$ and $e_T=1_{[-T,T]}(P)$.
The Fourier coordinate is the spectral variable of $P$, so
\eqref{connes core trace formula} gives
\[
\rho_T(x)=\frac1{2T}
\Tr\bigl(e_Te^{D/2}xe^{D/2}e_T\bigr).
\]
Thus the density $e^D$ cancels the spectral weight in the trace.
Write $(M\rtimes_{\sigma^\varphi}\R)_c$ for the elements with a
norm-continuous $\sigma^\varphi$-orbit.

\begin{proposition}\label{connes KMS realization}
The restriction of $\rho$ to
$(M\rtimes_{\sigma^\varphi}\R)_c$ is a KMS state for
$\sigma^\varphi$. More precisely, if $x$ has compact
$\sigma^\varphi$-spectrum and $y\in M\rtimes_{\sigma^\varphi}\R$,
then
\begin{equation}\label{connes KMS identity}
\rho(xy)=\rho\bigl(y\sigma_{-\ri}^\varphi(x)\bigr).
\end{equation}
\end{proposition}

\begin{proof}
The densities $e_Te^D$ commute with the unitaries $u_t$, so
$\rho_T$ and $\rho$ are $\sigma^\varphi$-invariant.
Suppose that $\Sp_{\sigma^\varphi}(x)\subset[-L,L]$ and $T>L$.
Spectral transfer gives
\[
e_Tx(1-e_T)x^*e_T
\leq\|x\|^2 1_{[-T,-T+L]\cup[T-L,T]}(P),
\]
and the same estimate holds with $x$ replaced by $x^*$.
The state $\rho_T$ assigns mass $L/T$ to the displayed boundary
projection. Since it is supported on $e_T$, Cauchy--Schwarz gives
\begin{equation}\label{connes KMS boundary estimate}
\bigl|\rho_T(xy)-\rho_T(e_Txe_Tye_T)\bigr|
\leq\|x\|\|y\|\sqrt{L/T}.
\end{equation}

On the corner $e_T(M\rtimes_{\sigma^\varphi}\R)e_T$, the
faithful normal state $\rho_T$ has density $(1/(2T))e_Te^D$.
Its modular group is therefore the restriction of
$\sigma^\varphi$. The element $e_Txe_T$ is entire analytic, and
the KMS identity on this corner gives
\[
\rho_T(e_Txe_Tye_T)
=\rho_T\bigl(e_Tye_T\sigma_{-\ri}^\varphi(x)e_T\bigr).
\]
The second boundary estimate, applied to
$\sigma_{-\ri}^\varphi(x)$, bounds the difference between this
last expression and $\rho_T(y\sigma_{-\ri}^\varphi(x))$ by
$\|y\|\|\sigma_{-\ri}^\varphi(x)\|\sqrt{L/T}$.
Taking the ultralimit proves \eqref{connes KMS identity}.
Compact-spectrum elements form a norm dense invariant analytic
$*$-subalgebra of $(M\rtimes_{\sigma^\varphi}\R)_c$, which proves
the KMS assertion.
\end{proof}

\subsection{The Plancherel function}
\label{subsec:connes-plancherel-function}

All Hilbert space valued functions in this section are strongly measurable
and therefore have essentially separable range. For any Hilbert space
$H$, reflexivity and the Radon--Nikod\'ym property identify
$\rL^\infty(\R,H)$ with the dual of $\rL^1(\R,H^*)$ under the pairing
$(F,g)\mapsto\int_\R g(p)(F(p))\,\rd p$.
This duality holds even when $H$ is nonseparable.

We use the standard left and right actions of $M$ on
$\rL^2(M,\varphi)$, with inner product linear in the first variable.
For an entire analytic element $a\in M$, the right action satisfies
\begin{equation}\label{connes modular right action}
\varphi^{1/2}\sigma_{\ri/2}^\varphi(a)=a\varphi^{1/2}.
\end{equation}
This identity accounts for the modular translate in the product
formula below.

\begin{proposition}\label{connes modular Plancherel}
There is a unique weak*-weak* continuous linear map
\[
\vartheta:M\rtimes_{\sigma^\varphi}\R
\longrightarrow\rL^\infty(\R,\rL^2(M,\varphi))
\]
such that
\[
\vartheta(a\widehat\lambda(k))(p)=k(p)a\varphi^{1/2}
\qquad(a\in M,\ k\in\rL^\infty(\R)).
\]
We call $\vartheta(x)$ the Plancherel function of $x$.
For $x,y\in M\rtimes_{\sigma^\varphi}\R$, it satisfies
\[
\langle\vartheta(x),\vartheta(y)\rangle
=\rE_{\rL(\R)}(y^*x),
\qquad
\|\vartheta(x)\|_\infty\leq\|x\|.
\]
In particular, $\vartheta$ is injective and
\begin{equation}\label{connes Plancherel mean identity}
\|x\|_{2,\rho}^2
=m\bigl(p\mapsto\|\vartheta(x)(p)\|_2^2\bigr).
\end{equation}
It agrees with the Plancherel identification
\[
\rL^2(M\rtimes_{\sigma^\varphi}\R,\widetilde\varphi)
\longrightarrow
\rL^2\bigl(\R,\rL^2(M,\varphi),\rd p\bigr)
\]
on bounded square-integrable elements. It also satisfies
\begin{equation}\label{connes Plancherel covariance}
\begin{aligned}
\vartheta(ax)&=a\vartheta(x),\\
\vartheta(x\widehat\lambda(k))&=k\vartheta(x),\\
\vartheta(\theta_q(x))(p)
 &=\vartheta(x)(p+q),\\
\vartheta(\sigma_t^\varphi(x))(p)
 &=\Delta_\varphi^{\ri t}\vartheta(x)(p),\\
\vartheta(\widetilde\beta^\varphi_q(x))(p)
 &=U^{\beta^\varphi}_q\vartheta(x)(p).
\end{aligned}
\end{equation}
Here $U^{\beta^\varphi}_q(a\varphi^{1/2})
=\beta^\varphi_q(a)\varphi^{1/2}$.
\end{proposition}

\begin{proof}
For $f\in\rL^1(\R)$ and $a\in M$, conditional
Cauchy--Schwarz gives
\[
\left|\int_\R f(p)\rE_{\rL(\R)}(a^*x)(p)\,\rd p\right|
\leq\|f\|_1\|a\varphi^{1/2}\|_2\|x\|.
\]
The projective tensor description of
$\rL^1(\R,\rL^2(M,\varphi)^*)$ and the duality above therefore
give a unique contractive map satisfying
\[
\int_\R f(p)
\langle\vartheta(x)(p),a\varphi^{1/2}\rangle\,\rd p
=\int_\R f(p)\rE_{\rL(\R)}(a^*x)(p)\,\rd p.
\]
All the functionals on the right are normal. Their continuous
extension to the space of integrable test functions proves weak*
continuity of $\vartheta$.
The equality
$\rE_{\rL(\R)}(a^*b\widehat\lambda(k))=\varphi(a^*b)k$
gives the required formula on products. These products span a
weak$^*$ dense subspace, proving uniqueness.

For $x=a\widehat\lambda(f)$ and $y=b\widehat\lambda(g)$,
the pointwise inner product is
\[
f\overline g\,\varphi(b^*a)
=\rE_{\rL(\R)}(y^*x).
\]
After integration against an $\rL^1$ function, both sides are
separately weak* continuous in $x$ and $y$. The identity therefore
holds for all bounded $x,y$. Faithfulness of the expectation gives
injectivity, and applying $m$ gives
\eqref{connes Plancherel mean identity}.
Integration against $\rd p$ gives the dual-weight
$\rL^2$ norm. On the dense set of vectors $a\lambda(f)$ with
$f\in\rL^1(\R)\cap\rL^2(\R)$, the image is
$\check f\,a\varphi^{1/2}$, proving the Plancherel assertion.
Finally, the module and covariance formulas hold on products and
extend by weak* continuity.
\end{proof}

\begin{proposition}\label{connes modular right product}
Let $x\in M\rtimes_{\sigma^\varphi}\R$ have compact dual
spectrum, and let $a\in M$ have compact modular spectrum.
Choose $h,k\in\mathcal D(\R)$ equal to $1$ near
$\Sp_{\theta}(x)$ and
$\Sp_{\sigma^\varphi}(a)$, respectively, and put
\[
\Phi(t,r)=h(t)k(r)e^{2\pi\ri rt}.
\]
Using our Fourier convention on $\R^2$ with Lebesgue measure, we have
\begin{equation}\label{connes product coefficient}
\vartheta(xa)(p)
=\int_{\R^2}\widehat\Phi(q,s)
\vartheta(x)(p+q)\sigma_{s+\ri/2}^\varphi(a)
\,\rd q\,\rd s.
\end{equation}
The integral converges in $\rL^2(M,\varphi)$ and defines a bounded
measurable function.
\end{proposition}

\begin{proof}
Compact modular spectrum makes $a$ entire analytic. The integrand
is bounded in norm by
$|\widehat\Phi(q,s)|\|x\|\|\sigma_{\ri/2}^\varphi(a)\|$,
so the integral converges. After testing against an integrable
$\rL^2(M,\varphi)$-valued function, the right-hand side is normal in
$x$. Indeed, translation of the test function and right multiplication
by $\sigma_{s+\ri/2}^\varphi(a)^*$ preserve its $\rL^1$ norm
up to the uniform factor $\|\sigma_{\ri/2}^\varphi(a)\|$.
Fubini's theorem gives a single integrable test function.

For $x=cu_t$, the dual frequency is $t$. Fourier inversion and the
cutoff property of $k$ give
\[
\int_{\R^2}\widehat\Phi(q,s)e^{2\pi\ri qt}
\sigma_{s+\ri/2}^\varphi(a)\,\rd q\,\rd s
=h(t)\sigma_{t+\ri/2}^\varphi(a).
\]
By \eqref{connes modular right action}, the proposed right-hand
side is therefore $h(t)\vartheta(cu_ta)(p)$.
Define
\[
x_h=\int_\R\widehat h(q)\theta_q(x)\,\rd q,
\]
with the integral taken in the weak$^*$ sense. Fourier inversion
and $\theta_q(cu_t)=e^{2\pi\ri qt}cu_t$ give
$(cu_t)_h=h(t)cu_t$.
Thus the normal map defined by the right-hand side agrees with
$x\mapsto\vartheta(x_ha)$ on every $cu_t$, and therefore on the
whole crossed product. Since $h=1$ near the dual spectrum of the given
$x$, we have $x_h=x$, proving the formula.
\end{proof}

\subsection{The bicentralizer in the GNS representation}
\label{subsec:connes-gns}

Let $M^\dagger$ be the fixed point algebra in
$M\rtimes_{\sigma^\varphi}\R$ of the action
\[
q\longmapsto\widetilde\beta^\varphi_q\circ\theta_q.
\]
The modular and dual actions preserve $M^\dagger$.
As in Subsection~\ref{subsec:bicentralizer-gns}, this fixed algebra
gives a second crossed-product decomposition.

\begin{proposition}\label{connes dagger state and orbits}
There is a faithful normal state $\varphi^\dagger$ on $M^\dagger$
such that
\[
\rE_{\rL(\R)}(x)=\varphi^\dagger(x)1,
\qquad
\rho_T(x)=\rho(x)=\varphi^\dagger(x)
\qquad(x\in M^\dagger).
\]
Its modular group is $\sigma^{\varphi^\dagger}_t
=\Ad(u_t)|_{M^\dagger}$, and
\[
M\rtimes_{\sigma^\varphi}\R
=M^\dagger\rtimes_{\sigma^{\varphi^\dagger}}\R.
\]
For every $x\in M^\dagger$, there is a unique
$\xi_x\in\rL^2(M,\varphi)$ such that
\[
\vartheta(x)=\bigl[p\longmapsto
U^{\beta^\varphi}_{-p}\xi_x\bigr].
\]
The map $x(\varphi^\dagger)^{1/2}\mapsto\xi_x$ extends to a unitary
\[
V:\rL^2(M^\dagger,\varphi^\dagger)\longrightarrow\rL^2(M,\varphi),
\qquad
V(\varphi^\dagger)^{1/2}=\varphi^{1/2}.
\]
It intertwines the modular unitary groups. The dual action preserves
$M^\dagger$ and $\varphi^\dagger$, and
\[
\xi_{\theta_q(x)}=U^{\beta^\varphi}_{-q}\xi_x.
\]
\end{proposition}

\begin{proof}
The expectation intertwines
$\widetilde\beta^\varphi_q\circ\theta_q$
with the translation $p\mapsto p+q$ on $\rL^\infty(\R)$.
Thus its restriction to $M^\dagger$ is scalar.
Faithfulness and normality of the expectation give the state
$\varphi^\dagger$, and the definitions of $\rho_T$ and $\rho$
give the restriction formulas. Bimodularity of the expectation
also gives $\varphi^\dagger\circ\Ad(u_t)=\varphi^\dagger$.

The defining action of $M^\dagger$ sends $u_t$ to
$e^{2\pi\ri qt}u_t$. The crossed-product characterization
\cite[Theorem~8.1]{Na77} therefore gives a decomposition with
coefficient algebra $M^\dagger$ and implementing group $u_t$.
Its Haar operator-valued weight is
\[
\rT_{M^\dagger}
=\int_\R
\widetilde\beta^\varphi_q\circ\theta_q\,\rd q.
\]
For positive $z$ in the continuous core, normality and equivariance
of the expectation give
\[
\varphi^\dagger(\rT_{M^\dagger}(z))
=\int_\R\rE_{\rL(\R)}(z)(p)\,\rd p
=\varphi(\rT(z)).
\]
These are equalities of extended positive values.
The modular group of the last dual weight is $\Ad(u_t)$.
Restriction to $M^\dagger$ identifies its modular group and proves
the stated crossed-product decomposition.

For $x\in M^\dagger$, Proposition~\ref{connes modular Plancherel}
gives
\[
\vartheta(x)(p)
=U^{\beta^\varphi}_q\vartheta(x)(p+q).
\]
Consequently the measurable function
$p\mapsto U^{\beta^\varphi}_p\vartheta(x)(p)$ is
translation invariant and is therefore constant almost everywhere.
Call its value $\xi_x$. The inner-product identity gives
\[
\|\xi_x\|_2^2=\varphi^\dagger(x^*x),
\]
so $V$ extends to an isometry.

The two crossed-product decompositions have the same dual weight,
as proved above. In its GNS space, the vectors represented by
$x\widehat\lambda(f)$, where $x\in M^\dagger$ and
$f\in\rL^2(\R)\cap\rL^\infty(\R)$, span a dense subspace.
Their Plancherel functions are
\[
p\longmapsto f(p)U^{\beta^\varphi}_{-p}\xi_x.
\]
Suppose that $\eta\in\rL^2(M,\varphi)$ is orthogonal to the range of $V$,
and choose a nonzero $g\in\rL^2(\R)\cap\rL^\infty(\R)$.
The function $p\mapsto g(p)U^{\beta^\varphi}_{-p}\eta$ is orthogonal
to every function in the preceding dense subspace. It is therefore
zero, and hence $\eta=0$. The isometry $V$ has closed range, so
it is onto. The identity $\vartheta(1)=\varphi^{1/2}$ gives its
value at the state vector. Covariance under $\Ad(u_t)$ gives
modular intertwining. The dual covariance gives the last identity
and shows that the dual action preserves $\varphi^\dagger$.
\end{proof}

We now pass to the GNS representation of the singular KMS state.
Let $\pi$ be the GNS representation of
$\rho|_{(M\rtimes_{\sigma^\varphi}\R)_c}$, let $\xi_\rho$ be its
cyclic vector, and put
\[
N=\pi\bigl((M\rtimes_{\sigma^\varphi}\R)_c\bigr)''.
\]

\begin{proposition}\label{connes normal GNS embeddings}
The vector state of $\xi_\rho$ is faithful and normal on $N$.
We denote it again by $\rho$. The representations of the modular
continuous parts of $M$ and $M^\dagger$ extend to faithful normal
embeddings of these algebras into $N$. Their state restrictions
are $\varphi$ and $\varphi^\dagger$, respectively, and both
embeddings intertwine the modular groups. In particular, the
inclusion $M\subset N$ is with expectation.
For $z\in(M\rtimes_{\sigma^\varphi}\R)_c$ with modular spectrum
contained in $[r_0,r_1]$, one has
\begin{equation}\label{connes KMS spectral estimates}
e^{2\pi r_0}\rho(z^*z)\leq\rho(zz^*)\leq e^{2\pi r_1}\rho(z^*z).
\end{equation}
\end{proposition}

\begin{proof}
The right multiplication construction for a KMS state gives
\begin{equation}\label{connes KMS right multiplication}
\|\pi(b)\pi(a)\xi_\rho\|
\leq\|\sigma_{\ri/2}^\varphi(a)\|
\|\pi(b)\xi_\rho\|
\end{equation}
for entire analytic $a$ and arbitrary $b$ in the modular continuous
part of the core. These right multiplication operators commute
with the left representation and have a dense orbit on
$\xi_\rho$. Thus $\xi_\rho$ is separating for $N$, and its vector
state is faithful and normal. The KMS identity identifies its
modular group with the extension of $\sigma^\varphi$.

Take a bounded net $(b_i)_i$ in the modular continuous part of
$M$ which converges strongly$^*$ to zero in $M$.
For each compact-spectrum $a$ in the core,
\eqref{connes KMS right multiplication} gives
\[
\|\pi(b_i)\pi(a)\xi_\rho\|
\leq\|\sigma_{\ri/2}^\varphi(a)\|\,\|b_i\|_\varphi
\longrightarrow0.
\]
The same holds for $b_i^*$. Density and uniform boundedness give
$\pi(b_i)\to0$ strongly$^*$. Bounded modular convolutions are
strongly$^*$ dense in $M$ on bounded sets, so $\pi$ extends
normally to $M$. The argument for $M^\dagger$ is identical,
using $\rho|_{M^\dagger}=\varphi^\dagger$.
Both extensions are faithful because their state restrictions are
faithful, and both intertwine the modular groups.
Takesaki's conditional expectation theorem
\cite[Theorem~${\rm IX}$.4.2]{Ta03} gives the expectation onto $M$.

The vector $\pi(z)\xi_\rho$ has spectral support in $[2\pi r_0,2\pi r_1]$
for $\log\Delta_\rho$. The identity
\[
\rho(zz^*)=\|\Delta_\rho^{1/2}\pi(z)\xi_\rho\|^2
\]
proves \eqref{connes KMS spectral estimates}.
\end{proof}

We identify $M$ with its normal image in $N$, and keep $\pi(x)$
for the image of $x\in M^\dagger$.

\begin{lemma}\label{connes scaled smoothing}
Fix $q\in\R$ and a bounded sequence $(a_n)_n$ in $M$ satisfying
\[
\|a_n\varphi-e^{2\pi q}\varphi a_n\|\longrightarrow0.
\]
Then
\[
\sup_{t\in K}\|\sigma_t^\varphi(a_n)-e^{-2\pi\ri qt}a_n\|_\varphi^\#
\longrightarrow0
\qquad(K\Subset\R).
\]
The sequence can be replaced, up to an error tending to zero
strongly$^*$, by a bounded sequence with modular spectrum contained
in $[-q-1,-q+1]$ and with the same scaled predual convergence.
For every $b\in M$, one also has
\[
\sup_{p\in\R}
\|a_n\beta^\varphi_p(b)
-\beta^\varphi_{p+q}(b)a_n\|_\varphi^\#
\longrightarrow0.
\]
\end{lemma}

\begin{proof}
On $\mathbb M_2(M)$, take the state
$(e^{2\pi q}\varphi\oplus\varphi)/(1+e^{2\pi q})$ and the elements
$e_{12}\otimes a_n$. Their predual commutators tend to zero.
The square-root estimate \cite[Lemma~2.8(b)]{Ha87}, applied also
to the adjoints, gives asymptotic commutation with the square root
of this state. Spectral calculus and
\[
\sup_{t\in K}|v^{\ri t}-1|\leq C_K|v^{1/2}-1|
\qquad(v>0)
\]
give compact-uniform modular invariance in the matrix algebra.
Its modular group acts on the off-diagonal corner by
$e^{2\pi\ri qt}\sigma_t^\varphi$, proving the first assertion.

Choose a nonnegative Schwartz function $f$ of integral one with
$\supp\check f\subset[-1,1]$, and replace $a_n$ by
\[
\int_\R f(t)e^{2\pi\ri qt}\sigma_t^\varphi(a_n)\,\rd t.
\]
The spectral cutoff formula gives the required spectrum and a
uniform norm bound. The first assertion and an $\rL^1$ tail
estimate give convergence of the difference strongly$^*$.
State preservation shows that the scaled predual defect does not
increase.

For the last assertion, apply $\beta^\varphi_{-p}$ to the
displayed error. The resulting elements
$\beta^\varphi_{-p}(a_n)$ have the same norm bound and the
same scaled predual defect, independently of $p$. Failure of
uniform convergence would give a sequence $p_n$ contradicting
the defining characterization of $\beta^\varphi_q$, since
$\rB(M,\varphi)=M$.
\end{proof}

\begin{theorem}\label{connes mean value commutation}
The normal images of $M$ and $M^\dagger$ in the GNS von Neumann
algebra of $\rho$ commute.
\end{theorem}

\begin{proof}
First take $x\in M^\dagger$ with compact spectrum for both
$\theta$ and $\Ad(u_t)$. Fix $q\in\R$ and
a bounded sequence $(a_n)_n$ satisfying
$\|a_n\varphi-e^{2\pi q}\varphi a_n\|\to0$.
We prove that $[a_n,\pi(x)]\to0$ strongly$^*$ in $N$.
By Lemma~\ref{connes scaled smoothing} and
normality of the embedding of $M$, we may assume
\[
\Sp_{\sigma^\varphi}(a_n)\subset[-q-1,-q+1].
\]
These elements are entire analytic and
\[
\sup_n\|\sigma_{\ri/2}^\varphi(a_n)\|<\infty.
\]
For example, choose a fixed Schwartz spectral cutoff equal to
one on this interval and apply its convolution formula at
$\ri/2$.

Write $F(p)=\vartheta(x)(p)=U^{\beta^\varphi}_{-p}\xi_x$.
Choose $h,k\in\mathcal D(\R)$ equal to one near
$\Sp_{\theta}(x)$ and $[-q-1,-q+1]$, respectively,
and put
\[
\Phi(t,r)=h(t)k(r)e^{2\pi\ri rt}.
\]
Proposition~\ref{connes modular right product} gives
\[
\vartheta(xa_n)(p)
=\int_{\R^2}\widehat\Phi(r,s)
F(p+r)\sigma^\varphi_{s+\ri/2}(a_n)\,\rd r\,\rd s.
\]
Right multiplication here is the canonical bounded right action
on $\rL^2(M,\varphi)$. The modular shift is necessary because
$\varphi^{1/2}\sigma_{\ri/2}^\varphi(a)=a\varphi^{1/2}$.

Set
\[
d_{n,s}=\sigma^\varphi_{s+\ri/2}(a_n)
-e^{-2\pi\ri qs}\sigma^\varphi_{\ri/2}(a_n).
\]
These elements are uniformly bounded in operator norm.
For $b\in M$, we have
\[
\begin{aligned}
\|F(p+r)d_{n,s}\|_2
&\leq C\|\xi_x-b\varphi^{1/2}\|_2\\
&\quad+\|b\|\,
\|\sigma^\varphi_s(a_n)-e^{-2\pi\ri qs}a_n\|_\varphi.
\end{aligned}
\]
Indeed, the right action is bounded by $\|d_{n,s}\|$, and
\[
\varphi^{1/2}d_{n,s}
=\bigl(\sigma^\varphi_s(a_n)-e^{-2\pi\ri qs}a_n\bigr)
\varphi^{1/2}.
\]
The estimate is uniform in $p,r$. Integrate against
$|\widehat\Phi(r,s)|\,\rd r\,\rd s$, use
compact-uniform modular convergence and an $\rL^1$ tail estimate,
and then approximate $\xi_x$ by $b\varphi^{1/2}$. It follows that
replacing $\sigma^\varphi_{s+\ri/2}(a_n)$ in the product formula
by $e^{-2\pi\ri qs}\sigma^\varphi_{\ri/2}(a_n)$ changes the function by a
quantity tending to zero in essential supremum norm.

Fourier inversion gives
\[
\int_\R\widehat\Phi(r,s)e^{-2\pi\ri qs}\,\rd s
=\widehat h(r+q).
\]
Since $h$ equals one near the dual spectrum of $x$, its spectral
cutoff acts as the identity on $F$. Hence
\[
\operatorname*{ess\,sup}_{p\in\R}
\|\vartheta(xa_n)(p)
-F(p-q)\sigma^\varphi_{\ri/2}(a_n)\|_2
\longrightarrow0.
\]
The last part of Lemma~\ref{connes scaled smoothing}
also gives
\[
\sup_{p\in\R}
\|a_nF(p)-F(p-q)\sigma^\varphi_{\ri/2}(a_n)\|_2
\longrightarrow0.
\]
To justify this for $\xi_x$, approximate it by $b\varphi^{1/2}$.
The approximation errors are uniform in $p$ because both
$(a_n)_n$ and $(\sigma^\varphi_{\ri/2}(a_n))_n$ are uniformly
bounded. On the approximating vector, the expression is
\[
\bigl(a_n\beta^\varphi_{-p}(b)
-\beta^\varphi_{-p+q}(b)a_n\bigr)\varphi^{1/2}.
\]
The Plancherel norm identity and left $M$-linearity give
\[
\rho([a_n,x]^*[a_n,x])\longrightarrow0.
\]

All the commutators have their $\Ad(u_t)$-spectra in one fixed
compact interval. The KMS spectral estimates give
$\rho([a_n,x][a_n,x]^*)\to0$ as well. Uniform boundedness and
faithfulness of the normal state on $N$ imply the desired
strong$^*$ convergence. The same conclusion holds for the original
sequence before smoothing, by normality of $M\subset N$.

Taking $q=0$ shows that
$\pi(x)\in\rB(M\subset N,\varphi)$.
For every $q$, the same convergence and pointwise uniqueness in
the characterization of the relative bicentralizer flow
\cite[Lemma~3.4 and proof of Theorem~A(ii)]{AHHM20} show that
\[
\beta^{\varphi,N}_q(\pi(x))=\pi(x).
\]
The relative fixed point theorem \cite[Theorem~C]{Ma26} gives
$\pi(x)\in M'\cap N$.

Both commuting actions preserve $M^\dagger$ and
$\varphi^\dagger$. Bounded convolutions in the two variables
approximate any element of $M^\dagger$ strongly$^*$ by elements
with compact spectra for both actions. The normal embedding of
$M^\dagger$ and strong closure of $M'\cap N$ finish the proof.
\end{proof}

\subsection{Resonance for the joint bicentralizer action}
\label{subsec:connes-resonance}

Let $\mathcal H_m$ be the Hilbert space obtained by separation and
completion of $\rL^\infty(\R,\rL^2(M,\varphi))$ for the seminorm
\[
\|F\|_m^2=m\bigl(p\mapsto\|F(p)\|_2^2\bigr).
\]
For $a,c\in M$, consider the functions
\[
\begin{aligned}
L(a\varphi^{1/2}\otimes c\varphi^{1/2})(p)
 &=a\beta^\varphi_{-p}(c)\varphi^{1/2},\\
R(a\varphi^{1/2}\otimes c\varphi^{1/2})(p)
 &=\beta^\varphi_{-p}(c)a\varphi^{1/2}.
\end{aligned}
\]

\begin{proposition}\label{connes coefficient isometries}
The maps $L$ and $R$ extend to isometries from
$\rL^2(M,\varphi)\otimes\rL^2(M,\varphi)$ into $\mathcal H_m$.
The formula
\[
I(F)(p)=U^{\beta^\varphi}_{-p}F(-p)
\]
defines a linear unitary involution on $\mathcal H_m$.
If $S$ is the tensor flip, then $IR=LS$ and $IL=RS$.
For every $\xi\in\rL^2(M,\varphi)$ and every entire analytic $a\in M$,
\[
\begin{aligned}
L(a\varphi^{1/2}\otimes\xi)
 &=[p\mapsto aU^{\beta^\varphi}_{-p}\xi],\\
R(a\varphi^{1/2}\otimes\xi)
 &=[p\mapsto (U^{\beta^\varphi}_{-p}\xi)\sigma^\varphi_{\ri/2}(a)].
\end{aligned}
\]
The multiplication on the right in the last line is the standard
right action of $M$ on $\rL^2(M,\varphi)$.
\end{proposition}

\begin{proof}
The bounded Ces\`aro averages of $\beta^\varphi_{-p}(z)$ converge weak$^*$
to $\varphi(z)1$, by ergodicity and the mean ergodic theorem.
Consequently, for $a,b,c,d\in M$,
\[
\begin{aligned}
&\langle R(a\varphi^{1/2}\otimes c\varphi^{1/2}),
 R(b\varphi^{1/2}\otimes d\varphi^{1/2})\rangle_m\\
&\qquad=m\bigl(p\mapsto
\varphi(b^*\beta^\varphi_{-p}(d^*c)a)\bigr)
=\varphi(b^*a)\varphi(d^*c).
\end{aligned}
\]
Thus $R$ is isometric. Reflection invariance of $m$ and
unitarity of $U^{\beta^\varphi}_{-p}$ show that $I$ is isometric, and $I^2=1$.
Direct calculation gives $IR=LS$, so $L$ is also isometric
and $IL=RS$.

For analytic $a$, the modular identity
\[
(c\varphi^{1/2})\sigma^\varphi_{\ri/2}(a)
=ca\varphi^{1/2}
\]
gives the last formula on vectors $\xi=c\varphi^{1/2}$.
The uniform bounds
\[
\sup_p\|aU^{\beta^\varphi}_{-p}\xi\|_2\leq\|a\|\|\xi\|_2,
\qquad
\sup_p\|(U^{\beta^\varphi}_{-p}\xi)\sigma^\varphi_{\ri/2}(a)\|_2
\leq\|\sigma^\varphi_{\ri/2}(a)\|\|\xi\|_2
\]
allow approximation by such vectors and prove both formulas.
\end{proof}

Let $\mathbb{E}_{\sigma^\varphi}$ and $\mathbb{E}_{\beta^\varphi}$ be the spectral measures
on $\rL^2(M,\varphi)$ with
\[
\Delta_\varphi^{\ri r}
=\int_\R e^{2\pi\ri rp}\,\rd \mathbb{E}_{\sigma^\varphi}(p),
\qquad
U^{\beta^\varphi}_q=\int_\R e^{2\pi\ri qt}\,\rd \mathbb{E}_{\beta^\varphi}(t).
\]
Define the unitary
\[
\Omega=\int_{\R^2}e^{2\pi\ri pt}
\,\rd(\mathbb{E}_{\sigma^\varphi}\otimes \mathbb{E}_{\beta^\varphi})(p,t),
\]
where the two spectral measures act on the first and second tensor
factors, respectively.

\begin{proposition}\label{connes tensor commutation}
We have $L=R\Omega^*$.
\end{proposition}

\begin{proof}
Take $a\in M$ with compact $\sigma^\varphi$-spectrum and
$x\in M^\dagger$ in the modular continuous part with compact
$\theta$-spectrum. Write
$\xi_x=V(x(\varphi^\dagger)^{1/2})$.
The covariance of $\vartheta$ gives
\[
\xi_{\theta_q(x)}=U^{\beta^\varphi}_{-q}\xi_x,
\]
so the $\beta^\varphi$-spectrum of $\xi_x$ is contained in
$-\Sp_{\theta}(x)$.
Choose $h,k\in\mathcal D(\R)$ equal to $1$ near
$\Sp_{\theta}(x)$ and
$\Sp_{\sigma^\varphi}(a)$, respectively, and put
\[
\Phi(t,r)=h(t)k(r)e^{2\pi\ri rt}.
\]
With our Fourier convention on $\R^2$, spectral calculus gives
\[
\Omega^*(a\varphi^{1/2}\otimes\xi_x)
=\int_{\R^2}\widehat\Phi(q,s)
\sigma^\varphi_s(a)\varphi^{1/2}
\otimes U^{\beta^\varphi}_{-q}\xi_x\,\rd q\,\rd s.
\]
Proposition~\ref{connes coefficient isometries} and the right multiplication
formula in Proposition~\ref{connes modular right product} give
\[
\begin{aligned}
R\Omega^*(a\varphi^{1/2}\otimes\xi_x)
&=\left[p\mapsto
\int_{\R^2}\widehat\Phi(q,s)
(U^{\beta^\varphi}_{-p-q}\xi_x)\sigma^\varphi_{\ri/2+s}(a)
\,\rd q\,\rd s\right]\\
&=\vartheta(xa).
\end{aligned}
\]
These integrals converge uniformly in the $\rL^2(M,\varphi)$ norm,
since their integrands have norm at most
$|\widehat\Phi(q,s)|\|\xi_x\|_2
\|\sigma^\varphi_{\ri/2}(a)\|$.
Left $M$-linearity gives
$L(a\varphi^{1/2}\otimes\xi_x)=\vartheta(ax)$.
Theorem~\ref{connes mean value commutation} yields
\[
\|\vartheta(ax)-\vartheta(xa)\|_m
=\|ax-xa\|_{2,\rho}=0.
\]
Smoothing with respect to the two commuting flows makes the
allowed $x$ dense in $\rL^2(M^\dagger,\varphi^\dagger)$.
Modular smoothing makes the allowed $a$ dense in
$\rL^2(M,\varphi)$. Surjectivity of $V$ and continuity prove the
operator identity.
\end{proof}

\begin{proposition}\label{connes joint resonance}
The joint spectrum $\Sigma=\Sp(\sigma^\varphi\times\beta^\varphi)$ satisfies
\begin{equation}\label{connes pairwise resonance}
pt'+p't\in\Z
\qquad((p,t),(p',t')\in\Sigma).
\end{equation}
In particular, $pt\in\tfrac12\Z$ for every $(p,t)\in\Sigma$.
\end{proposition}

\begin{proof}
Propositions~\ref{connes coefficient isometries} and~\ref{connes tensor commutation} give
$IR=LS$, $IL=RS$, and $L=R\Omega^*$.
Hence
\[
LS\Omega^*=IR\Omega^*=IL=RS=L\Omega S.
\]
Injectivity of $L$ gives $S\Omega^*=\Omega S$, or equivalently
$\Omega(S\Omega S)=1$.

Let $\mathbb{E}$ be the joint spectral measure on
$\rL^2(M,\varphi)$. Faithfulness of $\varphi$ identifies its support
with $\Sigma$.
The operator $\Omega(S\Omega S)$ is the image of the continuous
function
\[
((p,t),(p',t'))\longmapsto e^{2\pi\ri(pt'+p't)}
\]
under the spectral calculus for $\mathbb{E}\otimes\mathbb{E}$.
The support of this product spectral measure is
$\Sigma\times\Sigma$, so the function equals $1$ throughout this
set. Taking the same point twice proves the last assertion.
\end{proof}

\begin{theorem}\label{Connes bicentralizer problem}
Let $M$ be a type $\III_1$ factor. Then
\[
\rB(M,\varphi)=\C 1
\]
for every faithful normal state $\varphi$ on $M$.
\end{theorem}

\begin{proof}
By \eqref{connes self bicentralizing reduction}, it suffices to rule out a
type $\III_1$ factor $M$ with $\rB(M,\varphi)=M$.
For such a factor, the modular and bicentralizer flows commute and
are weakly mixing.
Their joint action is ergodic, so
\eqref{eq:ergodic-spectrum-kernel} makes
$\Sigma=\Sp(\sigma^\varphi\times\beta^\varphi)$ a closed subgroup of $\R^2$.
Proposition~\ref{connes joint resonance} gives $pt\in\tfrac12\Z$ for
$(p,t)\in\Sigma$. Hence $2\Sigma$ is a closed subgroup contained
in $\mathcal R_\R$. Proposition~\ref{resonant group for R} shows
that $2\Sigma$, and therefore $\Sigma$, is discrete or contained
in a coordinate axis.

If $\Sigma$ is discrete, its joint spectral measure is supported
on a countable set. The joint unitary representation is therefore
a direct sum of characters, so the modular flow is almost periodic.
Weak mixing forces $\rL^2(M,\varphi)=\C\varphi^{1/2}$.
If $\Sigma$ is contained in a coordinate axis, one of the two
flows is trivial. Ergodicity of that flow again gives $M=\C 1$.
Both alternatives contradict the type $\III_1$ assumption.
\end{proof}

\end{document}